\documentclass[11pt, twoside,leqno]{article}

\usepackage{amsmath}
\usepackage{amssymb}
\usepackage{titletoc}
\usepackage{mathrsfs}
\usepackage{amsthm}

\usepackage{indentfirst}
\usepackage{color}
\usepackage{txfonts}

\allowdisplaybreaks

\def\bint{{\ifinner\rlap{\bf\kern.30em--}
\int\else\rlap{\bf\kern.35em--}\int\fi}\ignorespaces}

\def\sbint{{\ifinner\rlap{\bf\kern.32em--}
\hspace{0.078cm}\int\else\rlap{\bf\kern.45em--}\int\fi}\ignorespaces}

\def\vpz{\vphantom}

\def\rr{\mathbb{R}}
\def\rn{\mathbb{R}^n}

\def\cc{\mathbb{C}}
\def\nn{\mathbb{N}}
\def\zz{\mathbb{Z}}

\def\dz{\delta}

\def\ez{\epsilon}

\def\bz{\beta}
\def\gz{{\gamma}}

\def\vz{\varphi}
\def\kz{{\kappa}}
\def\tz{\theta}

\def\wz{\widetilde}

\def\ls{\lesssim}

\def\fz{\infty}
\def\az{\alpha}

\def\ca{{\mathcal A}}
\def\cb{{\mathcal B}}

\def\cg{{\mathcal G}}
\def\ccg{{\mathring{\mathcal G}}}

\def\cm{{\mathcal M}}
\def\cn{{\mathcal N}}

\def\cx{{\mathcal X}}
\def\cy{{\mathcal Y}}

\def\ta{\widetilde\alpha}
\def\tit{\widetilde\tau}

\def\lp{{L^p(\mathcal{X})}}

\def\icgg{\mathcal{G}_0^\eta(\beta,\gamma)}
\def\cggi{\mathring{\mathcal{G}}_0^\eta(\beta,\gamma)}

\def\cggii{\mathring{\mathcal{G}}_0^\eta(\beta_2,\gamma_2)}
\def\cggt{\mathring{\mathcal{G}}_0^\eta(\widetilde{\beta},\widetilde{\gamma})}

\def\icggt{\mathcal{G}_0^\eta(\widetilde{\beta},\widetilde{\gamma})}
\def\hb{\dot{B}^s_{p,q}(\mathcal{X})}
\def\hf{\dot{F}^s_{p,q}(\mathcal{X})}
\def\ihb{B^s_{p,q}(\mathcal{X})}
\def\ihf{F^s_{p,q}(\mathcal{X})}
\def\hfi{\dot{F}^s_{\infty,q}(\mathcal{X})}
\def\hbi{\dot{B}^s_{\infty,q}(\mathcal{X})}
\def\ihfi{F^s_{\infty,q}(\mathcal{X})}
\def\ihbi{B^s_{\infty,q}(\mathcal{X})}

\def\r{\right}
\def\lf{\left}

\def\noz{{\nonumber}}

\def\r{\right}
\def\lf{\left}

\def\at{{\mathop\mathrm{\,at\,}}}
\def\supp{{\mathop\mathrm{\,supp\,}}}

\def\diam{{\mathop\mathrm{\,diam\,}}}

\def\loc{{\mathop\mathrm{\,loc\,}}}

\def\BMO{{\mathop\mathrm{\,BMO\,}}}
\def\bmo{{\mathop\mathrm{\,bmo\,}}}
\DeclareMathOperator*{\esssup}{ess\ sup}

\def\eqref#1{(\ref{#1})}

\def\func#1{\mathop{\mathrm{#1}}}
\def\diam{\func{diam}}
\def\supp{\func{supp}}

\def\ya{y_\alpha^{k,m}}
\def\qa{Q_\alpha^{k,m}}
\def\qo{Q_\alpha^{0,m}}

\def\qop{Q_{\alpha'}^{0,m'}}

\def\yap{y_{\alpha'}^{k',m'}}
\def\qap{Q_{\alpha'}^{k',m'}}

\def\red{\color{red}}

\newtheorem{theorem}{Theorem}[section]
\newtheorem{lemma}[theorem]{Lemma}

\newtheorem{proposition}[theorem]{Proposition}

\theoremstyle{definition}
\newtheorem{remark}[theorem]{Remark}
\newtheorem{definition}[theorem]{Definition}

\numberwithin{equation}{section}

\begin{document}

\title{\bf\Large Besov and Triebel--Lizorkin Spaces on Spaces of Homogeneous Type
with Applications to Boundedness of Calder\'on--Zygmund Operators
\footnotetext{\hspace{-0.35cm} 2010 {\it Mathematics Subject Classification}. Primary 46E35;
Secondary 42B25, 42B20, 42B35, 30L99.\endgraf
{\it Key words and phrases.} space of homogeneous type, Calder\'on reproducing formula,
Besov space, Triebel--Lizorkin space,
Calder\'on--Zygmund operator.\endgraf
This project is supported by the National Natural Science Foundation of China (Grant Nos.\
11971058, 11761131002 and 11671185).}}
\date{}
\author{Fan Wang, Yongsheng Han, Ziyi He and Dachun Yang\footnote{Corresponding author,
E-mail: \texttt{dcyang@bnu.edu.cn}/{\red March 25, 2020}/Final version.}}
\maketitle

\vspace{-0.8cm}

\begin{center}
\begin{minipage}{13cm}
{\small {\bf Abstract}\quad In this article, the authors introduce Besov and Triebel--Lizorkin
spaces on spaces of homogeneous type in the sense of Coifman and Weiss,  prove that these
(in)homogeneous Besov and Triebel--Lizorkin spaces are
independent of the choices of both exp-ATIs  (or exp-IATIs) and
underlying spaces of distributions,
and give some basic properties of these spaces.
As applications, the authors show that some known function spaces coincide with certain
special cases of Besov and Triebel--Lizorkin spaces and, moreover,
obtain the boundedness of Calder\'on--Zygmund operators on these Besov  and
Triebel--Lizorkin spaces. All these results strongly depend on the geometrical properties,
reflected via its dyadic cubes, of the considered space of homogeneous type.
Comparing with the known theory of these spaces on metric measure spaces,
a major novelty of this article is that all results presented in this article get rid of the
dependence on the reverse doubling assumption
of the considered measure of the underlying space and hence give a final real-variable
theory of these function spaces on spaces of homogeneous type.
}
\end{minipage}
\end{center}

\vspace{0.2cm}

\tableofcontents

\vspace{0.2cm}

\section{Introduction}\label{intro}

Between 1950s and 1980s, the scales of Besov spaces $B_{p,q}^s$ and Triebel--Lizorkin
spaces $F_{p,q}^s$ on
$\rn$ were introduced by several mathematicians. In 1951, Nikol'ski\u{\i} \cite{n51} introduced the
Nikol'ski\u{\i}--Besov spaces, which are nowadays denoted by $B^s_{p,\infty}(\rn)$, while he mentioned that
this was based on earlier works of Bern\v{s}te\v{\i}n \cite{b47} and Zygmund \cite{z45}. By introducing the
third index $q$, Besov \cite{b59,b61} complemented this scale.
We also refer the reader to Taibleson \cite{t64,t65,t66}
for the early investigations of Besov spaces. Lizorkin \cite{l72,l74} and Triebel \cite{t73}  independently started to
investigate the scale $F^s_{p,q}(\rn)$ around 1970. Further, we have to mention the contributions of Peetre
\cite{p73,p75,p76}, who extended the range of the admissible parameters $p$ and $q$ to values less than one.
We refer the reader to monographs \cite{bin,t83,t92,t06,w88} for a comprehensive treatment of these function
spaces and their history.

It is important that Besov and Triebel--Lizorkin spaces  provide a unified frame
to study many function spaces,
because they cover many well-known classical concrete function spaces, for
instance, Lebesgue spaces, Sobolev spaces, potential spaces, (local) Hardy spaces,
the space of functions of
bounded mean oscillations. We refer the reader to \cite{ysy} for the
relationship between Morrey spaces, Campanato spaces
and Besov--Triebel--Lizorkin spaces and to \cite{Sa18} for some new progress of Besov
and Triebel--Lizorkin spaces.

Moreover, as a generalization of $\rn$, the space of homogeneous type was introduced by Coifman
and Weiss \cite{cw71,cw77}, which provides a natural setting for the study of function spaces and the boundedness
of Calder\'on--Zygmund operators. To recall the notion of spaces of homogeneous type,
we need the following notion of quasi-metric spaces.

\begin{definition}\label{10.21.1}
A \emph{quasi-metric space} $(\cx, d)$ is a non-empty set $\cx$ equipped with a \emph{quasi-metric} $d$,
namely, a non-negative function defined on $\cx\times\cx$ satisfying that, for any $x,\ y,\ z \in \cx$,
\begin{enumerate}
\item[{\rm(i)}] $d(x,y)=0$ if and only if $x=y$;
\item[{\rm(ii)}] $d(x,y)=d(y,x)$;
\item[{\rm(iii)}] there exists a constant $A_0 \in [1, \infty)$, independent of $x$, $y$
and $z$, such that
$$d(x,z)\leq A_0[d(x,y)+d(y,z)].$$
\end{enumerate}
\end{definition}

The \emph{ball} $B$ of $\cx$, centered at $x_0 \in \cx$ with radius $r\in(0, \infty)$, is defined by setting
$$
B:=\{x\in\cx: d(x,x_0)<r\}=:B(x_0,r).
$$
For any ball $B$ and $\tau\in(0,\infty)$, we denote by $\tau B$ the ball with the same center as that of $B$ but
of radius $\tau$ times that of $B$.
\begin{definition}\label{10.21.2}
Let $(\cx,d)$ be a quasi-metric space and
$\mu$ a measure on $\cx$.
The triple $(\cx,d,\mu)$ is called
a \emph{space of homogeneous type} if $\mu$ satisfies the following doubling condition: there exists a
positive constant $C_{(\mu)}\in[1,\infty)$ such that, for any ball $B \subset \cx$,
$$
\mu(2B)\leq C_{(\mu)}\mu(B).
$$
\end{definition}

The above doubling condition is equivalent to that, for any ball $B$ and any $\lambda\in[1,\infty)$,
\begin{equation}\label{eq-doub}
\mu(\lambda B)\leq C_{(\mu)}\lambda^\omega\mu(B),
\end{equation}
where $\omega:= \log_2 C_{(\mu)}$ is called the \emph{upper dimension} of $\cx$.
If $A_0 = 1$, then
$(\cx,d,\mu)$ is called a \emph{metric measure space of homogeneous type} or,
simply, a \emph{doubling metric measure space}.

Both spaces of homogeneous type, with some additional assumptions,
and function spaces on them have been extensively investigated  in many
articles. For instance, the \emph{Ahlfors $d$-regular space} is
a special  space of homogeneous type
satisfying the following condition: there exists a positive constant $C$ such that, for any ball
$B(x,r)\subset\cx$ with center $x$ and radius $r\in(0,\diam \cx)$,
$$
C^{-1}r^d\leq \mu(B(x,r))\leq Cr^d,
$$
here and thereafter, $\diam \cx:=\sup_{x,\ y\in \cx} d(x,y)$. Another example is the
RD-space (see
\cite{j94,hmy06,hmy08} for instance), which is a doubling metric measure space with the
following \emph{reverse doubling condition}: there exist positive constants $\widetilde{C}_{(\mu)}\in(0,1]$
and $\kappa\in(0,\omega]$ such that, for any ball $B(x, r)$ with $r\in(0, \diam \cx/2)$ and
$\lambda\in[1,\diam \cx/(2r))$,
\begin{equation}\label{eq-rdoub}
\widetilde{C}_{(\mu)}\lambda^\kappa\mu(B(x, r))\leq\mu(B(x, \lambda r)).
\end{equation}
Obviously, an RD-space is a generalization of an  Ahlfors $d$-regular space.
We refer
the reader to \cite{yz11} for more equivalent characterizations of RD-spaces.
Function spaces on these underlying spaces and their applications have been extensively studied  in recent
years.  Mac\'ias and Segovia \cite{ms79,ms79b} investigated Hardy spaces and Lipschitz functions.
Then, Zhou et al. \cite{zsy16} introduced the Hardy spaces with variable exponents
on RD-spaces.
Moreover, Nakai \cite{n08} introduced a new kind of Hardy spaces via atoms on a space of homogeneous
type and proved that they coincide with those in \cite{cw77} under some certain circumstances
(see \cite[Remark 2.2]{n08} for details). In \cite{n17}, Nakai obtained the boundedness of fractional
integral operators on these Hardy spaces on  a regular space.
We refer the reader, for instance,  to \cite{hmy06,hmy08,gly08,gly091,yz10}
for more characterizations of Hardy spaces on RD-spaces, and to \cite{yz08, hyz09} for
more applications of these Hardy spaces.

Besov and Triebel--Lizorkin spaces on spaces of homogeneous type satisfying some additional
assumptions were also studied.

Let us first recall some developments of Besov and Triebel--Lizorkin spaces on Ahlfors $d$-regular spaces.
In 1994, Han and Sawyer \cite{hs} introduced homogeneous Besov and Triebel--Lizorkin spaces on
Ahlfors $d$-regular spaces. Later, in \cite{hly99, hly99b}, Han et al.  introduced their
inhomogeneous counterparts via the continuous inhomogeneous Calder\'on
reproducing formulae obtained in \cite{h97}. Using the discrete inhomogeneous Calder\'on
reproducing formulae obtained in \cite{hly01}, Han and Yang \cite{hy02} and \cite{hy03}
obtained several characterizations of inhomogeneous Besov and Triebel--Lizorkin spaces
on Ahlfors $d$-regular spaces. Moreover, Yang \cite{y041,y051,y052} established
various characterizations of Triebel--Lizorkin spaces on Ahlfors $d$-regular spaces,
such as, local principle and frame characterizations.
We also refer the reader to \cite{y031} for frame characterizations of Besov spaces
on Ahlfors $d$-regular spaces.
In \cite{dhy04}, Deng et al. established inhomogeneous Plancherel--P\^olya inequalities
and characterized Besov and Triebel--Lizorkin spaces on Ahlfors $d$-regular spaces
via Littlewood--Paley functions. As applications, Yang \cite{y033} obtained the $T1$ theorem
and \cite{y034} the boundedness of Riesz potentials on
Besov and Triebel--Lizorkin spaces on Ahlfors $d$-regular spaces.
Besides,  Yang obtained embedding theorem in \cite{y035} and real interpolations
of these spaces in \cite{y042}.
We refer the reader to \cite{y031,y032,dh09} for more applications of Besov and
Triebel--Lizorkin spaces on Ahlfors $d$-regular spaces and to \cite{am15}
for a sharp theory of Hardy spaces on Ahlfors $d$-regular spaces.

On another hand, Han et al.\ \cite{hmy08} and M\"uller and Yang \cite{my09}
introduced and studied both the homogeneous and the inhomogeneous
Besov and Triebel--Lizorkin spaces on RD-spaces.
Moreover, Yang and Zhou \cite{yz11} established a new
characterization of these Besov and Triebel--Lizorkin spaces.
Besides, Koskela et al. \cite{kyz10,kyz11} introduced the Haj\l asz--Besov and Triebel--Lizorkin spaces on RD-spaces.
Later, Grafakos et al. \cite{glmy14} developed a systematic theory of multilinear analysis of
Besov and Triebel--Lizorkin spaces on RD-spaces.

Observe that the most important tool used in all these studies on the real-variable
theory of function spaces on spaces of homogeneous type satisfying some additional
assumptions is the \emph{Calder\'on reproducing formulae}.

Now we return to an arbitrary space $(\cx,d,\mu)$ of homogeneous type in the sense of Coifman and
Weiss. If we make an abuse of the notion of the
reverse doubling condition, such an $\cx$ can be \emph{formally}
regarded to satisfy \eqref{eq-rdoub} with $\kz:=0$. But, this
leads some essential difficulties in establishing Calder\'{o}n reproducing formulae.
For instance, Auscher and Hyt\"{o}nen
\cite{ah13,ah13-2} pointed out that the inequality \cite[(3.48)]{hmy08}
may not hold true in this case, which is a key inequality
when establishing Calder\'{o}n reproducing formulae. To overcome these difficulties,
we found in \cite{hlyy} that one possible way  is to use the  following system of dyadic cubes of
$(\cx,d,\mu)$  established by Hyt\"onen and Kariema in \cite[Theorem 2.2]{hk},
which reveals the geometrical properties of the considered space of homogeneous type.

\begin{theorem}\label{10.22.1}
Suppose that constants $0 < c_0\leq C_0 <\infty$ and $\delta\in(0, 1)$ such that $12A_0^3C_0\delta\leq c_0$.
Assume that a set of points, $\{ z_\alpha^k : k \in\zz , \alpha \in \ca_k \} \subset \cx$ with $\ca_k$, for
any $k \in\zz$, being a set of indices, has the following properties: for any $k \in\zz$,
$$d(z^k_\alpha, z^k_\beta)\geq c_0\delta^k\quad\text{if}\quad\alpha\neq\beta,\quad
\text{and}\quad\min_{\alpha\in\ca_k}d(x,z_\alpha^k)<C_0\delta^k\quad\text{for any}\quad x\in\cx.$$
Then there exists a family of sets, $\{ Q_\alpha^k : k \in\zz , \alpha \in \ca_k \}$, satisfying
\begin{enumerate}
\item[{\rm(i)}] for any $k\in\zz$, $\bigcup_{\alpha\in\ca_k}
Q_\alpha^k=\cx$ and $\{ Q_\alpha^k : \alpha \in \ca_k \}$ is disjoint;
\item[{\rm(ii)}] if $l,\ k\in\zz$ and $k\leq l$, then, for any $\alpha\in\ca_k$ and $\beta\in\ca_l$,
either $Q_\beta^l\subset Q_\alpha^k$ or $Q_\beta^l\cap Q_\alpha^k=\emptyset$;
\item[{\rm(iii)}] for any $k\in\zz$ and $\alpha\in\ca_k$,
$B(z^k_\alpha, (3A_0^2)^{-1}c_0\delta^k)\subset Q_\alpha^k\subset B(z^k_\alpha, 2A_0C_0\delta^k)$.
\end{enumerate}
\end{theorem}

Throughout this article, for any $k\in\zz$, define
$$
\cg_k:=\ca_{k+1}\setminus\ca_k\quad\text{and}\quad\cy^k:=\left\{z^{k+1}_\alpha\right\}_{\alpha\in\cg_k}
=: \left\{y_\alpha^k\right\}_{\alpha\in\cg_k}
$$
and, for any $x \in \cx$, define
$$
d(x,\cy^k):=\inf_{y\in\cy^k}d(x,y)\quad\text{and}\quad V_{\delta^k}(x):=\mu(B(x,\delta^k)).
$$
Based on \cite[Theorem 2.2]{hk} (Theorem \ref{10.22.1} is only a part of it),
Auscher and Hyt\"{o}nen \cite{ah13} established a wavelet system
on $(\cx,d,\mu)$. Moreover, they also obtained
a very useful inequality (see \cite[Lemma 8.3]{ah13}), which
can be seen as a substitution of \cite[(3.48)]{hmy08} if the reverse doubling condition is dropped.
Motivated by this, He et al.\ \cite[Definition 2.7]{hlyy} introduced a new kind of approximations of the
identity with exponential decay (see Definitions \ref{10.23.2} and \ref{iati}
below for details).
Besides, He et al.\ \cite{hlyy} established (in)homogeneous continuous/discrete Calder\'on reproducing
formulae (see, for instance, Lemmas \ref{crf} and \ref{icrf} below) on a space of homogeneous type by using
these approximations of the identity.

Motivated by the wavelet system in \cite{ah13} and the new Calder\'{o}n reproducing formulae established
in \cite{hlyy}, a real-variable theory of function spaces on a space of homogeneous type has been developed
rapidly. As a first attempt, Han et al.\ \cite{hlw} established the wavelet reproducing formulae by using the
wavelets in \cite{ah13}. Using these formulae, Han et al.\ \cite{hhl16} characterized the atomic Hardy
spaces via discrete Littlewood--Paley square functions. Moreover,
some criteria of the boundedness of Calder\'on--Zygmund
operators on Hardy spaces and their dual spaces were also obtained in \cite{hhl16}. Later, Han et al.\
\cite{hhl18} introduced a new kind of Hardy spaces by using another kind of distribution spaces.

Right after the Calder\'{o}n reproducing formulae were established in \cite{hlyy}, He et al.\ \cite{hhllyy}
obtained a complete real-variable theory of atomic Hardy spaces on a
space $(\cx,d,\mu)$ of homogeneous type
with $\mu(\cx)=\fz$, which is equivalent to $\diam\cx=\fz$
(see, for instance, Nakai and Yabuta \cite[Lemma 5.1]{ny97}
or Auscher and Hyt\"{o}nen \cite[Lemma 8.1]{ah13}). Recently, He et al.\ \cite{hyy19}
established a real-variable theory of local Hardy spaces on
$(\cx,d,\mu)$ without the assumption $\mu(\cx)=\fz$. We point out that,
in both \cite{hhllyy} and \cite{hyy19}, He et al.\ gave a complete answer to an open question asked by Coifman
and Weiss \cite[p.\ 642]{cw77} on the radial maximal function characterization
of Hardy spaces over spaces of homogeneous type (see also \cite[p.\ 5]{bdl20}). Later,
Fu et al.\ \cite{fmt19} obtained a real-variable theory of Musielak--Orlicz Hardy spaces on $\cx$.
Besides, Zhou et al. \cite{zyh20} established a real-variable theory of
Hardy--Lorentz spaces on spaces of homogeneous type.
On another hand, Duong and Yan \cite{dy03} investigated Hardy spaces defined
by means of the area integral function associated with the Poisson semigroup.
Later, Song and Yan \cite{sy18} obtained the maximal function
characterizations of Hardy spaces associated with operators. Moreover, Bui et al.\ \cite{bdl,bdl20,bdk,bd20}
obtained the maximal function characterizations of a new local-type Hardy spaces associated with operators.
Besides, S. Yang and D. Yang \cite{yy19} established atomic and maximal function
characterizations of Musielak--Orlicz--Hardy spaces associated to non-negative
self-adjoint operators on spaces of homogeneous type.

It should be mentioned that Hardy spaces on  spaces of homogeneous type
have some applications. Ky \cite{ky15}
established the linear decompositions of elements in $H^1(\cx)\times\BMO(\cx)$ on an RD-space $\cx$.
Later, Fu et al.\ \cite{fyl17} improved the result in \cite{ky15} to a bilinear decomposition
of elements in $H^1_\at(\cx)\times\BMO(\cx)$ on a space $\cx$ of homogeneous type via first establishing
some wavelet characterizations of $H^1_\at(\cx)$ (see \cite{fy18}) in terms of regular
wavelets from \cite{ah13}, where $H^1_\at(\cx)$ denotes the atomic Hardy space introduced by
Coifman and Weiss in \cite{cw77}. Moreover, Liu et al.\
\cite{lyy18} established the bilinear decomposition for products of Hardy spaces and their dual spaces
on a space of homogeneous type. Using this decomposition, Liu et al.\ \cite{lcfy17,lcfy18} obtained the
endpoint boundedness of commutators on a space of homogeneous type. Recently, Bui et al.\ \cite{bbd18}
introduced weighted Besov and Triebel--Lizorkin spaces associated with operators on a space of homogeneous
type and showed that, if the operator has some good properties,
these function spaces coincide with classical
counterparts on a space of homogeneous type defined via atoms.

In \cite{hmy08}, (in)homogeneous continuous/discrete Calder\'on reproducing formulae were established and
further used to build a complete real-variable theory of (in)homogeneous Besov and Triebel--Lizorkin
spaces on RD-spaces. Observe that Calder\'on reproducing formulae
play a very important and essential role in developing a real-variable
theory of Besov and Triebel--Lizorkin spaces on RD-spaces.
It is a natural question whether or not
we can develop a complete real-variable theory of Besov and Triebel--Lizorkin spaces on
spaces of homogeneous type by using these new obtained
Calder\'on reproducing formulae in \cite{hlyy}, and hence further
generalize and complete the real-variable theory of function spaces on
spaces of homogeneous type developed in \cite{hhllyy,hyy19}.
The main target of this article is to give an affirmative answer to this
question. To be precise, in this article, we introduce Besov and Triebel--Lizorkin
spaces on spaces of homogeneous type in the sense of Coifman and Weiss,
prove that these (in)homogeneous Besov and Triebel--Lizorkin spaces are
independent of the choices of both exp-ATIs  (or exp-IATIs) and
underlying spaces of distributions, and give some basic properties of these spaces.
As applications, we show that some known function spaces coincide with certain
special cases of Besov and Triebel--Lizorkin spaces and, moreover,
obtain the boundedness of Calder\'on--Zygmund operators on these Besov and
Triebel--Lizorkin spaces.

To limit the length of this article, in \cite{hwy20}, He et al.
established wavelet characterizations of these Besov
and Triebel--Lizorkin spaces on spaces of homogeneous type and
then, via first obtaining the boundedness of almost diagonal operators on sequence Besov
and Triebel--Lizorkin spaces on spaces of homogeneous type, He et al. further
established the molecular and various Littlewood--Paley function characterizations of Besov
and Triebel--Lizorkin spaces on spaces of homogeneous type. Besides, Han et al.
\cite{hhhlp20} showed that the embedding theorems for Besov
and Triebel--Lizorkin spaces on spaces of homogeneous type hold true if and only if
the considered measure $\mu$ of the underlying space has a (local) lower bound.

All these results in this article, \cite{hwy20} and \cite{hhhlp20}
strongly depend on the geometrical properties,
reflected via its dyadic cubes, of the considered space of homogeneous type.
Comparing with the known theory of Besov and Triebel--Lizorkin spaces on metric measure spaces,
a major novelty of all these three articles is that all obtained results get rid of the
dependence on the reverse doubling assumption of the considered measure of the underlying space.
Thus, these results give a \emph{final} real-variable theory of Besov and Triebel--Lizorkin spaces
on spaces of homogeneous type.

The organization of this article is as follows.

In Section \ref{Scrf}, we recall some known facts on exp-(I)ATIs, spaces of both test functions
and distributions, and the Calder\'on reproducing formulae from \cite{hlyy},
which are the main tools of this article. Observe that these exp-(I)ATIs and
Calder\'on reproducing formulae subtly use the geometrical properties of
the considered space of homogeneous type (see Remark \ref{geo} for more details).

In Section \ref{s2}, we introduce homogeneous Besov spaces $\hb$ and Triebel--Lizorkin spaces $\hf$ with
$p<\fz$ on a space of homogeneous type with $\mu(\cx)=\infty$.
Their inhomogeneous counterparts are introduced in Section \ref{s3}.
In Section \ref{s4}, we introduce the Triebel--Lizorkin spaces with $p=\infty$. In these sections, by
establishing the so-called (in)homogeneous Plancherel--P\^olya inequality,
we show that (in)homogeneous Besov and Triebel--Lizorkin spaces are independent of the choices
of exp-ATIs and the considered spaces of distributions.
We point out that, even comparing with the results on RD-spaces in \cite{hmy08},
the ranges of some indices appearing in some results of this article are optimal (see Remarks \ref{addre2},
\ref{addre1} and \ref{addre3} below). It is also worth mentioning that, to obtain
the inhomogeneous Plancherel--P\^olya inequality (see Lemma \ref{6.9.1} below), we have to use
the Banach--Steinhaus theorem from functional analysis to deal with the terms related to cube averages,
which is caused by the speciality of the inhomogeneous spaces.
Moreover, when we establish the (in)homogeneous Plancherel--P\^olya inequalities, respectively,
associated with (in)homogeneous Triebel--Lizorkin spaces in the case $p=\fz$
(see Lemmas \ref{11.14.1} and \ref{12.3.2} below), due to the special structures
of the considered quasi-norms of (in)homogeneous Triebel--Lizorkin spaces in the case $p=\fz$,
we need to subtly classify all dyadic cubes in any fixed level based
to all dyadic cubes in any given level [see \eqref{eq-x1} below], which fully uses
the geometrical properties of the considered space of homogeneous type
reflected via its dyadic cubes.

In Section \ref{s6}, we establish the relationship between these Besov
and Triebel--Lizorkin spaces and some  classical spaces
on spaces of homogeneous type, including Lebesgue spaces, H\"older spaces and $\BMO(\cx)$.
To achieve these, we need to use some tools from probability
and functional analysis, such as the Khinchin inequality, the Cotlar--Stein lemma and
the stopping time [see, respectively, Lemma \ref{khin}, \eqref{t_en_2} and \eqref{stop_l} below],
and also need to use the geometrical structure, presented via its dyadic cubes,
of the considered underlying space [see, for instance, $E^l(x)$ in \eqref{6.8x} below].
Meanwhile, we also need to use the independence of exp-(I)ATIs of Besov
and Triebel--Lizorkin spaces, which is obtained in the previous sections
of this article.

In Section \ref{s5}, we establish the boundedness of Calder\'on--Zygmund operators on these Besov
and Triebel--Lizorkin spaces. To do this, the main difficulty we encounter is that
we can \emph{not} construct an exp-ATI with bounded support. To overcome this difficulty,
motivated by \cite{ns}, we establish  decompositions of an $\exp$-(I)ATI
(see Lemmas \ref{4.9.1} and \ref{4.25.1} below). To achieve these, we
first need to construct some ``smooth" functions having bounded supports and
some other good properties (see Lemma \ref{4.8.1} below),
which are obtained via using both the dyadic reference points of the considered
space of homogeneous type and some good functions from \cite[Corollary 4.2]{ah13}
(which is actually similar to the
Urysohn lemma). Both obviously need to use the geometrical properties of the
considered underlying space. Using these decompositions, we further reduce the action of
Calder\'on--Zygmund operators on an $\exp$-(I)ATI into a sum
of the actions of Calder\'on--Zygmund operators on aforementioned good functions,
and then separately estimate each term via using different
methods to obtain a unified overall estimate.

Finally, let us make some conventions on notation.  The \emph{Lebesgue space} $L^p(\cx)$ for any given
$p\in(0,\infty]$ is defined by setting, when $p\in(0,\infty)$,
$$
L^p(\cx):=\left\{f \ \text{is measurable on} \ \cx :\
\|f\|_{\lp}:=\left[\int_{\cx}|f(x)|^p\,d\mu(x)\right]^{1/p}<\infty\right\},
$$
and
$$
L^\infty(\cx):=\left\{f \ \text{is measurable on} \ \cx :\  \|f\|_{L^\infty(\cx)}:=
\displaystyle{\esssup_{x\in\cx}|f(x)|<\infty}\right\}.
$$
Throughout this article, we use $A_0$ to denote the positive constant appearing in the \emph{quasi-triangle
inequality} of $d$ (see Definition \ref{10.21.1}), the parameter $\omega$ means the \emph{upper dimension}
in Definition \ref{10.21.2} [see \eqref{eq-doub}], and $\eta$ is defined to be the smoothness index of the exp-ATI in
Definition \ref{10.23.2} below. Moreover, $\delta$ is a small positive number, for instance,
$\delta\leq(2A_0)^{-10}$, coming from the construction of the
dyadic cubes on $\cx$ (see Theorem \ref{10.22.1}). For any $p\in[1,\infty]$,
we use $p'$ to denote its conjugate index, that is, $1/p + 1/p' = 1$. For any $r\in \rr$, $r_+$ is defined as
$r_+:=\max\{0,r\}$. For any $a,\ b\in\rr$, $a\wedge b:=\min\{a,b\}$ and $a\vee b :=\max\{a,b\}$.
The symbol $C$ denotes a positive constant which is independent
of the main parameters involved, but may vary
from line to line. We use $C_{(\alpha,\beta,\dots)}$
to denote a positive constant depending on the indicated
parameters $\alpha,\ \beta,\ \dots$.
The symbol $A\lesssim B$ means that $A\leq CB$ for some positive constant
$C$, while $A\sim B$ means $A\lesssim B\lesssim A$.
We also use the following convention: If $f\le Cg$ and $g=h$
or $g\le h$, we then write $f\ls g\sim h$ or $f\ls g\ls h$, \emph{rather than} $f\ls g=h$ or $f\ls g\le h$.
For any $r\in(0,\infty)$ and $x,\ y\in\cx$ with $x\neq y$, define $V(x,y):=\mu(B(x,d(x,y)))$
and $V_r(x):=\mu(B(x,r))$. For any  $\beta,\ \gamma\in(0,\eta)$ and $s\in (-(\beta\wedge\gamma),\beta\wedge\gamma)$,
we let
\begin{equation}\label{pseta}
p(s,\beta\wedge\gamma):=\max\left\{\frac{\omega}{\omega+(\beta\wedge\gamma)},\frac{\omega}{\omega+(\beta\wedge\gamma)+s}\right\},
\end{equation}
where $\omega$ and $\eta$ are, respectively, as in \eqref{eq-doub} and Definition \ref{10.23.2}.
The operator $M$ always denotes the \emph{Hardy--Littlewood maximal operator}, which is defined by setting,
for any locally integral function $f$ on $\cx$ and any $x\in\cx$,
\begin{equation}\label{m}
M(f)(x):=\sup_{r\in(0,\infty)}\frac{1}{\mu(B(x,r))}\int_{B(x,r)}|f(y)|\,d\mu(y).
\end{equation}
Finally, for any set $E\subset\cx$, we use $\mathbf 1_E$ to denote its characteristic function and,
for any set $J$, we use $\#J$ to denote its
\emph{cardinality}.

\section{Calder\'on reproducing formulae}\label{Scrf}

In this section, we recall Calder\'on reproducing
formulae established in \cite{hlyy}.
To this end, we first recall the notions of test functions and distributions, whose
following versions were originally given in \cite{hmy08} (see also \cite{hmy06}).

\begin{definition}[test functions]
Let $x_1\in\cx$, $r\in(0,\infty)$, $\beta \in (0,1]$ and $\gamma \in (0,\infty)$. A function $f$ on $\cx$ is
called a \emph{test function of type $(x_1, r, \beta, \gamma)$}, denoted by $f \in \cg(x_1, r, \beta,
\gamma)$, if there exists a positive constant $C$ such that
\begin{enumerate}
\item[{\rm(i)}] for any $x\in\cx$,
\begin{equation}\label{12.5.9}
|f(x)|\leq C \frac{1}{V_r(x_1)+V(x_1,x)}\left[\frac{r}{r+d(x_1,x)}\right]^\gamma;
\end{equation}
\item[{\rm(ii)}] for any $x,\ y \in \cx$ satisfying $d(x, y)\leq(2A_0)^{-1}[r + d(x_1, x)]$,
\begin{equation}\label{12.5.10}
|f(x)-f(y)|\leq C\left[\frac{d(x,y)}{r+d(x_1,x)}\right]^\beta
\frac{1}{V_r(x_1)+V(x_1,x)}\left[\frac{r}{r+d(x_1,x)}\right]^\gamma.
\end{equation}
\end{enumerate}
For any  $f\in \cg(x_1, r, \beta, \gamma)$, the norm $\|f\|_{\cg(x_1, r, \beta, \gamma)}$ is defined by
setting
$$
\|f\|_{\cg(x_1, r, \beta, \gamma)}:=\inf\{C\in(0,\infty):\, \text{\eqref{12.5.9} and \eqref{12.5.10} hold
true}\}.
$$
The subspace $\mathring{\cg}(x_1, r, \beta, \gamma)$ is defined by setting
$$
\mathring{\cg}(x_1, r, \beta, \gamma):=\left\{f\in\cg(x_1, r, \beta, \gamma):\int_{\cx}f(x)\,d\mu(x)=0\right\}
$$
equipped with the norm $\|\cdot\|_{\mathring{\cg}(x_1, r, \beta, \gamma)}:=\|\cdot\|_{\cg(x_1, r, \beta,
\gamma)}$.
\end{definition}

Fix $x_0\in\cx$ and $r=1$. We denote $\cg(x_0, 1, \beta, \gamma)$ and
$\mathring{\cg}(x_0, 1, \beta, \gamma)$ simply, respectively, by $\cg(\beta,\gamma)$ and $\mathring{\cg}(\beta,\gamma)$.
Note that, for any fixed $x\in\cx$ and $r \in(0,\infty)$,
$\cg(x,r,\beta,\gamma) = \cg(\beta,\gamma)$ and $\mathring\cg(x,r,\beta,\gamma)=\mathring\cg(\beta,\gamma)$
with equivalent norms, but the positive equivalence  constants may depend on $x$ and $r$.

Fix $\epsilon\in (0, 1]$ and $\beta,\ \gamma \in (0, \epsilon]$. Let $\cg^\epsilon_0(\beta,\gamma)$ [resp.,
$\mathring{\cg}^\epsilon_0(\beta, \gamma)$] be the completion of the set $\cg(\epsilon, \epsilon)$ [resp.,
$\mathring{\cg}(\epsilon, \epsilon)$] in $\cg(\beta,\gamma)$ [resp., $\mathring{\cg}(\bz,\gz)$]. Furthermore,
the norm of $\cg^\epsilon_0(\beta,\gamma)$ [resp., $\mathring{\cg}^\epsilon_0(\beta, \gamma)$] is defined
by setting $\|\cdot\|_{\cg^\epsilon_0(\beta,\gamma)}:=\|\cdot\|_{\cg(\beta,\gamma)}$ [resp.,
$\|\cdot\|_{\mathring{\cg}^\epsilon_0(\beta,\gamma)}:=\|\cdot\|_{\cg(\beta,\gamma)}$]. The dual space
$(\cg^\epsilon_0(\beta,\gamma))'$ [resp., $(\mathring{\cg}^\epsilon_0(\beta,\gamma))'$] is defined to be the
set of all continuous linear functionals from $\cg^\epsilon_0(\beta,\gamma)$ [resp.,
$\mathring{\cg}^\epsilon_0(\beta,\gamma)$] to $\mathbb{C}$, equipped with the weak-$\ast$ topology. The
spaces $(\cg^\epsilon_0(\beta,\gamma))'$ and $(\mathring{\cg}^\epsilon_0(\beta,\gamma))'$ are called the
\emph{spaces of distributions}.

Now we recall the notion of approximations of the identity with exponential decay from \cite{hlyy}.

\begin{definition}\label{10.23.2}
A sequence $\{Q_k\}_{k\in\zz}$ of bounded linear integral operators on $L^2(\cx)$ is called an
\emph{approximation of the identity with exponential decay} (for short, exp-ATI) if there exist constants $C$,
$\nu\in(0, \infty)$, $a\in(0, 1]$ and $\eta\in(0,1)$ such that,
for any $k \in\zz$, the kernel of the operator $Q_k$, a
function on $\cx \times \cx$ , which is still denoted by $Q_k$,
satisfies the following conditions:
\begin{enumerate}
\item[{\rm(i)}] (the \emph{identity condition}) $\sum_{k=-\infty}^\infty Q_k=I$
in $L^2(\cx)$, where $I$ denotes the identity operator on $L^2(\cx)$;
\item[{\rm(ii)}] (the \emph{size condition}) for any $x,\ y\in\cx$,
\begin{align*}
|Q_k(x,y)|&\leq C \frac{1}{\sqrt{V_{\delta^k}(x)V_{\delta^k}(y)}}
\exp\left\{-\nu\left[\frac{d(x,y)}{\delta^k}\right]^a\right\}\\
&\qquad\times\exp\left\{-\nu\left[\frac{\max\{d(x,\cy^k),d(y,\cy^k)\}}{\delta^k}\right]^a\right\};
\end{align*}
\item[{\rm(iii)}] (the \emph{regularity condition}) for any
$x,\ x',\ y\in\cx$ with $d(x, x')\leq\delta^k$,
\begin{align*}
&|Q_k(x,y)-Q_k(x',y)|+|Q_k(y,x)-Q_k(y,x')|\\
&\quad\leq C\left[\frac{d(x,x')}{\delta^k}\right]^\eta
\frac{1}{\sqrt{V_{\delta^k}(x)V_{\delta^k}(y)}}\exp\left\{-\nu\left[\frac{d(x,y)}{\delta^k}\right]^a\right\}\\
&\qquad\quad\times\exp\left\{-\nu\left[\frac{\max\{d(x,\cy^k),d(y,\cy^k)\}}{\delta^k}\right]^a\right\};
\end{align*}
\item[{\rm(iv)}] (the \emph{second difference regularity condition}) for any
$x,\ x',\ y,\ y'\in\cx$ with $d(x, x')\leq\delta^k$ and $d(y, y')\leq\delta^k$,
\begin{align*}
&|[Q_k(x,y)-Q_k(x',y)]-[Q_k(x,y')-Q_k(x',y')]|\\
&\quad\leq C\left[\frac{d(x,x')}{\delta^k}\right]^\eta\left[\frac{d(y,y')}{\delta^k}\right]^\eta
\frac{1}{\sqrt{V_{\delta^k}(x)V_{\delta^k}(y)}}\exp\left\{-\nu\left[\frac{d(x,y)}{\delta^k}\right]^a\right\}\\
&\qquad\quad\times\exp\left\{-\nu\left[\frac{\max\{d(x,\cy^k),d(y,\cy^k)\}}{\delta^k}\right]^a\right\};
\end{align*}
\item[{\rm(v)}] (the \emph{cancellation condition}) for any $x,\ y \in \cx$,
$$\int_\cx Q_k(x,y')\,d\mu(y')=0=\int_\cx Q_k(x',y)\,d\mu(x').$$
\end{enumerate}
\end{definition}

The existence of such an exp-ATI on spaces of homogeneous type
is guaranteed by \cite[Theorem 7.1]{ah13} with $\eta$ same as in \cite[Theorem 3.1]{ah13}
which might be very small (see also \cite[Remark 2.8(i)]{hlyy}).
Moreover, if $d$ of $\cx$ is a metric,
then $\eta$ can be taken arbitrarily close to 1 (see \cite[Corollary 6.13]{ht14}).

The following lemma is the homogeneous continuous Calder\'on reproducing formula,
which was obtained in \cite[Theorem 4.18]{hlyy}.

\begin{lemma}\label{h_c_crf}
Let $\{Q_k\}_{k=-\infty}^\infty$ be an {\rm exp-ATI}
and $\beta,\ \gamma \in (0, \eta)$ with $\eta$ as in Definition \ref{10.23.2}.
Then there exists a sequence $\{\widetilde{Q}_k\}_{k=-\infty}^\infty$
of bounded linear integral operators on $L^2(\mathcal{X})$ such that,
for any $f \in (\cggi)'$,
$$f= \sum_{k=-\infty}^\infty\widetilde{Q}_kQ_kf \qquad \text{in}\quad (\cggi)',$$
and, moreover, there exists a positive constant
$C$ such that, for any $k\in\zz$, the kernel of
$\widetilde{Q}_k$, still denoted by
$\widetilde{Q}_k$, satisfies that
\begin{enumerate}
\item[{\rm(i)}] for any $x,\ y \in\cx$,
\begin{equation}\label{4.23x}
\left|\widetilde{Q}_k(x,y)\right|\leq C\frac{1}{V_{\delta^k}(x)+V(x,y)}\left[\frac{\delta^k}{\delta^k+d(x,y)}\right]^\gamma;
\end{equation}
\item[{\rm(ii)}] for any $x,\ x',\ y\in\cx$ with $d(x,x')\leq(2A_0)^{-1}[\delta^k+d(x,y)]$,
\begin{align}\label{4.23y}
&\left|\widetilde{Q}_k(x,y)-\widetilde{Q}_k(x',y)\right|\leq C\left[\frac{d(x,x')}{\delta^k+d(x,y)}\right]^\beta
\frac{1}{V_{\delta^k}(x)+V(x,y)}\left[\frac{\delta^k}{\delta^k+d(x,y)}\right]^\gamma;
\end{align}
\item[{\rm(iii)}]for any $x\in\mathcal{X}$,
\begin{equation}\label{4.23a}
\int_{\mathcal{X}}\widetilde{Q}_k(x,y)\,d\mu(y)=0=
\int_{\mathcal{X}}\widetilde{Q}_k(y,x)\,d\mu(y).
\end{equation}
\end{enumerate}
\end{lemma}

To recall the homogeneous discrete Calder\'on reproducing formulae obtained in
\cite[Theorem 5.11]{hlyy}, we need  more notions.
We point out that these homogeneous reproducing formulae need the assumption that $\mu(\cx)=\fz$, which is
also the root why the homogeneous Besov and Triebel--Lizorkin spaces need this additional assumption.
Let $j_0\in\nn$ be sufficiently large such that $\delta^{j_0}\leq (2A_0)^{-3}C_0$. Based on Theorem
\ref{10.22.1}, for any $k\in\zz$ and $\alpha\in\ca_k$, let
$$
\cn(k,\alpha):=\{\tau\in\ca_{k+j_0}:\  Q_\tau^{k+j_0}\subset Q_\alpha^k\}
$$
and $N(k,\alpha):=\#\cn(k,\alpha)$. From Theorem \ref{10.22.1}, it follows that
$N(k,\alpha) \lesssim \delta^{-j_0\omega}$ and $\bigcup_{\tau\in\cn(k,\alpha)}
Q_\tau^{k+j_0}= Q_\alpha^k$. We
rearrange the set $\{Q_\tau^{k+j_0}:\tau\in\cn(k,\alpha)\}$ as $\{\qa\}_{m=1}^{N(k,\alpha)}$. Also, denote by
$\ya$ an arbitrary point in $\qa$ and $z_\alpha^{k,m}$ the ``center" of $\qa$.

\begin{lemma}\label{crf}
Let $\{Q_k\}_{k=-\infty}^\infty$ be an {\rm exp-ATI}
and $\beta,\ \gamma \in (0, \eta)$ with $\eta$ as in Definition \ref{10.23.2}.
For any $k\in\zz$, $\alpha\in\ca_k$ and $m\in\{1,\dots,N(k,\alpha)\}$,
suppose that $\ya$ is an arbitrary point in $\qa$.
Then there exists a sequence $\{\widetilde{Q}_k\}_{k=-\infty}^\infty$ of
bounded linear integral operators on $L^2(\mathcal{X})$ such that,
for any $f \in (\cggi)'$,
$$f(\cdot) = \sum_{k=-\infty}^\infty\sum_{\alpha \in \ca_k}
\sum_{m=1}^{N(k,\alpha)}\mu\left(\qa\right)\widetilde{Q}_k(\cdot,\ya)Q_kf
\left(\ya\right) \qquad \text{in}\quad (\cggi)'.$$
Moreover, there exists a positive constant $C$,
independent of the choices of both
$\ya$, with $k\in\zz,\ \alpha\in\ca_k$ and $m\in\{1,\dots,N(k,\alpha)\}$,
and $f$, such that,
for any $k\in\zz$,
the kernel of $\widetilde{Q}_k$
satisfies \eqref{4.23x}, \eqref{4.23y} and \eqref{4.23a}.
\end{lemma}

To recall the inhomogeneous Calder\'on
reproducing formulae, we first introduce the inhomogeneous approximation
of the identity with exponential decay (see \cite[Definition 6.1]{hlyy}).

\begin{definition}\label{iati}
Let $\eta\in(0,1)$ be as in Definition \ref{10.23.2}.
A sequence $\{Q_k\}_{k=0}^\infty$ of bounded linear integral operators on
$L^2(\cx)$ is called an  \emph{inhomogeneous approximation
of the identity with exponential decay} (for short, exp-IATI) if $\{Q_k\}_{k=0}^\infty$ has the following properties:
\begin{enumerate}
\item[{\rm(i)}] $\sum_{k=0}^\infty Q_k=I$ in $L^2(\cx)$;
\item[{\rm(ii)}] for any $k\in\nn$, $Q_k$
satisfies (ii) through (v) of Definition \ref{10.23.2};
\item[{\rm(iii)}] $Q_0$ satisfies (ii), (iii) and (iv) of Definition \ref{10.23.2} with $k:=0$ but without the
term
$$\exp\left\{-\nu\left[\max\left\{d(x,\cy^0),d(y,\cy^0)\right\}\right]^a\right\};$$
moreover, for any $x\in\cx$,
$$\int_\cx Q_0(x,y)\,d\mu(y)=1=\int_\cx Q_0(y,x)\,d\mu(y).$$
\end{enumerate}
\end{definition}

The following lemma is the inhomogeneous continuous Calder\'on reproducing formula,
which was obtained in \cite[Theorem 6.13]{hlyy}.

\begin{lemma}\label{ih_c_crf}
Let $\{Q_k\}_{k\in\zz_+}$ be an {\rm exp-IATI}
and $\beta,\ \gamma \in (0, \eta)$ with $\eta$ as in Definition \ref{10.23.2}.
Then there exists a sequence $\{\widetilde{Q}_k\}_{k\in\zz_+}$
of bounded linear integral operators on $L^2(\mathcal{X})$ such that,
for any $f \in (\cggi)'$,
$$f= \sum_{k=0}^\infty\widetilde{Q}_kQ_kf \qquad \text{in}\quad (\cggi)',$$
where, for any $k\in\zz_+$, the kernel of $\widetilde{Q}_k$,
still denoted by $\widetilde{Q}_k$,  satisfies
\eqref{4.23x}, \eqref{4.23y} and the following integral condition: for any $x\in\mathcal{X}$,
$$\int_{\mathcal{X}}\widetilde{Q}_k(x,y)\,d\mu(y)=
\int_{\mathcal{X}}\widetilde{Q}_k(y,x)\,d\mu(y)=\begin{cases}
1 &\text{if } k \in \{0,\dots,N\},\\
0 &\text{if } k\in \{N+1,N+2,\ldots\}.
\end{cases}$$
\end{lemma}

The following inhomogeneous discrete Calder\'on
reproducing formulae were obtained in \cite[Theorems 6.10 and 6.13]{hlyy}.

\begin{lemma}\label{icrf}
Let $\{Q_k\}_{k\in\zz_+}$ be an {\rm exp-IATI} and $\beta,\ \gamma\in (0, \eta)$ with
$\eta$ as in Definition \ref{10.23.2}. For any $k\in\zz_+,$ $\alpha\in\ca_k$ and
$m\in\{1,\dots,N(k,\alpha)\}$, suppose that $\ya$ is an arbitrary point in $\qa$.
Then there exist an $N\in\nn$ and a sequence $\{\widetilde{Q}_k\}_{k\in\zz_+}$ of bounded
linear integral operators on $L^2(\mathcal{X})$ such
that, for any $f \in (\icgg)'$,
\begin{align*}
f(\cdot) &= \sum_{\alpha \in \ca_0}\sum_{m=1}^{N(0,\alpha)}\int_{\qo}
\widetilde{Q}_0(\cdot,y)\,d\mu(y)Q^{0,m}_{\alpha,1}(f)\\
&\qquad+\sum_{k=1}^N\sum_{\alpha \in \ca_k}\sum_{m=1}^{N(k,\alpha)}
\mu\left(\qa\right)\widetilde{Q}_k(\cdot,\ya)Q^{k,m}_{\alpha,1}(f)\\
& \qquad+\sum_{k=N+1}^\infty\sum_{\alpha \in \ca_k}\sum_{m=1}^{N(k,\alpha)}
\mu\left(\qa\right)\widetilde{Q}_k(\cdot,\ya)Q_kf\left(\ya\right)
\end{align*}
in $(\icgg)',$ where, for any $k\in\{0,\ldots,N\}$, $\alpha\in\ca_k$ and $m\in\{1,\dots,N(k,\alpha)\}$,
$$Q^{k,m}_{\alpha,1}(f):=\frac{1}{\mu(\qa)}\int_{\qa}Q_kf(u)\,d\mu(u).$$
Moreover, for any $k\in\zz_+$, the kernel of $\widetilde{Q}_k$,
still denoted by $\widetilde{Q}_k$, satisfies the same conditions as in Lemma
\ref{ih_c_crf} with the positive constant $C$ independent of $y_\az^{k,m}$, where $\alpha\in\ca_k$ and $m\in\{1,\dots,N(k,\alpha)\}$.
\end{lemma}

\begin{remark}\label{geo}
Compared with the corresponding ones on RD-spaces (see \cite{hmy08}), these exp-(I)ATIs and
Calder\'on reproducing formulae in this article have some essential differences
presented via some terms such as
$$
\exp\left\{-\nu\left[\frac{\max\{d(x,\cy^k),d(y,\cy^k)\}}{\delta^k}\right]^a\right\}.
$$
Observe that here $x$, $y\in\cx$, $\cy^k$ is a set of dyadic reference points appearing
in Theorem \ref{10.22.1} and $d(y,\cy^k)$
is the distance between $y$ and $\cy^k$. Thus, such terms closely connect with the geometry
of the considered space of homogeneous type.
\end{remark}

\section{Homogeneous Besov and Triebel--Lizorkin spaces}\label{s2}

In this section, we introduce homogeneous Besov spaces $\hb$ and Triebel--Lizorkin
spaces $\hf$ on a space $\cx$ of homogeneous type.
In the remainder of this section, we always assume that
$\mu(\cx)=\infty$, which is a quite natural assumption due to the homogeneity of $\hb$ and $\hf$.

\begin{definition}\label{h}
Let $\beta,\ \gamma \in (0, \eta)$ and $s\in(-\eta,\eta)$ with $\eta$ as in
Definition \ref{10.23.2}. Let $\{Q_k\}_{k\in\zz}$ be an exp-ATI.
\begin{enumerate}
\item[\rm{(i)}] Let $p\in(p(s,\beta\wedge\gamma),\infty]$,
with $p(s,\beta\wedge\gamma)$ as in \eqref{pseta}, and $q \in (0,\infty]$.
The \emph{homogenous Besov space} $\hb$ is defined by setting
$$
\hb := \left\{f  \in\lf(\cggi\r)' :\  \|f\|_{\hb}:=\left[\sum_{k \in \zz}
\delta^{-ksq}\|Q_k(f)\|_{\lp}^q\right]^{1/q}<\infty\right\}
$$
with usual modifications made when $p=\infty$ or $q=\infty$.
\item[\rm{(ii)}] Let $p\in(p(s,\beta\wedge\gamma),\infty)$ and $q \in (p(s,\beta\wedge\gamma),\infty]$.
The \emph{homogenous Triebel--Lizorkin space} $\hf$ is defined by setting
$$
\hf := \left\{f  \in \lf(\cggi\r)' :\  \|f\|_{\hf}:=\left\|\left[\sum_{k \in \zz}
\delta^{-ksq}|Q_k(f)|^q\right]^{1/q}\right\|_{\lp}<\infty\right\}
$$
with usual modification made when $q=\infty$.
\end{enumerate}
\end{definition}

We now show that the spaces  $\hb$ and $\hf$ in Definition \ref{h} are independent of the choice of exp-ATIs.
To this end, we need some technical lemmas. Let us begin with
recalling the following very useful inequality.

\begin{lemma}
For any $\theta\in (0,1]$ and $\{a_j\}_{j\in\nn} \subset \mathbb{C}$,
it holds true that
\begin{equation}\label{r}
\left(\sum_{j=1}^\infty|a_j|\right)^\theta\leq\sum_{j=1}^\infty|a_j|^\theta.
\end{equation}
\end{lemma}

We also need to use the following properties of exp-ATIs.
One can find more details in \cite[Remarks 2.8 and 2.9, and
Proposition 2.10]{hlyy}.

\begin{lemma}
Let $\{Q_k\}_{k\in\zz}$ be an {\rm exp-ATI} and $\eta\in(0,1)$ as in Definition \ref{10.23.2}.
Then, for any given $\Gamma \in (0,\infty)$, there exists a positive constant $C$ such that, for any $k\in\zz$, the
kernel $Q_k$ has the following properties:
\begin{enumerate}
\item[{\rm(i)}] for any $x,\ y \in\cx$,
\begin{equation}\label{10.23.3}
|Q_k(x,y)|\leq C\frac{1}{V_{\delta^k}(x)+V(x,y)}\left[\frac{\delta^k}{\delta^k+d(x,y)}\right]^\Gamma;
\end{equation}
\item[{\rm(ii)}] for any $x,\ x',\ y\in\cx$ with $d(x,x')\leq(2A_0)^{-1}[\delta^k+d(x,y)]$,
\begin{align}\label{10.23.4}
&|Q_k(x,y)-Q_k(x',y)|+|Q_k(y,x)-Q_k(y,x')|\\
&\quad\leq C\left[\frac{d(x,x')}{\delta^k+d(x,y)}\right]^\eta\frac{1}{V_{\delta^k}(x)+V(x,y)}
\left[\frac{\delta^k}{\delta^k+d(x,y)}\right]^\Gamma;\notag
\end{align}
\item[{\rm(iii)}] for any $x,\ x',\ y,\ y' \in \cx$ with $d(x,x')\leq(2A_0)^{-2}[\delta^k +d(x, y)]$ and
$d(y, y')\leq(2A_0)^{-2}[\delta^k +d(x, y)]$,
\begin{align*}
&|[Q_k(x,y)-Q_k(x',y)]-[Q_k(x,y')-Q_k(x',y')]|\\
&\quad\leq C\left[\frac{d(x,x')}{\delta^k+d(x,y)}\right]^\eta\left[\frac{d(y,y')}{\delta^k+d(x,y)}\right]^\eta
\frac{1}{V_{\delta^k}(x)+V(x,y)}\left[\frac{\delta^k}{\delta^k+d(x,y)}\right]^\Gamma.
\end{align*}
\end{enumerate}
\end{lemma}

The following lemma contains  some basic and very useful estimates
related to $d$ and $\mu$ on a space $(\cx,d,\mu)$ of homogeneous type.
One can find the details in \cite[Lemma 2.1]{hmy08} or \cite[Lemma 2.4]{hlyy}.

\begin{lemma}\label{6.15.1}
Let $\beta,\ \gamma\in(0,\infty)$.
\begin{enumerate}
\item[{\rm(i)}] For any $x,\ y\in \cx$ and $r \in (0, \infty)$,
$V(x, y)\sim V (y, x)$ and
$$V_r(x) + V_r(y) + V (x,y) \sim V_r(x) + V (x,y) \sim V_r(y) + V (x,y) \sim \mu(B(x,r + d(x,y)))$$
and, moreover, if $d(x,y)\leq r$, then $V_r(x)\sim V_r(y)$.
Here the positive equivalence  constants are independent of $x$, $y$ and $r$.
\item[{\rm(ii)}] There exists a positive constant $C$ such that,
for any $x_1 \in \cx$ and $r \in (0, \fz)$,
$$
\int_\cx \frac{1}{V_r(x_1)+V(x_1,y)}\left[\frac{r}{r+d(x_1,y)}\right]^\gamma\,d\mu(y)\leq C.
$$
\item[{\rm(iii)}] There exists a positive constant $C$ such that, for any $x \in \cx $ and $R \in (0, \fz)$,
$$\int_{\{z\in \cx:\ d(x,z)\leq R\}}\frac{1}{V(x,y)}
\left[\frac{d(x,y)}{R}\right]^\beta\,d\mu(y)\leq C$$
and
$$\quad \int_{\{z\in \cx:\ d(x,z)\geq R\}}\frac{1}{V(x,y)}
\left[\frac{R}{d(x,y)}\right]^\beta\,d\mu(y)\leq C.$$
\item[{\rm(iv)}] There exists a positive constant $C$ such that, for any $x_1\in\cx$ and $r,\ R\in (0,\fz)$,
$$\int_{\{x\in\cx:\ d(x_1,x)\geq R\}}\frac{1}{V_r(x_1)+V(x_1,x)}
\left[\frac{r}{r+d(x_1,x)}\right]^\gamma\,d\mu(x)\leq
C\left(\frac{r}{r+R}\right)^\gamma.$$
\end{enumerate}
\end{lemma}

The following two technical lemmas play a very important role, respectively, in
dealing with Besov spaces and Triebel--Lizorkin spaces on spaces of homogeneous type.

\begin{lemma}\label{9.14.1}
Let $\gamma \in (0,\infty)$ and $p \in (\omega/(\omega+\gamma),1]$ with $\omega$ as in
\eqref{eq-doub}. Then there exists a positive constant $C$ such that,
for any $k,\ k' \in \zz$, $x \in \mathcal{X}$ and  $\ya \in \qa$ with
$\alpha \in \ca_k$ and $m\in\{1,\dots,N(k,\alpha)\}$,
\begin{align}\label{5.2}
C^{-1}[V_{\delta^{k\wedge k'}}(x)]^{1-p}&\le\sum_{\alpha \in \ca_k}\sum_{m=1}^{N(k,\alpha)}\mu\left(\qa\right)
\left[\frac{1}{V_{\delta^{k\wedge k'}}(x)+V(x,\ya)}\right]^p
\left[\frac{\delta^{k\wedge k'}}{\delta^{k\wedge k'}
+d(x,\ya)}\right]^{\gamma p}\\
& \leq C[V_{\delta^{k\wedge k'}}(x)]^{1-p}.\noz
\end{align}
\end{lemma}

\begin{proof}
Notice that, for any $k,\ k'\in\zz$, $\az\in\ca_k$, $m\in\{1,\ldots,N(k,\az)\}$ and $z\in Q_\az^{k,m}$, we conclude
that $d(z,y_\az^{k,m})\ls\dz^k\ls\dz^{k\wedge k'}$. It then follows that
$\dz^{k\wedge k'}+d(x,y_\az^{k,m})\sim \dz^{k\wedge k'}+d(x,z)$ and hence, by Lemma \ref{6.15.1}(i),
$V_{\dz^{k\wedge k'}}(x)+V(x,y_\az^{k,m})\sim V_{\dz^{k\wedge k'}}(x)+V(x,z)$. By this
and Theorem \ref{10.22.1}, we conclude that, for any
$x\in\cx$,
\begin{align}\label{eq-int}
&\sum_{\alpha \in \ca_k}\sum_{m=1}^{N(k,\alpha)}\mu\left(\qa\right)
\left[\frac{1}{V_{\delta^{k\wedge k'}}(x)+V(x,\ya)}\right]^p
\left[\frac{\delta^{k\wedge k'}}{\delta^{k\wedge k'}+d(x,\ya)}\right]^{\gamma p}\\
&\quad=\sum_{\alpha \in \ca_k}\sum_{m=1}^{N(k,\alpha)}\int_{\qa}
\left[\frac{1}{V_{\delta^{k\wedge k'}}(x)+V(x,\ya)}\right]^p
\left[\frac{\delta^{k\wedge k'}}{\delta^{k\wedge k'}+d(x,\ya)}\right]^{\gamma p}\,d\mu(z)\noz\\
&\quad\sim \sum_{\alpha \in \ca_k}\sum_{m=1}^{N(k,\alpha)}\int_{\qa}
\left[\frac{1}{V_{\delta^{k\wedge k'}}(x)+V(x,z)}\right]^p
\left[\frac{\delta^{k\wedge k'}}{\delta^{k\wedge k'}+d(x,z)}\right]^{\gamma p}\,d\mu(z)\noz\\
&\quad\sim\int_\cx\left[\frac{1}{V_{\delta^{k\wedge k'}}(x)+V(x,z)}\right]^p
\left[\frac{\delta^{k\wedge k'}}{\delta^{k\wedge k'}+d(x,z)}\right]^{\gamma p}\,d\mu(z)\noz\\
&\quad\sim\int_{d(x,z)<\dz^{k\wedge k'}}\left[\frac{1}{V_{\delta^{k\wedge k'}}(x)+V(x,z)}\right]^p
\left[\frac{\delta^{k\wedge k'}}{\delta^{k\wedge k'}+d(x,z)}\right]^{\gamma p}\,d\mu(z)\noz\\
&\quad\qquad+\sum_{l=1}^\fz\int_{2^{l-1}\dz^{k\wedge k'}\le d(x,z)<2^{l}\dz^{k\wedge k'}}\cdots\noz\\
&\quad=:{\rm I}_0+\sum_{l=1}^\fz{\rm I}_l.\noz
\end{align}

We first consider ${\rm I}_0$. Indeed, since $d(x,z)<\dz^{k\wedge k'}$, it then follows that
$\delta^{k\wedge k'}+d(x,z)\sim \delta^{k\wedge k'}$ and
$V_{\delta^{k\wedge k'}}(x)+V(x,z)\sim V_{\delta^{k\wedge k'}}(x)$. By these, we obtain
\begin{equation}\label{eq-i0}
{\rm I}_0\sim\int_{d(x,z)<\dz^{k\wedge k'}}\left[\frac{1}{V_{\delta^{k\wedge k'}}(x)}\right]^p\,d\mu(z)
\sim\lf[V_{\delta^{k\wedge k'}}(x)\r]^{1-p}.
\end{equation}
Moreover, this, together with \eqref{eq-int}, further implies the first inequality of \eqref{5.2}.

Now we deal with ${\rm I}_l$ for any $l\in\nn$. Since $d(x,z)\sim 2^l\dz^{k\wedge k'}\gtrsim\dz^{k\wedge k'}$,
it then follows that $\delta^{k\wedge k'}+d(x,z)\sim 2^l\delta^{k\wedge k'}$ and
$V_{\delta^{k\wedge k'}}(x)+V(x,z)\sim V_{2^l\delta^{k\wedge k'}}(x)$ due to the doubling condition \eqref{eq-doub}.
By these and \eqref{eq-doub}, we have
\begin{align*}
{\rm I}_l&\sim\int_{2^{l-1}\dz^{k\wedge k'}\le d(x,z)<2^{l}\dz^{k\wedge k'}}
\left[\frac{1}{V_{2^l\delta^{k\wedge k'}}(x)}\right]^p\lf[\frac{\dz^{k\wedge k'}}{2^l\dz^{k\wedge k'}}\r]^{p\gz}\,d\mu(z)
\ls 2^{-lp\gz}\lf[V_{2^l\delta^{k\wedge k'}}(x)\r]^{1-p}\\
&\ls 2^{-l[p\gz-\omega(1-p)]}\lf[V_{\delta^{k\wedge k'}}(x)\r]^{1-p}.
\end{align*}
From this, \eqref{eq-i0}, \eqref{eq-int} and $p\in(\omega/(\omega+\gz),1]$, we deduce that
\begin{align*}
&\sum_{\alpha \in \ca_k}\sum_{m=1}^{N(k,\alpha)}\mu\left(\qa\right)
\left[\frac{1}{V_{\delta^{k\wedge k'}}(x)+V(x,\ya)}\right]^p
\left[\frac{\delta^{k\wedge k'}}{\delta^{k\wedge k'}+d(x,\ya)}\right]^{\gamma p}\\
&\quad\ls\lf[V_{\delta^{k\wedge k'}}(x)\r]^{1-p}\lf\{1+\sum_{l=1}^\fz 2^{-l[p\gz-\omega(1-p)]}\r\}
\sim \lf[V_{\delta^{k\wedge k'}}(x)\r]^{1-p}.
\end{align*}
This finishes the proof of the second inequality of \eqref{5.2} and hence of Lemma \ref{9.14.1}.
\end{proof}

The following lemma is just \cite [Lemma 5.3]{hmy08}, whose proof remains valid for any quasi-metric $d$ and
is independent of the reverse doubling property of the considered measure; we omit the details here.

\begin{lemma}\label{10.18.5}
Let $\gamma \in (0,\infty)$ and $r \in (\omega/(\omega+\gamma),1]$ with $\omega$ as in
\eqref{eq-doub}. Then there exists a positive constant $C$ such that,
for any $k,\ k' \in \zz$, $x \in \mathcal{X}$, $a_\alpha^{k,m}\in\cc$
and $\ya \in \qa$ with $\alpha \in \ca_k$ and $m\in\{1,\dots,N(k,\alpha)\}$,
\begin{align}\label{5.3}
&\sum_{\alpha \in \ca_k}\sum_{m=1}^{N(k,\alpha)}\mu\left(\qa\right)
\frac{1}{V_{\delta^{k\wedge k'}}(x)+V(x,\ya)}\left[\frac{\delta^{k\wedge k'}}
{\delta^{k\wedge k'}+d(x,\ya)}\right]^{\gamma}|a_\alpha^{k,m}|\\
&\quad \leq C \delta^{[k-(k\wedge k')]\omega(1-1/r)}\left[M\left(\sum_{\alpha \in \ca_k}
\sum_{m=1}^{N(k,\alpha)}|a_\alpha^{k,m}|^r\mathbf 1_{\qa}\right)(x)\right]^{1/r},\notag
\end{align}
where $M$ is as in \eqref{m}.
\end{lemma}

\begin{remark} Observe that Lemma \ref{9.14.1} is sharp
because \eqref{5.2} in Lemma \ref{9.14.1} is indeed an equivalence relation.
Moreover, in Lemma \ref{10.18.5}, if we let $a_\az^{k,m}:=1$ for any $k\in\zz$, $\az\in\ca_k$
and $m\in\{1,\ldots,N(k,\az)\}$, we find that,
for any given $r\in(\omega/(\omega+\gz),1]$, \eqref{5.3} becomes
\begin{align}\label{eq-intadd}
&\sum_{\alpha \in \ca_k}\sum_{m=1}^{N(k,\alpha)}\mu\left(\qa\right)
\frac{1}{V_{\delta^{k\wedge k'}}(x)+V(x,\ya)}\left[\frac{\delta^{k\wedge k'}}
{\delta^{k\wedge k'}+d(x,\ya)}\right]^{\gamma}\\
&\quad \leq C \delta^{[k-(k\wedge k')]\omega(1-1/r)}\left[M\left(\sum_{\alpha \in \ca_k}
\sum_{m=1}^{N(k,\alpha)}\mathbf 1_{\qa}\right)(x)\right]^{1/r}
= C \delta^{[k-(k\wedge k')]\omega(1-1/r)},\noz
\end{align}
where $C$ is a positive constant independent of $k$ and $k'$. If we choose $r:=1$ in \eqref{eq-intadd},
we then find that the right-hand side of
\eqref{eq-intadd} just becomes a positive constant $C$ and hence
recovers the best possible estimate in \eqref{5.2} in this case, namely,
in such a case, \eqref{eq-intadd} coincides with
the second inequality of \eqref{5.2} when $p:=1$ therein. In this sense,
we see that Lemma \ref{10.18.5} is also sharp.
\end{remark}

The following lemma is the Fefferman--Stein vector-valued maximal inequality,
which was established in \cite[Theorem 1.2]{gly09}.
\begin{lemma}\label{fsvv}
Let $p\in(1,\infty)$, $q\in(1,\infty]$ and $M$ be the Hardy--Littlewood maximal operator on $\mathcal{X}$ as
in \eqref{m}. Then there exists a positive constant $C$ such that, for any sequence $\{f_j\}_{j\in\zz}$ of
measurable functions on $\mathcal{X}$,
$$
\left\|\left\{\sum_{j\in\zz}[M(f_j)]^q\right\}^{1/q}\right\|_{L^p(\mathcal{X})}
\leq C\left\|\left(\sum_{j\in\zz}|f_j|^q\right)^{1/q}\right\|_{L^p(\mathcal{X})}
$$
with the usual modification made when $q = \infty$.
\end{lemma}

For compositions of $\{Q_k\}_{k=-\infty}^\infty$ and
$\{\widetilde{Q}_k\}_{k=-\infty}^\infty$, we have the following almost orthonormal estimate.

\begin{lemma}\label{4.23z}
Let $\{Q_k\}_{k=-\infty}^\infty$ be an {\rm exp-ATI}
and $\{\widetilde{Q}_k\}_{k=-\infty}^\infty$ satisfy
\eqref{4.23x}, \eqref{4.23y} and \eqref{4.23a} of Lemma \ref{h_c_crf} with $\beta$ and $\gamma$ same as therein.
Then, for any given $\eta'\in(0,\beta\wedge\gamma)$, there
exists a positive constant $C$ such that, for any $k,\ k'\in\zz$ and $z,\ y\in\cx$,
\begin{equation}\label{pp}
\lf|Q_k\widetilde{Q}_{k'}(z,y)\r|\le C\delta^{|k-k'|\eta'}
\frac{1}{V_{\delta^{k\wedge k'}}(y)+V(z,y)}\left[\frac{\delta^{k\wedge k'}}
{\delta^{k\wedge k'}+d(z,y)}\right]^{\gamma}.
\end{equation}
\end{lemma}

\begin{proof}
By \eqref{10.23.3} and \eqref{10.23.4}, we find that, for any $k\in\zz$, $Q_k$ also satisfies \eqref{4.23x}
and \eqref{4.23y}. Thus, to prove \eqref{pp}, by symmetry, it suffices to show \eqref{pp} when
$k\geq k'$ with $k,\ k'\in\zz$. Then, for any $z,\ y\in X$, by the cancellation condition of $Q_k$, we have
\begin{align*}
\lf|Q_k\widetilde{Q}_{k'}(z,y)\r|&=\left|\int_\cx Q_k(z,s)\widetilde{Q}_{k'}(s,y)\,d\mu(s)\right|\\
&=\left|\int_\cx Q_k(z,s)\left[\widetilde{Q}_{k'}(s,y)-\widetilde{Q}_{k'}(z,y)\right]\,d\mu(s)\right|\\
&\lesssim\int_{\{s\in\cx:\ d(s,z)\leq(2A_0)^{-1}[\delta^{k'}+d(z,y)]\}}|Q_k(z,s)|
\lf|\widetilde{Q}_{k'}(s,y)-\widetilde{Q}_{k'}(z,y)\right|\,d\mu(s)\\
&\qquad+\int_{\{s\in\cx:\ d(s,z)\geq(2A_0)^{-1}[\delta^{k'}+d(z,y)]\}}|Q_k(z,s)|
\left|\widetilde{Q}_{k'}(s,y)\right|\,d\mu(s)\\
&\qquad+\left|\widetilde{Q}_{k'}(z,y)\right|\int_{\{s\in\cx:\ d(s,z)\geq(2A_0)^{-1}[\delta^{k'}+d(z,y)]\}}
|Q_k(z,s)|\,d\mu(s)\\
&=:\rm{{\rm J}_1+{\rm J}_2+{\rm J}_3}.
\end{align*}

We first estimate $\rm J_1$. Indeed, by Lemma \ref{6.15.1}(iii) and $\eta'\in(0,\beta\wedge\gamma)$, we conclude
that
\begin{align*}
{\rm J}_1&\lesssim \int_{\{s\in\cx:\ d(s,z)\leq(2A_0)^{-1}[\delta^{k'}+d(z,y)]\}}
\frac{1}{V_{\delta^k}(z)+V(s,z)}\left[\frac{\delta^k}{\delta^k+d(z,s)}\right]^{\eta'}\\
&\qquad \times\left[\frac{d(s,z)}{\delta^{k'}
+d(z,y)}\right]^\beta\frac{1}{V_{\delta^{k'}}(z)
+V(z,y)}\left[\frac{\delta^{k'}}{\delta^{k'}
+d(z,y)}\right]^{\gamma}\,d\mu(s)\\
&\lesssim\delta^{k\eta'}\delta^{k'\gamma}
\frac{1}{V_{\delta^{k'}}(z)+V(z,y)}\left[\frac{1}
{\delta^{k'}+d(z,y)}\right]^{\beta+\gamma}\\
&\qquad \times \int_{\{s\in\cx:\ d(s,z)
\leq(2A_0)^{-1}[\delta^{k'}+d(z,y)]\}}[d(s,z)]^{\beta-\eta'}\frac{1}{V(s,z)}\,d\mu(s)\\
&\lesssim   \delta^{(k-k')\eta'}\frac{1}
{V_{\delta^{k'}}(z)+V(z,y)}\left[\frac{\delta^{k'}}{\delta^{k'}+d(z,y)}\right]^{\gamma},
\end{align*}
which is the desired estimate.

Now we estimate $\rm J_2$. For any $s\in \{s\in\cx:\ d(s,z)\geq(2A_0)^{-1}
[\delta^{k'}+d(z,y)]\}$, by Lemma
\ref{6.15.1}(i), we know that
$$
V(s,z)\sim \mu(B(z,d(s,z)))\gtrsim\mu(B(z,\delta^{k'}+d(z,y)))\sim V_{\delta^{k'}}(z)+V(z,y).
$$
Moreover, in this case, we have
$$
\delta^k+d(z,s)\gtrsim \delta^k+d(z,y)+\delta^{k'}\gtrsim\delta^{k'}.
$$
From these facts, \eqref{4.23x} and Lemma \ref{6.15.1}(ii), we deduce that
\begin{align*}
{\rm J}_2&\lesssim\int_{\{s\in\cx:\ d(s,z)\geq(2A_0)^{-1}[\delta^{k'}+d(z,y)]\}}
\frac{1}{V_{\delta^k}(z)+V(s,z)}\left[\frac{\delta^k}{\delta^k+d(z,s)}\right]^{\gamma}\\
&\qquad\times\frac{1}{V_{\delta^{k'}}+V(s,y)}
\left[\frac{\delta^{k'}}{\delta^{k'}+d(s,y)}\right]^\gamma\,d\mu(s)\\
&\lesssim \delta^{(k-k')\gamma}\frac{1}
{V_{\delta^{k'}}(z)+V(z,y)}\left[\frac{\delta^{k'}}{\delta^{k'}+d(z,y)}\right]^{\gamma},
\end{align*}
which is also the desired estimate.

For the term $\rm{{\rm J}_3}$, by \eqref{10.23.3}, \eqref{4.23x} and  Lemma \ref{6.15.1}(iv), we conclude that
\begin{align*}
{\rm J}_3&\lesssim \frac{1}{V_{k'}(z)+V(z,y)}
\left[\frac{\delta^{k'}}{\delta^{k'}+d(z,y)}\right]^{\gamma}\\
&\qquad\times\int_{\{s\in\cx:\ d(s,z)\geq(2A_0)^{-1}[\delta^{k'}
+d(z,y)]\}}\frac{1}{V_{\delta^k}(z)+V(s,z)}
\left[\frac{\delta^k}{\delta^k+d(z,s)}\right]^{\gamma}\,d\mu(s)\\
&\lesssim  \delta^{(k-k')\gamma}\frac{1}{V_{\dz^{k'}}(z)+V(z,y)}
\left[\frac{\delta^{k'}}{\delta^{k'}+d(z,y)}\right]^{\gamma},
\end{align*}
which is the desired estimate and then completes the proof of Lemma \ref{4.23z}.
\end{proof}

Using these lemmas, we now establish the \emph{homogeneous Plancherel--P\^olya inequality}.

\begin{lemma}\label{5.12.1}
Let $\{Q_k\}_{k=-\infty}^\infty$ and $\{P_k\}_{k=-\infty}^\infty$
be two {\rm $\exp$-ATIs} and
$\beta,\ \gamma \in (0, \eta)$
with $\omega$ and $\eta$, respectively, as in \eqref{eq-doub} and Definition \ref{10.23.2}.
Then there exists a positive constant $C$ such that, for any $f \in (\cggi)'$,
\begin{align}\label{bhpp}
&\left\{\sum_{k=-\infty}^\infty\delta^{-ksq}\left[\sum_{\alpha \in \ca_k}\sum_{m=1}^{N(k,\alpha)}
\mu\left(\qa\right)\left\{\sup_{z\in\qa}|P_k(f)(z)|\right\}^p\right]^{q/p}\right\}^{1/q}\\
&\quad\leq C\left\{\sum_{k=-\infty}^\infty\delta^{-ksq}\left[\sum_{\alpha \in \ca_k}\sum_{m=1}^{N(k,\alpha)}
\mu\left(\qa\right)\left\{\inf_{z\in\qa}|Q_k(f)(z)|\right\}^p\right]^{q/p}\right\}^{1/q}\notag
\end{align}
when $s\in(-(\beta\wedge\gamma),\,\beta\wedge\gamma)$,
$p\in (p(s,\beta\wedge\gamma),\infty]$ and $q \in (0,\infty]$ with
$p(s,\beta\wedge\gamma)$ as in \eqref{pseta} and usual
modifications made when $p=\infty$ or $q=\infty$,  and
\begin{align}\label{fhpp}
&\left\|\left\{\sum_{k=-\infty}^\infty\sum_{\alpha \in \ca_k}\sum_{m=1}^{N(k,\alpha)}\delta^{-ksq}
\left[\sup_{z\in\qa}|P_k(f)(z)|\right]^q\mathbf 1_{\qa}\right\}^{1/q}\right\|_{L^p(\mathcal{X})}\\
&\quad\leq C\left\|\left\{\sum_{k=-\infty}^\infty\sum_{\alpha \in \ca_k}\sum_{m=1}^{N(k,\alpha)}\delta^{-ksq}
\left[\inf_{z\in\qa}|Q_k(f)(z)|\right]^q\mathbf 1_{\qa}\right\}^{1/q}\right\|_{L^p(\mathcal{X})}\notag
\end{align}
when $s\in(-(\beta\wedge\gamma),\,\beta\wedge\gamma)$, $p\in (p(s,\beta\wedge\gamma),\infty)$ and
$q \in (p(s,\beta\wedge\gamma),\infty]$, with the usual modification made when $q=\infty$.
\end{lemma}

\begin{proof}
We first prove \eqref{bhpp}. By Lemma \ref{crf},  we know that there exist
$\{\widetilde{Q}_{k'}\}_{k'\in\zz}$  of bounded linear integral
operators on $L^2(\cx)$ such that, for any $f \in (\cggi)'$
with $\beta,\ \gamma$ as in this lemma,
\begin{equation}\label{11.10.1}
f(\cdot) = \sum_{k'=-\infty}^\infty\sum_{\alpha' \in \ca_{k'}}
\sum_{m'=1}^{N(k',\alpha')}\mu\left(\qap\right)\widetilde{Q}_{k'}(\cdot,\yap)
Q_{k'}f\left(\yap\right) \qquad \text{in}\  (\cggi)',
\end{equation}
where $\widetilde{Q}_{k'}$ satisfies the same conditions as in
Lemma \ref{crf} and $\yap$ is an arbitrary point in $\qap$
for any $k'\in\zz$, $\alpha'\in\ca_{k'}$ and
$m'\in\{1,\dots,N(k',\alpha')\}$.
Now we consider two cases.

{\it Case 1) $p\in (p(s,\beta\wedge\gamma),1]$.}
In this case, note that, for any $k\in\zz$, $\alpha\in\ca_k$, $m\in
\{1,\dots,N(k,\alpha)\}$ and $z,\ \ya\in\qa$, it holds true that
$\delta^{k\wedge k'}+d(z,\yap)\sim\delta^{k\wedge k'}+d(\ya,\yap)$ and
$$V_{\delta^{k\wedge k'}}\left(\yap\right)+V(z,\yap)
\sim V_{\delta^{k\wedge k'}}\left(\yap\right)+V(\yap,\ya).$$
Using these, \eqref{pp}, \eqref{5.2} and $p\in (p(s,\beta\wedge\gamma),1]$,
combined with \eqref{r}, we conclude that,
for any fixed $\beta'\in(0,\beta\wedge\gamma)$ and any $k\in\zz$,
\begin{align}\label{es-ex}
&\sum_{\alpha \in \ca_k}\sum_{m=1}^{N(k,\alpha)}\mu\left(\qa\right)
\left[\sup_{z\in\qa}|P_k(f)(z)|\right]^p\\
&\quad=\sum_{\alpha \in \ca_k}\sum_{m=1}^{N(k,\alpha)}\mu\left(\qa\right)\sup_{z\in\qa}
\left|\sum_{k'=-\infty}^\infty\sum_{\alpha' \in \ca_{k'}}\sum_{m'=1}^{N(k',\alpha')}
\mu\left(\qap\right)P_k\widetilde{Q}_{k'}(z,\yap)Q_{k'}f\left(\yap\right)\right|^p\notag\\
&\quad\lesssim\sum_{\alpha \in \ca_k}\sum_{m=1}^{N(k,\alpha)}\mu\left(\qa\right)
\sum_{k'=-\infty}^\infty\sum_{\alpha' \in \ca_{k'}}\sum_{m'=1}^{N(k',\alpha')}
\left[\mu\left(\qap\right)\right]^p\delta^{|k-k'|\beta'p}\notag\\
&\qquad\quad\times\left[\frac{1}{V_{\delta^{k\wedge k'}}(\yap)
+V(\yap,\ya)}\right]^p\left[\frac{\delta^{k\wedge k'}}{\delta^{k\wedge k'}
+d(\yap,\ya)}\right]^{\gamma p}\left|Q_{k'}f\left(\yap\right)\right|^p\notag\\
&\quad\lesssim\sum_{k'=-\infty}^\infty\sum_{\alpha' \in \ca_{k'}}
\sum_{m'=1}^{N(k',\alpha')}\left[\mu\left(\qap\right)\right]^p\delta^{|k-k'|\beta'p}
\left|Q_{k'}f\left(\yap\right)\right|^p\notag\\
&\qquad\quad\times\sum_{\alpha \in \ca_k}\sum_{m=1}^{N(k,\alpha)}\mu\left(\qa\right)
\left[\frac{1}{V_{\delta^{k\wedge k'}}(\yap)+V(\yap,\ya)}\right]^p
\left[\frac{\delta^{k\wedge k'}}{\delta^{k\wedge k'}+d(\yap,\ya)}\right]^{\gamma p}\notag\\
&\quad \lesssim\sum_{k'=-\infty}^\infty\sum_{\alpha' \in \ca_{k'}}
\sum_{m'=1}^{N(k',\alpha')}\left[\mu\left(\qap\right)\right]^p\delta^{|k-k'|\beta'p}
\left[V_{\delta^{k\wedge k'}}\left(\yap\right)\right]^{1-p}\left|Q_{k'}f\left(\yap\right)\right|^p.\notag
\end{align}
From this, $s\in(-(\bz\wedge\gz),\bz\wedge\gz)$, $p\in(p(s,\bz\wedge\gz),1]$ and the fact that
\begin{equation}\label{11.10.2}
V_{\delta^{k\wedge k'}}\left(\yap\right)\lesssim \delta^{(k\wedge k'-k')\omega}
V_{\delta^{k'}}\left(\yap\right)\sim\delta^{(k\wedge k'-k')\omega}\mu\left(\qap\right),
\end{equation}
together with the H\"older inequality when $q/p\in[1,\infty)$, or \eqref{r} when $q/p \in(0,1)$,
it follows that, if we choose $\beta'\in(0,\beta\wedge\gamma)$
such that $s\in(-\beta',\beta')$ and
$p\in(\max\{\omega/(\omega+\beta'),\omega/(\omega+\beta'+s)\},1]$,
then
\begin{align}\label{6.2.1}
&\left\{\sum_{k=-\infty}^\infty\delta^{-ksq}\left[\sum_{\alpha \in \ca_k}
\sum_{m=1}^{N(k,\alpha)}\mu\left(\qa\right)\left\{\sup_{z\in\qa}|P_k(f)(z)|\right\}^p\right]^{q/p}\right\}^{1/q}\\
&\quad\lesssim\left\{\sum_{k=-\infty}^\infty\left[\sum_{k'\in\zz}
\sum_{\alpha' \in \ca_{k'}}\sum_{m'=1}^{N(k',\alpha')}\delta^{-ksp}
\left\{\mu\left(\qap\right)\right\}^p\delta^{|k-k'|\beta'p}
\left\{V_{\delta^{k\wedge k'}}\left(\yap\right)\right\}^{1-p}\r.\r.\noz\\
&\qquad\quad\lf.\Biggl.{}\times\left|Q_{k'}f\left(\yap\right)\right|^p\Biggr]^{q/p}\right\}^{1/q}\notag\\
&\quad\lesssim \left\{\sum_{k'\in\zz}\delta^{-k'sq}\left[\sum_{\alpha' \in \ca_{k'}}
\sum_{m'=1}^{N(k',\alpha')}\mu\left(\qap\right)\left|Q_{k'}f\left(\yap\right)\right|^p\right]^{q/p}\right\}^{1/q}.\notag
\end{align}
By the arbitrariness of $\yap$, we know that
\begin{align}\label{limpp}
&\left\{\sum_{k=-\infty}^\infty\delta^{-ksq}\left[\sum_{\alpha \in \ca_k}
\sum_{m=1}^{N(k,\alpha)}\mu\left(\qa\right)\left\{\sup_{z\in\qa}|P_k(f)(z)|\right\}^p\right]^{q/p}\right\}^{1/q}\\
&\quad\lesssim\left\{\sum_{k=-\infty}^\infty\delta^{-ksq}\left[\sum_{\alpha \in \ca_k}
\sum_{m=1}^{N(k,\alpha)}\mu\left(\qa\right)\left\{\inf_{z\in\qa}|Q_k(f)(z)|\right\}^p\right]^{q/p}\right\}^{1/q}.\notag
\end{align}
This finishes the proof of \eqref{bhpp} in this
case.

{\it Case 2) $p\in(1,\infty]$.}
In this case, due to the assumption that $s\in(-(\bz\wedge\gz),\bz\wedge\gz)$,
we can choose $\beta' \in (0,\beta\wedge\gamma)$ such
that $s\in(-\beta',\beta')$; then, by this, \eqref{5.2}, \eqref{pp} and the H\"older inequality, we know that
\begin{align*}
&\sum_{\alpha \in \ca_k}\sum_{m=1}^{N(k,\alpha)}
\mu\left(\qa\right)\left[\sup_{z\in\qa}|P_k(f)(z)|\right]^p\\
&\quad=\sum_{\alpha \in \ca_k}\sum_{m=1}^{N(k,\alpha)}
\mu\left(\qa\right)\sup_{z\in\qa}\left|\sum_{k'=-\infty}^\infty
\sum_{\alpha' \in \ca_{k'}}\sum_{m'=1}^{N(k',\alpha')}
\mu\left(\qap\right)P_k\widetilde{P}_{k'}\left(z,\yap\right)Q_{k'}f\left(\yap\right)\right|^p\\
&\quad\lesssim\sum_{\alpha \in \ca_k}
\sum_{m=1}^{N(k,\alpha)}\mu\left(\qa\right)
\left\{\sum_{k'=-\infty}^\infty\sum_{\alpha' \in \ca_{k'}}
\sum_{m'=1}^{N(k',\alpha')}\mu\left(\qap\right)\delta^{|k-k'|\beta'}
\frac{1}{V_{\delta^{k\wedge k'}}(\yap)+V(\yap,\ya)}\right.\\
&\qquad\quad\times\left.\left[\frac{\delta^{k\wedge k'}}
{\delta^{k\wedge k'}+d(\yap,\ya)}\right]^{\gamma }\left|Q_{k'}f\left(\yap\right)\right|\right\}^p\\
&\quad\lesssim\sum_{\alpha \in \ca_k}\sum_{m=1}^{N(k,\alpha)}
\mu\left(\qa\right)\left\{\sum_{k'=-\infty}^\infty\sum_{\alpha' \in \ca_{k'}}
\sum_{m'=1}^{N(k',\alpha')}\mu\left(\qap\right)\delta^{|k-k'|\beta'}
\delta^{(k'-k)s(p-1)}\left|Q_{k'}f\left(\yap\right)\right|^p\right.\\
&\qquad\quad\times\left.\frac{1}{V_{\delta^{k\wedge k'}}(\yap)+V(\yap,\ya)}
\left[\frac{\delta^{k\wedge k'}}{\delta^{k\wedge k'}+d(\yap,\ya)}\right]^{\gamma }\right\}\\
&\qquad\quad \times\left\{\sum_{k'=-\infty}^\infty\sum_{\alpha' \in \ca_{k'}}
\sum_{m'=1}^{N(k',\alpha')}\mu\left(\qap\right)\delta^{|k-k'|\beta'}\delta^{(k'-k)s(p-1)}\delta^{(k-k')s(p-1)p'}\right.\\
&\qquad\quad\times\left.\frac{1}{V_{\delta^{k\wedge k'}}(\yap)+V(\yap,\ya)}
\left[\frac{\delta^{k\wedge k'}}{\delta^{k\wedge k'}
+d(\yap,\ya)}\right]^{\gamma }\right\}^{\frac{p}{p'}}\\
&\quad \lesssim\sum_{k'=-\infty}^\infty
\sum_{\alpha' \in \ca_{k'}}\sum_{m'=1}^{N(k',\alpha')}
\delta^{|k-k'|\beta'}\delta^{(k'-k)s(p-1)}\mu\left(\qap\right)\left|Q_{k'}f\left(\yap\right)\right|^p.
\end{align*}
Using this and an argument similar to that used in the estimations of \eqref{6.2.1}
and \eqref{limpp}, we conclude that, when
 $p\in(1,\infty]$,
\begin{align}\label{6.2.2}
&\left\{\sum_{k=-\infty}^\infty\delta^{-ksq}
\left[\sum_{\alpha \in \ca_k}\sum_{m=1}^{N(k,\alpha)}
\mu\left(\qa\right)\left\{\sup_{z\in\qa}|P_k(f)(z)|\right\}^p\right]^{q/p}\right\}^{1/q}\\
&\quad\lesssim \left\{\sum_{k'\in\zz}\delta^{-k'sq}
\left[\sum_{\alpha' \in \ca_{k'}}\sum_{m'=1}^{N(k',\alpha')}
\mu\left(\qap\right)\left|Q_{k'}f\left(\yap\right)\right|^p\right]^{q/p}\right\}^{1/q}.\notag
\end{align}
This, combined with \eqref{6.2.1}, then finishes the proof of \eqref{bhpp}.

We now prove \eqref{fhpp}. Because $s\in(-(\bz\wedge\gz),\bz\wedge\gz)$, $p\in(p(s,\bz\wedge\gz),\fz)$
and $q\in(p(s,\bz\wedge\gz),\fz]$, we may choose $\bz'\in(0,\bz\wedge\gz)$ such that
$s\in(-\bz',\bz')$ and
\begin{equation}\label{eq-pq}
\min\{p,q,1\}>\max\lf\{\frac\omega{\omega+\beta'},\frac{\omega}{\omega+\beta'+s}\r\}.
\end{equation}
By this, Lemma \ref{crf}, \eqref{pp} and Lemma \ref{5.3}, we know that, for any fixed
$$
r\in\lf(\max\lf\{\frac\omega{\omega+\beta'},\frac{\omega}{\omega+\beta'+s}\r\},\min\{p,q,1\}\r)
$$
and any $x \in \cx$,
\begin{align*}
&\left\{\sum_{k=-\infty}^\infty\sum_{\alpha \in \ca_k}
\sum_{m=1}^{N(k,\alpha)}\delta^{-ksq}\left[\sup_{z\in\qa}
|P_k(f)(z)|\right]^q\mathbf 1_{\qa}(x)\right\}^{1/q}\\
&\quad\lesssim\left[\sum_{k=-\infty}^\infty
\left\{\sum_{k'\in\zz}\delta^{|k-k'|\beta'-(k-k')s}
\delta^{[k'-(k\wedge k')]\omega(1-1/r)}
\vphantom{\sum_{\alpha' \in \ca_{k'}}\sum_{m'=1}^{N(k',\alpha')}}\r.\r.\\
&\quad\qquad\times\left.\left.\left[M\left(\sum_{\alpha' \in \ca_{k'}}
\sum_{m'=1}^{N(k',\alpha')}\delta^{-k'sr}
\left|Q_{k'}f\left(\yap\right)\right|^r
\mathbf 1_{\qap}\right)(x)\right]^{1/r}\right\}^q\right]^{1/q}.
\end{align*}
From this, the H\"older inequality when $q\in(1,\infty]$, or \eqref{r} when $q\in(p(s,\bz\wedge\gz),1]$,
and Lemma \ref{fsvv}, we deduce that
\begin{align}\label{10.18.6}
&\left\|\left\{\sum_{k=-\infty}^\infty\sum_{\alpha \in \ca_k}\sum_{m=1}^{N(k,\alpha)}\delta^{-ksq}
\left[\sup_{z\in\qa}|P_k(f)(z)|\right]^q\mathbf 1_{\qa}\right\}^{1/q}\right\|_{L^p(\mathcal{X})}\\
&\quad\lesssim\left\|\left\{\sum_{k'=-\infty}^\infty\left[M\left(\sum_{\alpha' \in \ca_{k'}}
\sum_{m'=1}^{N(k',\alpha')}\delta^{-k'sr}
\left|Q_{k'}f\left(\yap\right)\right|^r\mathbf 1_{\qap}\right)\right]^{q/r}\right\}^{r/q}\right\|_{L^{p/r}(\mathcal{X})}^{1/r}\notag\\
&\quad\lesssim\left\|\left\{\sum_{k'=-\infty}^\infty\sum_{\alpha' \in \ca_{k'}}
\sum_{m'=1}^{N(k',\alpha')}\delta^{-k'sq}
\left|Q_{k'}f\left(\yap\right)\right|^q
\mathbf 1_{\qap}\right\}^{1/q}\right\|_{L^p(\mathcal{X})}.\notag
\end{align}
Then, by  the arbitrariness of $\yap$ for any $k'\in\zz$,
$\alpha'\in\ca_{k'}$ and $m'\in\{1,\dots,N(k',\alpha')\}$, we obtain \eqref{fhpp},
which completes the proof of Lemma \ref{5.12.1}.
\end{proof}

\begin{remark}
Recently, Jaming and Negreira obtained a Plancherel--P\^olya inequality on Besov spaces on
spaces of homogeneous type for some special indices with
a quite different purpose from this article (see \cite[Theorem 4.1]{jn}). Indeed, in their
article, they considered the \emph{regular spaces} (see \cite[(2.1) and (2.2)]{jn}), which
is a special case  of  spaces of homogeneous type.
\end{remark}

Using Lemma \ref{5.12.1}, we can immediately show that the spaces $\hb$ and $\hf$
are independent of the choice of exp-ATIs.

\begin{proposition}\label{6.4.1}
Let $\{P_k\}_{k\in\zz}$ and $\{Q_k\}_{k\in\zz}$ be two {\rm exp-AITs}, and $\beta,\ \gamma\in(0,\eta)$ with $\eta$
as in Definition \ref{10.23.2}. Then there exists a constant $C\in[1,\fz)$ such that, for any
$f\in(\cggi)'$,
\begin{align}\label{4.18.3}
 C^{-1}\left[\sum_{k \in \zz} \delta^{-ksq}\|Q_k(f)\|_{\lp}^q\right]^{1/q}&\leq\left[\sum_{k \in \zz}
 \delta^{-ksq}\|P_k(f)\|_{\lp}^q\right]^{1/q}\\
 &\leq C \left[\sum_{k \in \zz} \delta^{-ksq}\|Q_k(f)\|_{\lp}^q\right]^{1/q}\noz
\end{align}
when $s\in(-(\beta\wedge\gamma),\,\beta\wedge\gamma)$,
$p\in (p(s,\beta\wedge\gamma),\infty]$ and $q \in (0,\infty]$
with $p(s,\beta\wedge\gamma)$ as in \eqref{pseta} and the usual modifications made when $p=\infty$ or
$q=\infty$, and
\begin{align}\label{4.18.4}
C^{-1}\left\|\left[\sum_{k \in \zz} \delta^{-ksq}|Q_k(f)|^q\right]^{1/q}\right\|_{\lp}
&\leq\left\|\left[\sum_{k \in \zz} \delta^{-ksq}|P_k(f)|^q\right]^{1/q}\right\|_{\lp}\\
&\leq C\left\|\left[\sum_{k \in \zz} \delta^{-ksq}|Q_k(f)|^q\right]^{1/q}\right\|_{\lp}\noz
\end{align}
when $s\in(-(\beta\wedge\gamma),\,\beta\wedge\gamma)$,
$p\in (p(s,\beta\wedge\gamma),\infty)$ and $q \in (p(s,\beta\wedge\gamma),\infty]$,
with usual modification made when $q=\fz$.
\end{proposition}

\begin{proof}
We first show \eqref{4.18.3}. Indeed, by \eqref{bhpp}, we know that
\begin{align*}
\left[\sum_{k \in \zz} \delta^{-ksq}\|P_k(f)\|_{\lp}^q\right]^{1/q}&=\left\{\sum_{k \in \zz}
\delta^{-ksq}\left[\int_{\mathcal{X}}|P_k(f)|^pd\mu\right]^{q/p}\right\}^{1/q}\\
&\leq \left\{\sum_{k=-\infty}^\infty\delta^{-ksq}\left[\sum_{\alpha \in \ca_k}
\sum_{m=1}^{N(k,\alpha)}\mu\left(\qa\right)\left\{\sup_{z\in\qa}|P_k(f)(z)|\right\}^p\right]^{q/p}\right\}^{1/q}\\
&\lesssim\left\{\sum_{k=-\infty}^\infty\delta^{-ksq}\left[\sum_{\alpha \in \ca_k}
\sum_{m=1}^{N(k,\alpha)}\mu\left(\qa\right)\left\{\inf_{z\in\qa}|Q_k(f)(z)|\right\}^p\right]^{q/p}\right\}^{1/q}\\
&\lesssim \left[\sum_{k \in \zz} \delta^{-ksq}\|Q_k(f)\|_{\lp}^q\right]^{1/q},
\end{align*}
which, by symmetry, then completes the proof of \eqref{4.18.3}.

Now we show \eqref{4.18.4}.  By \eqref{fhpp}, we know that
\begin{align*}
\left\|\left[\sum_{k \in \zz} \delta^{-ksq}|P_k(f)|^q\right]^{1/q}\right\|_{\lp}
&=\left\|\left[\sum_{k \in \zz}\sum_{\alpha \in \ca_k}\sum_{m=1}^{N(k,\alpha)}
\delta^{-ksq}|P_k(f)|^q\mathbf 1_{\qa}\right]^{1/q}\right\|_{\lp}\\
&\leq\left\|\left\{\sum_{k=-\infty}^\infty\sum_{\alpha \in \ca_k}\sum_{m=1}^{N(k,\alpha)}\delta^{-ksq}
\left[\sup_{z\in\qa}|P_k(f)(z)|\right]^q\mathbf 1_{\qa}\right\}^{1/q}\right\|_{L^p(\mathcal{X})}\\
&\lesssim\left\|\left\{\sum_{k=-\infty}^\infty\sum_{\alpha \in \ca_k}\sum_{m=1}^{N(k,\alpha)}\delta^{-ksq}
\left[\inf_{z\in\qa}|Q_k(f)(z)|\right]^q\mathbf 1_{\qa}\right\}^{1/q}\right\|_{L^p(\mathcal{X})}\\
&\lesssim \left\|\left[\sum_{k \in \zz} \delta^{-ksq}|Q_k(f)|^q\right]^{1/q}\right\|_{\lp}.
\end{align*}
By this and symmetry, we finish the proof of \eqref{4.18.4} and hence of Proposition \ref{6.4.1}.
\end{proof}

The following proposition states the relation between $\hb$ and $\hf$
when $p$ and $q$ satisfy certain conditions.

\begin{proposition}\label{phb}
Let $\beta,\ \gamma\in(0,\eta)$ with $\eta$ as in Definition \ref{10.23.2} and $s\in(-(\beta\wedge\gamma),\,\beta\wedge\gamma)$.
\begin{enumerate}
\item[{\rm(i)}]  If $0<q_0\leq q_1\leq \infty$ and $p\in(p(s,\beta\wedge\gamma),\infty]$
with $p(s,\beta\wedge\gamma)$ as in \eqref{pseta},
then $\dot{B}^s_{p,q_0}(\mathcal{X})\subset\dot{B}^s_{p,q_1}(\mathcal{X})$;
if $p(s,\beta\wedge\gamma)<q_0\leq q_1\leq \infty$ and $p\in(p(s,\beta\wedge\gamma),\infty)$,
then $\dot{F}^s_{p,q_0}(\mathcal{X})\subset\dot{F}^s_{p,q_1}(\mathcal{X})$.
\item[{\rm(ii)}] If $p\in(p(s,\beta\wedge\gamma),\infty)$ and
$q\in(p(s,\beta\wedge\gamma),\infty]$, then
$$\dot{B}^s_{p,p\wedge q}(\mathcal{X})\subset\hf\subset\dot{B}^s_{p,p\vee q}(\mathcal{X}).$$
\item[{\rm(iii)}] If
\begin{equation}\label{6.14.1}
\widetilde{\beta}\in(s_+,\eta)\quad \textit{and}\quad
\widetilde{\gamma}\in\left(\max\left\{-s+\omega\left(\frac{1}{p}-1\right)_+,\omega\lf(\frac{1}{p}-1\r)_+\right\},\eta\right),
\end{equation}
then $\mathring{\cg}(\widetilde{\beta},\widetilde{\gamma})\subset \hb$ for
any $q\in(0,\infty]$ and $p \in(p(s,\beta\wedge\gamma),\infty]$, and
$\mathring{\cg}(\widetilde{\beta},\widetilde{\gamma})\subset \hf$ for  any $q\in(p(s,\beta\wedge\gamma),\infty]$
and $p \in(p(s,\beta\wedge\gamma),\infty)$.
\end{enumerate}
\end{proposition}

\begin{proof}
Property (i) is a corollary of \eqref{r}. Indeed,
by \eqref{r}, we know that
\begin{align*}
\|f\|_{\dot{B}^s_{p,q_1}(\mathcal{X})}&=\left[\sum_{k \in \zz}
\delta^{-ksq_1}\|Q_k(f)\|_{\lp}^{q_1}\right]^{1/q_1}=\left\{\sum_{k \in \zz}
\left[\delta^{-ksq_0}\|Q_k(f)\|_{\lp}^{q_0}\right]^{q_1/q_0}\right\}^{(q_0/q_1)(1/q_0)}\\
&\leq\left[\sum_{k \in \zz} \delta^{-ksq_0}\|Q_k(f)\|_{\lp}^{q_0}\right]^{1/q_0}
=\|f\|_{\dot{B}^s_{p,q_0}(\mathcal{X})},
\end{align*}
which proves (i) in the case of $\hb$. The proof of $\hf$
is similar to that of $\hb$; we omit the details.

Property (ii) can be deduced from the Minkowski integral inequality and \eqref{r}. Indeed, we have
\begin{align*}
\|f\|_{\hf}&=\left\{\int_\cx\left[\sum_{k=-\infty}^\infty\delta^{-ksq}
|Q_k(f)(z)|^q\right]^{\frac{p}{q}}\,d\mu(z)\right\}^{\frac{1}{p}}\\
&=\left\{\int_\cx\left[\sum_{k=-\infty}^\infty\delta^{-ksq}
|Q_k(f)(z)|^q\right]^{\frac{p\wedge q}{q}\frac{p}
{p\wedge q}}\,d\mu(z)\right\}^{\frac{p\wedge q}{p}\frac{1}{p\wedge q}}\\
&\leq \left\{\sum_{k=-\infty}^\infty\left[\int_\cx\delta^{-ksp}
|Q_k(f)(z)|^p\,d\mu(z)\right]^{\frac{p\wedge q}{p}}\right\}^{\frac{1}{p\wedge q}}
=\|f\|_{\dot{B}^s_{p,p\wedge q}(\mathcal{X})}.
\end{align*}
This shows the first inequality of (ii).
Now we show the second inequality of (ii).
Indeed, by \eqref{r} and the Minkowski inequality, we have
\begin{align*}
\|f\|_{\dot{B}^s_{p,p\vee q}(\mathcal{X})}
&=\left[\sum_{k\in\zz}\delta^{-ks(p\vee q)}
\|Q_k(f)\|_{L^p(\cx)}^{p\vee q}\right]^{\frac{1}{p\vee q}}
=\left\{\sum_{k\in\zz}\delta^{-ks(p\vee q)}
\left[\int_{\cx}|Q_k(f)(x)|^p\,d\mu(x)\right]^{\frac{p\vee q}{p}}
\right\}^{\frac{1}{p\vee q}}\\
&=\left\{\sum_{k\in\zz}
\left[\int_{\cx}\delta^{-ksp}|Q_k(f)(x)|^p\,d\mu(x)\right]^{\frac{p\vee q}{p}}
\right\}^{\frac{p}{p\vee q}\frac{1}{p}}\\
&\lesssim \left\{\int_{\cx}
\left[\sum_{k\in\zz}\delta^{-ks(p\vee q)}|Q_k(f)(x)|^{p\vee q}\right]^{\frac{p}{p\vee q}}
\,d\mu(x)\right\}^{\frac{1}{p}}\\
&\lesssim \left\{\int_{\cx}
\left[\sum_{k\in\zz}\delta^{-ksq}|Q_k(f)(x)|^q\right]^{\frac{p}{q}}
\,d\mu(x)\right\}^{\frac{1}{p}}
\lesssim \|f\|_{\hf}.
\end{align*}
This finishes the proof of (ii).

To show (iii), suppose $\psi\in \mathring{\cg}(\widetilde{\beta},\widetilde{\gamma})$.
We claim that, for any $k\in \zz_+$ and $y\in \cx$,
\begin{equation}\label{8.26.1}
|Q_k(\psi)(y)|\lesssim \delta^{k\widetilde{\beta}}\|f\|_{\cg(\widetilde{\beta},\widetilde{\gamma})}
\frac{1}{V_1(x_0)+V(x_0,y)}\frac{1}{[1+d(x_0,y)]^{\widetilde{\gamma}}}
\end{equation}
and, for any $k \in \zz\setminus\zz_+$ and $y\in\cx$,
\begin{equation}\label{8.26.1x}
|Q_k(\psi)(y)|\lesssim\delta^{-k\widetilde{\gamma}}\|f\|_{\cg(\widetilde{\beta},\widetilde{\gamma})}
\frac{1}{V_{\delta^k}(x_0)+V(x_0,y)}\frac{\delta^{k\widetilde{\gamma}}}{[\delta^k+d(x_0,y)]^{\widetilde{\gamma}}}.
\end{equation}

We first prove \eqref{8.26.1}. Indeed, notice that,
for any $k\in\zz_+$ and $d(z,y)>(2A_0)^{-1}[1+d(x_0,y)]$, we have
$$\delta^k+d(z,y)\gtrsim 1+d(x_0,y)$$
and
$$V(z,y)\gtrsim V_1(y)+V(x_0,y)\sim V_1(x_0)+V(x_0,y).$$
Therefore, we conclude that,
for any $k\in\zz_+$ and $y\in\cx$,
\begin{align*}
|Q_k(\psi)(y)|
&=\left|\int_\cx Q_k(z,y)[\psi(z)-\psi(y)]\,d\mu(z)\right|\\
&\lesssim\int_{d(z,y)\leq(2A_0)^{-1}[1+d(y,x_0)]}\left|Q_k(z,y)\right|
|\psi(z)-\psi(y)|\,d\mu(z)\\
&\qquad+\int_{d(z,y)\geq(2A_0)^{-1}[1+d(y,x_0)]}
\left|Q_k(z,y)\right||\psi(z)|\,d\mu(z)\\
&\qquad+\int_{d(z,y)\geq(2A_0)^{-1}[1+d(y,x_0)]}
\left|Q_k(z,y)\right||\psi(y)|\,d\mu(z)\\
&=: {\rm J_1+J_2+J_3}.
\end{align*}
For ${\rm J_1}$, by the size condition of $Q_k$ and
the regularity of $\psi$,
we know that, for any $k\in\zz_+$ and  $y\in\cx$,
\begin{align*}
{\rm J_1} &\lesssim \|\psi\|_{\cg(\widetilde{\beta},\widetilde{\gamma})}
\int_{d(z,y)\leq(2A_0)^{-1}[1+d(y,x_0)]} \frac{1}{V_{\delta^k}(y)+V(z,y)}
\left[\frac{\delta^k}{\delta^k+d(z,y)}\right]^{\eta}
\left[\frac{d(z,y)}{1+d(x_0,y)}\right]^{\widetilde{\beta}}\\
&\qquad\times\frac{1}{V_1(x_0)+V(x_0,y)}
\left[\frac{1}{1+d(x_0,y)}\right]^{\widetilde{\gamma}}\,d\mu(z)\\
&\lesssim \|\psi\|_{\cg(\widetilde{\beta},\widetilde{\gamma})}\int_{d(z,y)\leq \delta^k}
\frac{1}{V_{\delta^k}(y)+V(z,y)}
\left[\frac{\delta^k}{\delta^k+d(z,y)}\right]^{\eta}
\left[\frac{d(z,y)}{1+d(x_0,y)}\right]^{\widetilde{\beta}}\\
&\qquad\times\frac{1}{V_1(x_0)+V(x_0,y)}
\left[\frac{1}{1+d(x_0,y)}\right]^{\widetilde{\gamma}}\,d\mu(z)\\
&\qquad+ \|\psi\|_{\cg(\widetilde{\beta},\widetilde{\gamma})}\int_{\delta^k<d(z,y)
\leq (2A_0)^{-1}[1+d(y,x_0)]}\cdots\\
&=: {\rm J_{1,1}+J_{1,2}}.
\end{align*}
Observe that
\begin{align*}
{\rm J_{1,1}}&\lesssim \|\psi\|_{\cg(\widetilde{\beta},\widetilde{\gamma})}
\frac{1}{V_1(x_0)+V(x_0,y)}
\left[\frac{1}{1+d(x_0,y)}\right]^{\widetilde{\gamma}}
\int_{d(z,y)\leq \delta^k} \frac{1}{V_{\delta^k}(y)}\delta^{k\widetilde{\beta}}\,d\mu(z)\\
&\sim \delta^{k\widetilde{\beta}}\|\psi\|_{\cg(\widetilde{\beta},\widetilde{\gamma})}
\frac{1}{V_1(x_0)+V(x_0,y)}
\left[\frac{1}{1+d(x_0,y)}\right]^{\widetilde{\gamma}}
\end{align*}
and, from $\widetilde{\beta}\in(0,\eta)$,  \eqref{6.14.1} and
Lemma \ref{6.15.1}(iii), it follows that
\begin{align*}
{\rm J_{1,2}}&\lesssim \|\psi\|_{\cg(\widetilde{\beta},\widetilde{\gamma})}
\frac{1}{V_1(x_0)+V(x_0,y)}
\left[\frac{1}{1+d(x_0,y)}\right]^{\widetilde{\gamma}}\\
&\qquad\times\int_{\delta^k<d(z,y)
\leq (2A_0)^{-1}[1+d(y,x_0)]}
\frac{1}{V(z,y)}\left[\frac{\delta^k}{d(z,y)}\right]^{\eta}
\left[\frac{d(z,y)}{\delta^k}\right]^{\widetilde{\beta}}
\delta^{k\widetilde{\beta}}\,d\mu(z)\\
&\lesssim \delta^{k\widetilde{\beta}}\|\psi\|_{\cg(\widetilde{\beta},\widetilde{\gamma})}
\frac{1}{V_1(x_0)+V(x_0,y)}
\left[\frac{1}{1+d(x_0,y)}\right]^{\widetilde{\gamma}}
\int_{\delta^k<d(z,y)}\frac{1}{V(z,y)}
\left[\frac{\delta^k}{d(z,y)}\right]^{\eta-\widetilde{\beta}}\,d\mu(z)\\
&\lesssim \delta^{k\widetilde{\beta}}\|\psi\|_{\cg(\widetilde{\beta},\widetilde{\gamma})}
\frac{1}{V_1(x_0)+V(x_0,y)}
\left[\frac{1}{1+d(x_0,y)}\right]^{\widetilde{\gamma}}.
\end{align*}
Moreover, by the size conditions of $Q_k$ and $\psi$, and Lemma \ref{6.15.1}(ii), we conclude that
\begin{align*}
{\rm J_2} &\lesssim \|\psi\|_{\cg(\widetilde{\beta},\widetilde{\gamma})}
\int_{d(z,y) >(2A_0)^{-1}[1+d(y,x_0)]} \frac{1}{V_{\delta^k}(y)+V(z,y)}
\left[\frac{\delta^k}{\delta^k+d(z,y)}\right]^{\eta}\\
&\qquad\times\frac{1}{V_1(x_0)+V(x_0,z)}
\left[\frac{1}{1+d(x_0,z)}\right]^{\widetilde{\gamma}}\,d\mu(z)\\
&\lesssim \delta^{k\eta}\|\psi\|_{\cg(\widetilde{\beta},\widetilde{\gamma})}
\frac{1}{V_1(x_0)+V(x_0,y)}
\left[\frac{1}{1+d(x_0,y)}\right]^{\eta}
\end{align*}
and, by the size conditions of $Q_k$ and $\psi$,
and Lemma \ref{6.15.1}(iii), we conclude that
\begin{align*}
{\rm J_3} &\lesssim \|\psi\|_{\cg(\widetilde{\beta},\widetilde{\gamma})}
\int_{d(z,y) >(2A_0)^{-1}[1+d(y,x_0)]} \frac{1}{V_{\delta^k}(y)+V(z,y)}
\left[\frac{\delta^k}{\delta^k+d(z,y)}\right]^{\eta}\\
&\qquad\times\frac{1}{V_1(x_0)+V(x_0,y)}
\left[\frac{1}{1+d(x_0,y)}\right]^{\widetilde{\gamma}}\,d\mu(z)\\
&\lesssim \delta^{k\eta}\|\psi\|_{\cg(\widetilde{\beta},\widetilde{\gamma})}
\frac{1}{V_1(x_0)+V(x_0,y)}
\left[\frac{1}{1+d(x_0,y)}\right]^{\widetilde{\gamma}}\\
&\qquad \times \int_{d(z,y) >(2A_0)^{-1}[1+d(y,x_0)]}
\frac{1}{V(z,y)}\left[\frac{1}{d(z,y)}\right]^{\eta}\,d\mu(z)\\
&\lesssim \delta^{k\eta}\|\psi\|_{\cg(\widetilde{\beta},\widetilde{\gamma})}
\frac{1}{V_1(x_0)+V(x_0,y)}
\left[\frac{1}{1+d(x_0,y)}\right]^{\widetilde{\gamma}}.
\end{align*}
Combining the estimates  as above, we know that \eqref{8.26.1} holds true.

Now we show \eqref{8.26.1x}.
Since $\psi\in\mathring{\cg}(\widetilde{\beta},\widetilde{\gamma})$,
it then follows that, for any $k\in\zz\setminus\zz_+$ and $y\in\cx$,
\begin{align*}
|Q_k(\psi)(y)|
&=\left|\int_\cx\lf[Q_k(z,y)-Q_k(x_0,y)\r]\psi(z)\,d\mu(z)\right|\\
&\lesssim \int_{d(z,x_0)\leq(2A_0)^{-1}[1+d(y,x_0)]}
\left|Q_k(z,y)-Q_k(x_0,y)\right|
|\psi(z)|\,d\mu(z)\\
&\qquad+\int_{d(z,x_0)\geq(2A_0)^{-1}[1+d(y,x_0)]}
\left|Q_k(z,y)\right||\psi(z)|\,d\mu(z)\\
&\qquad+\int_{d(z,x_0)\geq(2A_0)^{-1}[1+d(y,x_0)]}
\left|Q_k(x_0,y)\right||\psi(z)|\,d\mu(z)\\
&=: {\rm J_4+J_5+J_6}.
\end{align*}
For ${\rm J_4}$, from the regularity of $Q_k$,
the  size condition of $\psi$, $\widetilde{\gamma}\in(0,\eta)$,
\eqref{6.14.1}  and Lemma \ref{6.15.1}(iii), we deduce that
\begin{align*}
{\rm J_4}&\lesssim\|\psi\|_{\cg(\widetilde{\beta},\widetilde{\gamma})}\int_{d(z,x_0)\leq(2A_0)^{-1}[1+d(y,x_0)]}
\left[\frac{d(z,x_0)}{\delta^k+d(x_0,y)}\right]^{\eta}
\frac{1}{V_{\delta^k}(x_0)+V(x_0,y)}
\left[\frac{\delta^k}{\delta^k+d(x_0,y)}\right]^{\eta}\\
&\qquad\times\frac{1}{V_1(x_0)+V(x_0,z)}
\left[\frac{1}{1+d(x_0,z)}\right]^{\widetilde{\gamma}}\,d\mu(z)\\
&\lesssim \delta^{-k\widetilde{\gamma}}\|\psi\|_{\cg(\widetilde{\beta},\widetilde{\gamma})}
\frac{1}{V_{\delta^k}(x_0)+V(x_0,y)}
\left[\frac{\delta^k}{\delta^k+d(x_0,y)}\right]^{\eta}\\
&\qquad\times\int_{d(z,x_0)\leq(2A_0)^{-1}[1+d(y,x_0)]}
\frac{1}{V(x_0,z)}
\left[\frac{d(z,x_0)}{\delta^k+d(x_0,y)}\right]^{\eta-\widetilde{\gamma}}\,d\mu(z)\\
&\lesssim \delta^{-k\widetilde{\gamma}}\|\psi\|_{\cg(\widetilde{\beta},\widetilde{\gamma})}
\frac{1}{V_{\delta^k}(x_0)+V(x_0,y)}
\left[\frac{\delta^k}{\delta^k+d(x_0,y)}\right]^{\eta}.
\end{align*}
Using some arguments similar to those used in the
estimations of ${\rm J_2}$ and ${\rm J_3}$, we conclude that
$${\rm J_5+J_6}\lesssim \delta^{-k\widetilde{\gamma}}\|\psi\|_{\cg(\widetilde{\beta},\widetilde{\gamma})}
\frac{1}{V_{\delta^k}(x_0)+V(x_0,y)}\left[\frac{\delta^k}
{\delta^k+d(x_0,y)}\right]^{\widetilde{\gamma}},$$
which, together with the estimate of ${\rm J_1}$, then completes the proof of \eqref{8.26.1x}.

From \eqref{6.14.1} and $\widetilde{\gamma}\in(\omega(1/p-1)_+,\eta)$, it follows that
\begin{align}
&\left\{\int_\cx\frac{1}{[V_1(x_0)+V(x_0,y)]^p}\frac{1}{[1+d(x_0,y)]^{\widetilde{\gamma} p}}\,d\mu(y)\right\}^{1/p}\\
&\quad\lesssim\left\{\frac{1}{[V_1(x_0)]^{p-1}}+\sum_{i=0}^\infty
\frac{1}{[V_{\delta^{-i}}(x_0)]^{p-1}}\delta^{i\widetilde{\gamma} p}\right\}^{1/p}
\lesssim \frac{1}{[V_1(x_0)]^{1-1/p}}\notag
\end{align}
and, for any $k \in \zz\setminus\zz_+$,
\begin{align}\label{8.26.2}
&\left\{\int_\cx\frac{1}{[V_{\delta^k}(x_0)+V(x_0,y)]^p}
\frac{\delta^{k\widetilde{\gamma} p}}{[\delta^k+d(x_0,y)]^{\widetilde{\gamma} p}}\,d\mu(y)\right\}^{1/p}\\
&\quad\lesssim\left\{\frac{1}{[V_{\delta^k}(x_0)]^{p-1}}
+\sum_{i=0}^\infty\frac{1}{[V_{\delta^{-i}\delta^k}(x_0)]^{p-1}}\delta^{i\widetilde{\gamma} p}\right\}^{1/p}
\lesssim \delta^{k\omega(1/p-1)_+}\left[V_1(x_0)\right]^{1/p-1}.\notag
\end{align}

Observe that, by \eqref{6.14.1}, we have $\widetilde{\beta}\in(s,\eta)$
and $\widetilde{\gamma}\in(\omega(1/p-1)_+-s,\eta)$.
From this, \eqref{8.26.1} through \eqref{8.26.2}, we deduce that
$$\|f\|_{\hb}\lesssim\|f\|_{\cg(\widetilde{\beta},\widetilde{\gamma})}[V_1(x_0)]^{1/p-1}
\left\{\sum_{k=0}^\infty\delta^{k(\widetilde{\beta}-s)q}
+\sum_{k=-\infty}^{-1}\delta^{k[\omega(1/p-1)_+-s-\widetilde{\gamma}]q}\right\}^{1/q}
\lesssim\|f\|_{\cg(\widetilde{\beta},\widetilde{\gamma})}.$$
Therefore, $\mathring{\cg}(\widetilde{\beta},\widetilde{\gamma})\subset \hb$, which,
combined with (ii), implies that $\mathring{\cg}(\widetilde{\beta},\widetilde{\gamma}) \subset \hf$.
This finishes the proof of (iii) and hence of Proposition \ref{phb}.
\end{proof}

\begin{remark}\label{addre2}
A counterpart on RD-spaces of Proposition \ref{phb}(iii) is
\cite[Proposition 5.10(iv)]{hmy08} in which $\gz$ is required to satisfy:
\begin{equation}\label{eq-grd}
\gamma\in\left(\max\left\{\omega\lf(\frac{1}{p}-1\r)_+, -s+\omega\left(\frac{1}{p}-1\right)_+
-\kappa\left(1-\frac{1}{p}\right)_+\right\}, \eta\right),
\end{equation}
where $\kappa\in(0,\omega]$ is as in \eqref{eq-rdoub}. As we mention in the introduction of this article,
a space $\cx$ of homogeneous type can be  formally regarded as a special RD-space satisfying
\eqref{eq-rdoub} with $\kz:=0$. Thus, if we let $\kz:=0$ in \eqref{eq-grd}, then it coincides with
\eqref{6.14.1}. In this sense, we may say that Proposition \ref{phb}(iii) is a generalization of
\cite[Proposition 5.10(iv)]{hmy08} and that the range of
$\widetilde{\gamma}$ in \eqref{6.14.1} is the \emph{optimal range} of $\widetilde{\gz}$.
\end{remark}

Using Proposition \ref{phb}, we now show that the spaces $\hb$ and $\hf$ are independent of the
choice of the spaces of distributions. More precisely, we have the following proposition.

\begin{proposition}\label{6.5.1}
Let $\bz,\ \gz\in(0,\eta)$ with $\eta$ as in Definition \ref{10.23.2},
$s\in(-(\beta\wedge\gamma),\,\beta\wedge\gamma)$
and $p,\ q\in(0,\fz]$ satisfy
\begin{equation}\label{6.14.1x}
\beta\in\left(\max\left\{0,-s+\omega\left(\frac{1}{p}-1\right)_+\right\},\eta\right)\quad \textit{and}\quad
\gamma\in\left(\max\left\{s,\omega\left(\frac{1}{p}-1\right)_+\right\},\eta\right).
\end{equation}
\begin{enumerate}
\item[{\rm (i)}] If $p\in(p(s,\beta\wedge\gamma),\infty]$
with $p(s,\beta\wedge\gamma)$ as in \eqref{pseta},
then, for any $f\in\dot B^s_{p,q}(\cx)\subset(\mathring\cg_0^\eta(\bz,\gz))'$, $f\in(\cggt)'$ with
$\widetilde{\beta}$ and $\widetilde{\gamma}$ satisfying $s\in(-(\widetilde{\beta}\wedge\widetilde{\gamma}),\,\widetilde{\beta}\wedge\widetilde{\gamma})$, $p\in(p(s,\widetilde{\beta}\wedge\widetilde{\gamma}),\infty]$
with $p(s,\widetilde{\beta}\wedge\widetilde{\gamma})$
as in \eqref{pseta} via replacing $\beta$ and $\gamma$ respectively
by $\widetilde{\beta}$ and $\widetilde{\gamma}$,
and \eqref{6.14.1x} via replacing $\beta$ and $\gamma$ respectively
by $\widetilde{\beta}$ and $\widetilde{\gamma}$, and there exists a positive constant $C$,
independent of $f$, such that
$$
\|f\|_{(\cggt)'}\leq C\|f\|_{\hb}.
$$
\item[{\rm (ii)}] If $p\in(p(s,\beta\wedge\gamma),\infty)$ and
$q\in(p(s,\beta\wedge\gamma),\infty]$,
then, for any $f\in\dot F^s_{p,q}(\cx)\subset(\mathring\cg_0^\eta(\bz,\gz))'$, $f\in(\cggt)'$ with
$\widetilde{\beta}$ and $\widetilde{\gamma}$ satisfying $s\in(-(\widetilde{\beta}\wedge\widetilde{\gamma}),
\,\widetilde{\beta}\wedge\widetilde{\gamma})$,
$p\in(p(s,\widetilde{\beta}\wedge\widetilde{\gamma}),\infty),\
q\in(p(s,\widetilde{\beta}\wedge\widetilde{\gamma}),\infty]$
with $p(s,\widetilde{\beta}\wedge\widetilde{\gamma})$
as in \eqref{pseta} via replacing $\beta$ and $\gamma$ respectively
by $\widetilde{\beta}$ and $\widetilde{\gamma}$,
and \eqref{6.14.1x} via replacing $\beta$ and $\gamma$ respectively
by $\widetilde{\beta}$ and $\widetilde{\gamma}$, and there exists a positive constant $C$,
independent of $f$, such that
$$
\|f\|_{(\cggt)'}\leq C\|f\|_{\hf}.
$$
\end{enumerate}
\end{proposition}

\begin{proof}
Let $\{Q_k\}_{k\in\zz}$ be an exp-ATI, $\psi\in\mathring{\cg}(\eta,\eta)$ and  $x_0$ a fixed point in $\cx$.
Assume that $f\in(\cggi)'$.
Choosing $\beta_0,\ \gamma_0\in(\max\{\beta,\widetilde{\beta},\gamma,\widetilde{\gamma}\}, \eta)$,
then $f\in(\mathring{\cg}_0^\eta(\beta_0,\gamma_0))'$.
By Lemma \ref{crf}, we know that there exists
$\{\widetilde{Q}_k\}_{k=-\infty}^\infty$ such that
$$f(\cdot) = \sum_{k=-\infty}^\infty\sum_{\alpha \in \ca_k}
\sum_{m=1}^{N(k,\alpha)}\mu\left(\qa\right)\widetilde{Q}_k(\cdot,\ya)Q_kf
\left(\ya\right) \qquad \text{in}\quad (\mathring{\cg}_0^\eta(\beta_0,\gamma_0))'.$$
Using some arguments similar to those used in the estimations of \eqref{8.26.1} and
\eqref{8.26.1x}, we conclude that, for any $k\in \zz_+$ and $y\in\cx$,
\begin{equation}\label{6.14.2}
\left|\left\langle\widetilde{Q}_k(\cdot, y),\psi\right\rangle\right|
\lesssim\delta^{k\widetilde{\beta}}\|\psi\|_{\cg(\widetilde{\beta},\widetilde{\gamma})}
\frac{1}{V_1(x_0)+V(x_0,y)}\left[\frac{1}{1+d(x_0,y)}\right]^{\widetilde{\gamma}}
\end{equation}
and, for any $k \in \zz\setminus\zz_+$ and $y\in\cx$,
\begin{equation}\label{6.14.3}
\left|\left\langle\widetilde{Q}_k(\cdot, y),\psi\right\rangle\right|
\lesssim\delta^{-k\widetilde{\gamma}}\|\psi\|_{\cg(\widetilde{\beta},\widetilde{\gamma})}
\frac{1}{V_{\delta^k}(x_0)+V(x_0,y)}\left[\frac{\delta^{k\widetilde{\gamma}}}
{\delta^k+d(x_0,y)}\right]^{\widetilde{\gamma}}.
\end{equation}

By Lemma \ref{crf}, \eqref{6.14.2}, \eqref{6.14.3}
and Lemma \ref{6.15.1}, we know that
\begin{align}\label{6.25.1}
|\langle f,\psi\rangle|&=\left|\sum_{k=-\infty}^\infty
\sum_{\alpha\in\ca_k}\sum_{m=1}^{N(k,\alpha)}
\mu\left(\qa\right)Q_k(f)\left(\ya\right)
\left\langle \widetilde{Q}_k(\cdot,\ya),\psi\right\rangle\right|\\
&\lesssim \|\psi\|_{\cg(\widetilde{\beta},\widetilde{\gamma})}
\left\{\sum_{k=0}^\infty\sum_{\alpha\in\ca_k}
\sum_{m=1}^{N(k,\alpha)}\delta^{k\widetilde{\beta}}\mu\left(\qa\right)
\left|Q_k(f)\left(\ya\right)\right|
\right.\notag\\
&\qquad\times \frac{1}{V_1(x_0)+V(x_0,\ya)}
\left[\frac{1}{1+d(x_0,\ya)}\right]^{\widetilde{\gamma}}\notag\\
&\qquad+\sum_{k=-\infty}^{-1}\sum_{\alpha\in\ca_k}
\sum_{m=1}^{N(k,\alpha)}\delta^{-k\widetilde{\gamma}}\mu\left(\qa\right)
\left|Q_k(f)\left(\ya\right)\right|\noz\\
&\qquad\times\lf.
\frac{1}{V_{\delta^k}(x_0)+V(x_0,\ya)}
\left[\frac{\delta^k}{\delta^k+d(x_0,\ya)}\right]^{\widetilde{\gamma}}\right\}.
\notag
\end{align}
If $p\in (p(s,\widetilde{\beta}\wedge\widetilde{\gamma}),1]$, by \eqref{r}, we find that
\begin{align}\label{6.23.1}
|\langle f,\psi\rangle|&\lesssim \|\psi\|_{\cg(\widetilde{\beta},\widetilde{\gamma})}
\left[\sum_{k=0}^\infty\delta^{k(\widetilde{\beta}+s)}\left\{\sum_{\alpha\in\ca_k}
\sum_{m=1}^{N(k,\alpha)}\delta^{-ksp}
\left[\mu\left(\qa\right)\right]^p
\left|Q_k(f)\left(\ya\right)\right|^p\right.\right.\\
&\qquad \times \left.
\vphantom{\sum_{\alpha\in\ca_k}\sum_{m=1}^{N(k,\alpha)}}\left[\frac{1}{V_1(x_0)+V(x_0,\ya)}\right]^p
\left[\frac{1}{1+d(x_0,\ya)}\right]^{\widetilde{\gamma}p}\right\}^{1/p}
\notag\\
&\qquad+\sum_{k=-\infty}^{-1}\delta^{-k(\widetilde{\gamma}-s)}
\left\{\sum_{\alpha\in\ca_k}\sum_{m=1}^{N(k,\alpha)}\delta^{-ksp}
\left[\mu\left(\qa\right)\right]^p
\left|Q_k(f)\left(\ya\right)\right|^p\right.\notag\\
&\qquad\left.\left.\times
\left[\frac{1}{V_{\delta^k}(x_0)+V(x_0,\ya)}\right]^p
\left[\frac{\delta^k}{\delta^k+d(x_0,\ya)}\right]^{\widetilde{\gamma}p}
\right\}^{1/p}\right].\notag
\end{align}
Observe that, from Lemma \ref{6.15.1}, $p \in (p(s,\widetilde{\beta}\wedge\widetilde{\gamma}),1]$,
$\widetilde{\gamma}\in(\omega(1/p-1),\eta)$, and \eqref{6.14.1x}
via replacing $\gamma$ by $\widetilde{\gamma}$, we deduce that,
for any $k\in\zz_+$, $\alpha\in\ca_k$ and $m\in\{1,\dots,N(k,\alpha)\}$,
\begin{align}\label{6.22.1}
&\frac{[\mu(\qa)]^{p-1}}{[V_1(x_0)+V(x_0,\ya)]^p}
\left[\frac{1}{1+d(x_0,\ya)}\right]^{\widetilde{\gamma}p}\\
&\quad\lesssim\frac{1}{V_1(x_0)}\left[\frac{V_1(\ya)+V(x_0,\ya)}{\mu(\qa)}\right]^{1-p}
\left[\frac{1}{1+d(x_0,\ya)}\right]^{\widetilde{\gamma}p}
\lesssim\delta^{-k\omega(1-p)}\frac{1}{V_1(x_0)};\notag
\end{align}
similarly, for any $k\in\zz\setminus\zz_+$,  we have $V_1(x_0)\lesssim V_{\delta^k}(x_0)$,
which further implies that,
for any $\alpha\in\ca_k$ and $m\in\{1,\dots,N(k,\alpha)\}$,
\begin{align}\label{6.22.2}
&\frac{[\mu(\qa)]^{p-1}}{[V_{\delta^k}(x_0)+V(x_0,\ya)]^p}
\left[\frac{\delta^k}
{\delta^k+d(x_0,\ya)}\right]^{\widetilde{\gamma}p}\\
&\quad\lesssim \frac{1}{V_1(x_0)}\left[\frac{V_{\delta^k}(\ya)+V(x_0,\ya)}{\mu(\qa)}\right]^{1-p}
\left[\frac{\delta^k}
{\delta^k+d(x_0,\ya)}\right]^{\widetilde{\gamma}p}
\lesssim\frac{1}{V_1(x_0)}.\notag
\end{align}
Using this, \eqref{6.23.1}, \eqref{6.22.1}, \eqref{bhpp}, \eqref{6.14.1x}, the H\"older inequality
when $q\in (1,\infty)$, or \eqref{r} when $q\in(0,1]$,
we know that
\begin{align}\label{6.25.2}
|\langle f,\psi\rangle|&\ls\frac{\|\psi\|_{\cg(\widetilde{\beta},\widetilde{\gamma})}}{[V_1(x_0)]^{1/p}}
\left\{\sum_{k=0}^\infty\delta^{-k[\omega(1/p-1)-(\widetilde{\beta}+s)]}\left[\sum_{\alpha\in\ca_k}
\sum_{m=1}^{N(k,\alpha)}\delta^{-ksp}\mu\left(\qa\right)
\left|Q_k(f)\left(\ya\right)\right|^p\right]^{1/p}\right.\\
&\quad\quad\left.+\sum_{k=-\infty}^{-1}\delta^{-k(\widetilde{\gamma}-s)}\left[\sum_{\alpha\in\ca_k}
\sum_{m=1}^{N(k,\alpha)}\delta^{-ksp}\mu\left(\qa\right)
\left|Q_k(f)\left(\ya\right)\right|^p\right]^{1/p}\right\}\notag\\
&\lesssim \frac{1}{[V_1(x_0)]^{1/p}}\|\psi\|_{\cg(\widetilde{\beta},\widetilde{\gamma})}\|f\|_{\hb}. \notag
\end{align}

If $p\in(1,\infty)$, then, by \eqref{6.25.1}, \eqref{6.14.1x}, Lemma \ref{6.15.1}(ii) and the H\"older
inequality, we conclude that
\begin{align}\label{6.25.3}
|\langle f,\psi\rangle|
&\lesssim\|\psi\|_{\cg(\widetilde{\beta},\widetilde{\gamma})}\left\{\sum_{k=0}^\infty\sum_{\alpha\in\ca_k}
\sum_{m=1}^{N(k,\alpha)}\delta^{-ks}
\left[\mu\left(\qa\right)\right]^{1/p}
\left|Q_k(f)\left(\ya\right)\right|\right.\\
&\qquad \times\delta^{k(\widetilde{\beta}+s)}
\left[\mu\left(\qa\right)\right]^{1/p'}
\frac{1}{V_1(x_0)+V(x_0,\ya)}
\left[\frac{1}{1+d(x_0,\ya)}\right]^{\widetilde{\gamma}}\notag\\
&\qquad+\sum_{k=-\infty}^{-1}\sum_{\alpha\in\ca_k}\sum_{m=1}^{N(k,\alpha)}\delta^{-ks}
\left[\mu\left(\qa\right)\right]^{1/p}
\left|Q_k(f)\left(\ya\right)\right|\notag\\
&\qquad\left. {}\times \delta^{k(-\widetilde{\gamma}+s)}
\left[\mu\left(\qa\right)\right]^{1/p'}
\frac{1}{V_{\delta^k}(x_0)+V(x_0,\ya)}
\left[\frac{\delta^k}
{\delta^k+d(x_0,\ya)}\right]^{\widetilde{\gamma}}\right\}.\notag\\
&\lesssim \frac{1}{[V_1(x_0)]^{1/p}}\|\psi\|_{\cg(\widetilde{\beta},\widetilde{\gamma})}
\left\{\sum_{k=0}^\infty\delta^{k(\widetilde{\beta}+s)}\left[\int_\cx
\frac{1}{V_1(x_0)+V(x_0,y)}\right.\right.\notag\\
&\qquad \times \left.\left\{\frac{1}{1+d(x_0,y)}\right\}^{\widetilde{\gamma}p'}\,d\mu(y)\right]^{1/p'}
\sum_{\alpha\in\ca_k}\sum_{m=1}^{N(k,\alpha)}
\left[\delta^{-ksp}\mu\left(\qa\right)
\left|Q_k(f)\left(\ya\right)\right|^p\right]^{1/p}\notag\\
&\qquad+\sum_{k=-\infty}^{-1}\delta^{k(-\widetilde{\gamma}+s)}
\left[\int_\cx\frac{1}{V_{\delta^k}(x_0)+V(x_0,y)}
\left\{\frac{\delta^{k\widetilde{\gamma}p'}}{\delta^k+d(x_0,y)}\right\}^{\widetilde{\gamma}p'}
\,d\mu(y)\right]^{1/p'}\notag\\
&\qquad\left.\times \sum_{\alpha\in\ca_k}\sum_{m=1}^{N(k,\alpha)}
\left[\delta^{-ksp}\mu\left(\qa\right)
\left|Q_k(f)\left(\ya\right)\right|^p
\right]^{1/p}\right\}
\notag\\
&\lesssim \frac{1}{[V_1(x_0)]^{1/p}}\|\psi\|_{\cg(\widetilde{\beta},\widetilde{\gamma})}\|f\|_{\hb}. \notag
\end{align}

To summarize, we conclude that, for any $\psi\in\mathring{\cg}(\eta,\eta)$,
$$
|\langle f,\psi\rangle|\ls\frac{1}{[V_1(x_0)]^{1/p}}\|\psi\|_{\cg(\widetilde{\beta},\widetilde{\gamma})}\|f\|_{\hb}.
$$

Now let $h\in \mathring{\cg}_0^\eta(\widetilde{\beta},\widetilde{\gamma})$. Then there exists a sequence
$\{h_n\}_{n=1}^\infty\subset\mathring{\cg}(\eta,\eta)$ such that $\|h-h_n\|_{\cg(\widetilde{\beta},\widetilde{\gamma})}\to 0$ as
$n\to \infty$. By \eqref{6.25.2} and \eqref{6.25.3}, we know that, for any $m,\ n\in\nn$,
$$|\langle f,h_m-h_n\rangle|\lesssim\|f\|_{\hb}\|h_m-h_n\|_{\cg(\widetilde{\beta},\widetilde{\gamma})},$$
which implies that $\lim_{n\to \infty}\langle f,h_n\rangle$ exists. Therefore, define
$$\langle f, h \rangle := \lim_{n\to \infty}\langle f, h_n \rangle.$$
It is obvious that $\langle f, h \rangle$ is independent of the choice of $\{h_n\}_{n=1}^\infty$.
Moreover, by \eqref{6.25.2} and \eqref{6.25.3}, we conclude that
$$|\langle f, h \rangle| = \lim_{n\to \infty}|\langle f, h_n \rangle|
\lesssim \lim_{n\to \infty}\|f\|_{\hb}\|h_n\|_{\cg(\widetilde{\beta},\widetilde{\gamma})}
\sim\|f\|_{\hb}\|h\|_{\cg(\widetilde{\beta},\widetilde{\gamma})},$$
which implies that $f\in(\cggii)'$ and hence completes the proof of (i).

Property (ii) can be deduced from (i) and Proposition \ref{phb}(ii), and we omit the details.
This finishes the proof  of Proposition \ref{6.5.1}.
\end{proof}

\begin{remark}\label{addre1}
From Propositions \ref{6.4.1} and \ref{6.5.1}, we deduce that
the spaces $\hb$ and $\hf$ are independent of the choices of both exp-ATIs
and the distribution spaces $(\cggi)'$ if $\beta$, $\gamma$ satisfy \eqref{6.14.1}.
Recall that, in \cite[Proposition 5.7]{hmy08}  on $\hb$ and $\hf$ with $\cx$ being an RD-space,
$\bz$ and $\gz$ are required to satisfy
\begin{equation}\label{betagamma}
\beta\in\left(\max\left\{0,-s+\omega\lf(\frac{1}{p}-1\r)_+\right\},\eta\right)\quad\hbox{and}\quad
\gamma\in\left(\max\left\{s-\frac{\kappa}{p},\omega\lf(\frac{1}{p}-1\r)_+\right\}, \eta\right),
\end{equation}
where $\kappa\in(0,\omega]$ is as in \eqref{eq-rdoub}. Similarly to Remark \ref{addre2}, we find that,
when $\kz:=0$ in \eqref{betagamma}, the above inequality coincides with the range
of $\widetilde{\gamma}$ in Proposition \ref{6.5.1}(iii). In this sense, we may say that Proposition
\ref{6.5.1} generalizes \cite[Proposition 5.7]{hmy08} and that the range
of $\widetilde{\gamma}$ in Proposition \ref{6.5.1} is \emph{optimal}.
\end{remark}

\section{Inhomogeneous Besov and Triebel--Lizorkin spaces}\label{s3}

In this section, we introduce inhomogeneous Besov and Triebel--Lizorkin spaces on
spaces of homogeneous type. Moreover, in this section, $\mu(\cx)$ can be finite or infinite.

For any measurable set $E\subset \cx$ with $\mu(E)\in(0,\fz)$
and non-negative function $f$, let
\begin{equation}\label{int_m}
m_E(f):=\frac{1}{\mu(E)}\int_E f(x)\,d\mu(x).
\end{equation}
Now we introduce the inhomogeneous Besov and Triebel--Lizorkin spaces on spaces of homogeneous type.

\begin{definition}\label{ih}
Let  $\beta,\ \gamma \in (0, \eta)$ and $s\in(-\eta,\eta)$
with $\eta$ as in
Definition \ref{10.23.2}. Let $\{Q_k\}_{k\in\zz}$ be an exp-IATI and $N\in\nn$ as in Lemma \ref{icrf}.
\begin{enumerate}
\item[\rm{(i)}] Let $p\in(p(s,\beta\wedge\gamma),\infty]$ with $p(s,\beta\wedge\gamma)$ as in \eqref{pseta},
and $q \in (0,\infty]$. The
\emph{inhomogeneous Besov space $\ihb$} is defined by setting
\begin{align*}
\ihb :={}& \left\{f  \in (\icgg)' :\  \|f\|_{\ihb}:=
\left\{\sum_{k=0}^N\sum_{\alpha \in \ca_k}\sum_{m=1}^{N(k,\alpha)}
\mu\left(\qa\right)\left[m_{\qa}\left(|Q_k(f)|\right)\r]^p\right\}^{1/p}\right.\\
&\qquad\qquad\qquad\qquad\qquad\qquad\qquad\left.+\left[\sum_{k=N+1}^{\infty} \delta^{-ksq}
\|Q_k(f)\|_{\lp}^q\right]^{1/q}<\infty\right\}
\end{align*}
with the usual modifications made when
$p=\infty$ or  $q=\infty$.
\item[\rm{(ii)}] Let $p\in(p(s,\beta\wedge\gamma),\infty)$ and
$q \in (p(s,\beta\wedge\gamma),\infty]$. The \emph{inhomogeneous
Triebel--Lizorkin space $\ihf$} is defined by setting
\begin{align*}
\ihf :={}& \left\{f  \in (\icgg)' :\  \|f\|_{\ihf}:=\left\{\sum_{k=0}^N
\sum_{\alpha \in \ca_k}\sum_{m=1}^{N(k,\alpha)}\mu\left(\qa\right)\left[m_{\qa}
\left(|Q_k(f)|\right)\r]^p\right\}^{1/p}\r.\\
&\qquad\qquad\qquad\qquad\qquad\qquad\qquad\left.+\left\|\left[\sum_{k=N+1}^\infty \delta^{-ksq}|Q_k(f)|^q\right]^{1/q}\right\|_{\lp}<\infty\right\}
\end{align*}
with usual modification made when $q=\infty$.
\end{enumerate}
\end{definition}

Using some arguments similar to those used in the homogeneous case, we first show that the spaces $\ihb$ and
$\ihf$ are independent of the choice of exp-IATIs.
To this end, we need several technical lemmas.

The following lemma is the \emph{inhomogeneous Plancherel--P\^olya inequality}.

\begin{lemma}\label{6.9.1}
Let $\{Q_k\}_{k=0}^\infty$ and $\{P_k\}_{k=0}^\infty$ be two \emph{$\exp$-IATIs}, and
$\beta,\ \gamma\in (0, \eta)$ with
$\eta$ as in Definition \ref{10.23.2}.
Suppose that $N\in\nn$ and $N'\in\nn$ are as in Lemma
\ref{icrf} associated, respectively, with $\{Q_k\}_{k=0}^\fz$ and $\{P_k\}_{k=0}^\fz$.
Then there exists a positive constant $C$ such that, for any $f\in(\cggi)'$,
\begin{align}\label{bihpp}
&\left\{\sum_{k=0}^N\sum_{\alpha \in \ca_k}\sum_{m=1}^{N(k,\alpha)}
\mu\left(\qa\right)\left[m_{\qa}(|Q_k(f)|)\r]^p\right\}^{1/p}\\
&\qquad+\left\{\sum_{k=N+1}^\infty\delta^{-ksq}\left[\sum_{\alpha \in \ca_k}
\sum_{m=1}^{N(k,\alpha)}\mu\left(\qa\right)\left\{\sup_{z\in\qa}|Q_k(f)(z)|\right\}^p\right]^{q/p}\right\}^{1/q}
\notag\\
&\quad\leq C\left\{\sum_{k'=0}^{N'}\sum_{\alpha' \in \ca_{k'}}
\sum_{m'=1}^{N(k',\alpha')}\mu\left(\qap\right)\left[m_{\qap}(|P_{k'}(f)|)\r]^p\right\}^{1/p}\notag\\
&\qquad+\left\{\sum_{k'=N'+1}^\infty\delta^{-k'sq}\left[\sum_{\alpha' \in \ca_{k'}}
\sum_{m'=1}^{N(k',\alpha')}\mu\left(\qap\right)\left\{\inf_{z\in\qap}|P_{k'}(f)(z)|\right\}^p\right]^{q/p}\right\}^{1/q}
\notag
\end{align}
when $s\in(-(\beta\wedge\gamma),\,\beta\wedge\gamma)$, $p\in (p(s,\beta\wedge\gamma),\infty]$ with
$p(s,\beta\wedge\gamma)$ as in \eqref{pseta},  and $q \in (0,\infty]$,
with the usual modifications made when $p=\infty$ or $q=\infty$, and
\begin{align}\label{fihpp}
&\left\{\sum_{k=0}^N\sum_{\alpha \in \ca_k}
\sum_{m=1}^{N(k,\alpha)}\mu\left(\qa\right)\left[m_{\qa}(|Q_k(f)|)\r]^p\right\}^{1/p}\\
&\qquad+\left\|\left\{\sum_{k=N+1}^\infty\sum_{\alpha \in \ca_k}
\sum_{m=1}^{N(k,\alpha)}\delta^{-ksq}\left[\sup_{z\in\qa}|Q_k(f)(z)|\right]^q
\mathbf 1_{\qa}\right\}^{1/q}\right\|_{L^p(\mathcal{X})}\notag\\
&\quad\leq C\left\{\sum_{k'=0}^{N'}\sum_{\alpha' \in \ca_{k'}}
\sum_{m'=1}^{N(k',\alpha')}\mu\left(\qap\right)\left[m_{\qap}(|P_{k'}(f))|\r]^p\right\}^{1/p}\notag\\
&\qquad+\left\|\left\{\sum_{k'=N'+1}^\infty\sum_{\alpha' \in \ca_{k'}}
\sum_{m'=1}^{N(k',\alpha')}\delta^{-k'sq}\left[\inf_{z\in\qap}|P_{k'}(f)(z)|\right]^q
\mathbf 1_{\qap}\right\}^{1/q}\right\|_{L^p(\mathcal{X})}\notag
\end{align}
when $s\in(-(\beta\wedge\gamma),\,\beta\wedge\gamma)$,
$p\in (p(s,\beta\wedge\gamma),\infty)$ and $q \in (p(s,\beta\wedge\gamma),\infty]$,
with the usual modification made when
$q=\fz$.
\end{lemma}

\begin{proof}
Let $\eta$ be as in Definition \ref{10.23.2}. We first prove \eqref{bihpp}. For the first term of the left-hand side of \eqref{bihpp}, by Lemma \ref{icrf}
with $\{Q_k\}_{k\in\zz_+}$ and $\{\widetilde{Q}_k\}_{k\in\zz_+}$
respectively replaced by $\{P_k\}_{k\in\zz_+}$ and $\{\widetilde{P}_k\}_{k\in\zz_+}$, we
know that, for any $k\in\zz_+$ and $f\in(\icgg)'$ with $\beta,\ \gamma\in(0,\eta)$,
\begin{align}\label{10.15.5}
|Q_k(f)(\cdot)|&\leq \sum_{\alpha' \in \ca_0}\sum_{m'=1}^{N(0,\alpha')}
\int_{\qop}\left|Q_k\widetilde{P}_0(\cdot,y)\right|\,d\mu(y)\left|P^{0,m'}_{\alpha',1}(f)\right|\\
&\qquad+\sum_{k'=1}^{N'}\sum_{\alpha' \in \ca_{k'}}\sum_{m'=1}^{N(k',\alpha')}
\mu\left(\qap\right)\left|Q_k\widetilde{P}_{k'}\left(\cdot,\yap\right)\right|\left|P^{k',m'}_{\alpha',1}(f)\right|\notag\\
& \qquad+\sum_{k'=N'+1}^\infty\sum_{\alpha' \in \ca_{k'}}
\sum_{m'=1}^{N(k',\alpha')}\mu\left(\qap\right)
\left|Q_k\widetilde{P}_{k'}\left(\cdot,\yap\right)\right|\left|P_{k'}f\left(\yap\right)\right|.\notag
\end{align}
We claim that, for any $f\in(\cg_0^\eta(\beta,\gamma))'$,
\begin{equation}\label{4.26.1}
\left|\left\langle f, \frac{1}{\mu(\qap)}\int_{\qap}P_{k'}(x,\cdot)\,d\mu(x)\right\rangle\right|
\leq\frac{1}{\mu(\qap)}\int_{\qap}|\langle f,P_{k'}(x,\cdot)\rangle|\,d\mu(x).
\end{equation}
Indeed, if $f\in L^2(\cx)$, then \eqref{4.26.1}
obviously holds true by the Fubini theorem. Now, assume that $f\in (\cg_0^\eta(\beta,\gamma))'$.
Then, by Lemma \ref{icrf}, we know that there
exists a sequence of functions $\{f_n\}_{n=1}^\infty \subset L^2(\cx)$ such that $f_n\to f$ in
$(\cg_0^\eta(\beta,\gamma))'$ as $n\to\infty$. From the Banach--Steinhaus
theorem (see, for instance, \cite[Proposition 3.5]{b11}), we deduce
that $\{\|f_n\|_{(\cg_0^\eta(\beta,\gamma))'}\}_{n\in\nn}$
is bounded. Then, by the Lebesgue dominated convergence theorem, we conclude that
\begin{align*}
&\left|\left\langle f, \frac{1}{\mu(\qap)}\int_{\qap}P_{k'}(x,\cdot)\,d\mu(x)\right\rangle\right|\\
&\quad=\lim_{n\to \infty}
\left|\left\langle f_n, \frac{1}{\mu(\qap)}\int_{\qap}P_{k'}(x,\cdot)\,d\mu(x)\right\rangle\right|\\
&\quad\leq \frac{1}{\mu(\qap)}\lim_{n\to \infty}\int_{\qap}|\langle f_n,P_{k'}(x,\cdot)\rangle|\,d\mu(x)
=\frac{1}{\mu(\qap)}\int_{\qap}\lim_{n\to \infty}|\langle f_n,P_{k'}(x,\cdot)\rangle|\,d\mu(x)\\
&\quad=\frac{1}{\mu(\qap)}\int_{\qap}|\langle f,P_{k'}(x,\cdot)\rangle|\,d\mu(x),
\end{align*}
namely, \eqref{4.26.1} holds true.

From \eqref{4.26.1}, it follows that, for any $k'\in\{0,\dots,N'\}$,
$\alpha'\in\ca_{k'}$ and  $m'\in N(k',\alpha')$,
\begin{align}\label{10.15.6}
\left|P^{k',m'}_{\alpha',1}(f)\right|&=\left|\left\langle f, P^{k',m'}_{\alpha',1}\right\rangle\right|=\left|
\left\langle f, \frac{1}{\mu(\qap)}\int_{\qap}P_{k'}(x,\cdot)\,d\mu(x)\right\rangle\right|\\
&\leq\frac{1}{\mu(\qap)}\int_{\qap}|\langle f,P_{k'}(x,\cdot)\rangle|\,d\mu(x)=m_{\qap}(|P_{k'}(f)|),\notag
\end{align}
which, together with \eqref{10.15.5}, implies that,
for any $k\in\{0,\dots,N\}$, $\alpha \in\ca_k$ and $m\in\{0,\dots,N(k,\alpha)\}$,
\begin{align}\label{12.3.4}
&m_{\qa}(|Q_k(f)|)\\
&\quad= \frac{1}{\mu(\qa)}\int_{\qa}|Q_k(f)(z)|\,d\mu(z)\noz\\
&\quad\leq \sum_{\alpha' \in \ca_0}\sum_{m'=1}^{N(0,\alpha')}m_{\qop}(|P_0(f)|)
\frac{1}{\mu(\qa)}\int_{\qa}\int_{\qop}\left|Q_k\widetilde{P}_{k'}(z,y)\right|\,d\mu(y)\,d\mu(z)\notag\\
&\quad\qquad+\sum_{k'=1}^{N'}\sum_{\alpha' \in \ca_{k'}}\sum_{m'=1}^{N(k',\alpha')}
m_{\qap}(|P_{k'}(f)|)\mu\left(\qap\right)\frac{1}{\mu(\qa)}\int_{\qa}
\left|Q_k\widetilde{P}_{k'}\left(z,\yap\right)\right|\,d\mu(z)\notag\\
&\quad\qquad+\sum_{k'=N'+1}^\infty\sum_{\alpha' \in \ca_{k'}}\sum_{m'=1}^{N(k',\alpha')}
\left|P_{k'}f\left(\yap\right)\right|\mu\left(\qap\right)
\frac{1}{\mu(\qa)}\int_{\qa}\left|Q_k\widetilde{P}_{k'}\left(z,\yap\right)\right|\,d\mu(z)\notag\\
&\quad=:\rm{{\rm I}_1+{\rm I}_2+{\rm I}_3}.\notag
\end{align}

Note that, for any $k\in\{0,\dots,N\}$
and $k'\in\{0,\dots, N'\}$, $\delta^k\sim 1 \sim \delta^{k'}$.
By this,  Lemma \ref{6.15.1}(ii) and the fact that
$$\{s\in\cx:\ d(z,s)\geq(2A_0)^{-1}d(y,z)\}\cup\{s\in\cx:\ d(y,s)\geq(2A_0)^{-1}d(y,z)\}=\cx,$$
we know that, for any $k\in\{0,\dots,N\}$, $k'\in\{0,\dots, N'\}$ and $y,\ z\in\cx$,
\begin{align}\label{3.6x}
\lf|Q_k\widetilde{P}_{k'}(z,y)\r|&=\left|\int_\cx Q_k(z,s)\widetilde{P}_{k'}(s,y)\,d\mu(s)\right|
\leq\int_\cx |Q_k(z,s)|\left|\widetilde{P}_{k'}(s,y)\right|\,d\mu(s)\\
&\lesssim\int_\cx\frac{1}{V_1(z)+V(z,s)}\left[\frac{1}{1+d(z,s)}\right]^\gamma
\frac{1}{V_1(y)+V(y,s)}\left[\frac{1}{1+d(y,s)}\right]^\gamma\,d\mu(s)\notag\\
&\lesssim\frac{1}{V_1(z)+V(z,y)}\left[\frac{1}{1+d(z,y)}\right]^\gamma\notag\\
&\qquad\times\left\{\int_{\{s\in\cx:\ d(z,s)\geq(2A_0)^{-1}d(y,z)\}}
\frac{1}{V_1(y)+V(y,s)}\left[\frac{1}{1+d(y,s)}\right]^\gamma\,d\mu(s)\r.\notag\\
&\qquad\left.+\int_{\{s\in\cx:\ d(y,s)\geq(2A_0)^{-1}d(y,z)\}}
\frac{1}{V_1(z)+V(z,s)}\left[\frac{1}{1+d(z,s)}\right]^\gamma\,d\mu(s)\r\}\notag\\
&\lesssim\frac{1}{V_1(z)+V(z,y)}\left[\frac{1}{1+d(z,y)}\right]^\gamma.\notag
\end{align}
Observe that, for any $k\in\{0,\dots,N\}$
and $k'\in\{0,\dots, N'\}$, by Lemma \ref{6.15.1}(i), we find that, for any $y_1,\ y_2\in\qap$ and
$z_1,\ z_2\in\qa$,
$$V_1(z_1)+V(z_1,y_1)\sim\mu(B(z_1,1+
d(y_1,z_1)))\sim\mu(B(z_2,1+d(y_2,z_2)))\sim V_1(z_2)+V(z_2,y_2).$$
This, combined with \eqref{3.6x}, further implies that
\begin{equation}\label{qkpk}
\sup_{z\in\qa}\sup_{y\in\qap}
\left|Q_k\widetilde{P}_{k'}(z,y)\right|
\lesssim\inf_{z\in\qa}\inf_{y\in\qap}\frac{1}{V_1(z)+V(z,y)}
\left[\frac{1}{1+d(z,y)}\right]^\gamma.
\end{equation}
From this, we deduce that
\begin{align}\label{10.14.1}
\rm{{\rm I}_1+{\rm I}_2}&\lesssim\sum_{k'=0}^{N'}\sum_{\alpha' \in \ca_{k'}}
\sum_{m'=1}^{N(k',\alpha')}m_{\qap}(|P_{k'}(f)|)\mu\left(\qap\right)\\
&\qquad\times\inf_{z\in\qa}\inf_{y\in\qap}\frac{1}{V_1(z)+V(z,y)}\left[\frac{1}{1+d(z,y)}\right]^\gamma.\notag
\end{align}
If $p\in(p(s,\beta\wedge\gamma),1]$, then $p\in(\omega/(\omega+\gamma), 1]$.
This, together with \eqref{10.14.1}, \eqref{r} and  Lemma \ref{9.14.1}, implies that
\begin{align}\label{10.15.2}
&\left[\sum_{k=0}^N\sum_{\alpha \in \ca_k}\sum_{m=1}^{N(k,\alpha)}
\mu\left(\qa\right)({\rm I}_1+{\rm I}_2)^p\right]^{1/p}\\
&\quad\lesssim \left\{\sum_{k=0}^N\sum_{\alpha \in \ca_k}
\sum_{m=1}^{N(k,\alpha)}\mu\left(\qa\right)\sum_{k'=0}^{N'}\sum_{\alpha' \in \ca_{k'}}
\sum_{m'=1}^{N(k',\alpha')}\left[m_{\qap}(|P_{k'}(f)|)\right]^p\left[\mu\left(\qap\right)\right]^p\r.\notag\\
&\qquad\quad\times\left.\inf_{z\in\qa}\inf_{y\in\qap}\left[\frac{1}{V_1(z)+V(z,y)}\r]^p
\left[\frac{1}{1+d(z,y)}\right]^{\gamma p}\right\}^{1/p}\notag\\
&\quad\lesssim \left\{\sum_{k'=0}^{N'}\sum_{\alpha' \in \ca_{k'}}
\sum_{m'=1}^{N(k',\alpha')}\left[m_{\qap}(|P_{k'}(f)|)\right]^p
\left[\mu\left(\qap\right)\right]^p
\left[V_1\left(\yap\right)\right]^{1-p}\right\}^{1/p}\notag\\
&\quad\lesssim\left\{\sum_{k'=0}^{N'}\sum_{\alpha' \in \ca_{k'}}
\sum_{m'=1}^{N(k',\alpha')}\mu\left(\qap\right)\left[m_{\qap}(|P_{k'}(f)|)\right]^p\right\}^{1/p}.\notag
\end{align}
If $p\in(1,\infty]$, from \eqref{10.14.1}, the H\"older inequality and Lemma \ref{9.14.1}, we deduce that
\begin{align*}
\rm{{\rm I}_1+{\rm I}_2}&\lesssim\left\{\sum_{k'=0}^{N'}\sum_{\alpha' \in \ca_{k'}}
\sum_{m'=1}^{N(k',\alpha')}\left[m_{\qap}(|P_{k'}(f)|)\right]^p\mu\left(\qap\right)\r.\\
&\qquad\times\left.\inf_{z\in\qa}\inf_{y\in\qap}\frac{1}{V_1(z)+V(z,y)}
\left[\frac{1}{1+d(z,y)}\right]^\gamma\right\}^{1/p},
\end{align*}
which, combined with Lemma \ref{9.14.1}, further implies that
\begin{align}\label{10.15.3}
&\left[\sum_{k=0}^N\sum_{\alpha \in \ca_k}\sum_{m=1}^{N(k,\alpha)}
\mu\left(\qa\right)({\rm I}_1+{\rm I}_2)^p\right]^{1/p}\\
&\quad\lesssim\left\{\sum_{k'=0}^{N'}\sum_{\alpha' \in \ca_{k'}}
\sum_{m'=1}^{N(k',\alpha')}\mu\left(\qap\right)\left[m_{\qap}(|P_{k'}(f)|)\right]^p\right\}^{1/p}.\notag
\end{align}

For the term $\rm{{\rm I}_3}$, by the cancellation of $\widetilde{P}_{k'}$ and \eqref{pp},
we conclude that, for any fixed $ \eta'\in(0,\beta\wedge\gamma)$, and any $k\in\{0,\dots,N\}$ and
$k'\in\{N'+1,N'+2,\dots\}$,
\begin{equation}\label{qkpkm}
\sup_{z\in \qa}\left|Q_k\widetilde{P}_{k'}(z,\yap)\right|\lesssim\delta^{k'\eta'}
\inf_{z\in\qa}\frac{1}{V_1(z)+V(z,\yap)}\left[\frac{1}{1+d(z,\yap)}\right]^{\gz},
\end{equation}
which implies that
\begin{align}\label{10.14.2}
\rm{{\rm I}_3}&\lesssim\sum_{k'=N'+1}^\infty\delta^{k'\eta'}
\sum_{\alpha' \in \ca_{k'}}\sum_{m'=1}^{N(k',\alpha')}\mu\left(\qap\right)\left|P_{k'}f\left(\yap\right)\right|\\
&\qquad\times\inf_{z\in\qa}\frac{1}{V_1(z)+V(z,\yap)}\left[\frac{1}{1+d(z,\yap)}\right]^{\gz}.\notag
\end{align}
Observe that, for any $\yap\in\qap$, $V_1(\yap)\lesssim\delta^{k'\omega}\mu(\qap)$.
If $p\in(p(s,\beta\wedge\gamma), 1]$,  choose $\eta'\in(0,\bz\wedge\gz)$ such
that $s\in(-\eta',\eta')$. From this, \eqref{r}, Lemma \ref{9.14.1},
$p\in(\omega/[\omega+(\bz\wedge\gz)],1]$ and the H\"older inequality when
$q/p\in(1,\infty]$, we deduce that
\begin{align}\label{10.15.1}
&\left[\sum_{k=0}^N\sum_{\alpha \in \ca_k}\sum_{m=1}^{N(k,\alpha)}\mu\left(\qa\right)({\rm I}_3)^p\right]^{1/p}\\
&\quad\lesssim\left\{\sum_{k'=N'+1}^\infty\delta^{-k'sq}\left[\sum_{\alpha' \in \ca_{k'}}
\sum_{m'=1}^{N(k',\alpha')}\mu\left(\qap\right)\left|P_{k'}(f)\left(\yap\right)\right|^p\right]^{q/p}\right\}^{1/q}.\notag
\end{align}
If $p\in(1,\infty]$, choose $\eta'\in(0,\beta\wedge\gamma)$ such
that $s\in(-\eta',\eta')$. Then, by \eqref{10.14.2}, the H\"older inequality and Lemma \ref{9.14.1},
we know that, for any fixed $\eta'\in(\max\{-s,0\}, \eta)$,
\begin{align*}
\rm{{\rm I}_3}&\lesssim\left\{\sum_{k'=N'+1}^\infty\delta^{k'\eta'p}
\sum_{\alpha' \in \ca_{k'}}\sum_{m'=1}^{N(k',\alpha')}\mu\left(\qap\right)
\left|P_{k'}f\left(\yap\right)\right|^p\r.\\
&\qquad\times\left.\inf_{z\in\qa}\frac{1}{V_1(z)+V(z,\yap)}\left[\frac{1}{1+d(z,\yap)}\right]^\gamma\r\}^{1/p},
\end{align*}
which, together with Lemma \ref{9.14.1} and the H\"older inequality when $q/p\in(1,\infty]$, or \eqref{r} when
$q/p\in(0,1]$, further implies that \eqref{10.15.1} also holds true for $p\in(1,\infty]$.

By \eqref{10.15.2}, \eqref{10.15.3}, \eqref{10.15.1}, the arbitrariness
of $\yap$ and an argument similar to that used in the estimation of \eqref{limpp},
we conclude that, for any $p\in(p(s,\beta\wedge\gamma),\infty]$
and $q\in(0,\infty]$,
\begin{align}\label{10.15.4}
&\left\{\sum_{k=0}^N\sum_{\alpha \in \ca_k}\sum_{m=1}^{N(k,\alpha)}
\mu\left(\qa\right)\left[m_{\qa}(|Q_k(f)|)\r]^p\right\}^{1/p}\\
&\quad\lesssim\left\{\sum_{k'=0}^{N'}\sum_{\alpha' \in \ca_{k'}}
\sum_{m'=1}^{N(k',\alpha')}\mu\left(\qap\right)\left[m_{\qap}(|P_{k'}(f)|)\r]^p\right\}^{1/p}\notag\\
&\quad\qquad+\left\{\sum_{k'=N'+1}^\infty\delta^{-k'sq}\left[\sum_{\alpha' \in \ca_{k'}}
\sum_{m'=1}^{N(k',\alpha')}\mu\left(\qap\right)
\left\{\inf_{z\in\qap}|P_{k'}(f)(z)|\right\}^p\right]^{q/p}\right\}^{1/q}.\notag
\end{align}

Now we estimate the second term of the left-hand side of \eqref{bihpp}.
From \eqref{10.15.5} and \eqref{10.15.6}, we deduce that,
for any $k\in\{N+1,N+2,\dots\}$ and $z\in\cx$,
\begin{align*}
|Q_k(f)(z)|&\leq \sum_{\alpha' \in \ca_0}\sum_{m'=1}^{N(0,\alpha')}
m_{\qop}(|P_{0}(f)|)\int_{\qop}
\left|Q_k\widetilde{P}_0(z,y)\right|\,d\mu(y)\\
&\qquad+\sum_{k'=1}^{N'}\sum_{\alpha' \in \ca_{k'}}\sum_{m'=1}^{N(k',\alpha')}
\mu\left(\qap\right)m_{\qap}(|P_{k'}(f)|)\left|Q_k\widetilde{P}_{k'}\left(z,\yap\right)\right|\\
&\qquad+\sum_{k'=N'+1}^\infty\sum_{\alpha' \in \ca_{k'}}
\sum_{m'=1}^{N(k',\alpha')}\mu\left(\qap\right)
\left|Q_k\widetilde{P}_{k'}\left(z,\yap\right)\right|
\left|P_{k'}f\left(\yap\right)\right|\\
&=:\rm{{\rm I}_4+{\rm I}_5+{\rm I}_6}.
\end{align*}
Using an argument similar to that used in the estimation of \eqref{bhpp}, we conclude that
\begin{align}\label{10.15.8}
&\left\{\sum_{k=N+1}^\infty\delta^{-ksq}\left[\sum_{\alpha \in \ca_k}
\sum_{m=1}^{N(k,\alpha)}\mu\left(\qa\right)\left\{\sup_{z\in\qa}|\rm{{\rm I}_6}|\right\}^p\right]^{q/p}\right\}^{1/q}\\
&\qquad\lesssim\left\{\sum_{k'=N'+1}^\infty\delta^{-k'sq}
\left[\sum_{\alpha' \in \ca_{k'}}\sum_{m'=1}^{N(k',\alpha')}
\mu\left(\qap\right)\left\{\inf_{z\in\qap}|P_{k'}(f)(z)|\right\}^p\right]^{q/p}\right\}^{1/q}.\notag
\end{align}
Observe that $Q_k$ has the cancellation property when $k\in\{N+1,N+2,\dots\}$,
and $\widetilde{P}_{k'}$ satisfies the regularity condition on
the first variable for any $k'\in\zz$.
By this, we have
\begin{align*}
\left|Q_k\widetilde{P}_{k'}(z,y)\right|&=\left|\int_\cx Q_k(z,s)\widetilde{P}_{k'}(s,y)\,d\mu(s)\right|\\
&=\left|\int_\cx Q_k(z,s)\left[\widetilde{P}_{k'}(s,y)-\widetilde{P}_{k'}(z,y)\right]\,d\mu(s)\right|\\
&\lesssim\int_{\{s\in\cx:\ d(s,z)\leq(2A_0)^{-1}
[\delta^{k'}+d(z,y)]\}}|Q_k(z,s)|\left|\widetilde{P}_{k'}(s,y)-\widetilde{P}_{k'}(z,y)\right|\,d\mu(s)\\
&\qquad+\int_{\{s\in\cx:\ d(s,z)\geq(2A_0)^{-1}
[\delta^{k'}+d(z,y)]\}}|Q_k(z,s)|\left|\widetilde{P}_{k'}(s,y)\right|\,d\mu(s)\\
&\qquad+\left|\widetilde{P}_{k'}(z,y)\right|\int_{\{s\in\cx:\ d(s,z)\geq(2A_0)^{-1}
[\delta^{k'}+d(z,y)]\}}|Q_k(z,s)|\,d\mu(s)\\
&=:\rm{{\rm J}_1+{\rm J}_2+{\rm J}_3}.
\end{align*}
From $\delta^{k'}\sim 1$ when $k'\in\{0,\dots,N'\}$,
\eqref{10.23.3} and Lemma \ref{6.15.1}(iii), it follows that,
for any fixed $\gamma'\in(0,\gamma)$ and $\Gamma\in(0,\beta)$,
\begin{align*}
\rm{{\rm J}_1}&\lesssim \int_{\{s\in\cx:\ d(s,z)\leq(2A_0)^{-1}[\delta^{k'}+d(z,y)]\}}
\frac{1}{V_{\delta^k}(z)+V(s,z)}\left[\frac{\delta^k}{\delta^k+d(z,s)}\right]^{\Gamma}\\
&\qquad \times\left[\frac{d(s,z)}{1+d(z,y)}\right]^\beta\frac{1}{V_1(z)+V(z,y)}
\left[\frac{1}{1+d(z,y)}\right]^{\gamma'}\,d\mu(s)\\
&\lesssim\delta^{k\Gamma}\frac{1}{V_1(z)+V(z,y)}\left[\frac{1}{1+d(z,y)}\right]^{\beta+\gamma'}
\int_{\{s\in\cx:\ d(s,z)\leq(2A_0)^{-1}[\delta^{k'}+d(z,y)]\}}
[d(s,z)]^{\beta-\Gamma}\frac{1}{V(s,z)}\,d\mu(s)\\
&\lesssim   \delta^{k\Gamma}\frac{1}{V_1(z)+V(z,y)}\left[\frac{1}{1+d(z,y)}\right]^{\Gamma+\gz'}.
\end{align*}
For any $k\in\zz_+$, by the size condition
of $\widetilde{P}_{k'}$
and Lemma \ref{6.15.1}(iii), we have $\int_{\cx}|\widetilde{P}_{k'}(s,y)|\,d\mu(s)\lesssim 1$.
Besides, for any $s\in \{s\in\cx:\ d(s,z)\geq(2A_0)^{-1}[\delta^{k'}+d(z,y)]\}$,
by Lemma \ref{6.15.1}(i), we know that
$$V(s,z)\sim \mu(B(z,d(s,z)))\gtrsim\mu(B(z,\delta^{k'}+d(z,y)))\sim V_1(z)+V(z,y).$$
From these facts, we deduce that
\begin{align*}
\rm{{\rm J}_2}&\lesssim\int_{\{s\in\cx:\ d(s,z)\geq(2A_0)^{-1}[\delta^{k'}+d(z,y)]\}}
\frac{1}{V_{\delta^k}(z)+V(s,z)}\left[\frac{\delta^k}{\delta^k
+d(z,s)}\right]^{\Gamma}\left|\widetilde{P}_{k'}(s,y)\right|\,d\mu(s)\\
&\lesssim \delta^{k\Gamma}\frac{1}{V_1(z)+V(z,y)}\left[\frac{1}{1+d(z,y)}\right]^{\Gamma}.
\end{align*}
For the term $\rm{{\rm J}_3}$, by Lemma \ref{6.15.1}(iv), we conclude that
\begin{align*}
\rm{{\rm J}_3}&\lesssim \frac{1}{V_1(z)+V(z,y)}\int_{\{s\in\cx:\ d(s,z)\geq(2A_0)^{-1}
[\delta^{k'}+d(z,y)]\}}
\frac{1}{V_{\delta^k}(z)+V(s,z)}\left[\frac{\delta^k}{\delta^k+d(z,s)}\right]^{\Gamma}\,d\mu(s)\\
&\lesssim \delta^{k\Gamma}\frac{1}{V_1(z)+V(z,y)}\left[\frac{1}{1+d(z,y)}\right]^{\Gamma}.
\end{align*}
Combining  the estimates of $\rm{{\rm J}_1,{\rm J}_2}$ and $\rm{{\rm J}_3}$, we know that, for any fixed
$\Gamma\in(0,\bz)$, any $k\in\{N+1,N+2,\dots\}$, $k'\in\{0,\dots,N'\}$ and $z,\ y\in\cx$,
\begin{equation*}
\left|Q_k\widetilde{P}_{k'}(z,y)\right|\lesssim
\delta^{k\Gamma}\frac{1}{V_1(z)+V(z,y)}\left[\frac{1}{1+d(z,y)}\right]^{\Gamma}.
\end{equation*}
Using this and an argument similar to
that used in the estimation of \eqref{10.14.1},
we obtain
\begin{align}\label{10.18.8}
\sup_{z\in\qa}(\rm{{\rm I}_4+{\rm I}_5})&\lesssim\delta^{k\Gamma}
\sum_{k'=0}^{N'}\sum_{\alpha' \in \ca_{k'}}\sum_{m'=1}^{N(k',\alpha')}m_{\qap}(|P_{k'}(f)|)\mu\left(\qap\right)\\
&\qquad\times\inf_{z\in\qa}\inf_{y\in\qap}\frac{1}{V_1(z)+V(z,y)}\left[\frac{1}{1+d(z,y)}\right]^{\Gamma}.\notag
\end{align}
Therefore, if $p\in(\omega/[\omega+(\beta\wedge\gamma)], 1]$
and $s\in(-(\beta\wedge\gamma),\,\beta\wedge\gamma)$, choose $\Gamma\in(0,\bz)$ such that
$s\in(-\Gamma,\Gamma)$ and $p\in(\omega/(\omega+\Gamma),1]$. By this, \eqref{r} and Lemma \ref{9.14.1}, we
know that
\begin{align}\label{10.18.2}
&\left\{\sum_{k=N+1}^\infty\delta^{-ksq}\left[\sum_{\alpha \in \ca_k}
\sum_{m=1}^{N(k,\alpha)}\mu\left(\qa\right)\left\{\sup_{z\in\qa}|\rm{{\rm I}_4
+{\rm I}_5}|\right\}^p\right]^{q/p}\right\}^{1/q}\\
&\quad \lesssim \left\{\sum_{k'=0}^{N'}\sum_{\alpha' \in \ca_{k'}}
\sum_{m'=1}^{N(k',\alpha')}\mu\left(\qap\right)\left[m_{\qap}(|P_{k'}(f)|)\r]^p\right\}^{1/p}.\notag
\end{align}
If $p\in(1,\infty]$, choosing $\Gamma\in(s_+,\bz)$ and using the H\"older inequality and Lemma \ref{9.14.1},
we conclude that
\begin{align}\label{10.18.3}
&\left\{\sum_{k=N+1}^\infty\delta^{-ksq}\left[\sum_{\alpha \in \ca_k}
\sum_{m=1}^{N(k,\alpha)}\mu\left(\qa\right)\left\{\sup_{z\in\qa}|{\rm I}_4
+{\rm I}_5|\right\}^p\right]^{q/p}\right\}^{1/q}\\
&\quad \lesssim \left\{\sum_{k'=0}^{N'}\sum_{\alpha' \in \ca_{k'}}
\sum_{m'=1}^{N(k',\alpha')}\mu\left(\qap\right)\left[m_{\qap}(|P_{k'}(f)|)\r]^p\right\}^{1/p}.\notag
\end{align}
From \eqref{10.15.8}, \eqref{10.18.2} and \eqref{10.18.3}, we deduce that
\begin{align*}
&\left\{\sum_{k=N+1}^\infty\delta^{-ksq}\left[\sum_{\alpha \in \ca_k}
\sum_{m=1}^{N(k,\alpha)}\mu\left(\qa\right)\left\{\sup_{z\in\qa}|Q_k(f)(z)|\right\}^p\right]^{q/p}\right\}^{1/q}\\
&\quad\sim\left\{\sum_{k'=0}^{N'}\sum_{\alpha' \in \ca_{k'}}
\sum_{m'=1}^{N(k',\alpha')}\mu\left(\qap\right)\left[m_{\qap}(|P_{k'}(f)|\r]^p\right\}^{1/p}\\
&\quad\qquad+\left\{\sum_{k'=N'+1}^\infty\delta^{-k'sq}
\left[\sum_{\alpha' \in \ca_{k'}}\sum_{m'=1}^{N(k',\alpha')}\mu\left(\qap\right)
\left\{\inf_{z\in\qap}|P_{k'}(f)(z)|\right\}^p\right]^{q/p}\right\}^{1/q},
\end{align*}
which, combined with \eqref{10.15.4} and the symmetry, further implies \eqref{bihpp}.

Now we show \eqref{fihpp}. Using an argument similar to that used in the estimation of \eqref{10.15.1}, by
\eqref{10.14.2}, Lemmas \ref{10.18.5} and  \ref{fsvv}, and \eqref{r} when $q\in(p(s,\beta\wedge\gamma),1]$,
or the H\"older inequality when $q\in(1,\infty]$, we know that, for any fixed
$r\in(\omega/(\omega+\gz),\min\{1,p,q\})$,
\begin{align*}
&\left[\sum_{k=0}^N\sum_{\alpha \in \ca_k}\sum_{m=1}^{N(k,\alpha)}\mu\left(\qa\right)({\rm I}_3)^p\right]^{1/p}\\
&\quad\lesssim\left\|\left\{\sum_{k'=N'+1}^\infty \left[M\left(\sum_{\alpha' \in \ca_{k'}}
\sum_{m'=1}^{N(k',\alpha')}\delta^{-k'sr}\left|P_{k'}(f)\left(\yap\right)\right|^r
\mathbf 1_{\qap}\right)\right]^{q/r}\right\}^{1/q}\right\|_{L^p(\cx)}\\
&\quad\lesssim\left\|\left[\sum_{k'=N'+1}^\infty \sum_{\alpha' \in \ca_{k'}}
\sum_{m'=1}^{N(k',\alpha')}\delta^{-k'sq}\left|P_{k'}(f)\left(\yap\right)\right|^q
\mathbf 1_{\qap}\right]^{1/q}\right\|_{L^p(\cx)}.
\end{align*}
From this and \eqref{10.15.3}, it follows that,
for any $p\in(p(s,\beta\wedge\gamma),\infty)$
and $q\in(p(s,\beta\wedge\gamma),\infty]$,
\begin{align}\label{10.18.4}
&\left\{\sum_{k=0}^N\sum_{\alpha \in \ca_k}\sum_{m=1}^{N(k,\alpha)}
\mu\left(\qa\right)\left[m_{\qa}(|Q_k(f)|)\r]^p\right\}^{1/p}\\
&\quad\lesssim\left\{\sum_{k'=0}^{N'}\sum_{\alpha' \in \ca_{k'}}
\sum_{m'=1}^{N(k',\alpha')}\mu\left(\qap\right)\left[m_{\qap}(|P_{k'}(f)|)\r]^p\right\}^{1/p}\notag\\
&\quad\qquad+\left\|\left\{\sum_{k'=N'+1}^\infty\sum_{\alpha' \in \ca_{k'}}
\sum_{m'=1}^{N(k',\alpha')}\delta^{-k'sq}\left[\inf_{z\in\qap}|P_{k'}(f)(z)|\right]^q
\mathbf 1_{\qap}\right\}^{1/q}\right\|_{L^p(\mathcal{X})}.\notag
\end{align}

For the second term of \eqref{fihpp}, on one hand, by \eqref{10.18.8}, and \eqref{r} when $p/q\in(0,1]$,
or the H\"older inequality when $p/q\in(1,\infty)$, we know that,
for any fixed $\Gamma \in (0,\eta)$,
\begin{align*}
&\left\|\left\{\sum_{k=N+1}^\infty\sum_{\alpha \in \ca_k}\sum_{m=1}^{N(k,\alpha)}
\delta^{-ksq}\left[\sup_{z\in\qa}|\rm{{\rm I}_4+{\rm I}_5}|\right]^q
\mathbf 1_{\qa}\right\}^{1/q}\right\|_{L^p(\mathcal{X})}^p\\
&\quad\lesssim\sum_{k=N+1}^\infty\delta^{k(\Gamma-s)(p\wedge q)}\sum_{\alpha \in \ca_k}
\sum_{m=1}^{N(k,\alpha)}\mu\left(\qa\right)\left\{\sum_{k'=0}^{N'}\sum_{\alpha' \in \ca_{k'}}
\sum_{m'=1}^{N(k',\alpha')}m_{\qap}(|P_{k'}(f)|)\mu\left(\qap\right)\r.\\
&\quad\qquad\times\lf.\inf_{z\in\qa}\inf_{y\in\qap}\frac{1}{V_1(z)+V(z,y)}
\left[\frac{1}{1+d(z,y)}\right]^{\Gamma}\right\}^p.
\end{align*}
Choose $\Gamma$ such that $\Gamma\in(s,\beta\wedge\gamma)$ and $p\in(\omega/(\omega+\Gamma),\fz)$.
From this, \eqref{r} when $p\in(p(s,\beta\wedge\gamma),1]$, or the H\"older inequality when $p\in(1,\infty)$,
and Lemma \ref{9.14.1}, we deduce that
\begin{align}\label{10.19.1}
&\left\|\left\{\sum_{k=N+1}^\infty\sum_{\alpha \in \ca_k}\sum_{m=1}^{N(k,\alpha)}\delta^{-ksq}
\left[\sup_{z\in\qa}|{\rm I}_4+{\rm I}_5|\right]^q\mathbf 1_{\qa}\right\}^{1/q}\right\|_{L^p(\mathcal{X})}^p\\
&\quad \lesssim\sum_{k'=0}^{N'}\sum_{\alpha' \in \ca_{k'}}\sum_{m'=1}^{N(k',\alpha')}
\mu\left(\qap\right)\left[m_{\qap}(|P_{k'}(f)|)\r]^p.\notag
\end{align}

On the other hand, using an argument similar to that
used in the estimation of \eqref{10.18.6}, we have
\begin{align*}
&\left\|\left\{\sum_{k=N+1}^\infty\sum_{\alpha \in \ca_k}\sum_{m=1}^{N(k,\alpha)}\delta^{-ksq}\
\left[\sup_{z\in\qa}|{\rm I}_6|\right]^q\mathbf 1_{\qa}\right\}^{1/q}\right\|_{L^p(\mathcal{X})}\\
&\quad\lesssim\left\|\left\{\sum_{k'=N'+1}^\infty\sum_{\alpha' \in \ca_{k'}}
\sum_{m'=1}^{N(k',\alpha')}\delta^{-k'sq}
\left|P_{k'}f\left(\yap\right)\right|^q\mathbf 1_{\qap}\right\}^{1/q}\right\|_{L^p(\mathcal{X})},
\end{align*}
which, together with \eqref{10.19.1}, completes  the proof of \eqref{fihpp}.
This finishes the proof of Lemma \ref{6.9.1}.
\end{proof}

Using an argument similar to that used in  the proof of  Proposition \ref{6.4.1},
we have the following proposition and
we omit the details.

\begin{proposition}
Let $\{Q_k\}_{k=0}^\infty$ and $\{P_k\}_{k=0}^\infty$ be two \emph{$\exp$-IATIs},
$\beta,\ \gamma\in (0, \eta)$ with
$\eta$ as in Definition \ref{10.23.2}, and
$s\in(-(\beta\wedge\gamma),\,\beta\wedge\gamma)$.
Suppose that $N\in\nn$ and $N'\in\nn$ are as in Lemma
\ref{icrf} associated, respectively, with $\{Q_k\}_{k=0}^\fz$ and $\{P_k\}_{k=0}^\fz$.
\begin{enumerate}
\item[{\rm(i)}] If $p\in(p(s,\beta\wedge\gamma),\infty]$
with $p(s,\beta\wedge\gamma)$ as in \eqref{pseta}, and $q\in(0,\infty]$, then
\begin{align*}
&\left\{\sum_{k=0}^N\sum_{\alpha \in \ca_k}\sum_{m=1}^{N(k,\alpha)}\mu\left(\qa\right)\left[m_{\qa}(|Q_k(f)|)
\r]^p\right\}^{1/p}+\left[\sum_{k=N+1}^{\infty} \delta^{-ksq}\|Q_k(f)\|_{\lp}^q\right]^{1/q}\\
&\quad\sim \left\{\sum_{k'=0}^{N'}\sum_{\alpha' \in \ca_{k'}}\sum_{m'=1}^{N(k',\alpha')}\mu\left(\qap\right)
\left[m_{\qap}(|P_{k'}(f)|)\r]^p\right\}^{1/p}
+\left[\sum_{k'=N'+1}^{\infty} \delta^{-k'sq}\|P_{k'}(f)\|_{\lp}^q\right]^{1/q}
\end{align*}
with the usual modifications made when $p=\infty$ or $q=\infty$, and the positive equivalence constants are independent of $f$.
\item[{\rm(ii)}] If $p\in(p(s,\beta\wedge\gamma),\infty)$ and
$q\in(p(s,\beta\wedge\gamma),\infty]$, then
\begin{align*}
&\left\{\sum_{k=0}^N\sum_{\alpha \in \ca_k}\sum_{m=1}^{N(k,\alpha)}
\mu\left(\qa\right)\left[m_{\qa}(|Q_k(f)|\r]^p\right\}^{1/p}
+\left\|\left[\sum_{k=N+1}^\infty \delta^{-ksq}|Q_k(f)|^q\right]^{1/q}\right\|_{\lp}\\
&\quad\sim \left\{\sum_{k'=0}^{N'}\sum_{\alpha' \in \ca_{k'}}
\sum_{m'=1}^{N(k',\alpha')}\mu\left(\qap\right)\left[m_{\qap}(|P_{k'}(f)|\r]^p\right\}^{1/p}
+\left\|\left[\sum_{k'=N'+1}^\infty \delta^{-k'sq}|P_{k'}(f)|^q\right]^{1/q}\right\|_{\lp}
\end{align*}
with the usual modification made when $q=\infty$, and the positive equivalence constants are independent of $f$.
\end{enumerate}
\end{proposition}

We now prove that the spaces $\ihb$ and
$\ihf$ are independent of the choice
of the underlying spaces of distributions. Indeed, we have the following proposition.

\begin{proposition}\label{10.19.2}
Let  $\beta,\ \gamma \in (0, \eta)$
with $\eta$ as in
Definition \ref{10.23.2}, $s\in(-\eta,\,\eta)$
and $p,\ q\in(0,\infty]$
satisfy
\begin{equation}\label{10.19.3}
\beta\in\left(\max\left\{0,-s+\omega
\left(\frac{1}{p}-1\right)_+\right\},\eta\right)
\qquad\text{and}\qquad
\gamma\in\left(\omega\left(\frac{1}{p}-1\right)_+,\eta\right).
\end{equation}
\begin{enumerate}
\item[\rm (i)] If $p\in(p(s,\beta\wedge\gamma),\infty]$
with $p(s,\beta\wedge\gamma)$ as in \eqref{pseta},
then, for any $f\in B^s_{p,q}(\cx)\subset(\cg_0^\eta(\bz,\gz))'$, $f\in(\icggt)'$ with
$\widetilde{\beta}$ and $\widetilde{\gamma}$ satisfying $p\in(p(s,\widetilde{\beta}\wedge\widetilde{\gamma}),\infty]$
with $p(s,\widetilde{\beta}\wedge\widetilde{\gamma})$
as in \eqref{pseta} via replacing $\beta$ and $\gamma$ respectively
by $\widetilde{\beta}$ and $\widetilde{\gamma}$,
and \eqref{10.19.3} via replacing $\beta$ and $\gamma$ respectively
by $\widetilde{\beta}$ and $\widetilde{\gamma}$, and there exists a positive constant $C$,
independent of $f$, such that
$$
\|f\|_{(\icggt)'}\leq C\|f\|_{\ihb}.
$$
\item[{\rm (ii)}] If $p\in(p(s,\beta\wedge\gamma),\infty)$ and
$q\in(p(s,\beta\wedge\gamma),\infty]$,
then, for any $f\in F^s_{p,q}(\cx)\subset(\cg_0^\eta(\bz,\gz))'$, $f\in(\icggt)'$ with
$\widetilde{\beta}$ and $\widetilde{\gamma}$ satisfying
$p\in(p(s,\widetilde{\beta}\wedge\widetilde{\gamma}),\infty)$
and $q\in(p(s,\widetilde{\beta}\wedge\widetilde{\gamma}),\infty]$
with $p(s,\widetilde{\beta}\wedge\widetilde{\gamma})$
as in \eqref{pseta} via replacing $\beta$ and $\gamma$ respectively
by $\widetilde{\beta}$ and $\widetilde{\gamma}$, and \eqref{10.19.3} via replacing $\beta$ and $\gamma$ respectively
by $\widetilde{\beta}$ and $\widetilde{\gamma}$, and there exists a positive constant $C$,
independent of $f$, such that
$$
\|f\|_{(\icggt)'}\leq C\|f\|_{\ihf}.
$$
\end{enumerate}
\end{proposition}

\begin{proof}
Let $\psi\in\cg(\eta,\eta)$. We adopt the notation from Lemma \ref{icrf}.
We first claim that, for any $k\in \zz_+$ and $y\in\cx$,
\begin{equation}\label{10.19.4}
\left|\langle\widetilde{Q}_k(\cdot, y),\psi\rangle\right|
\lesssim\delta^{k\widetilde{\beta}}
\|\psi\|_{\cg(\widetilde{\beta},\widetilde{\gamma})}\frac{1}{V_1(x_1)+V(x_1,y)}
\left[\frac{1}{1+d(x_1,y)}\right]^{\widetilde{\gamma}}.
\end{equation}
Indeed, when $k\in\{N+1,N+2,\dots\}$, using an argument similar to that
used in the estimation of \eqref{6.14.2}, we obtain \eqref{10.19.4} directly.
When $k\in\{0,\dots,N\}$, by $\delta^k\sim 1$, Lemma \ref{6.15.1}(ii),
$\widetilde{\gamma}\in(0,\eta)$ and the fact that
$$\{z\in\cx:\ d(z,y)\geq(2A_0)^{-1}d(x_1,y)\}
\bigcup\{z\in\cx:\ d(z,x_1)\geq(2A_0)^{-1}d(x_1,y)\}=\cx,$$
we know that
\begin{align*}
&\left|\left\langle\widetilde{Q}_k(\cdot, y),\psi\right\rangle\right|\\
&\quad=\left|\int_{\cx}\widetilde{Q}_k(z, y)\psi(z)\,d\mu(z)\right|\\
&\quad\lesssim\|\psi\|_{\cg(\widetilde{\beta},\widetilde{\gamma})}\int_{\cx}\frac{1}{V_1(y)+V_1(z,y)}
\left[\frac{1}{1+d(z,y)}\right]^{\widetilde{\gamma}}\frac{1}{V_1(x_1)+V(x_1,z)}
\left[\frac{1}{1+d(x_1,z)}\right]^{\widetilde{\gamma}}\,d\mu(z)\\
&\quad\lesssim\|\psi\|_{\cg(\widetilde{\beta},\widetilde{\gamma})}\frac{1}{V_1(x_1)+V(x_1,y)}\\
&\quad\qquad\times\left\{\left[\frac{1}{1+d(x_1,y)}\right]^{\widetilde{\gamma}}
\int_{\{z\in\cx:\ d(z,y)\geq(2A_0)^{-1}d(x_1,y)\}}\frac{1}{V_1(x_1)+V(x_1,z)}
\left[\frac{1}{1+d(x_1,z)}\right]^{\widetilde{\gamma}}\,d\mu(z)\right.\\
&\quad\qquad+\left.\left[\frac{1}{1+d(x_1,y)}\right]^{\widetilde{\gamma}}
\int_{\{z\in\cx:\ d(z,x_1)\geq(2A_0)^{-1}d(x_1,y)\}}\frac{1}{V_1(y)+V_1(z,y)}
\left[\frac{1}{1+d(z,y)}\right]^{\widetilde{\gamma}}\,d\mu(z)\right\}\\
&\quad\lesssim\delta^{k\widetilde{\beta}}\|\psi\|_{\cg(\widetilde{\beta},\widetilde{\gamma})}
\frac{1}{V_1(x_1)+V(x_1,y)}\left[\frac{1}{1+d(x_1,y)}\right]^{\widetilde{\gamma}},
\end{align*}
which implies \eqref{10.19.4}.

From Lemma \ref{icrf}, \eqref{10.15.6} and \eqref{10.19.4}, it follows that
\begin{align*}
&|\langle f,\psi\rangle|\\
&\quad=\left|\sum_{\alpha \in \ca_0}\sum_{m=1}^{N(0,\alpha)}
\int_{\qo}\left\langle\widetilde{Q}_0(\cdot,y),\psi\right\rangle\,d\mu(y)
Q^{0,m}_{\alpha,1}(f)\right.\\
&\quad\qquad+\sum_{k=1}^N\sum_{\alpha \in \ca_k}
\sum_{m=1}^{N(k,\alpha)}\mu\left(\qa\right)\left\langle\widetilde{Q}_k(\cdot,\ya),
\psi\right\rangle Q^{k,m}_{\alpha,1}(f)\\
&\quad\qquad\left.+\sum_{k=N+1}^\infty\sum_{\alpha \in \ca_k}
\sum_{m=1}^{N(k,\alpha)}\mu\left(\qa\right)
\left\langle\widetilde{Q}_k(\cdot,\ya),\psi\right\rangle Q_kf\left(\ya\right)\right|\\
&\quad\lesssim\|\psi\|_{\cg(\widetilde{\beta},\widetilde{\gamma})}\left\{\sum_{k=0}^N
\sum_{\alpha \in \ca_k}\sum_{m=1}^{N(k,\alpha)}\mu\left(\qa\right)m_{\qa}(|Q_k(f)|)
\frac{1}{V_1(x_1)+V(x_1,\ya)}\left[\frac{1}{1+d(x_1,\ya)}\right]^{\widetilde{\gamma}}\right.\\
&\qquad\quad\left.+\sum_{k=N+1}^\infty\delta^{k\widetilde{\beta}}
\sum_{\alpha \in \ca_k}\sum_{m=1}^{N(k,\alpha)}\mu\left(\qa\right)\left|Q_kf\left(\ya\right)\right|
\frac{1}{V_1(x_1)+V(x_1,\ya)}\left[\frac{1}{1+d(x_1,\ya)}\right]^{\widetilde{\gamma}}\right\}.
\end{align*}
If $p\in(p(s,\widetilde{\beta}\wedge\widetilde{\gamma}), 1]$, then, by \eqref{r}, \eqref{6.22.1} and
\eqref{10.19.3}, we find that
\begin{align}\label{10.19.5}
|\langle f,\psi\rangle|&\lesssim\|\psi\|_{\cg(\widetilde{\beta},\widetilde{\gamma})}
\left\{\sum_{k=0}^N\sum_{\alpha \in \ca_k}\sum_{m=1}^{N(k,\alpha)}
\mu\left(\qa\right)\left[m_{\qa}(|Q_k(f)|)\right]^p\right.\\
&\quad\quad\left.+\sum_{k=N+1}^\infty\delta^{k[\widetilde{\beta}p-
\omega(1-p)+sp]}\delta^{-ksp}\left[\sum_{\alpha \in \ca_k}
\sum_{m=1}^{N(k,\alpha)}\mu\left(\qa\right)
\left|Q_kf\left(\ya\right)\right|^p\right]\right\}^{1/p}\notag\\
&\lesssim \|\psi\|_{\cg(\widetilde{\beta},\widetilde{\gamma})}\|f\|_{\ihb},\notag
\end{align}
where, in the last inequality, we used \eqref{r} when $q/p\in(0,1]$, or the H\"older inequality when $q/p\in(1,\infty]$.

If $p\in(1,\infty]$, then, by the H\"older inequality and $\widetilde{\bz}\in(\max\{0,-s\},\eta)$, we have
\begin{align}\label{10.19.6}
|\langle f,\psi\rangle|&\lesssim\|\psi\|_{\cg(\widetilde{\beta},\widetilde{\gamma})}
\left(\left\{\sum_{k=0}^N\sum_{\alpha \in \ca_k}\sum_{m=1}^{N(k,\alpha)}
\mu\left(\qa\right)\left[m_{\qa}(|Q_k(f)|)\right]^p\right\}^{1/p}\right.\\
&\qquad\times\left\{\int_\cx\frac{1}{V_1(x_1)+V(x_1,y)}
\left[\frac{1}{1+d(x_1,y)}\right]^{\widetilde{\gamma}}\,d\mu(y)\right\}^{1/p'}\notag\\
&\qquad\left.+\sum_{k=N+1}^\infty\delta^{k\widetilde{\beta}}
\left[\sum_{\alpha \in \ca_k}\sum_{m=1}^{N(k,\alpha)}\mu\left(\qa\right)
\left|Q_kf\left(\ya\right)\right|^p\right]\right)^{1/p}\notag\\
&\qquad\times\left\{\int_\cx\frac{1}{V_1(x_1)+V(x_1,y)}\left[\frac{1}{1+d(x_1,y)}\right]^{\widetilde{\gamma}}
\,d\mu(y)\right\}^{1/p'}\notag\\
&\lesssim \|\psi\|_{\cg(\widetilde{\beta},\widetilde{\gamma})}\|f\|_{\ihb},\notag
\end{align}
where, in the last inequality, we used \eqref{r} when $q\in(0,1]$, or the H\"older inequality when $q\in(1,\infty]$.

Using \eqref{10.19.5}, \eqref{10.19.6} and an argument similar to that used in the proof of Proposition
\ref{6.5.1}, we obtain (i) and (ii), and  we omit the details here.
This finishes the proof of Proposition \ref{10.19.2}.
\end{proof}

Similarly to Proposition \ref{phb}, we have the following properties of $\ihb$ and $\ihf$.

\begin{proposition}\label{10.20.1}
Let $\beta,\ \gamma\in (0,\eta)$
with $\eta$ as in Definition \ref{10.23.2} and
$s \in (-(\beta\wedge\gamma), \beta\wedge\gamma)$.
\begin{enumerate}
\item[{\rm(i)}]  If $0<q_0\leq q_1\leq \infty$ and $p\in(p(s,\beta\wedge\gamma),\infty]$ with
$p(s,\beta\wedge\gamma)$ as in \eqref{pseta}, then
$B^s_{p,q_0}(\mathcal{X})\subset B^s_{p,q_1}(\mathcal{X})$.
If $p(s,\beta\wedge\gamma)<q_0\leq q_1\leq \infty$ and
$p\in(p(s,\beta\wedge\gamma),\infty)$, then $F^s_{p,q_0}(\mathcal{X})\subset F^s_{p,q_1}(\mathcal{X})$.
\item[{\rm(ii)}] If $p\in(p(s,\beta\wedge\gamma),\infty)$ and $q\in(p(s,\beta\wedge\gamma),\infty]$, then
$$B^s_{p,p\wedge q}(\mathcal{X})\subset\ihf\subset B^s_{p,p\vee q}(\mathcal{X}).$$
\item[{\rm(iii)}] If $\widetilde{\beta},\ \widetilde{\gamma}$ satisfy the following conditions,
\begin{equation}\label{4.23.x}
\widetilde{\beta}\in(s_+,\eta)\quad
\text{and}\quad \widetilde{\gamma}\in\left(\omega\left(\frac{1}{p}-1\right)_+,\eta\right),
\end{equation}
then $\cg(\widetilde{\beta},\widetilde{\gamma})\subset \ihb$ when $q\in(0,\infty]$ and $p \in(p(s,\beta\wedge\gamma),\infty]$,
and $\cg(\widetilde{\beta},\widetilde{\gamma})\subset \ihf$ when $q\in(p(s,\beta\wedge\gamma),\infty]$
and $p \in(p(s,\beta\wedge\gamma),\infty)$.
\item[{\rm(iv)}] If $\theta \in ((-\eta-s)_+,\eta-s)$, then $B^{s+\theta}_{p,q_0}(\mathcal{X})
\subset B^s_{p,q_1}(\mathcal{X})$ when $q_0,\  q_1\in(0,\infty]$ and $p\in(p(s,\beta\wedge\gamma),\infty]$, and
$F^{s+\theta}_{p,q_0}(\mathcal{X})\subset F^s_{p,q_1}(\mathcal{X})$ when $q_0,\
q_1\in(p(s,\beta\wedge\gamma), \infty]$ and $p\in(p(s,\beta\wedge\gamma),\infty)$.
\end{enumerate}
\end{proposition}

\begin{proof}
The proofs of (i), (ii) and (iii) are similar,  respectively,
to those of (i), (ii) and (iii) of Proposition
\ref{phb} and hence we omit the details.

For (iv), observe that, for any $\{c_k\}_{k=1}^\infty \subset \mathbb{C}$,
\begin{align*}
\left[\sum_{k=0}^\infty\delta^{-skq_1}|c_k|^{q_1}\right]^{1/q_1}&\leq\sup_{k\in\zz_+}
\delta^{-(s+\theta)k}|c_k|\left[\sum_{j=0}^\infty\delta^{\theta jq_1}\right]^{1/q_1}\\
&\lesssim \sup_{k\in\zz_+}\delta^{-(s+\theta)k}|c_k|
\lesssim\left[\sum_{k=0}^\infty\delta^{-(s+\theta)kq_0}|c_k|^{q_0}
\right]^{1/q_0}.
\end{align*}
This, combined with Definition \ref{ih}, implies that (iv) holds true and hence finishes the proof of
Proposition \ref{10.20.1}.
\end{proof}

\section{Triebel--Lizorkin spaces with $p=\infty$}\label{s4}

In this section, we introduce Triebel--Lizorkin spaces with $p=\infty$ in a way similar to that used in
\cite{fj90} and \cite{hmy08}.

\subsection{Homogeneous Triebel--Lizorkin spaces $\hfi$}

In this section, we concentrate on the homogeneous Triebel--Lizorkin space $\hfi$
and hence we
always need to assume that $\mu(\cx)=\infty$.
Let us begin with its definition.

\begin{definition}\label{hfi}
Let $\beta,\ \gamma \in (0, \eta)$, $s\in(-\eta,\eta)$ and $q\in (p(s,\beta\wedge\gamma),\infty]$ with
$\eta$ as in Definition \ref{10.23.2} and $p(s,\beta\wedge\gamma)$ as in
\eqref{pseta}. Let $\{Q_k\}_{k\in\zz}$ be an exp-ATI. For  any $k\in\zz$ and $\alpha\in \ca_k$,
let $Q_\az^k$ be as in Theorem \ref{10.22.1}.
Then the \emph{homogeneous Triebel--Lizorkin space $\hfi$} is defined by setting
\begin{align*}
\hfi := \Bigg\{f  \in (\cggi)' :\  \|f\|_{\hfi}:={}&\sup_{l \in \zz}
\sup_{\alpha\in\ca_l}\left[\frac{1}{\mu(Q_\alpha^l)}\r.\\
&\quad\left.\left.\times\int_{Q_\alpha^l}\sum_{k=l}^\infty\delta^{-ksq}
|Q_k(f)(x)|^q\,d\mu(x)\right]^{1/q}<\infty\right\}
\end{align*}
with the usual modification made when $q=\infty$.
\end{definition}

We now establish the \emph{homogenous Plancherel--P\^olya inequality} in the case $p=\infty$.

\begin{lemma}\label{11.14.1}
Let $\beta,\ \gamma \in (0, \eta)$,
$s \in (-(\beta\wedge\gamma), \beta\wedge\gamma)$
and $q\in (p(s,\beta\wedge\gamma),\infty]$ with $\eta$
as in Definition \ref{10.23.2} and $p(s,\beta\wedge\gamma)$
as in \eqref{pseta}. Let $\{P_k\}_{k\in\zz}$ and $\{Q_k\}_{k\in\zz}$
be two \emph{$\exp$-ATIs}.
Then there exists a positive constant $C$ such that,
for any $f\in(\cggi)'$,
\begin{align*}
&\sup_{l \in \zz} \sup_{\alpha\in\ca_l}\left\{\frac{1}{\mu(Q_\alpha^l)}
\sum_{k=l}^\infty\sum_{\tau\in\ca_k}\sum_{m=1}^{N(k,\tau)}\delta^{-ksq}\mu(Q_\tau^{k,m})
\mathbf{1}_{\{(\tau,m):\ Q_\tau^{k,m}\subset Q_\alpha^l\}}(\tau,m)
\left[\sup_{x\in Q_\tau^{k,m}}|P_k(f)(x)|\right]^q\right\}^{1/q}\\
&\quad\leq C\sup_{l \in \zz} \sup_{\alpha\in\ca_l}\left\{\frac{1}{\mu(Q_\alpha^l)}
\sum_{k=l}^\infty\sum_{\tau\in\ca_k}\sum_{m=1}^{N(k,\tau)}\delta^{-ksq}
\mu(Q_\tau^{k,m})\mathbf{1}_{\{(\tau,m):\ Q_\tau^{k,m}\subset Q_\alpha^l\}}(\tau,m)
\left[\inf_{x\in Q_\tau^{k,m}}|Q_k(f)(x)|\right]^q\right\}^{1/q}.
\end{align*}
\end{lemma}

To prove Lemma \ref{11.14.1}, we need to use some
basic properties on the geometry of spaces of homogeneous type.
The following lemma comes from \cite[p.\ 4, (2) and (4)]{hk}
and we omit the details here.

\begin{lemma}\label{11.14.2}
Let $(\cx,d,\mu)$ be a space of homogeneous type.
\begin{enumerate}
\item[{\rm(i)}] For any $t\in(0,1]$, $r \in(0,\infty)$ and $x \in\cx$, the ball $B(x,r)$
contains at most $A_0t^{-\log_2A_0}$ pairwise disjoint balls with radius $tr$.
\item[{\rm(ii)}] For any $r\in(0,\infty)$ and $x,\ y\in\cx$ with $d(x,y)\geq r$,
the balls $B(x,(2A_0)^{-1}r)$ and $B(y,(2A_0)^{-1}r)$ are disjoint.
\end{enumerate}
\end{lemma}

Now we prove Lemma \ref{11.14.1}.

\begin{proof}[Proof of Lemma \ref{11.14.1}]
Let all the notation be as in this lemma. Then, by \eqref{11.10.1} and \eqref{pp},
we know that, for any fixed $ \beta'\in(0,\beta\wedge\gamma)$
and any $l\in\zz$, $\alpha\in\ca_l$ and $f\in(\cggi)'$,
\begin{align*}
&\frac{1}{\mu(Q_\alpha^l)}\sum_{k=l}^\infty\sum_{\tau\in\ca_k}
\sum_{m=1}^{N(k,\tau)}\delta^{-ksq}\mu(Q_\tau^{k,m})
\mathbf{1}_{\{(\tau,m):\ Q_\tau^{k,m}\subset Q_\alpha^l\}}(\tau,m)
\left[\sup_{x\in Q_\tau^{k,m}}|P_k(f)(x)|\right]^q\\
&\quad \lesssim \frac{1}{\mu(Q_\alpha^l)}\sum_{k=l}^\infty
\sum_{\tau\in\ca_k}\sum_{m=1}^{N(k,\tau)}\delta^{-ksq}\mu\left(Q_\tau^{k,m}\right)
\mathbf{1}_{\{(\tau,m):\ Q_\tau^{k,m}\subset Q_\alpha^l\}}(\tau,m)\\
&\quad\qquad\times\left\{\sum_{k'=l}^\infty\sum_{\alpha' \in \ca_{k'}}
\sum_{m'=1}^{N(k',\alpha')}\delta^{|k-k'|\beta'}\mu\left(\qap\right)
\left|Q_{k'}f\left(\yap\right)\right|
\frac{1}{V_{\delta^{k\wedge k'}}(\yap)+V(\yap,y_\tau^{k,m})}\right.\\
&\quad\qquad\times\left.\left[\frac{\delta^{k\wedge k'}}
{\delta^{k\wedge k'}+d(\yap,y_\tau^{k,m})}\right]^{\gamma}\right\}^q
+ \frac{1}{\mu(Q_\alpha^l)}\sum_{k=l}^\infty\sum_{\tau\in\ca_k}
\sum_{m=1}^{N(k,\tau)}\delta^{-ksq}\mu\left(Q_\tau^{k,m}\right)
\mathbf{1}_{\{(\tau,m):\ Q_\tau^{k,m}\subset Q_\alpha^l\}}(\tau,m)\\
&\quad\qquad\times\left\{\sum_{k'=-\infty}^{l-1}
\sum_{\alpha' \in \ca_{k'}}\sum_{m'=1}^{N(k',\alpha')}\delta^{|k-k'|\beta'}
\mu\left(\qap\right)\left|Q_{k'}f\left(\yap\right)\right|
\frac{1}{V_{\delta^{k\wedge k'}}(\yap)+V(\yap,y_\tau^{k,m})}\right.\\
&\quad\qquad\times\left.\left[\frac{\delta^{k\wedge k'}}{\delta^{k\wedge k'}
+d(\yap,y_\tau^{k,m})}\right]^\gamma\right\}^q\\
&\quad=: {\rm I}_1+{\rm I}_2.
\end{align*}

For any $l\in\zz$ and $\alpha\in\ca_l$, let 
$$\cb_\alpha^{l,0}:=\lf\{\ta\in\ca_l:\ d(z_{\ta}^l,z_\alpha^l)<5A_0^3C_0\delta^l\r\}.$$
We claim that there exists a positive number $m_1$,
independent of $l$ and $\alpha$, such that $\#\cb_\alpha^{l,0}\leq m_1$.
Indeed, for any $\ta \in\cb_\alpha^{l,0}$ and $x\in B(z_{\ta}^l,(2A_0C_0)^{-1}\delta^l)$, we know that
$$d(x,z_\alpha^l)\leq A_0[d(x,z_{\ta}^l)+d(z_{\ta}^l,z_\alpha^l)]\leq 6A_0^4C_0\delta^l,$$
which implies that $B(z_{\ta}^l,(2A_0C_0)^{-1}\delta^l)\subset B(z_\alpha^l,6A_0^4C_0\delta^l)$. By
this, Theorem \ref{10.22.1} and Lemma \ref{11.14.2}(i),
we find that there exists a positive number $m_1$, independent of $l$ and $\alpha$, such that
$\#\cb_\alpha^{l,0}\le m_1$. Moreover, from Theorem \ref{10.22.1}(iii), we deduce that, for any
$\ta\in\cb_\alpha^{l,0}$,
$$\mu(Q_{\ta}^l)\leq\mu(B(z_\alpha^l, 2A_0C_0\delta^l))\lesssim
\mu(B(z_\alpha^l, (3A_0^2)^{-1}c_0\delta^l))\lesssim\mu(Q_\alpha^l).$$
By symmetry, we have
$\mu(Q_{\ta}^l)\sim\mu(Q_\alpha^l)$ for any $l\in\zz$, $\alpha\in\ca_l$
and $\ta\in\cb_\alpha^{l,0}$.

From this and Theorem \ref{10.22.1}(ii), it follows that, for any $l\in\zz$ and $\alpha\in\ca_l$,
\begin{align*}
{\rm I}_1 &\lesssim \frac{1}{\mu(Q_\alpha^l)}\sum_{k=l}^\infty
\sum_{\tau\in\ca_k}\sum_{m=1}^{N(k,\tau)}\delta^{-ksq}\mu(Q_\tau^{k,m})
\mathbf{1}_{\{(\tau,m):\ Q_\tau^{k,m}\subset Q_{\alpha}^l\}}(\tau,m)\\
&\qquad\times\left\{\sum_{k'=l}^\infty\sum_{\alpha' \in \ca_{k'}}
\sum_{m'=1}^{N(k',\alpha')}\delta^{|k-k'|\beta'}\mu\left(\qap\right)\left|Q_{k'}f\left(\yap\right)\right|
\mathbf{1}_{\{(\alpha',m'):\ Q_{\alpha'}^{k',m'}\subset
(\cup_{\ta\in\cb_\alpha^{l,0}}Q_{\ta}^l)\}}(\alpha',m')\right.\\
&\qquad\times\left.\frac{1}{V_{\delta^{k\wedge k'}}(\yap)
+V(\yap,y_\tau^{k,m})}\left[\frac{\delta^{k\wedge k'}}{\delta^{k\wedge k'}
+d(\yap,y_\tau^{k,m})}\right]^{\gamma'}\right\}^q\\
&\qquad +\frac{1}{\mu(Q_\alpha^l)}\sum_{k=l}^\infty\sum_{\tau\in\ca_k}
\sum_{m=1}^{N(k,\tau)}\delta^{-ksq}\mu\left(Q_\tau^{k,m}\right)
\mathbf{1}_{\{(\tau,m):\ Q_\tau^{k,m}\subset Q_\alpha^l\}}(\tau,m)\\
&\qquad\times\left\{\sum_{k'=l}^\infty
\sum_{\alpha' \in \ca_{k'}}\sum_{m'=1}^{N(k',\alpha')}\delta^{|k-k'|\beta'}\mu\left(\qap\right)
\left|Q_{k'}f\left(\yap\right)\right|
\mathbf{1}_{\{(\alpha',m'):\ Q_{\alpha'}^{k',m'}\cap (\cup_{\ta\in\cb_\alpha^{l,0}}Q_{\ta}^l)=\emptyset\}}
(\alpha',m')\right.\\
&\qquad\times\left.\frac{1}{V_{\delta^{k\wedge k'}}(\yap)+V(\yap,y_\tau^{k,m})}
\left[\frac{\delta^{k\wedge k'}}{\delta^{k\wedge k'}+d(\yap,y_\tau^{k,m})}\right]^{\gamma'}\right\}^q\\
&=:{\rm I}_{1,1}+{\rm I}_{1,2}.
\end{align*}

We first estimate $\rm{{\rm I}_{1,1}}$. If $q\in(p(s,\beta\wedge\gamma),1]$, we then have
$q\in(\omega/(\omega+\gamma),1]$.
By this, \eqref{r}, Lemma \ref{9.14.1}, \eqref{11.10.2}, we know that,
for any $l\in\zz$ and $\alpha\in\ca_l$,
\begin{align*}
\rm{{\rm I}_{1,1}} &\lesssim \frac{1}{\mu(Q_\alpha^l)}\sum_{\ta\in\cb_\alpha^{l,0}}
\sum_{k=l}^\infty\sum_{\tau\in\ca_k}\sum_{m=1}^{N(k,\tau)}\delta^{-ksq}
\mu\left(Q_\tau^{k,m}\right)\mathbf{1}_{\{(\tau,m):\ Q_\tau^{k,m}\subset Q_\alpha^l\}}(\tau,m)\\
&\qquad\times\sum_{k'=l}^\infty\sum_{\alpha' \in \ca_{k'}}
\sum_{m'=1}^{N(k',\alpha')}\delta^{|k-k'|\beta' q}\left[\mu\left(\qap\right)\right]^q
\left|Q_{k'}f\left(\yap\right)\right|^q
\mathbf{1}_{\{(\alpha',m'):\ Q_{\alpha'}^{k',m'}\subset Q_{\ta}^l\}}(\alpha',m')\\
&\qquad\times\left[\frac{1}{V_{\delta^{k\wedge k'}}(\yap)+V(\yap,y_\tau^{k,m})}
\right]^q\left[\frac{\delta^{k\wedge k'}}{\delta^{k\wedge k'}+d(\yap,y_\tau^{k,m})}\right]^{\gamma q}\\
&\lesssim \frac{1}{\mu(Q_\alpha^l)}\sum_{\ta\in\cb_\alpha^{l,0}}
\sum_{k'=l}^\infty\sum_{\alpha' \in \ca_{k'}}\sum_{m'=1}^{N(k',\alpha')}\delta^{-k'sq}
\left[\mu\left(\qap\right)\right]^q
\left|Q_{k'}f\left(\yap\right)\right|^q\mathbf{1}_{\{(\alpha',m'):\ Q_{\alpha'}^{k',m'}
\subset Q_{\ta}^l\}}(\alpha',m')\\
&\qquad\times\sum_{k=l}^\infty\delta^{k'sq}\delta^{-ksq}\delta^{|k-k'|\beta' q}\sum_{\tau\in\ca_k}
\sum_{m=1}^{N(k,\tau)}\mu\left(Q_\tau^{k,m}\right)\mathbf{1}_{\{(\tau,m):\ Q_\tau^{k,m}\subset Q_\alpha^l\}}(\tau,m)\\
&\qquad\times \left[\frac{1}{V_{\delta^{k\wedge k'}}(\yap)+V(\yap,y_\tau^{k,m})}\right]^q
\left[\frac{\delta^{k\wedge k'}}{\delta^{k\wedge k'}+d(\yap,y_\tau^{k,m})}\right]^{\gamma q}\\
&\lesssim \frac{1}{\mu(Q_\alpha^l)}\sum_{\ta\in\cb_\alpha^{l,0}}
\sum_{k'=l}^\infty\sum_{\alpha' \in \ca_{k'}}\sum_{m'=1}^{N(k',\alpha')}
\delta^{-k'sq}\mu\left(\qap\right)
\left|Q_{k'}f\left(\yap\right)\right|^q\mathbf{1}_{\{(\alpha',m'):\ Q_{\alpha'}^{k',m'}
\subset Q_{\ta}^l\}}(\alpha',m')\\
&\qquad \times\sum_{k=l}^\infty\delta^{k'sq}\delta^{-ksq}
\delta^{|k-k'|\beta' q}\delta^{(k\wedge k'-k')\omega(1- q)}
\end{align*}
By this and choosing $ \beta'\in(s_+,\eta)$ such that, we know that, for any $l\in\zz$ and $\alpha\in\ca_l$,
$q\in(\omega/(\omega+\bz'+s),1]$
\begin{equation*}
\rm{{\rm I}_{1,1}}\lesssim \sum_{\ta\in\cb_\alpha^{l,0}}\frac{1}{\mu(Q_{\ta}^l)}
\sum_{k'=l}^\infty\sum_{\alpha' \in \ca_{k'}}\sum_{m'=1}^{N(k',\alpha')}
\delta^{-k'sq}\mu\left(\qap\right)
\left|Q_{k'}f\left(\yap\right)\right|^q\mathbf{1}_{\{(\alpha',m'):\ Q_{\alpha'}^{k',m'}
\subset Q_{\ta}^l\}}(\alpha',m').
\end{equation*}

If $q\in(1, \infty]$, then, by the H\"{o}lder inequality and Lemma \ref{9.14.1},
we conclude that, for any fixed $\beta'\in(|s|,\beta\wedge\gamma)$ and for any $l\in\zz$ and $\alpha\in\ca_l$,
\begin{align*}
\rm{{\rm I}_{1,1}}&\lesssim \frac{1}{\mu(Q_\alpha^l)}
\sum_{\ta\in\cb_\alpha^{l,0}}\sum_{k=l}^\infty\sum_{\tau\in\ca_k}
\sum_{m=1}^{N(k,\tau)}\mu\left(Q_\tau^{k,m}\right)
\mathbf{1}_{\{(\tau,m):\ Q_\tau^{k,m}\subset Q_\alpha^l\}}(\tau,m)\\
&\qquad\times\left(\left\{\sum_{k'=l}^\infty\sum_{\alpha' \in \ca_{k'}}
\sum_{m'=1}^{N(k',\alpha')}\delta^{-k'sq}\delta^{(k'-k)s}\delta^{|k-k'|\beta'}\mu\left(\qap\right)
\left|Q_{k'}f\left(\yap\right)\right|^q\r.\r.\\
&\qquad\times\left.\mathbf{1}_{\{(\alpha',m'):\ Q_{\alpha'}^{k',m'}\subset Q_{\ta}^l\}}(\alpha',m')
\frac{1}{V_{\delta^{k\wedge k'}}(\yap)+V(\yap,y_\tau^{k,m})}
\left[\frac{\delta^{k\wedge k'}}{\delta^{k\wedge k'}+d(\yap,y_\tau^{k,m})}\right]^{\gamma'}\r\}^{1/q}\\
&\qquad\times\left\{\sum_{k'=l}^\infty\sum_{\alpha' \in \ca_{k'}}
\sum_{m'=1}^{N(k',\alpha')}\delta^{(k'-k)s}\delta^{|k-k'|\beta'}\mu\left(\qap\right)\r.\\
&\qquad\times\left.\left.\frac{1}{V_{\delta^{k\wedge k'}}(\yap)+V(\yap,\ya)}
\left[\frac{\delta^{k\wedge k'}}{\delta^{k\wedge k'}+d(\yap,\ya)}\right]^{\gamma'}\r\}^{1/q'}\r)^q\\
&\lesssim \sum_{\ta\in\cb_\alpha^{l,0}}\frac{1}{\mu(Q_{\ta}^l)}
\sum_{k'=l}^\infty\sum_{\alpha' \in \ca_{k'}}\sum_{m'=1}^{N(k',\alpha')}
\delta^{-k'sq}\mu\left(\qap\right)
\left|Q_{k'}f\left(\yap\right)\right|^q\mathbf{1}_{\{(\alpha',m'):\ Q_{\alpha'}^{k',m'}
\subset Q_{\ta}^l\}}(\alpha',m').
\end{align*}
To summarize, we know that, for any fixed
$q\in(p(s,\beta\wedge\gamma),\infty]$
and for any $l\in\zz$ and $\alpha\in\ca_l$,
\begin{equation}\label{11.11.1}
{\rm I}_{1,1}\lesssim
\sum_{\ta\in\cb_\alpha^{l,0}}\frac{1}{\mu(Q_{\ta}^l)}
\sum_{k'=l}^\infty\sum_{\alpha' \in \ca_{k'}}\sum_{m'=1}^{N(k',\alpha')}
\delta^{-k'sq}\mu\left(\qap\right)
\left|Q_{k'}f\left(\yap\right)\right|^q\mathbf{1}_{\{(\alpha',m'):\ Q_{\alpha'}^{k',m'}
\subset Q_{\ta}^l\}}(\alpha',m').
\end{equation}

We next estimate ${\rm I}_{1,2}$. To this end, for
any $l\in\zz$, $j\in\nn$ and $\alpha\in\ca_l$,  define
$$\ca_l^j:=\{\ta\in\ca_l:\ 5A_0^3C_0\delta^{l-j+1}
\leq d(z_{\ta}^l,z_\alpha^l)<5A_0^3C_0\delta^{l-j}\}$$
and
$$\cb_\alpha^{l,j}:=\{\tit\in\ca_{l-j}:\ Q_{\ta}^l\subset Q_{\tit}^{l-j}
\quad\text{for some}\quad \ta\in\ca_l^j\}.$$
From Theorem \ref{10.22.1}(i), it easily follows that, for any
$l\in\zz$, $j\in\nn$ and $\alpha\in\ca_l$,
\begin{equation}\label{eq-x1}
\bigcup_{\ta\in\ca_l^j}Q_{\ta}^l\subset\bigcup_{\tit\in\cb_\alpha^{l,j}}Q_{\tit}^{l-j}.
\end{equation}
We claim that there exists a positive number $m_2$, independent of $\alpha$, $l$ and $j$,
such that
\begin{equation}\label{11.15.1}
\#\cb_\alpha^{l,j} \leq m_2.
\end{equation}
Indeed, if $\cb_\alpha^{l,j}=\emptyset$, \eqref{11.15.1} holds true obviously.
If $\cb_\alpha^{l,j}\neq\emptyset$, then there exists a
$\tit_0 \in\cb_\alpha^{l,j}$. By the definition of $\ca_l^j$, we find  that
there exist $\ta_0\in \ca_l^j$ and $x\in \cx$ such that $x\in Q_{\ta_0}^l\cap Q_{\tit_0}^{l-j}$.
Then, from Theorem \ref{10.22.1}(iii), we deduce that,
for any $\ta\in \ca_l^j$ and $y\in Q_{\ta}^l$,
\begin{align*}
d(y, z_{\tit_0}^{l-j})&\leq A_0\left[d(y, z_{\ta}^l)+d( z_{\ta}^l,z_{\tit_0}^{l-j})\right]\\
&<2A_0^2C_0\delta^l+A_0^2\left[d(z_{\ta}^l, z_\alpha^l)+d( z_\alpha^l,z_{\tit_0}^{l-j})\right]\\
&<2A_0^2C_0\delta^l+5A_0^5C_0\delta^{l-j}+A_0^3
\left[d( z_\alpha^l,z_{\ta_0}^l)+d(z_{\ta_0}^l,z_{\tit_0}^{l-j})\right]\\
&<2A_0^2C_0\delta^l+5A_0^5C_0\delta^{l-j}+5A_0^6C_0\delta^{l-j}+A_0^4
\left[d(z_{\ta_0}^l,x)+d(x,z_{\tit_0}^{l-j})\right]\\
&<2A_0^2C_0\delta^l+5A_0^5C_0\delta^{l-j}+5A_0^6C_0\delta^{l-j}
+2A_0^5C_0\delta^l+2A_0^5C_0\delta^{l-j}\\
&\leq 16A_0^6C_0\delta^{l-j},
\end{align*}
which further implies that, for any $\ta\in \ca_l^j$,
$Q_{\ta}^l\subset B(z_{\tit_0}^{l-j}, 16A_0^6C_0\delta^{l-j})$. By this, we
conclude that, for any $\tit \in\cb_\alpha^{l,j}$, there exists a $z\in Q_{\tit}^{l-j}\cap
B(z_{\tit_0}^{l-j},16A_0^6C_0\delta^{l-j})\neq\emptyset$. Therefore, for any  $s\in Q_{\tit}^{l-j}$,
$$d(s,z_{\tit_0}^{l-j})\leq A_0\{d(s,z_{\tit}^{l-j})+A_0[d(z_{\tit}^{l-j},z)
+d(z,z_{\tit_0}^{l-j})]\}<20A_0^8C_0\delta^{l-j},$$
which, together with Theorem \ref{10.22.1}(ii),   implies that,
for any $\tit \in\cb_\alpha^{l,j}$,
$$B(z_{\tit}^{l-j},(2A_0C_0)^{-1}\delta^{l-j})\subset B(z_{\tit_0}^{l-j},20A_0^8C_0\delta^{l-j}).$$
From this, Theorem \ref{10.22.1} and Lemma \ref{11.14.2},
we deduce that \eqref{11.15.1} holds true.

Moreover, for any $\tit \in \cb_\alpha^{l,j}$,
$x\in Q_{\tit}^{l-j}$ and $\yap\in\qap\subset Q_{\ta}^l$ with $\ta\in\ca_l^j$,
\begin{align*}
d(x,\yap)&\leq A_0[d(x,z_{\tit}^{l-j})+d(z_{\tit}^{l-j},\yap)]\\
&<2A_0C_0\delta^{l-j}+A_0^2[d(z_{\tit}^{l-j},z_{\tit_0}^{l-j})+d(z_{\tit_0}^{l-j},\yap)]\\
&<36A_0^{10}C_0\delta^{l-j},
\end{align*}
which implies that
$$\bigcup_{\tit\in\cb_\alpha^{l,j}}Q_{\tit}^{l-j}\subset B(\yap,36A_0^{10}C_0\delta^{l-j})$$
and hence
\begin{equation}\label{11.16.1}
\mu\left(\bigcup_{\tit\in\cb_\alpha^{l,j}}Q_{\tit}^{l-j}\right)\lesssim V_{\delta^{l-j}}\left(\yap\right).
\end{equation}
We also notice that, for any $\yap\in\qap\subset Q_{\ta}^l$ with $
\ta\in\ca_l^j$ and $y\in Q_\alpha^l$,
\begin{align*}
5A_0^3\delta^{l-j+1}&\leq d(z_{\ta}^l, z_\alpha^l)\leq A_0[d(z_{\ta}^l,\yap)+d(\yap,z_\alpha^l)]\\
&<2A_0^2\delta^l+A_0^2[d(\yap,y)+d(y,z_\alpha^l)]<4A_0^3\delta^l+A_0^2d(\yap,y)\\
&\leq 4A_0^3\delta^{l-j+1}+A_0^2d(\yap,y),
\end{align*}
which implies that
\begin{equation}\label{11.16.2}
d\left(\yap,y\right)\geq A_0\delta^{l-j+1}
\end{equation}
and hence, by \eqref{11.16.1}, we have, for any $\tit \in \cb_\alpha^{l,j}$,
\begin{equation}\label{vyy}
\mu\left(Q_{\tit}^{l-j}\right)\leq \mu\left(\bigcup_{\tit\in\cb_\alpha^{l,j}}
Q_{\tit}^{l-j}\right)\lesssim V_{\delta^{l-j}}\left(\yap\right)\lesssim V\left(\yap,y\right).
\end{equation}

Now we estimate ${\mathrm I}_{1,2}$. If $ q\in(p(s,\bz\wedge\gz),1]$, then, by
\eqref{r}, we find that, for any fixed $\beta'\in(s,\beta\wedge\gamma)$
and for any $l\in\zz$ and $\alpha\in\ca_l$,
\begin{align*}
{\mathrm I}_{1,2}&\lesssim\frac{1}{\mu(Q_\alpha^l)}\sum_{k=l}^\infty\sum_{\tau\in\ca_k}
\sum_{m=1}^{N(k,\tau)}\delta^{-ksq}\mu(Q_\tau^{k,m})
\mathbf{1}_{\{(\tau,m):\ Q_\tau^{k,m}\subset Q_\alpha^l\}}(\tau,m)\\
&\qquad\times\sum_{k'=l}^\infty\sum_{\alpha' \in \ca_{k'}}\sum_{m'=1}^{N(k',\alpha')}
\delta^{|k-k'|\beta'q}\left[\mu\left(\qap\right)\right]^q\left|Q_{k'}f\left(\yap\right)\right|^q
\mathbf{1}_{\{(\alpha',m'):\ Q_{\alpha'}^{k',m'}\cap (\cup_{\ta\in\cb_\alpha^{l,0}}Q_{\ta}^l)=\emptyset\}}
(\alpha',m')\notag\\
&\qquad\times\left[\frac{1}{V_{\delta^{k\wedge k'}}(\yap)+V(\yap,y_{\tau}^{k,m})}
\right]^q\left[\frac{\delta^{k\wedge k'}}{\delta^{k\wedge k'}+d(\yap,y_{\tau}^{k,m})}\right]^{\gamma q}.\notag
\end{align*}
From this and \eqref{eq-x1}, we deduce that,
for any $l\in\zz$ and $\alpha\in\ca_l$,
\begin{align}\label{11.18.1}
{\mathrm I}_{1,2}&\lesssim\frac{1}{\mu(Q_\alpha^l)}\sum_{k=l}^\infty\sum_{\tau\in\ca_k}
\sum_{m=1}^{N(k,\tau)}\delta^{-ksq}\mu\left(Q_\tau^{k,m}\right)
\mathbf{1}_{\{(\tau,m):\ Q_\tau^{k,m}\subset Q_\alpha^l\}}(\tau,m)\\
&\qquad\times\sum_{k'=l}^\infty\sum_{\alpha' \in \ca_{k'}}\sum_{m'=1}^{N(k',\alpha')}
\delta^{|k-k'|\beta'q}\left[\mu\left(\qap\right)\right]^q\left|Q_{k'}f\left(\yap\right)\right|^q\notag\\
&\qquad\times\sum_{j=1}^\infty\sum_{\ta\in\ca_j^l}
\mathbf{1}_{\{(\alpha',m'):\ \qap\subset Q_{\ta}^{l}\}}(\alpha',m')
\left[\frac{1}{V(\yap,y_{\tau}^{k,m})}\right]^q
\left[\frac{\delta^{k\wedge k'}}{d(\yap, y_{\tau}^{k,m})}\right]^{\gamma q}\notag\\
&\lesssim\frac{1}{\mu(Q_\alpha^l)}\sum_{k=l}^\infty\sum_{\tau\in\ca_k}
\sum_{m=1}^{N(k,\tau)}\delta^{-ksq}\mu\left(Q_\tau^{k,m}\right)
\mathbf{1}_{\{(\tau,m):\ Q_\tau^{k,m}\subset Q_\alpha^l\}}(\tau,m)\noz\\
&\qquad\times\sum_{k'=l}^\infty\sum_{\alpha' \in \ca_{k'}}\sum_{m'=1}^{N(k',\alpha')}
\delta^{|k-k'|\beta'q}\left[\mu\left(\qap\right)\right]^q
\left|Q_{k'}f\left(\yap\right)\right|^q\notag\\
&\qquad\times\sum_{j=1}^\infty\sum_{\tit\in\cb_\az^{l,j}}
\mathbf{1}_{\{(\alpha',m'):\ \qap\subset Q_{\tit}^{l-j}\}}(\alpha',m')
\left[\frac{1}{V(\yap,y_{\tau}^{k,m})}\right]^q
\left[\frac{\delta^{k\wedge k'}}{d(\yap, y_{\tau}^{k,m})}\right]^{\gamma q}.\noz
\end{align}
Notice that, for any $j\in\nn$, $\tit\in\cb_\alpha^{l,j}$, $Q_{\az'}^{k',m'}\subset Q_{\tit}^{l-j}$
and $\yap\in\qap$, we have $V_{\dz^{l-j}}(\yap)\sim\mu(Q_{\tit}^{l-j})$.
By this, $y_\tau^{k,m}\in Q_\tau^{k,m}\subset Q_\az^l$, \eqref{11.16.2} and \eqref{vyy}, we conclude that
\begin{align}\label{eq-x2}
&\mathbf{1}_{\{(\alpha',m'):\ \qap\subset Q_{\tit}^{l-j}\}}(\alpha',m')
\left[\frac{1}{V(\yap,y_{\tau}^{k,m})}\right]^q
\left[\frac{\delta^{k\wedge k'}}{d(\yap, y_{\tau}^{k,m})}\right]^{\gamma q}\\
&\quad\ls\mathbf{1}_{\{(\alpha',m'):\ \qap\subset Q_{\tit}^{l-j}\}}(\alpha',m')
\left[\frac{1}{\mu(Q_{\tit}^{l-j})}\right]^q\dz^{(k\wedge k'-l+j)\gamma'q}\noz\\
&\quad\sim\mathbf{1}_{\{(\alpha',m'):\ \qap\subset Q_{\tit}^{l-j}\}}(\alpha',m')
\frac{1}{\mu(Q_{\tit}^{l-j})}\lf[V_{\dz^{l-j}}\lf(\yap\r)\r]^{1-q}\dz^{(k\wedge k'-l+j)\gamma'q}.\noz
\end{align}
Moreover, since $l\le k'$ and $j\in\nn$, it then follows that $\dz^{l-j}>\dz^{k'}$,
which, combined with
\eqref{eq-doub}, $\yap\in\qap$ and Lemma \ref{10.22.1}(iii), further implies that
$$
\lf[V_{\dz^{l-j}}\lf(\yap\r)\r]^{1-q}\ls\dz^{(l-j-k')\omega(1-q)}\lf[\mu\lf(\qap\r)\r]^{1-q}.
$$
Notice that, since $l\leq k'$ and $q\in (p(s,\beta\wedge\gamma),1]$, it follows that, if choosing
 $\bz'\in(0,\bz\wedge\gz)$ such that $s\in(-\bz',\bz')$ and
$q\in(\omega/(\omega+\bz'+s),1]$, then
\begin{align}\label{eq-x30}
&\sum_{k=l}^{k'}\delta^{|k-k'|\beta'q}
\delta^{(k'-k)sq}\delta^{(k\wedge k')\gamma q}\delta^{-k'\omega(1-q)}
\delta^{l\omega(1-q)-l\gamma q}\\
&\qquad =\sum_{k=l}^{k'}\delta^{(k'-k)\beta'q}
\delta^{(k'-k)sq}\delta^{ k\gamma q}\delta^{-k'\omega(1-q)}
\delta^{l\omega(1-q)-l\gamma q}\notag\\
&\qquad \leq\sum_{k=l}^{k'}\delta^{(k'-k)\beta'q}
\delta^{(k'-k)sq}\delta^{ k\gamma q}\delta^{-k'\omega(1-q)}
\delta^{k\omega(1-q)-k\gamma q}\notag\\
&\qquad = \sum_{k=l}^{k'}\delta^{(k'-k)[\beta'q+sq-\omega(1-q)]}\lesssim 1\notag
\end{align}
and
\begin{align}\label{eq-x31}
&\sum_{k=k'+1}^{\infty}\delta^{|k-k'|\beta'q}
\delta^{(k'-k)sq}\delta^{(k\wedge k')\gamma q}\delta^{-k'\omega(1-q)}
\delta^{l\omega(1-q)-l\gamma q}\\
&\qquad =\sum_{k=k'+1}^{\infty}\delta^{(k-k')\beta'q}
\delta^{(k-k')sq}\delta^{ k'\gamma q}\delta^{-k'\omega(1-q)}
\delta^{l\omega(1-q)-l\gamma q}\noz\\
&\qquad =\sum_{k=k'+1}^{\fz}\delta^{(k-k')(\beta'q-sq)}
\delta^{(k'-l)[\gamma q-\omega(1-q)]}\lesssim 1,\noz
\end{align}
where we used the fact $q\in(\omega/(\omega+\gz),1]$.
By this, \eqref{11.18.1}, \eqref{eq-x2}, \eqref{11.15.1}, \eqref{eq-x30} and \eqref{eq-x31},
we conclude that, for any $l\in\zz$ and $\alpha\in\ca_l$,
\begin{align}\label{eq-x3}
\mathrm I_{1,2}&\lesssim\frac{1}{\mu(Q_\alpha^l)}\sum_{k=l}^\infty\sum_{\tau\in\ca_k}
\sum_{m=1}^{N(k,\tau)}\delta^{-ksq}\mu(Q_\tau^{k,m})
\mathbf{1}_{\{(\tau,m):\ Q_\tau^{k,m}\subset Q_\alpha^l\}}(\tau,m)\\
&\qquad\times\sum_{k'=l}^\infty\sum_{\alpha' \in \ca_{k'}}
\sum_{m'=1}^{N(k',\alpha')}\delta^{|k-k'|\beta'q}\mu\left(\qap\right)
\left|Q_{k'}f\left(\yap\right)\right|^q\notag\\
&\qquad\times\sum_{j=1}^\infty
\sum_{\tit\in\cb_\alpha^{l,j}}\mathbf{1}_{\{(\alpha',m'):\ \qap\subset Q_{\tit}^{l-j}\}}(\alpha',m')
\frac{1}{\mu(Q_{\tit}^{l-j})}\delta^{(l-j-k')\omega(1-q)}
\delta^{[(k\wedge k')-l+j]\gamma q}\notag\\
&\lesssim\sum_{j=1}^\infty\delta^{ j[\gamma q-\omega(1-q)]}\sum_{\tit\in\cb_\alpha^{l,j}}
\frac{1}{\mu(Q_{\tit}^{l-j})}\sum_{k'=l}^\infty\sum_{\alpha'\in\ca_{k'}}
\sum_{m'=1}^{N(k',\alpha')}\delta^{-k'sq}\mu\left(\qap\right)\notag\\
&\qquad\times\mathbf{1}_{\{(\alpha',m'):\ \qap\subset Q_{\tit}^{l-j}\}}(\alpha',m')
\left|Q_{k'}(f)\left(\yap\right)\right|^q\notag\\
&\qquad\times\left[\sum_{k=l}^\infty\delta^{|k-k'|\beta'q}
\delta^{(k'-k)sq}\delta^{(k\wedge k')\gamma q}\delta^{-k'\omega(1-q)}
\delta^{l\omega(1-q)-l\gamma q}\right]\notag\\
&\lesssim \sum_{j=1}^\infty\delta^{j[\gamma q-\omega(1-q)]}
\sum_{\tit\in\cb_\alpha^{l,j}}\frac{1}{\mu(Q_{\tit}^{l-j})}\sum_{k'=l-j}^\infty
\sum_{\alpha'\in\ca_{k'}}\sum_{m'=1}^{N(k',\alpha')}\delta^{-k'sq}\mu\left(\qap\right)\notag\\
&\qquad\times\mathbf{1}_{\{(\alpha',m'):\ \qap\subset Q_{\tit}^{l-j}\}}(\alpha',m')
\left|Q_{k'}(f)\left(\yap\right)\right|^q\notag\\
&\lesssim \sup_{l\in\zz} \sup_{\tit\in\ca_l}\frac{1}{\mu(Q_{\tit}^{l})}\sum_{k'=l}^\infty
\sum_{\alpha'\in\ca_{k'}}\sum_{m'=1}^{N(k',\alpha')}\delta^{-k'sq}\mu\left(\qap\right)\noz\\
&\qquad\times\mathbf{1}_{\{(\alpha',m'):\ \qap\subset Q_{\tit}^{l}\}}(\alpha',m')
\left|Q_{k'}(f)\left(\yap\right)\right|^q.\notag
\end{align}
Here we chose $\bz'\in(0,\bz\wedge\gz)$ such that $s\in(-\bz',\bz')$ and
$q\in(\omega/(\omega+\bz'+s),1]$.

On the other hand, if $q \in (1,\infty]$, then, by
the H\"{o}lder inequality,  Lemma \ref{9.14.1}, \eqref{11.16.2}, \eqref{11.15.1}, \eqref{11.16.1},
and  an argument similar to that used in the estimation of \eqref{eq-x3}, we choose $\bz'\in(0,\bz\wedge\gz)$
such that $s\in(-\bz',\bz')$ and hence conclude that, for any $l\in\zz$ and $\alpha\in\ca_l$,
\begin{align}\label{11.18.2}
\rm{{\rm I}_{1,2}}&\lesssim \frac{1}{\mu(Q_\alpha^l)}\sum_{k=l}^\infty
\sum_{\tau\in\ca_k}\sum_{m=1}^{N(k,\tau)}\mu(Q_\tau^{k,m})
\mathbf{1}_{\{(\tau,m):\ Q_\tau^{k,m}\subset Q_\alpha^l\}}(\tau,m)\\
&\qquad\times\sum_{k'=l}^\infty\sum_{\alpha' \in \ca_{k'}}
\sum_{m'=1}^{N(k',\alpha')}\delta^{|k-k'|\beta'}\delta^{k'sq}\delta^{(k'-k)s}
\mu\left(\qap\right)\left|Q_{k'}f\left(\yap\right)\right|^q\noz\\
&\qquad\times\mathbf{1}_{\{(\alpha',m'):\ Q_{\alpha'}^{k',m'}
\cap(\cup_{\ta\in\cb_\alpha^{l,0}}Q_{\ta}^l)=\emptyset\}}(\alpha',m')\notag\\
&\qquad\times\frac{1}{V_{\delta^{k\wedge k'}}(\yap)+V(\yap,\ya)}
\left[\frac{\delta^{k\wedge k'}}{\delta^{k\wedge k'}+d(\yap,\ya)}\right]^{\gamma}\notag\\
&\lesssim \sum_{j=0}^\infty\delta^{ j\gamma}\sum_{\tit\in\cb_\az^{l,j}}
\frac{1}{\mu(Q_{\tit}^{l-j})}\sum_{k'=l-j}^\infty
\sum_{\alpha'\in\ca_{k'}}\sum_{m'=1}^{N(k',\alpha')}\delta^{-k'sq}\mu\left(\qap\right)\notag\\
&\qquad\times\mathbf{1}_{\{(\alpha',m'):\ \qap\subset Q_{\tit}^{l-j}\}}(\alpha',m')
\left|Q_{k'}(f)\left(\yap\right)\right|^q\notag\\
&\lesssim \sup_{l\in\zz} \sup_{\tit\in\ca_l}\frac{1}{\mu(Q_{\tit}^{l})}
\sum_{k'=l}^\infty\sum_{\alpha'\in\ca_{k'}}
\sum_{m'=1}^{N(k',\alpha')}\delta^{-k'sq}\mu\left(\qap\right)\notag\\
&\qquad\times\mathbf{1}_{\{(\alpha',m'):\ \qap\subset Q_{\tit}^{l}\}}(\alpha',m')
\left|Q_{k'}(f)\left(\yap\right)\right|^q\notag.
\end{align}
Now we estimate $\rm{{\rm I}_2}$. Notice that, for any
$k'\in \zz$, $\alpha'\in\ca_{k'}$ and $m'\in\{1,\dots,N(k',\alpha')\}$,
$$\mu\left(Q_{\alpha'}^{k'}\right)\lesssim \mu\left(\qap\right),$$
which implies that, for any $k'\in\zz$,
\begin{align}\label{11.19.1}
&\sum_{\alpha'\in\ca_{k'}}\sum_{m'=1}^{N(k',\alpha')}\delta^{-k'sq}
\left|Q_{k'}(f)\left(\yap\right)\right|^q\\
&\quad\lesssim\sum_{\alpha'\in\ca_{k'}}\sum_{m'=1}^{N(k',\alpha')}\delta^{-k'sq}
\frac{\mu(\qap)}{\mu(Q_{\alpha'}^{k'})}\mathbf{1}_{\{(\alpha',m'):\ \qap\subset
Q_{\alpha'}^{k'}\}}(\alpha',m')\left|Q_{k'}(f)\left(\yap\right)\right|^q\notag\\
&\quad\lesssim \sup_{l \in \zz} \sup_{\alpha\in\ca_l}\frac{1}{\mu(Q_\alpha^l)}
\sum_{k=l}^\infty\sum_{\tau\in\ca_k}\sum_{m=1}^{N(k,\tau)}\delta^{-ksq}\mu(Q_\tau^{k,m})
\mathbf{1}_{\{(\tau,m):\ Q_\tau^{k,m}\subset Q_\alpha^l\}}(\tau,m)
\left|Q_k(f)\left(\ya\right)\right|^q.\notag
\end{align}
On one hand, if $q\in(p(s,\bz\wedge\gz),1]$, by \eqref{r} and \eqref{11.19.1}, we
choose $\bz'\in(s_+,\bz\wedge\gz)$ and then conclude that,
for any $l\in\zz$ and $\alpha\in\ca_l$,
\begin{align}\label{11.19.2}
\rm{{\rm I}_2}&\lesssim \frac{1}{\mu(Q_\alpha^l)}
\sum_{k=l}^\infty\sum_{\tau\in\ca_k}\sum_{m=1}^{N(k,\tau)}\delta^{-ksq}
\mu\left(Q_\tau^{k,m}\right)\mathbf{1}_{\{(\tau,m):\ Q_\tau^{k,m}
\subset Q_\alpha^l\}}(\tau,m)\\
&\qquad\times\sum_{k'=-\infty}^{l-1}\sum_{\alpha' \in \ca_{k'}}
\sum_{m'=1}^{N(k',\alpha')}\delta^{(k-k')\beta' q}
\left|Q_{k'}f\left(\yap\right)\right|^q\notag\\
&\lesssim \sum_{k=l}^\infty\delta^{k(\beta'-s)q}\left[\frac{1}{\mu(Q_\alpha^l)}
\sum_{\tau\in\ca_k}\sum_{m=1}^{N(k,\tau)}\mu\left(Q_\tau^{k,m}\right)
\mathbf{1}_{\{(\tau,m):\ Q_\tau^{k,m}\subset Q_\alpha^l\}}(\tau,m)\right]\notag\\
&\qquad\times\sum_{k'=-\infty}^{l-1}\delta^{k'(s-\beta')q}
\sum_{\alpha' \in \ca_{k'}}\sum_{m'=1}^{N(k',\alpha')}\delta^{k'sq}
\left|Q_{k'}f\left(\yap\right)\right|^q\notag\\
&\lesssim\sup_{l \in \zz} \sup_{\alpha\in\ca_l}\frac{1}{\mu(Q_\alpha^l)}
\sum_{k=l}^\infty\sum_{\tau\in\ca_k}\sum_{m=1}^{N(k,\tau)}\delta^{-ksq}
\mu\left(Q_\tau^{k,m}\right)\mathbf{1}_{\{(\tau,m):\ Q_\tau^{k,m}\subset Q_\alpha^l\}}(\tau,m)
\left|Q_k(f)\left(\yap\right)\right|^q.\notag
\end{align}
On another hand, if $q\in (1,\infty)$, we then choose $\beta'\in(|s|,\beta\wedge\gamma)$ and
$\widetilde{q}\in(1,q)$ such that $sq<\beta'\widetilde{q}<\beta'q$.
From this, \eqref{11.19.1}, Lemma \ref{9.14.1} and the H\"{o}lder inequality,
we deduce that, for any $l\in\zz$ and $\alpha\in\ca_l$,
\begin{align}\label{i02}
\rm{{\rm I}_2}&\lesssim \frac{1}{\mu(Q_\alpha^l)}\sum_{k=l}^\infty\sum_{\tau\in\ca_k}
\sum_{m=1}^{N(k,\tau)}\delta^{-ksq}\mu\left(Q_\tau^{k,m}\right)
\mathbf{1}_{\{(\tau,m):\ Q_\tau^{k,m}\subset Q_\alpha^l\}}(\tau,m)\\
&\qquad\times\left\{\left[\sum_{k'=-\infty}^{l-1}\sum_{\alpha' \in\ca_{k'}}
\sum_{m'=1}^{N(k',\alpha')}\delta^{(k-k')\beta'\widetilde{q}}
\left|Q_{k'}f\left(\yap\right)\right|^q\right]^{1/q}
\left[\sum_{k=-\infty}^{l-1}\delta^{(k-k')
\beta'q'/(q/\widetilde{q})'}\right]^{1/q'}\right\}^q\notag\\
&\lesssim \sum_{k=l}^\infty\delta^{k(\beta'\widetilde{q}-sq)}
\left[\frac{1}{\mu(Q_\alpha^l)}\sum_{\tau\in\ca_k}\sum_{m=1}^{N(k,\tau)}
\mu\left(Q_\tau^{k,m}\right)\mathbf{1}_{\{(\tau,m):\ Q_\tau^{k,m}
\subset Q_\alpha^l\}}(\tau,m)\right]\notag\\
&\qquad\times\sum_{k'=-\infty}^{l-1}\delta^{k'(sq-\beta'\widetilde{q})}
\sum_{\alpha' \in \ca_{k'}}\sum_{m'=1}^{N(k',\alpha')}\delta^{k'sq}
\left|Q_{k'}f\left(\yap\right)\right|^q\notag\\
&\lesssim\sup_{l \in \zz} \sup_{\alpha\in\ca_l}\frac{1}{\mu(Q_\alpha^l)}
\sum_{k=l}^\infty\sum_{\tau\in\ca_k}\sum_{m=1}^{N(k,\tau)}\delta^{-ksq}
\mu\left(Q_\tau^{k,m}\right)\mathbf{1}_{\{(\tau,m):\ Q_\tau^{k,m}\subset Q_\alpha^l\}}(\tau,m)
\left|Q_k(f)\left(\yap\right)\right|^q.\notag
\end{align}
If $q=\infty$,  it is easy to see that \eqref{i02} holds true.
By this, together with Lemma \ref{11.14.2}, \eqref{11.11.1}, \eqref{11.18.1}, \eqref{11.18.2}, \eqref{11.19.2}, the
arbitrariness of $\yap$ and an argument similar
to that used in the estimation of \eqref{limpp},
we then complete the proof of  Lemma \ref{11.14.1}.
\end{proof}

Using an argument similar to that used in the proof of Proposition \ref{6.4.1}, we can prove the following
proposition and we omit the details here.
\begin{proposition}
Let $\beta,\ \gamma \in (0, \eta)$,
$s \in (-(\beta\wedge\gamma), \beta\wedge\gamma)$
and $q\in (p(s,\beta\wedge\gamma),\infty]$ with $\eta$
as in Definition \ref{10.23.2} and $p(s,\beta\wedge\gamma)$
as in \eqref{pseta}. Let $\{P_k\}_{k\in\zz}$ and $\{Q_k\}_{k\in\zz}$
be two \emph{$\exp$-ATIs}. Then there exists a constant $C\in[1,\fz)$ such that,
for any $f\in(\cggi)'$,
\begin{align*}
&C^{-1}\sup_{l \in \zz} \sup_{\alpha\in\ca_l}\left[\frac{1}{\mu(Q_\alpha^l)}
\int_{Q_\alpha^l}\sum_{k=l}^\infty\delta^{-ksq}|P_k(f)(x)|^q\,d\mu(x)\right]^{1/q}\\
&\quad \le\sup_{l \in \zz} \sup_{\alpha\in\ca_l}
\left[\frac{1}{\mu(Q_\alpha^l)}\int_{Q_\alpha^l}\sum_{k=l}^\infty
\delta^{-ksq}|Q_k(f)(x)|^q\,d\mu(x)\right]^{1/q}\\
&\quad\le C\sup_{l \in \zz} \sup_{\alpha\in\ca_l}\left[\frac{1}{\mu(Q_\alpha^l)}
\int_{Q_\alpha^l}\sum_{k=l}^\infty\delta^{-ksq}|P_k(f)(x)|^q\,d\mu(x)\right]^{1/q}.
\end{align*}
\end{proposition}

The following proposition shows that the space $\hfi$ is
independent of the choice of the space of
distributions.

\begin{proposition}\label{12.2.7}
Let $\beta,\ \gamma\in(0,\eta)$ and $s\in(-\eta,\eta)$,
with $\eta$ be as in Definition \ref{10.23.2},  satisfy
\begin{equation}\label{12.1.2}
\beta\in((-s)_+,\eta)\qquad\text{and}\qquad\ \gamma\in(s_+,\eta).
\end{equation}
If  $q\in(p(s,\beta\wedge\gamma),\infty]$ with $p(s,\beta\wedge\gamma)$ as in \eqref{pseta},
then there exists a positive
constant $C$ such that, for any $f \in \hfi\subset(\cggi)'$, $f \in (\cggt)'$ with $\widetilde{\beta}$ and
$\widetilde{\gamma}$ satisfying $q\in(p(s,\widetilde{\beta}\wedge\widetilde{\gamma}),\infty]$
with $p(s,\widetilde{\beta}\wedge\widetilde{\gamma})$ as
in \eqref{pseta} via replacing $\beta$ and $\gamma$ respectively by $\widetilde{\beta}$ and
$\widetilde{\gamma}$, and \eqref{12.1.2}
via replacing $\beta$ and $\gamma$ respectively by $\widetilde{\beta}$ and
$\widetilde{\gamma}$, and there exists a positive constant $C$,
independent of $f$, such that
$$\|f\|_{(\cggt)'}\leq C\|f\|_{\hfi}.$$
\end{proposition}

\begin{proof}
Let $\eta$ be as in Definition \ref{10.23.2},
$\psi\in\mathring{\cg}(\eta,\eta)$ and adopt the notation from Lemma \ref{crf}.
If $q\in(p(s,\wz\beta\wedge\wz\gamma),1]$ with $\widetilde{\beta}$
and $\widetilde{\gamma}$ as in this proposition, by Lemma \ref{crf},
\eqref{6.14.2}, \eqref{6.14.3} and \eqref{r},
we know that, for any $f\in\hfi\subset(\cggi)'$ with $\beta$
and $\gamma$ as in this proposition,
\begin{align}\label{12.2.1}
|\langle f,\psi\rangle|&=\left|\sum_{k=-\infty}^\infty\sum_{\alpha\in\ca_k}
\sum_{m=1}^{N(k,\alpha)}\mu\left(\qa\right)Q_k(f)\left(\ya\right)
\left\langle \widetilde{Q}_k\left(\cdot,\ya\right),\psi\right\rangle\right|\\
&\lesssim \|\psi\|_{\cg(\widetilde{\beta},\widetilde{\gamma})}\left\{\sum_{k=0}^\infty
\sum_{\alpha\in\ca_k}\sum_{m=1}^{N(k,\alpha)}\delta^{k\widetilde{\beta}q}
\left[\mu\left(\qa\right)\left|Q_k(f)\left(\ya\right)\right|\right]^q\right.\notag\\
&\qquad\times\left[ \frac{1}{V_1(x_1)+V(x_1,\ya)}
\frac{1}{(1+d(x_1,\ya))^{\widetilde{\gamma}}}\right]^q\notag\\
&\qquad +\sum_{k=-\infty}^{-1}\sum_{\alpha\in\ca_k}\sum_{m=1}^{N(k,\alpha)}
\delta^{-k\widetilde{\gamma}q}\left[\mu\left(\qa\right)
\left|Q_k(f)\left(\ya\right)\right|\right]^q\notag\\
&\qquad \left.\times \left[\frac{1}{V_{\delta^k}(x_1)+V(x_1,\ya)}
\frac{\delta^{k\widetilde{\gamma}}}{(\delta^k+d(x_1,\ya))^{\widetilde{\gamma}}}\right]^q\right\}^{1/q}.\notag
\end{align}
On one hand, for any $k\in\zz_+$, if $d(x_1,\ya)< 1$, then 
$$\mu(\qa)\sim V_{\delta^k}(\ya)\lesssim V_1(\ya)\sim V_1(x_1);$$
otherwise, there exists an $l\in\zz_+$ such that  $\delta^{-l}\leq d(x_1,\ya)<\delta^{-(l+1)}$
and hence  $\mu(\qa)\sim V_{\delta^k}(\ya)\lesssim V_{\delta^{-l}}(\ya)\lesssim V(x_1,\ya)$.
Therefore, for any $k\in\zz_+$,
\begin{equation}\label{12.2.2}
\frac{\mu(\qa)}{V_1(x_1)+V(x_1,\ya)}\lesssim 1.
\end{equation}
On the other hand, for any $k\in\zz\setminus\zz_+$, if $d(x_1,\ya)< \delta^k$, then $\mu(\qa)\sim
V_{\delta^k}(\ya)\lesssim V_{\delta^k}(x_1)$; if $\delta^{-l}\delta^k\leq d(x_1,\ya)<\delta^{-(l+1)}\delta^k$
for some $l\in\zz_+$, then 
$$\mu(\qa)\sim V_{\delta^k}(\ya)\lesssim V_{\delta^{-l}\delta^k}(\ya)\lesssim
V(x_1,\ya).$$ 
Therefore, \eqref{12.2.2} also holds true when $k\in\zz\setminus\zz_+$.

By \eqref{12.1.2}, we know that $\wz\beta\in((-s)_+,\eta)$
and $ \wz\gamma\in(s_+,\eta)$,
and hence $\wz\beta\in(-s,\eta)$ and $ \wz\gamma\in(s,\eta)$. Using this, \eqref{11.19.1} and \eqref{12.2.1},
\eqref{12.2.2}, we conclude  that, for any $f\in\hfi\subset(\cggi)'$ with $\beta$
and $\gamma$ as in this proposition,
\begin{equation}\label{12.2.4}
|\langle f,\psi\rangle|\lesssim \|\psi\|_{\cg(\widetilde{\beta},\widetilde{\gamma})}
\|f\|_{\hfi}\left\{\sum_{k=0}^\infty\delta^{k(\widetilde{\beta}+s)q}
+\sum_{k=-\infty}^{-1}\delta^{k(s-\widetilde{\gamma})q}\right\}^{1/q}
\lesssim\|\psi\|_{\cg(\widetilde{\beta},\widetilde{\gamma})}\|f\|_{\hfi}.
\end{equation}

If $q\in(1,\infty]$, by Lemma \ref{crf}, \eqref{6.14.2},
\eqref{6.14.3} and the H\"{o}lder inequality, we
find that, for any $f\in\hfi\subset(\cggi)'$ with $\beta$
and $\gamma$ as in this proposition,
\begin{align}\label{12.2.5}
|\langle f,\psi\rangle|&=\left|\sum_{k=-\infty}^\infty\sum_{\alpha\in\ca_k}
\sum_{m=1}^{N(k,\alpha)}\mu\left(\qa\right)Q_k(f)\left(\ya\right)
\left\langle \widetilde{Q}_k\left(\cdot,\ya\right),\psi\right\rangle\right|\\
&\lesssim \|\psi\|_{\cg(\widetilde{\beta},\widetilde{\gamma})}\left\{\sum_{k=0}^\infty
\delta^{k\widetilde{\beta}}\delta^{ks}\left[\sum_{\alpha\in\ca_k}\sum_{m=1}^{N(k,\alpha)}
\delta^{-ksq}\left\{\mu\left(\qa\right)\left|Q_k(f)\left(\ya\right)\right|\right\}^q\right.\right.\notag\\
&\qquad\times\left. \frac{1}{V_1(x_1)+V(x_1,\ya)}\frac{1}{[1+d(x_1,\ya)]^{\widetilde{\gamma}}}\right]^{1/q}\notag\\
&\qquad\times\left[\int_{\cx}
\frac{1}{V_1(x_1)+V(x_1,y)}\frac{1}{[1+d(x_1,y)]^{\widetilde{\gamma}}}\,d\mu(y)\right]^{1/q'}\notag\\
&\qquad +\sum_{k=-\infty}^{-1}\delta^{-k\widetilde{\gamma}}\delta^{ks}
\left[\sum_{\alpha\in\ca_k}\sum_{m=1}^{N(k,\alpha)}\delta^{-ksq}
\left\{\mu\left(\qa\right)\left|Q_k(f)\left(\ya\right)\right|\right\}^q\right.\notag\\
&\qquad \times\left. \frac{1}{V_{\delta^k}(x_1)+V(x_1,\ya)}
\frac{\delta^{k\widetilde{\gamma}}}{[\delta^k+d(x_1,\ya)]^{\widetilde{\gamma}}}\right]^{1/q}\notag\\
&\qquad\times\left.\left[\int_{\cx}\frac{1}{V_{\delta^k}(x_1)+V(x_1,y)}
\frac{\delta^{k\widetilde{\gamma}}}{[\delta^k+d(x_1,y)]^{\widetilde{\gamma}}}\,d\mu(y)\right]^{1/q'}\right\}.\notag
\end{align}
By \eqref{12.1.2}, we know that $\wz\beta\in(-s,\eta)$ and $\wz\gamma\in(s,\eta)$. Using this,
\eqref{12.2.5}, \eqref{11.19.1}, \eqref{12.2.1}, \eqref{12.2.2}
and Lemma \ref{6.15.1}(ii), we conclude that,
for any $f\in\hfi\subset(\cggi)'$ with $\beta$
and $\gamma$ as in this proposition,
\begin{align}\label{12.2.6}
|\langle f,\psi\rangle|\lesssim \|\psi\|_{\cg(\widetilde{\beta},\widetilde{\gamma})}\|f\|_{\hfi}
\left\{\sum_{k=0}^\infty\delta^{k(\widetilde{\beta}+s)}
+\sum_{k=-\infty}^{-1}\delta^{k(s-\widetilde{\gamma})}\right\}\lesssim\|\psi\|_{\cg(\widetilde{\beta},
\widetilde{\gamma})}\|f\|_{\hfi}.
\end{align}

Combining \eqref{12.2.4} and \eqref{12.2.6}, and using an argument similar to that used in the proof of
Proposition  \ref{6.5.1}, we prove that $f \in (\cggii)'$, which completes the proof of Proposition
\ref{12.2.7}.
\end{proof}

Some basic properties of $\hfi$ are stated in the following proposition.

\begin{proposition}\label{12.3.1}
Let $\beta,\ \gamma \in (0, \eta)$, $s\in(-\eta,\eta)$ and
$q\in (p(s,\beta\wedge\gamma),\infty]$ with $\eta$ and $p(s,\beta\wedge\gamma)$, respectively, as in
Definition \ref{10.23.2} and \eqref{pseta}.
\begin{enumerate}
\item[{\rm(i)}] If $p(s,\beta\wedge\gamma)<q_1\leq q_2\leq\infty$, then
$\dot{F}^s_{\infty,q_1}(\cx)\subset\dot{F}^s_{\infty,q_2}(\cx)$;
\item[{\rm(ii)}] $\dot{B}^s_{\infty,q}(\cx)\subset\dot{F}^s_{\infty,q}(\cx)\subset\dot{B}^s_{\infty,\infty}(\cx)$;
\item[{\rm(iii)}] If  $\widetilde{\beta},\ \widetilde{\gamma}\in(|s|,\eta]$, then
$\ccg(\widetilde{\beta},\widetilde{\gamma})\subset\hfi$.
\end{enumerate}
\end{proposition}

\begin{proof}
Property (i) is a direct corollary of \eqref{r} and Definition \ref{hfi}; we omit the details here.

To show (ii), on one hand, by the definitions of
$\dot{B}^s_{\infty,q}(\cx)$ and $\dot{F}^s_{\infty,q}(\cx)$, we know that, for any $f\in(\cggi)'$ with $\beta$ and $\gamma$ as in  this proposition,
\begin{align*}
\|f\|_{\hfi}&=\sup_{l \in \zz} \sup_{\alpha\in\ca_l}\left[\frac{1}{\mu(Q_\alpha^l)}
\int_{Q_\alpha^l}\sum_{k=l}^\infty\delta^{-ksq}|Q_k(f)(x)|^q\,d\mu(x)\right]^{1/q}\\
&\leq\left[\sum_{k=-\infty}^{\infty}\delta^{-ksq}\|Q_k(f)\|_{L^\infty(\cx)}^q\right]^{1/q}
= \|f\|_{\dot{B}^s_{\infty,q}(\cx)},
\end{align*}
which implies that $\dot{B}^s_{\infty,q}(\cx)\subset\dot{F}^s_{\infty,q}(\cx)$.
On the other hand, by Theorem \ref{10.22.1}(i), we find that,
for any $f\in\dot{F}^s_{\infty,\infty}(\cx)$,
\begin{align*}
 \|f\|_{\dot{B}^s_{\infty,\infty}(\cx)}&=\sup_{l\in\zz}\delta^{-ls}\|Q_l(f)\|_{L^\infty(\cx)}
 =\sup_{l\in\zz}\|\delta^{-ls}|Q_l(f)|\|_{L^\infty(\cx)}\\
&\leq \sup_{l\in\zz}\left\|\sup_{k\geq l}\delta^{-ks}|Q_k(f)|\right\|_{L^\infty(\cx)}
\leq \sup_{l\in\zz}\sup_{\alpha\in\ca_l}\left\|\sup_{k\geq l}
\delta^{-ks}|Q_k(f)|\right\|_{L^\infty(Q_\alpha^l)}= \|f\|_{\dot{F}^s_{\infty,\infty}(\cx)},
\end{align*}
which, combined with Property (i), implies that
$\dot{F}^s_{\infty,q}(\cx)\subset\dot{F}^s_{\infty,\infty}(\cx)
\subset\dot{B}^s_{\infty,\infty}(\cx)$. This  finishes the proof of (ii).

Property (iii) comes from  (ii) and Proposition \ref{phb}(iii) directly; we
omit the details here. This finishes the proof of Proposition \ref{12.3.1}.
\end{proof}

\begin{remark}\label{addre3}
A counterpart on RD-spaces of Proposition \ref{12.3.1}(iii) is
\cite[Proposition 6.9(iv)]{hmy08} in which $\gz$ is required to satisfy
\begin{equation}\label{5.21x}
\gamma\in(-s-\kappa,\eta),
\end{equation}
where $\kappa\in(0,\omega]$ is as in \eqref{eq-rdoub}. Similarly to Remarks  \ref{addre2} and \ref{addre1},
we find that, when $\kz:=0$ in \eqref{5.21x}, the above inequality coincides with the range
of $\wz\gz$ in Proposition \ref{12.3.1}(iii). In this sense, we may say that Proposition \ref{12.3.1}(iii)
generalizes \cite[Proposition 6.9(iv)]{hmy08} and that the range of $\widetilde{\gz}$ in Proposition \ref{12.3.1}(iii) is \emph{optimal}.
\end{remark}

\subsection{Inhomogeneous Triebel--Lizorkin spaces $\ihfi$}

In this section, we concentrate on the inhomogeneous Triebel--Lizorkin spaces $\ihfi$
and hence we do not need the assumption
$\mu(\cx)=\infty$. Let us begin with its  definition.

\begin{definition}\label{ihfi}
Let  $\beta,\ \gamma \in (0, \eta)$ and $s\in(-\eta,\eta)$
with $\eta$ as in
Definition \ref{10.23.2}. Let $\{Q_k\}_{k\in\zz}$ be an exp-IATI and $N\in\nn$ as in Lemma \ref{icrf}.
The \emph{inhomogeneous Triebel--Lizorkin space $\ihfi$} is defined by setting
\begin{align*}
\ihfi := &{}\left\{f  \in (\icgg)' :\  \|f\|_{\ihfi}:={}\max\left[\sup_{k\in\{0,\dots,N\}}
\sup_{\alpha \in \ca_k}\sup_{m\in\{1,\dots,N(k,\alpha)\}}m_{\qa}(|Q_k(f)|),\r.\r.\\
&\qquad\qquad\qquad\qquad\left.\left.\sup_{l \in \nn,\ l>N} \sup_{\alpha\in\ca_l}\left\{\frac{1}{\mu(Q_\alpha^l)}
\int_{Q_\alpha^l}\sum_{k=l}^\infty\delta^{-ksq}|Q_k(f)(x)|^q\,d\mu(x)\right\}^{1/q}\r]<\infty\right\}
\end{align*}
with usual modification made when $q=\infty$.
\end{definition}

We now establish the inhomogeneous Plancherel--P\^olya inequality for the case $p=\infty$.

\begin{lemma}\label{12.3.2}
Let $\{Q_k\}_{k=0}^\infty$ and $\{P_k\}_{k=0}^\infty$ be two {\rm $\exp$-IATIs}.
Let $\beta,\ \gamma\in(0,\eta)$,
$s\in(-(\beta\wedge\gamma), \beta\wedge\gamma)$ and $q \in (p(s,\beta\wedge\gamma), \infty]$ with
$\eta$ and $p(s,\beta\wedge\gamma)$, respectively, as
in Definition \ref{10.23.2} and \eqref{pseta}. Let $N\in\nn$ and $N'\in\nn$ be as in Lemma
\ref{icrf} associated, respectively, with $\{Q_k\}_{k=0}^\infty$ and $\{P_k\}_{k=0}^\infty$.
Then there exists a positive constant $C$ such that, for any $f\in (\icgg)'$,
\begin{align*}
&\max\left\{\sup_{k\in\{0,\dots,N\}}\sup_{\alpha \in \ca_k}
\sup_{m\in\{1,\dots,N(k,\alpha)\}}m_{\qa}(|Q_k(f)|),\r.\\
&\qquad\quad\left.\sup_{l \in \nn, l>N} \sup_{\alpha\in\ca_l}\left\{\frac{1}{\mu(Q_\alpha^l)}
\sum_{k=l}^\infty\sum_{\tau\in\ca_k}\sum_{m=1}^{N(k,\tau)}
\delta^{-ksq}\mu(Q_\tau^{k,m})\mathbf{1}_{\{(\tau,m):\ Q_\tau^{k,m}
\subset Q_\alpha^l\}}(\tau,m)\r.\r.\\
&\qquad\quad\times\left.\left.\left[\sup_{x\in Q_\tau^{k,m}}|Q_k(f)(x)|\right]^q\right\}^{1/q}\right\}\\
&\qquad\leq C \max\left\{\sup_{k'\in\{0,\dots,N'\}}\sup_{\alpha' \in \ca_{k'}}
\sup_{m'\in\{1,\dots,N(k',\alpha')\}}m_{\qap}(|P_{k'}(f)|),\r.\\
&\qquad\quad\left.\sup_{l' \in \nn, l'>N'} \sup_{\alpha'\in\ca_{l'}}
\left\{\frac{1}{\mu(Q_{\alpha'}^{l'})}\sum_{k'=l'}^\infty
\sum_{\tau'\in\ca_k'}\sum_{m'=1}^{N(k',\tau')}\delta^{-k'sq}
\mu(Q_{\tau'}^{k',m'})\mathbf{1}_{\{(\tau',m'):\ Q_{\tau'}^{k',m'}
\subset Q_{\alpha'}^{l'}\}}(\tau',m')\r.\r.\\
&\quad\qquad\times\left.\left.\left[\inf_{x\in Q_{\tau'}^{k',m'}}
|P_{k'}(f)(x)|\right]^q\right\}^{1/q}\right\}.
\end{align*}
\end{lemma}

\begin{proof}
To prove this lemma, we consider two cases.

{\it Case 1) $k\in\{0,\dots,N\}$.}
In this case, by \eqref{12.3.4}, \eqref{10.14.1} and Lemma \ref{6.15.1}(ii),
we know that, for any $f\in (\icgg)'$ with $\beta$ and $\gamma$ as in this lemma,
\begin{align}\label{12.3.5}
&\sum_{\alpha' \in \ca_0}\sum_{m'=1}^{N(0,\alpha')}m_{\qop}(|P_0(f)|)
\frac{1}{\mu(\qa)}\int_{\qa}\int_{\qop}|Q_k\widetilde{P}_{k'}(z,y)|\,d\mu(y)\,d\mu(z)\\
&\qquad+\sum_{k'=1}^{N'}\sum_{\alpha' \in \ca_{k'}}\sum_{m'=1}^{N(k',\alpha')}m_{\qap}(|P_{k'}(f)|)
\mu\left(\qap\right)'\frac{1}{\mu(\qa)}\int_{\qa}
\left|Q_k\widetilde{P}_{k'}\left(z,\yap\right)\right|\,d\mu(z)\notag\\
&\quad\lesssim \sum_{k'=0}^{N'}\sum_{\alpha' \in \ca_{k'}}
\sum_{m'=1}^{N(k',\alpha')}m_{\qap}(|P_{k'}(f)|)\mu\left(\qap\right)
\inf_{y\in\qap}\frac{1}{V_1(z_\alpha^{k,m})+V(z_\alpha^{k,m},y)}
\left[\frac{1}{1+d(z,y)}\right]^\gamma\notag\\
&\quad\lesssim \sup_{k'\in\{0,\dots,N'\}}\sup_{\alpha' \in \ca_{k'}}
\sup_{m'\in\{1,\dots,N(k',\alpha')\}}m_{\qap}(|P_{k'}(f)|)\noz\\
&\qquad\times\int_{\cx}\frac{1}{V_1(z_\alpha^{k,m})+V(z_\alpha^{k,m},y)}
\left[\frac{1}{1+d(z,y)}\right]^\gamma\,d\mu(y)\notag\\
&\quad\lesssim \sup_{k'\in\{0,\dots,N'\}}\sup_{\alpha' \in \ca_{k'}}
\sup_{m'\in\{1,\dots,N(k',\alpha')\}}m_{\qap}(|P_{k'}(f)|).\notag
\end{align}

If $q\in(p(s,\beta\wedge\gamma),1]$, from \eqref{r}, \eqref{10.14.2},
\eqref{11.19.1}, $\mu(\qap)\lesssim V_1(\yap)$,
$|s|\in(0,\eta)$, the arbitrariness of $\yap$,
and an argument similar to that used in the estimation of \eqref{limpp},
we deduce that, for any $f\in (\icgg)'$ with
$\beta$ and $\gamma$ as in this lemma,
\begin{align}\label{12.4.1}
&\sum_{k'=N'+1}^\infty\sum_{\alpha' \in \ca_{k'}}\sum_{m'=1}^{N(k',\alpha')}
\left|P_{k'}f\left(\yap\right)\right|\mu\left(\qap\right)\frac{1}{\mu(\qa)}\int_{\qa}
\left|Q_k\widetilde{P}_{k'}\left(z,\yap\right)\right|\,d\mu(z)\\
&\quad\lesssim \sum_{k'=N'+1}^\infty\delta^{k'\eta'+k's}\left[\sum_{\alpha' \in \ca_{k'}}
\sum_{m'=1}^{N(k',\alpha')}\delta^{-k'sq}
\left|P_{k'}f\left(\yap\right)\right|^q\right]^{1/q}\notag\\
&\quad\lesssim \sup_{l' \in \nn, l'>N'} \sup_{\alpha'\in\ca_{l'}}
\left\{\frac{1}{\mu(Q_{\alpha'}^{l'})}\sum_{k'=l'}^\infty
\sum_{\tau'\in\ca_k'}\sum_{m'=1}^{N(k',\tau')}\delta^{-k'sq}
\mu(Q_{\tau'}^{k',m'})\mathbf{1}_{\{(\tau',m'):\ Q_{\tau'}^{k',m'}
\subset Q_{\alpha'}^{l'}\}}(\tau',m')\r.\notag\\
&\qquad\quad\times\left.\left.\left[\inf_{x\in Q_{\tau'}^{k',m'}}
|P_{k'}(f)(x)|\right]^q\right\}^{1/q}\right\}\notag.
\end{align}
Here we chose $\eta'\in((-s)_+,\bz\wedge\gz)$.

If $q\in(1,\infty]$, by \eqref{10.14.2}, \eqref{11.19.1}, the H\"older inequality,
(i) and  (ii) of Lemma \ref{6.15.1},
$\mu(\qap)\lesssim V_1(\yap)$, $|s|\in(0,\bz\wedge\gz)$,
the arbitrariness of $\yap$,
and an argument similar to that used in the estimation of \eqref{limpp},
we choose $\eta'\in((-s)_+,\bz\wedge\gz)$ and hence conclude that,
for any $f\in (\icgg)'$ with
$\beta$ and $\gamma$ as in this lemma,
\begin{align}\label{12.4.2}
&\sum_{k'=N'+1}^\infty\sum_{\alpha' \in \ca_{k'}}
\sum_{m'=1}^{N(k',\alpha')}\left|P_{k'}f\left(\yap\right)\right|\mu\left(\qap\right)\frac{1}{\mu(\qa)}
\int_{\qa}\left|Q_k\widetilde{P}_{k'}(z,\yap)\right|\,d\mu(z)\\
&\quad\lesssim \sum_{k'=N'+1}^\infty\delta^{k'\eta'+k's}
\left[\sum_{\alpha' \in \ca_{k'}}\sum_{m'=1}^{N(k',\alpha')}
\delta^{-k'sq}\left|P_{k'}f\left(\yap\right)\right|^q\right]^{1/q}\notag\\
&\quad\qquad\times\left[\int_{\cx}\frac{1}{V_1(z_\alpha^{k,m})+V(z_\alpha^{k,m},y)}
\frac{1}{[1+d(z_\alpha^{k,m},y)]^{\gamma}}\,d\mu(y)\right]^{1/q'}\notag\\
&\quad\lesssim \sup_{l' \in \nn, l'>N'} \sup_{\alpha'\in\ca_{l'}}
\left\{\frac{1}{\mu(Q_{\alpha'}^{l'})}\sum_{k'=l'}^\infty
\sum_{\tau'\in\ca_k'}\sum_{m'=1}^{N(k',\tau')}\delta^{-k'sq}
\mu\left(Q_{\tau'}^{k',m'}\right)\mathbf{1}_{\{(\tau',m'):\ Q_{\tau'}^{k',m'}
\subset Q_{\alpha'}^{l'}\}}(\tau',m')\r.\notag\\
&\quad\qquad\times\left.\left.\left[\inf_{x\in Q_{\tau'}^{k',m'}}
|P_{k'}(f)(x)|\right]^q\right\}^{1/q}\right\}\notag.
\end{align}

Combining the above two inequalities, we obtain the desired estimate.

{\it Case 2) $k\in\{N+1,N+2,\dots\}$.}
In this case, by \eqref{10.15.5} and \eqref{10.15.6}, we find that,
for any $f\in (\icgg)'$ with
$\beta$ and $\gamma$ as in this lemma, and $z\in\cx$,
\begin{align*}
|Q_k(f)(z)|&\leq \sum_{\alpha' \in \ca_0}\sum_{m'=1}^{N(0,\alpha')}m_{\qop}(|P_{0}(f)|)
\int_{\qop}\left|Q_k\widetilde{P}_0(z,y)\right|\,d\mu(y)\\
&\qquad+\sum_{k'=1}^{N'}\sum_{\alpha' \in \ca_{k'}}\sum_{m'=1}^{N(k',\alpha')}
\mu\left(\qap\right)m_{\qap}(|P_{k'}(f)|)\left|Q_k\widetilde{P}_{k'}\left(z,\yap\right)\right|\\
&\qquad+\sum_{k'=N'+1}^\infty\sum_{\alpha' \in \ca_{k'}}
\sum_{m'=1}^{N(k',\alpha')}\mu\left(\qap\right)
\left|Q_k\widetilde{P}_{k'}\left(z,\yap\right)\right|\left|P_{k'}f\left(\yap\right)\right|\\
&=:\rm{{\rm I}_1+{\rm I}_2+{\rm I}_3}.
\end{align*}
Using an argument similar to that used in the proof of Lemma \ref{11.14.1},
we conclude that, for any $f\in (\icgg)'$ with
$\beta$ and $\gamma$ as in this lemma,
\begin{align}\label{12.5.1}
&\sup_{l \in \nn,\ l>N} \sup_{\alpha\in\ca_l}\left\{\frac{1}{\mu(Q_\alpha^l)}\sum_{k=l}^\infty
\sum_{\tau\in\ca_k}\sum_{m=1}^{N(k,\tau)}\delta^{-ksq}\mu\left(Q_\tau^{k,m}\right)
\mathbf{1}_{\{(\tau,m):\ Q_\tau^{k,m}\subset Q_\alpha^l\}}(\tau,m)
\left[\sup_{x\in Q_\tau^{k,m}}|{\rm I}_3|\right]^q\right\}^{1/q}\\
&\quad\lesssim \sup_{l' \in \nn,\ l'>N'} \sup_{\alpha'\in\ca_{l'}}
\left\{\frac{1}{\mu(Q_{\alpha'}^{l'})}\sum_{k'=l'}^\infty\sum_{\tau'\in\ca_k'}
\sum_{m'=1}^{N(k',\tau')}\delta^{-k'sq}\mu\left(Q_{\tau'}^{k',m'}\right)
\mathbf{1}_{\{(\tau',m'):\ Q_{\tau'}^{k',m'}\subset Q_{\alpha'}^{l'}\}}(\tau',m')\r.\notag\\
&\quad\qquad\times\left.\left.\left[\inf_{x\in Q_{\tau'}^{k',m'}}|P_{k'}(f)(x)|\right]^q
\right\}^{1/q}\right\}\notag.
\end{align}

Now we estimate $\rm{{\rm I}_1}$ and $\rm{{\rm I}_2}$, respectively. By \eqref{10.18.8} and Lemma \ref{6.15.1}(ii),
for any fixed $\Gamma\in(0,\bz)$, we obtain
\begin{align*}
\sup_{z\in\qa}(\rm{{\rm I}_1+{\rm I}_2})&\lesssim\delta^{k\Gamma}
\sum_{k'=0}^{N'}\sum_{\alpha' \in \ca_{k'}}\sum_{m'=1}^{N(k',\alpha')}m_{\qap}(|P_{k'}(f)|)\mu\left(\qap\right)\\
&\qquad\times\inf_{z\in\qa}\inf_{y\in\qap}\frac{1}{V_1(z)+V(z,y)}
\left[\frac{1}{1+d(z,y)}\right]^{\gz'}\notag\\
&\lesssim \delta^{k\Gamma}\sup_{k'\in\{0,\dots,N'\}}
\sup_{\alpha' \in \ca_{k'}}\sup_{m'\in\{1,\dots,N(k',\alpha')\}}m_{\qap}(|P_{k'}(f)|)\notag\\
&\qquad\times\sum_{k'=0}^{N'}\sum_{\alpha' \in \ca_{k'}}
\sum_{m'=1}^{N(k',\alpha')}\mu\left(\qap\right)\inf_{y\in\qap}
\frac{1}{V_1(z_\alpha^{k,m})+V(z_\alpha^{k,m},y)}
\left[\frac{1}{1+d(z_\alpha^{k,m},y)}\right]^{\gz'}\notag\\
&\lesssim\delta^{k\Gamma}\sup_{k'\in\{0,\dots,N'\}}
\sup_{\alpha' \in \ca_{k'}}\sup_{m'\in\{1,\dots,N(k',\alpha')\}}m_{\qap}(|P_{k'}(f)|)\notag\\
&\qquad\times\int_{\cx}\frac{1}{V_1(z_\alpha^{k,m})+V(z_\alpha^{k,m},y)}
\left[\frac{1}{1+d(z_\alpha^{k,m},y)}\right]^{\gz'}\,d\mu(y)\notag\\
&\lesssim \delta^{k\Gamma}\sup_{k'\in\{0,\dots,N'\}}
\sup_{\alpha' \in \ca_{k'}}\sup_{m'\in\{1,\dots,N(k',\alpha')\}}m_{\qap}(|P_{k'}(f)|).
\end{align*}
Using this and choosing $\Gamma\in(|s|,\bz)$, we have, for any $f\in (\icgg)'$ with
$\beta$ and $\gamma$ as in this lemma,
\begin{align}\label{12.5.2}
&\sup_{l \in \nn, l>N} \sup_{\alpha\in\ca_l}\left\{\frac{1}{\mu(Q_\alpha^l)}
\sum_{k=l}^\infty\sum_{\tau\in\ca_k}\sum_{m=1}^{N(k,\tau)}\delta^{-ksq}
\mu(Q_\tau^{k,m})\mathbf{1}_{\{(\tau,m):Q_\tau^{k,m}\subset Q_\alpha^l\}}(\tau,m)
\left[\sup_{x\in Q_\tau^{k,m}}|\rm{{\rm I}_1+{\rm I}_2}|\right]^q\right\}^{1/q}\\
&\quad\lesssim \sup_{k'\in\{0,\dots,N'\}}\sup_{\alpha' \in \ca_{k'}}
\sup_{m'\in\{1,\dots,N(k',\alpha')\}}m_{\qap}(|P_{k'}(f)|)\notag\\
&\quad\qquad \times \sup_{l \in \nn, l>N} \sup_{\alpha\in\ca_l}
\left[\sum_{k=l}^\infty\delta^{-ksq+k\Gamma q}
\frac{1}{\mu(Q_\alpha^l)}\sum_{\tau\in\ca_k}\sum_{m=1}^{N(k,\tau)}
\mu(Q_\tau^{k,m})\mathbf{1}_{\{(\tau,m):Q_\tau^{k,m}\subset Q_\alpha^l\}}(\tau,m)\right]^{1/q}\notag\\
&\quad\lesssim \sup_{k'\in\{0,\dots,N'\}}\sup_{\alpha' \in \ca_{k'}}
\sup_{m'\in\{1,\dots,N(k',\alpha')\}}m_{\qap}(|P_{k'}(f)|)\notag.
\end{align}
Combining \eqref{12.3.5} through
\eqref{12.5.2}, we obtain the desired estimate and then complete the proof of Lemma \ref{12.3.2}.
\end{proof}

Using an argument similar to that used in the proof of Proposition
\ref{6.4.1}, we have the following proposition and  we omit the details here.

\begin{proposition}
Let $\{Q_k\}_{k=0}^\infty$ and $\{P_k\}_{k=0}^\infty$ be two {\rm $\exp$-IATIs}.
Let $\beta,\ \gamma\in(0,\eta)$, $s\in(-(\bz\wedge\gz),\bz\wedge\gz)$ and $q \in (p(s,\beta\wedge\gamma), \infty]$
with $\eta$ and $p(s,\beta\wedge\gamma)$, respectively, as
in Definition \ref{10.23.2} and \eqref{pseta}. Let $N\in\nn$ and $N'\in\nn$ be as in Lemma
\ref{icrf} associated, respectively, with $\{Q_k\}_{k=0}^\infty$ and $\{P_k\}_{k=0}^\infty$.
Then, for any $f \in (\icgg)'$,
\begin{align*}
&\max\left\{\sup_{k\in\{0,\dots,N\}}\sup_{\alpha \in \ca_k}
\sup_{m\in\{1,\dots,N(k,\alpha)\}}m_{\qa}(|Q_k(f)|),\r.\\
&\qquad\quad\left.\sup_{l \in \nn,\ l>N} \sup_{\alpha\in\ca_l}\left[\frac{1}
{\mu(Q_\alpha^l)}\int_{Q_\alpha^l}\sum_{k=l}^\infty\delta^{-ksq}|Q_k(f)(x)|^q\,d\mu(x)\right]^{1/q}\r\}\\
&\quad\sim \max\left\{\sup_{k'\in\{0,\dots,N'\}}\sup_{\alpha' \in \ca_k'}
\sup_{m'\in\{1,\dots,N(k',\alpha')\}}m_{\qap}(|P_{k'}(f)|),\r.\\
&\qquad\quad\left.\sup_{l' \in \nn,\ l'>N'} \sup_{\alpha'\in\ca_{l'}}
\left[\frac{1}{\mu(Q_{\alpha'}^{l'})}\int_{Q_{\alpha'}^{l'}}
\sum_{k'=l'}^\infty\delta^{-k'sq}|P_{k'}(f)(x)|^q\,d\mu(x)\right]^{1/q}\r\},
\end{align*}
where the positive equivalence constants are  independent of $f$.
\end{proposition}

The following proposition shows that the space $\ihfi$ is
independent of the choice of the space of distribution $(\icgg)'$
if $\beta$ and $\gamma$ satisfy some certain conditions.

\begin{proposition}\label{12.5.3}
Let $\beta,\ \gamma\in(0,\eta)$, with $\eta$ be as in Definition \ref{10.23.2}, and $s\in(-\eta,\eta)$ satisfy
\begin{equation}\label{12.5.4}
\beta\in((-s)_+,\eta)\qquad\text{and}\qquad\ \gamma\in(0,\eta).
\end{equation}
If $q\in(p(s,\beta\wedge\gamma),\infty]$,
then there exists a positive constant $C$ such that, for any $f \in \ihfi\subset(\icgg)'$,
$f \in (\icggt)'$ with $\widetilde{\beta}$ and  $\widetilde{\gamma}$ satisfying $q\in(p(s,\widetilde{\beta}\wedge\widetilde{\gamma}),\infty]$
with $p(s,\widetilde{\beta}\wedge\widetilde{\gamma})$ as in \eqref{pseta} via replacing $\beta$ and $\gamma$ respectively by $\widetilde{\beta}$ and
$\widetilde{\gamma}$, and \eqref{12.5.4}
via replacing $\beta$ and $\gamma$ respectively by $\widetilde{\beta}$ and
$\widetilde{\gamma}$, and there exists a positive constant $C$,
independent of $f$, such that
$$\|f\|_{(\icggt)'}\leq C\|f\|_{\ihfi}.$$
\end{proposition}

\begin{proof}
Let $\eta$ be as in Definition \ref{10.23.2}, $f \in (\icgg)'$ with
$\beta$ and $\gamma$ as in this proposition
belong to $\ihfi$ and $\psi\in\cg(\eta,\eta)$.
Let $\widetilde{\beta}$ and $\widetilde{\gamma}$ as in this proposition and  all the other notation be the
same as in Lemma \ref{icrf}. Note that, by the definition
of $\|\cdot\|_{\ihfi}$, we know that, for any $k\in \{0,\dots,N\}$,
$\alpha\in\ca_k$ and $m\in\{1,\dots,N(k,\alpha)\}$,
\begin{equation}\label{12.5.5}
m_{\qa}(|Q_{k}(f)|)\leq \|f\|_{\ihfi}.
\end{equation}
From \eqref{11.19.1}, we
deduce that, for any given $q\in(p(s,\widetilde{\beta}\wedge\widetilde{\gamma}),\infty]$
and any $k\in\{N+1,N+2,\dots\}$,
\begin{equation}\label{12.5.6}
\left[\sum_{\alpha\in\ca_k}\sum_{m=1}^{N(k,\alpha)}\delta^{-ksq}
\left|Q_k(f)\left(\ya\right)\right|^q\right]^{1/q}\lesssim \|f\|_{\ihfi}.
\end{equation}
If $q\in(p(s,\widetilde{\beta}\wedge\widetilde{\gamma}),1]$, then, by \eqref{r}, \eqref{10.15.6},
\eqref{10.19.4}, \eqref{12.2.2}, \eqref{12.5.4}, \eqref{12.5.5},
\eqref{12.5.6}, Lemma \ref{icrf} and Lemma \ref{6.15.1}(ii),
we conclude  that
\begin{align}\label{12.5.7}
|\langle f,\psi\rangle|&=\left|\sum_{\alpha \in \ca_0}\sum_{m=1}^{N(0,\alpha)}
\int_{\qo}\left\langle\widetilde{Q}_0(\cdot,y),\psi\right\rangle\,d\mu(y)Q^{0,m}_{\alpha,1}
(f)\right.\\
&\qquad+\sum_{k=1}^N\sum_{\alpha \in \ca_k}\sum_{m=1}^{N(k,\alpha)}
\mu\left(\qa\right)\left\langle\widetilde{Q}_k(\cdot,\ya),\psi\right\rangle Q^{k,m}_{\alpha,1}(f)\notag\\
&\qquad\left.+\sum_{k=N+1}^\infty\sum_{\alpha \in \ca_k}
\sum_{m=1}^{N(k,\alpha)}\mu\left(\qa\right)\left\langle\widetilde{Q}_k(\cdot,\ya),\psi\right\rangle Q_kf\left(\ya\right)\right|\notag\\
&\lesssim\|\psi\|_{\cg(\widetilde{\beta},\widetilde{\gamma})}
\left\{\sum_{k=0}^N\sum_{\alpha \in \ca_k}\sum_{m=1}^{N(k,\alpha)}
\mu\left(\qa\right)m_{\qa}(|Q_k(f)|)\frac{1}{V_1(x_1)+V(x_1,\ya)}\r.\noz\\
&\qquad\times\left[\frac{1}{1+d(x_1,\ya)}\right]^{\widetilde{\gamma}}+\sum_{k=N+1}
^\infty\delta^{k\widetilde{\beta}}\sum_{\alpha \in \ca_k}\sum_{m=1}^{N(k,\alpha)}
\mu\left(\qa\right)\left|Q_kf\left(\ya\right)\right|\notag\\
&\qquad\times\left.\frac{1}{V_1(x_1)+V(x_1,\ya)}
\left[\frac{1}{1+d(x_1,\ya)}\right]^{\widetilde{\gamma}}\right\}\notag\\
&\lesssim \|\psi\|_{\cg(\widetilde{\beta},\widetilde{\gamma})}\|f\|_{\ihfi}\notag\\
&\qquad\times\left\{\int_{\cx}\frac{1}{V_1(x_1)+V(x_1,y)}\left[\frac{1}{1+d(x_1,y)}\right]^{\widetilde{\gamma}}\,d\mu(y)
+\sum_{k=N+1}^\infty\delta^{k\widetilde{\beta}+ks}\right\}\notag\\
&\lesssim \|\psi\|_{\cg(\widetilde{\beta},\widetilde{\gamma})}\|f\|_{\ihfi}.\notag
\end{align}
If $q\in(1,\infty]$, then, by \eqref{10.15.6}, \eqref{10.19.4}, \eqref{12.2.2},
\eqref{12.5.4}, \eqref{12.5.5}, \eqref{12.5.6}, the H\"older inequality,
Lemmas \ref{icrf} and \ref{6.15.1}(ii),
and  an argument similar to that used in the estimation of \eqref{12.5.7}, we find that
\begin{align}\label{12.5.8}
|\langle f,\psi\rangle|&\lesssim\|\psi\|_{\cg(\widetilde{\beta},\widetilde{\gamma})}\left\{\sum_{k=0}^N
\sum_{\alpha \in \ca_k}\sum_{m=1}^{N(k,\alpha)}\mu\left(\qa\right)m_{\qa}(|Q_k(f)|)
\frac{1}{V_1(x_1)+V(x_1,\ya)}\right.\\
&\qquad\times\left[\frac{1}{1+d(x_1,\ya)}\right]^{\widetilde{\gamma}}
\sum_{k=N+1}^\infty\delta^{k\widetilde{\beta}}\sum_{\alpha \in \ca_k}
\sum_{m=1}^{N(k,\alpha)}\mu\left(\qa\right)
\left|Q_kf\left(\ya\right)\right|\notag\\
&\qquad\times\lf.\frac{1}{V_1(x_1)+V(x_1,\ya)}\left[\frac{1}{1+d(x_1,\ya)}\right]^{\widetilde{\gamma}}\r\}\noz\\
&\lesssim \|\psi\|_{\cg(\widetilde{\beta},\widetilde{\gamma})}\|f\|_{\ihfi}\left[\int_{\cx}
\frac{1}{V_1(x_1)+V(x_1,y)}\left\{\frac{1}{1+d(x_1,y)}\right\}^{\widetilde{\gamma}}\,d\mu(y)\right.\notag\\
&\qquad\left.+\sum_{k=N+1}^\infty\delta^{k\widetilde{\beta}+ks}\left\{\int_{\cx}\frac{1}{V_1(x_1)+V(x_1,y)}
\left[\frac{1}{1+d(x_1,y)}\right]^{\widetilde{\gamma}}\,d\mu(y)\right\}^{1/q'}\right]\notag\\
&\lesssim \|\psi\|_{\cg(\widetilde{\beta},\widetilde{\gamma})}\|f\|_{\ihfi}.\notag
\end{align}

Combining \eqref{12.5.7}, \eqref{12.5.8} with an  argument similar to that used in the proof of Proposition
\ref{6.5.1}, we complete the proof of Proposition \ref{12.5.3}.
\end{proof}

Now we state some basic properties of $\ihfi$ and we omit the details here.

\begin{proposition}\label{proihfi}
Let $\beta,\ \gamma \in (0, \eta)$,
$s\in(-(\beta\wedge\gamma),\beta\wedge\gamma)$ and
$q\in (p(s,\beta\wedge\gamma),\infty]$ with $\eta$ and $p(s,\beta\wedge\gamma)$, respectively, as in
Definition \ref{10.23.2} and \eqref{pseta}.
\begin{enumerate}
\item[{\rm(i)}] If $p(s,\beta\wedge\gamma)<q_1\leq q_2\leq\infty$, then $F_{\infty,q_1}^s(\cx)\subset F_{\infty,q_2}^s(\cx)$;
\item[{\rm(ii)}] If $\theta \in ((-\eta-s)_+,\eta-s)$ and
$q_1,\ q_2\in p(s,\beta\wedge\gamma),\infty]$, then $F_{\infty,q_1}^{s+\theta}(\cx)\subset F_{\infty,q_2}^s(\cx)$;
\item[{\rm(iii)}] $B_{\infty,q}^s(\cx)\subset F_{\infty,q}^s(\cx)\subset B_{\infty,\infty}^s(\cx)$;
\item[{\rm(iv)}] If $\widetilde{\beta}\in(s_+,\eta)$ and $\widetilde{\gamma}\in(0,\eta)$, then
$\cg(\widetilde{\beta},\widetilde{\gamma})\subset \ihfi$.
\end{enumerate}
\end{proposition}

\section[Relationships among Besov spaces,  Triebel--Lizorkin spaces \\ and other function spaces]
{Relationships among Besov spaces,  Triebel--Lizorkin spaces and other function spaces}\label{s6}

In this section, we establish the relations among Besov spaces, Triebel--Lizorkin spaces
and other function spaces.

\subsection[Relationships between homogeneous Besov and Triebel--Lizorkin spaces \\
and other function spaces]
{Relationships between homogeneous Besov and Triebel--Lizorkin spaces and other function spaces}

In this subsection, we establish the connections
between homogeneous Besov and
Triebel--Lizorkin spaces and other function spaces,
and hence we always assume that $\mu(\cx)=\infty$.
Let us begin with their relations with Lebesgue spaces.

\begin{theorem}\label{l_btl}
Let $p\in(1,\infty)$  and $\beta,\ \gamma\in(0,\eta)$ with $\eta$ as in
Definition \ref{10.23.2}. Then, as  subspaces of $(\cggi)'$, $\dot{F}_{p,2}^0(\cx)=L^p(\cx)$
in the following sense:
\begin{enumerate}
\item[{\rm(i)}] there exists a positive constant $C$ such that, for any $f\in L^p(\cx)$,
$$\|f\|_{\dot{F}_{p,2}^0(\cx)}\leq C\|f\|_{L^p(\cx)}.$$
\item[{\rm(ii)}] for any $f\in \dot{F}_{p,2}^0(\cx)$, there exists a $g\in L^p(\cx)$ such that
$$f=g\quad\text{in}\quad (\cggi)' \quad\text{and}\quad \|g\|_{L^p(\cx)}
\leq C\|f\|_{\dot{F}_{p,2}^0(\cx)},$$
where $C$ is a positive constant independent of $f$.
\end{enumerate}
\end{theorem}

To show Theorem \ref{l_btl}, we need the following several lemmas.
The next lemma is the well-known \emph{Khinchin inequality}
(see, for instance, \cite[(2.2)]{m92}).

\begin{lemma}\label{khin}
Let $E:=\{1,-1\}$. Then there exists a positive constant $C$ such that,
for any $n\in\nn$ and $\{a_j\}_{j=1}^n\subset\cc$,
$$\left(\sum_{j=1}^n|a_j|^2\right)^{1/2}\leq C
2^{-n}\sum_{\varepsilon_1\in E}\cdots\sum_{\varepsilon_n\in E}|
\varepsilon_1a_1+\cdots+\varepsilon_na_n|.$$
\end{lemma}

The following lemma was obtained in \cite[Lemma 4.9]{hlyy}.

\begin{lemma}\label{con_soe}
If $a,\ c\in(0, \infty)$, then there exists a positive constant $C$ such that,
for any $x,\ y\in\cx$,
$$\sum_{k=-\infty}^\infty\frac{1}{V_{\delta^k}(x)}
\exp\left\{-c\left[\frac{d(x,y)}{\delta^k}\right]^a\right\}
\exp\left\{-c\left[\frac{d(x, \cy^k)}{\delta^k}\right]^a\right\}
\leq \frac{C}{V(x,y)}.$$
\end{lemma}

Now we prove Theorem \ref{l_btl}.

\begin{proof}[Proof of Theorem \ref{l_btl}]
We first show (i). Let $E:=\{1,-1\}$ and $\{Q_k\}_{k\in\zz}$ be an exp-ATI.
Assume $f\in L^p(\cx)$ with $p\in(1,\infty)$.
By Lemma \ref{khin} and the Minkowski inequality, we know that, for any $N\in\nn$,
\begin{align}\label{t_en}
\left\|\left\{\sum_{k=-N}^{N}|Q_k(f)|^2\right\}^{1/2}\right\|_{L^p(\cx)}
&\lesssim \left\|\frac{1}{2^{2N}}\sum_{\varepsilon_{-N}\in E}\cdots
\sum_{\varepsilon_N\in E}\left|\sum_{k=-N}^{N}\varepsilon_{k}Q_k(f)\right|\right\|_{L^p(\cx)}\\
&\lesssim\frac{1}{2^{2N}}\sum_{\varepsilon_{-N}\in E}\cdots
\sum_{\varepsilon_N\in E}\left\|\sum_{k=-N}^{N}\varepsilon_{k}Q_k(f)\right\|_{L^p(\cx)}.\notag
\end{align}
For any fixed $\varepsilon :=\{\varepsilon_k\}_{k=-N}^N\subset E$,
let $T_N^\varepsilon:=\sum_{k=-N}^N\varepsilon_kQ_k$ and
denote its kernel by $K_N^\varepsilon$.

We claim that $K_N^\varepsilon$ and $(K_N^\varepsilon)^\ast$
are standard Calder\'on--Zygmund kernels, that is, they satisfy Definition \ref{4.11.2}(iii) below,
with the implicit positive constant $C$ therein independent
of $\varepsilon$ and $N$. Indeed,
by symmetry, it suffices to show the case of $K_N^\varepsilon$. Notice that, for any $x,\ y\in \cx$,
$$K_N^\varepsilon(x,y)=\sum_{k=-N}^N\varepsilon_kQ_k(x,y).$$
From  Definition \ref{10.23.2}(ii) and Lemma \ref{con_soe},  we deduce that, for any $x,\ y\in\cx$ with $x\neq y$,
\begin{align*}
|K_N^\varepsilon(x,y)|
&\leq\sum_{k=-N}^N|Q_k(x,y)|
\lesssim \sum_{k=-\infty}^\infty\frac{1}{V_{\delta^k}(x)}\exp
\left\{-\nu\left[\frac{d(x,y)}{\delta^k}\right]^a\right\}
\exp\left\{-\nu\left[\frac{d(x, \cy^k)}{\delta^k}\right]^a\right\}\\
&\lesssim \frac{1}{V(x,y)}.
\end{align*}
Moreover, by \cite[Remark 2.9(iii$'$)]{hlyy} and Lemma \ref{con_soe}, we find that,
for any $x,\ y,\ y'\in\cx$ with $x\neq y$ and
$d(x,y')\leq (2A_0)^{-1}d(x,y)$,
\begin{align*}
|K_N^\varepsilon(x,y)-K_N^\varepsilon(x,y')|
&\leq \sum_{k=-N}^{N}|Q_k(x,y)-Q_k(x,y')|\\
&\lesssim\sum _{k=-N}^{N} \left[\frac{d(y,y')}{d(x,y)+\delta^k}\right]^\eta
\frac{1}{V_{\delta^k}(x)}
\exp\left\{-\nu\left[\frac{d(x,y)}{\delta^k}\right]^a\right\}
\exp\left\{-\nu\left[\frac{d(x, \cy^k)}{\delta^k}\right]^a\right\}\\
&\lesssim \left[\frac{d(y,y')}{d(x,y)}\right]^\eta
\sum_{k=-\infty}^\infty\frac{1}{V_{\delta^k}(x)}\exp\left\{-\nu
\left[\frac{d(x,y)}{\delta^k}\right]^a\right\}
\exp\left\{-\nu\left[\frac{d(x, \cy^k)}{\delta^k}\right]^a\right\}\\
&\lesssim \left[\frac{d(y,y')}{d(x,y)}\right]^\eta\frac{1}{V(x,y)},
\end{align*}
which completes the proof of the above claim.

Now we prove that, for any given $p\in(1,\infty)$ and any $f\in L^p(\cx)$,
\begin{equation}\label{t_en_p}
\|T_N^\varepsilon f\|_{L^p(\cx)}+\|(T_N^\varepsilon)^\ast f\|_{L^p(\cx)}
\lesssim\|f\|_{L^p(\cx)},
\end{equation}
where the implicit positive constant is independent of $\varepsilon$, $N$ and $f$.
We first show \eqref{t_en_p} when $p=2$.
Indeed, by \cite[Lemma 4.1]{hlyy},  we have, for any $k_1,\ k_2\in\zz$,
\begin{equation*}
\|\varepsilon_{k_1}Q_{k_1}(\varepsilon_{k_2}Q_{k_2})^\ast\|_{L^2(\cx)\to L^2(\cx)}
\lesssim \delta^{|k_1-k_2|\eta}
\end{equation*}
and
\begin{equation*}
\|(\varepsilon_{k_1}Q_{k_1})^\ast\varepsilon_{k_2}Q_{k_2}\|_{L^2(\cx)\to L^2(\cx)}
\lesssim \delta^{|k_1-k_2|\eta}.
\end{equation*}
From these and the Cotlar--Stein lemma (see, for instance,
\cite[p.\,60, Lemma 6]{my97}), we deduce that, for any $f\in L^2(\cx)$,
\begin{equation}\label{t_en_2}\|T_N^\varepsilon f\|_{L^2(\cx)}+\|(T_N^\varepsilon)^\ast f\|_{L^2(\cx)}
\lesssim\|f\|_{L^2(\cx)},
\end{equation}
which shows that \eqref{t_en_p} holds true when $p=2$.
By \eqref{t_en_2} and \cite[p.\,22, Corollary]{s93},
we know that \eqref{t_en_p} holds true for any $f\in L^p(\cx)$ with $p\in(1,2]$,
which, together with a duality argument, further implies that
\eqref{t_en_p} holds true for any $p\in (1,\infty)$.
Combining the argument as above, we conclude  that, for any given $p\in(1,\infty)$ and any $f\in L^p(\cx)$,
$$
\left\|\left\{\sum_{k=-N}^{N}|Q_k(f)|^2\right\}^{1/2}\right\|_{L^p(\cx)}
\lesssim \|f\|_{L^p(\cx)}.
$$
This, together with the Fatou lemma, further implies that, for any given $p\in(1,\infty)$ and any $f\in L^p(\cx)$,
$$
\|f\|_{\dot{F}_{p,2}^0(\cx)}\ls\|f\|_{L^p(\cx)}.
$$
Thus, we finish the proof of (i).

Now we prove (ii). Assume that $p\in(1,\infty)$ and
$f\in(\cggi)'$, for some $\beta,\ \gamma\in(0,\eta)$ with $\eta$ as in Definition \ref{10.23.2},
belongs to $\dot{F}_{p,2}^0(\cx)$.
By Lemma \ref{h_c_crf}, we know that there exists a sequence $\{\widetilde{Q}_k\}_{k=-\infty}^\infty$
of bounded linear integral operators on $L^2(\mathcal{X})$ such that
\begin{equation}\label{expan_f}
f= \sum_{k=-\infty}^\infty\widetilde{Q}_kQ_kf \qquad \text{in}\quad (\cggi)'.
\end{equation}
For any $n\in\nn$, let $f_n:=\sum_{k=-n}^n\widetilde{Q}_kQ_kf$.
It is easy to see that, for any $n\in\nn$, $f_n$ is a  measurable function on $\cx$.
Using an argument similar to that used in the estimation of \cite[(3.108)]{hmy08},
we find that, for any $h\in L^{p'}(\cx)$,
$$\left\|\left\{\sum_{k=-\infty}^\infty\left|\widetilde{Q}_k^\ast(h)\right|^2\right\}^{1/2}
\right\|_{L^{p'}(\cx)}\lesssim\|h\|_{L^{p'}(\cx)}.$$
By this and the H\"older inequality, we conclude that,
for any $l,\ m\in\nn$ with $l>m$,
\begin{align}\label{f_n_lp}
\|f_l-f_m\|_{L^p(\cx)}&=\sup_{h\in L^{p'}(\cx),\  \|h\|_{L^{p'}(\cx)}\leq 1}|\langle f_l-f_m, h\rangle|\\
&=\sup_{h\in L^{p'}(\cx),\ \|h\|_{L^{p'}(\cx)}\leq 1} \lf|\left\langle
\sum_{m<|k|\leq l}\widetilde{Q}_kQ_kf, h\right\rangle\r|\notag\\
&= \sup_{h\in L^{p'}(\cx),\ \|h\|_{L^{p'}(\cx)}\leq 1} \lf|
\sum_{m<|k|\leq l}\lf\langle Q_k(f), \widetilde{Q}_k^\ast(h)\r\rangle\r|\notag\\
&\leq \sup_{h\in L^{p'}(\cx),\ \|h\|_{L^{p'}(\cx)}\leq 1} \int_{\cx}
\lf|\sum_{m<|k|\leq l}Q_k(f)(x)\widetilde{Q}_k^\ast(h)(x)\r|\,d\mu(x)\notag\\
&\leq \sup_{h\in L^{p'}(\cx),\ \|h\|_{L^{p'}(\cx)}\leq 1}
\left\|\left\{\sum_{m<|k|\leq l}|Q_k(f)|^2\right\}^{1/2}\right\|_{L^p(\cx)}
\left\|\left\{\sum_{m<|k|\leq l}\left|\widetilde{Q}_k^\ast(h)\right|^2\right\}^{1/2}\right\|_{L^{p'}(\cx)}\notag\\
&\lesssim \left\|\left\{\sum_{|k|>m}|Q_k(f)|^2\right\}^{1/2}\right\|_{L^p(\cx)}.\notag
\end{align}
Since $f\in\dot{F}_{p,2}^0(\cx)$, it follows that
$$\left\|\left\{\sum_{k=-\infty}^\infty|Q_k(f)|^2\right\}^{1/2}\right\|_{L^p(\cx)}<\infty,$$
which, combined with the dominated convergence theorem,
further implies that, for any $\epsilon\in(0,\infty)$,
there exists an $m_0\in\nn$ such that, for any $l,\ m>m_0$,
$$\|f_l-f_m\|_{L^p(\cx)}\lesssim\left\|\left\{\sum_{|k|>m}
|Q_k(f)|^2\right\}^{1/2}\right\|_{L^p(\cx)}\lesssim\epsilon.$$
Moreover, using an argument similar to that used in the estimation of \eqref{f_n_lp},
we find that, for any $n\in\nn$,
$$\|f_n\|_{L^p(\cx)}\lesssim\left\|\left\{\sum_{|k|\leq n}
|Q_k(f)|^2\right\}^{1/2}\right\|_{L^p(\cx)}\lesssim\|f\|_{\dot{F}_{p,2}^0(\cx)}.$$
From these, we deduce that $\{f_n\}_{n\in\nn}$ is a Cauchy sequence in $L^p(\cx)$.
By the completeness of $L^p(\cx)$ with any given $p\in(1,\infty)$, we can let
$$g:=\lim_{n\to\infty} f_n\quad\text{in}\quad L^p(\cx).$$
Using \eqref{expan_f},  we have, for any $\varphi\in \cggi\subset L^{p'}(\cx)$,
$$\langle f, \varphi\rangle
=\lim_{n\to\infty}\left\langle \sum_{k=-n}^n\widetilde{Q}_kQ_kf, \varphi\right\rangle
=\lim_{n\to\infty}\langle f_n, \varphi\rangle
=\langle g, \varphi\rangle$$
and hence
$$f=g\qquad \text{in}\quad (\cggi)'.$$
Moreover, by an argument similar to that used in the estimation of \eqref{f_n_lp},
we conclude that
$$\|g\|_{L^p(\cx)}=\lim_{n\to\infty}\|f_n\|_{L^p(\cx)}\lesssim
\lim_{n\to\infty}\left\|\left\{\sum_{|k|\leq n}|Q_k(f)|^2\right\}^{1/2}\right\|_{L^p(\cx)}
\lesssim \|f\|_{\dot{F}_{p,2}^0(\cx)},$$
which completes the proof of (ii) and hence of Theorem \ref{l_btl}.
\end{proof}

\begin{remark}
By the Littlewood--Paley $g$-function characterization of $H^p(\cx)$ in \cite[Theorem 5.10]{hhllyy},
we easily know that, for any given $p\in (\omega/(\omega+\eta),1]$,
$\dot{F}_{p,2}^0(\cx)=H^p(\cx)$ with equivalent quasi-norms. Therefore,
Theorem \ref{l_btl} further complements this fact via showing that it is still true when $p\in(1,\infty)$.
\end{remark}

Now we establish the relationships among H\"older spaces,
Besov spaces and Triebel--Lizorkin spaces.
Let us begin with the notion of H\"older spaces.
In what follows, the \emph{symbol $C(\cx)$} denotes the set of
all continuous functions on $\cx$.

\begin{definition}\label{c_s}
Let $s\in (0,1]$. The  \emph{homogeneous H\"older space} $\dot{C}^s(\cx)$ and
the \emph{inhomogeneous H\"older space} $C^s(\cx)$ are defined, respectively,  by setting
$$\dot{C}^s(\cx):=\left\{f\in C(\cx):\ \|f\|_{\dot{C}^s(\cx)}
:=\sup_{x,\ y\in\cx,\ x\neq y}\frac{|f(x)-f(y)|}{[d(x,y)]^s}<\infty\right\}$$
and
$$C^s(\cx):=\left\{f\in C(\cx):\ \|f\|_{C^s(\cx)}:=\|f\|_{\dot{C}^s(\cx)}+\|f\|_{L^\infty(\cx)}<\infty\right\}.$$
\end{definition}

By Proposition \ref{12.3.1}(ii),
we know that $\dot{B}^s_{\infty,\infty}(\cx)
=\dot{F}^s_{\infty,\infty}(\cx)$ with equivalent norms.
The following theorem states the relationship between $\dot{C}^s(\cx)$ and
$\dot{B}^s_{\infty,\infty}(\cx)$.

\begin{theorem}\label{c_b}
Let $\beta,\ \gamma\in(0,\eta)$ with $\eta$
as in Definition \ref{10.23.2}, and $s\in(0,\beta\wedge\gamma)$.
Then $\dot{C}^s(\cx)=\dot{B}^s_{\infty,\infty}(\cx)$ in the following sense:
\begin{enumerate}
\item[{\rm(i)}] there exists a positive constant $C$ such that, for any $f\in \dot{C}^s(\cx)$,
$$\|f\|_{\dot{B}^s_{\infty,\infty}(\cx)}\leq C\|f\|_{\dot{C}^s(\cx)};$$
\item[{\rm(ii)}] for any $f\in \dot{B}^s_{\infty,\infty}(\cx)$,
there exists a $g\in \dot{C}^s(\cx)$ such that
$$f=g\quad\text{in}\quad (\cggi)' \quad\text{and}
\quad \|g\|_{\dot{C}^s(\cx)}\leq C\|f\|_{\dot{B}^s_{\infty,\infty}(\cx)},$$
where $C$ is a positive constant independent of $f$.
\end{enumerate}
\end{theorem}

\begin{proof}
We first show (i). Assume $f \in \dot{C}^s(\cx)$ with $s\in(0,\beta\wedge\gamma)$,
where $\beta,\ \gamma\in(0,\eta)$ with $\eta$ as in Definition \ref{10.23.2}.
By Definition \ref{c_s}, we know that, for any $x,\ y\in\cx$,
$$|f(x)-f(y)|\leq \|f\|_{\dot{C}^s(\cx)}[d(x,y)]^s.$$
From this, Definition \ref{10.23.2}(v), \eqref{10.23.3},
Lemma \ref{6.15.1}(ii) and $s\in(0,\bz\wedge\gz)$, we deduce that, for any $k\in\zz$ and $x\in\cx$,
\begin{align*}
|Q_k(f)(x)|&=\left|\int_{\cx}Q_k(x,y)f(y)\,d\mu(y)\right|
=\left|\int_{\cx}Q_k(x,y)[f(y)-f(x)]\,d\mu(y)\right|\\
&\lesssim \|f\|_{\dot{C}^s(\cx)}\int_{\cx}\frac{1}{V_{\delta^k}(x)+V(x,y)}
\left[\frac{\delta^k}{\delta^k+d(x,y)}\right]^\gamma[d(x,y)]^s\,d\mu(y)\\
&\lesssim \delta^{ks}\|f\|_{\dot{C}^s(\cx)},
\end{align*}
which implies that
$$\|f\|_{\dot{B}^s_{\infty,\infty}(\cx)}=
\sup_{k\in\zz}\delta^{-ks}\|Q_k(f)\|_{L^\infty(\cx)}
\lesssim\|f\|_{\dot{C}^s(\cx)}$$
and hence shows (i).

We next prove (ii). Assume $f\in\dot{B}^s_{\infty,\infty}(\cx)$ with $s$ as in this  theorem.
From Proposition \ref{12.2.7}, we deduce that $f\in(\cggi)'$ with $\beta,\ \gamma$
satisfying \eqref{12.1.2}. Notice that, for any $k\in\zz$,
\begin{equation}\label{iii}
\left\|\widetilde{Q}_kQ_kf\right\|_{L^\infty(\cx)}\lesssim \|Q_kf\|_{L^\infty(\cx)}
\lesssim\|f\|_{\dot{B}^s_{\infty,\infty}(\cx)}\delta^{ks}<\infty.
\end{equation}
Thus, we may choose $x_1\in\cx$ such that, for any $k\in\zz$,
$|\widetilde{Q}_kQ_kf(x_1)|\leq\|\widetilde{Q}_kQ_kf\|_{L^\infty(\cx)}<\infty$.
By this  and \eqref{expan_f}, we know that, for any $\varphi\in\cggi$,
$$\langle f,\varphi\rangle
=\sum_{k=-\infty}^\infty\left\langle\widetilde{Q}_kQ_kf,\varphi\right\rangle
=\sum_{k=-\infty}^\infty\left\langle\widetilde{Q}_kQ_kf-
\widetilde{Q}_kQ_kf(x_1),\varphi\right\rangle.$$
Thus,
$$f=\sum_{k=-\infty}^\infty\left[\widetilde{Q}_kQ_kf-
\widetilde{Q}_kQ_kf(x_1)\right]\qquad \text{in}\quad (\cggi)'.$$
For any $x\in\cx$, let
\begin{align*}
\widetilde{g}(x):=&\sum_{k=-\infty}^\infty\left|\widetilde{Q}_kQ_kf(x)-
\widetilde{Q}_kQ_kf(x_1)\right|\\
=&\sum_{k=-\infty}^\infty\left|\int_{\cx}\left[\widetilde{Q}_k(x,y)-
\widetilde{Q}_k(x_1,y)\right]Q_kf(y)\,d\mu(y)\right|.
\end{align*}
Notice that, for any fixed $x\in\cx$ and $x\neq x_1$,
there exists a unique $k_0\in\zz$ such that $(2A_0)^{-1}\delta^{k_0+1}< d(x,x_1)
\leq (2A_0)^{-1}\delta^{k_0}$.
Using this $k_0$, we write
\begin{align*}
\widetilde{g}(x)
&=\sum_{k=-\infty}^{k_0-1}\left|\int_{\cx}\left[\widetilde{Q}_k(x,y)-
\widetilde{Q}_k(x_1,y)\right]Q_kf(y)\,d\mu(y)\right|
+\sum_{k=k_0}^\infty\cdots\\
&=: \widetilde{g}_1(x)+\widetilde{g}_2(x).
\end{align*}

We first estimate $\widetilde{g}_1$.
Observe that, for any $k\in\zz$ with $k\leq k_0-1$ and $y\in\cx$,
$$d(x,x_1)\leq (2A_0)^{-1}\delta^{k_0}<(2A_0)^{-1}[\delta^k+d(x,y)],$$
which, together with \eqref{4.23y}, \eqref{iii} and Lemma
\ref{6.15.1}(ii), implies that,
for any $x\in\cx$,
\begin{align*}
\widetilde{g}_1(x)
&\lesssim \sum_{k=-\infty}^{k_0-1}\int_{\cx}
\left[\frac{d(x,x_1)}{\delta^k+d(x,y)}\right]^\beta
\frac{1}{V_{\delta^k}(x)+V(x,y)}
\left[\frac{\delta^k}{\delta^k+d(x,y)}\right]^\gamma
|Q_k(f)(y)|\,d\mu(y)\\
&\lesssim \|f\|_{\dot{B}^s_{\infty,\infty}(\cx)}[d(x,x_1)]^s
\sum_{k=-\infty}^{k_0-1}\int_{\cx}
\left[\frac{d(x,x_1)}{\delta^k+d(x,y)}\right]^{\beta-s}
\frac{1}{V_{\delta^k}(x)+V(x,y)}
\left[\frac{\delta^k}{\delta^k+d(x,y)}\right]^{\gamma+s}\,d\mu(y)\\
&\lesssim  \|f\|_{\dot{B}^s_{\infty,\infty}(\cx)}[d(x,x_1)]^s
\sum_{k=-\infty}^{k_0-1}\delta^{k_0(\beta-s)}\delta^{-k(\beta-s)}
\lesssim  \|f\|_{\dot{B}^s_{\infty,\infty}(\cx)}[d(x,x_1)]^s.
\end{align*}
Next we estimate $\widetilde{g}_2$. Notice that $1\lesssim [d(x,x_1)]^s\delta^{-(k_0+1)s}$.
Then, by \eqref{4.23x} and Lemma \ref{6.15.1}(ii), we have, for any $x\in\cx$,
\begin{align*}
\widetilde{g}_2(x)&\leq\sum_{k=k_0}^\infty\int_{\cx}
\left[\left|\widetilde{Q}_k(x,y)\right|+\left|
\widetilde{Q}_k(x_1,y)\right|\right]|Q_kf(y)|\,d\mu(y)\\
&\lesssim \|f\|_{\dot{B}^s_{\infty,\infty}(\cx)}[d(x,x_1)]^s\delta^{-(k_0+1)s}
\sum_{k=k_0}^\infty \delta^{ks}
\lesssim  \|f\|_{\dot{B}^s_{\infty,\infty}(\cx)}[d(x,x_1)]^s.
\end{align*}
Combining the estimates of $\widetilde{g}_1$ and $\widetilde{g}_2$,
we obtain, for any $x\in\cx$,
\begin{equation}\label{est_tg}
\widetilde{g}(x)\lesssim  \|f\|_{\dot{B}^s_{\infty,\infty}(\cx)}[d(x,x_1)]^s.
\end{equation}

Now, for any $x\in\cx$, define
$$g(x):=\sum_{k=-\infty}^\infty\left[\widetilde{Q}_kQ_kf(x)-\widetilde{Q}_kQ_kf(x_1)\right].$$
From \eqref{est_tg}, we deduce that $g$ is well defined and,
for any $x\in\cx$,
$$|g(x)|\leq\widetilde{g}(x)\lesssim  \|f\|_{\dot{B}^s_{\infty,\infty}(\cx)}[d(x,x_1)]^s.$$
Notice that, for any $x,\ y\in\cx$,
$$g(x)-g(y)=\sum_{k=-\infty}^\infty\left[\widetilde{Q}_kQ_kf(x)-\widetilde{Q}_kQ_kf(y)\right].$$
Using this and repeating the estimation of \eqref{est_tg} via replacing $x_1$ by $y$,
we then conclude that, for any $x,\ y\in\cx$,
$$|g(x)-g(y)|\lesssim\|f\|_{\dot{B}^s_{\infty,\infty}(\cx)}[d(x,y)]^s,$$
which implies that
$$\|g\|_{\dot{C}^s(\cx)}\lesssim\|f\|_{\dot{B}^s_{\infty,\infty}(\cx)}$$
and hence completes the proof of (ii).
This finishes the proof of Theorem \ref{c_b}.
\end{proof}

At the end of this section, we establish the relationship
between $\BMO(\cx)$ and $\dot{F}^s_{\infty,2}(\cx)$.
Let us begin with the notions of 1-exp-ATI (see,
for instance \cite[Definition 2.8]{hhllyy}), $\BMO^p(\cx)$
(see, for instance, \cite{cw77}) and the Hardy space $H^1(\cx)$
(see, for instance, \cite[Section 3]{hhllyy}).

\begin{definition}\label{1-exp-ati}
Let $\eta\in(0,1)$ be as in Definition \ref{10.23.2}.
A sequence $\{P_k\}_{k=-\infty}^\infty$ of bounded linear integral operators
on $L^2(\cx)$ is called  an \emph{approximation of the identity with exponential
decay and integration 1} (for short, 1-exp-ATI) if $\{P_k\}_{k=-\infty}^\infty$ has the following properties:
\begin{enumerate}
\item[{\rm(i)}] for any $k\in\zz$, $P_k$
satisfies (ii) and (iii) of Definition \ref{10.23.2} but without the term
$$\exp\left\{-\nu\left[\max\{d(x,\cy^k),d(y,\cy^k)\}\right]^a\right\};$$
\item[{\rm(iii)}] for any $k\in\zz$ and $x\in\cx$,
$$\int_\cx P_k(x,y)\,d\mu(y)=1=\int_\cx P_k(y,x)\,d\mu(y);$$
\item[{\rm(ii)}] for any $k\in\zz$, letting $Q_k:=P_k-P_{k-1}$,
then $\{Q_k\}_{k\in\zz}$ is an exp-ATI.
\end{enumerate}
\end{definition}

\begin{definition}
For any $p\in[1,\infty)$,
the space $\BMO^p(\cx)$ is defined by setting
$$\BMO^p(\cx):=\left\{f\in L^1_{\loc}(\cx):\ \sup_{B\subset\cx}
\left[\frac{1}{\mu(B)}\int_B|f(x)-m_B(f)|^p\,d\mu(x)\right]^{1/p}<\infty\right\},$$
where, for any ball $B$, $m_B(f)$ is as \eqref{int_m} with $E$ replaced by $B$.
\end{definition}

The space $\BMO^p(\cx)$ when $p=1$ is simply denoted by $\BMO(\cx)$.

\begin{definition}\label{Hp}
Let $\beta,\ \gamma\in (0,\eta)$ with $\eta$ as in Definition \ref{10.23.2}.
Let $p\in(0,\infty)$ and $\{P_k\}_{k\in\zz}$ be a 1-exp-ATI as in Definition \ref{1-exp-ati}.
Then the Hardy space $H^p(\cx)$ is defined by setting
$$
H^p(\cx):=\left\{f\in(\icgg)':\ \|f\|_{H^p(\cx)}:=\left\|\sup_{k\in\zz}|P_kf|\right\|_{L^p(\cx)}<\infty\right\}.
$$
\end{definition}

It is well known that the dual space of $H^1(\cx)$ is $\BMO(\cx)$ (see, for instance,
\cite[Theorem B]{cw77}). Moreover, by the atomic characterization of $H^p(\cx)$ established in \cite{cw77}
and \cite[Theorem 4.16]{hhllyy}, it is easy to see  that,
for any given $p\in(1,\infty)$,
\begin{equation}\label{bmop_b}
\BMO^p(\cx)=\BMO(\cx)
\end{equation}
with equivalent norms (see, for instance,
\cite[pp.\ 632--633, Proof of Theorem B]{cw77} for the details).

The following theorem establishes the relationship
between $\BMO(\cx)$ and $\dot{F}^0_{\infty,2}(\cx)$.

\begin{theorem}\label{bmo_f}
Let $\eta$ be as in Definition \ref{10.23.2}.
Then $\BMO(\cx)=\dot{F}^0_{\infty,2}(\cx)$ in the following sense:
\begin{enumerate}
\item[{\rm(i)}] there exists a positive constant $C$ such that,
for any $f\in \BMO(\cx)$,
$$\|f\|_{\dot{F}^0_{\infty,2}(\cx)}\leq C\|f\|_{\BMO(\cx)};$$
\item[{\rm(ii)}] for any $f\in \dot{F}^0_{\infty,2}(\cx)$,
the linear functional
$$L_f :\ g \mapsto L_f(g):= \langle f, g\rangle,$$
initially defined on $\mathring{\cg}(\eta,\eta)$, has a bounded linear extension to $H^1(\cx)$;
moreover, there exists a positive constant $C$, independent of $f$, such that
$$\|L_f\|_{(H^1(\cx))^\ast}\leq C\|f\|_{\dot{F}^0_{\infty,2}(\cx)}.$$
\end{enumerate}
\end{theorem}

For any $l\in\zz$ and $x\in\cx$, let
\begin{equation}\label{6.8x}
E^l(x):=\bigcup_{\{\alpha\in\ca_l:\ Q_\alpha^l\cap B(x,(2A_0^2+A_0)\delta^l)
\neq\emptyset\}}Q_\alpha^l.
\end{equation}
Let $\{Q_k\}_{k\in\zz}$ be an exp-ATI. For any $s\in\rr$, $q\in(0,\infty]$, $f\in(\cggi)'$
with $\beta,\ \gamma\in(0,\eta)$ and $x\in\cx$,  define
\begin{equation}\label{m_q_s}
\dot{\cm}_q^s(f)(x)
:=\sup_{l\in\zz}\left[\frac{1}{\mu(E^l(x))}
\int_{E^l(x)}\sum_{k=l}^\infty\delta^{-ksq}|Q_k(f)(y)|^q\,d\mu(y)\right]^{1/q},
\end{equation}
where the usual modification is made when $q=\infty$.

To prove Theorem \ref{bmo_f}, we need the  following  technical lemma,
which on RD-spaces was obtained in \cite[Proposition 6.10]{hmy08}.
Via carefully checking the proof of \cite[Proposition 6.10]{hmy08},
we find that it does not depend on the reverse doubling condition \eqref{eq-rdoub} at all
and still holds true on  spaces of homogeneous type; we omit the details here.

\begin{lemma}\label{e_mqs}
Let $\beta,\ \gamma\in(0,\eta)$, with $\eta$ as in Definition \ref{10.23.2}, and
$s\in(-(\beta\wedge\gamma),\beta\wedge\gamma)$ satisfy \eqref{12.1.2}
and $q\in(p(s,\beta\wedge\gamma),\infty]$ with $p(s,\beta\wedge\gamma)$ as in \eqref{pseta}.
Then $f\in\dot{F}^s_{\infty,q}(\cx)$ if and only if $f\in(\cggi)'$
and $\dot{\cm}_q^s(f)\in L^\infty(\cx)$.
Moreover, there exists a positive constant $C$, independent of $f$, such that
$$C^{-1}\|f\|_{\dot{F}^s_{\infty,q}(\cx)}
\leq \|\dot{\cm}_q^s(f)\|_{L^\infty(\cx)}
\leq C\|f\|_{\dot{F}^s_{\infty,q}(\cx)}.$$
\end{lemma}

Now we show Theorem \ref{bmo_f}.

\begin{proof}[Proof of Theorem \ref{bmo_f}]
We first show (i).
Assume $f\in \BMO(\cx)$ and $\beta,\ \gamma$ are as in \eqref{12.1.2}.
Let $\{Q_k\}_{k=-\fz}^\fz$ be an exp-ATI. By $\mathring{\cg}^\eta_0(\beta,\gamma)\subset H^1(\cx)$
and $(H^1(\cx))^\ast=\BMO(\cx)$, we know that $f\in(\cggi)'$.
Fix $l\in\zz$ and $\alpha\in \ca_l$.
Define $B_\alpha^l:=B(z_\alpha^l, 3A_0^2C_0\delta^l)$ and write
$$f=\left[f-m_{B_\alpha^l}(f)\right]\mathbf{1}_{B_\alpha^l}
+\left[f-m_{B_\alpha^l}(f)\right]\mathbf{1}_{\cx\setminus B_\alpha^l}
+m_{B_\alpha^l}(f)
=:f_{\alpha,l}^{(1)}+f_{\alpha,l}^{(2)}+f_{\alpha,l}^{(3)}.$$
From Definition \ref{10.23.2}(v), it follows that,
for any $k\in\zz$,
\begin{equation}\label{f_a_3}
Q_k(f_{\alpha,l}^{(3)})=0.
\end{equation}

Now we estimate $f_{\alpha,l}^{(1)}$. By \eqref{bmop_b},
we know that $f_{\alpha,l}^{(1)}\in L^2(\cx)$. Then, from
Theorem \ref{l_btl} and \eqref{bmop_b} with $p=2$, we deduce that
\begin{align}\label{f_a_1}
&\left\{\frac{1}{\mu(Q_\alpha^l)}\int_{Q_\alpha^l}
\sum_{k=l}^\infty |Q_k(f_{\alpha,l}^{(1)})(x)|^2\,d\mu(x)\right\}^{1/2}\\
&\qquad\leq \left\{\frac{1}{\mu(Q_\alpha^l)}\int_{\cx}
\sum_{k=-\infty}^\infty |Q_k(f_{\alpha,l}^{(1)})(x)|^2\,d\mu(x)\right\}^{1/2}
\lesssim \left\{\frac{1}{\mu(Q_\alpha^l)}\int_{\cx}
|f_{\alpha,l}^{(1)}(x)|^2\,d\mu(x)\right\}^{1/2}\notag\\
&\qquad\sim \left\{\frac{1}{\mu(B_\alpha^l)}\int_{B_\alpha^l}
|f(x)-m_{B_\alpha^l}(f)|^2\,d\mu(x)\right\}^{1/2}\notag
\lesssim \|f\|_{\BMO(\cx)}.\notag
\end{align}
Next we estimate $f_{\alpha,l}^{(2)}$.
Notice that,
for any $x\in Q_\alpha^l$ and $y\in\cx\setminus B_\alpha^l$,
$$d(y,z_\alpha^l)\leq A_0[d(y,x)+d(x,z_\alpha^l)]
\leq A_0d(y,x)+2A_0^2C_0\delta^l$$
and hence $\delta^l\lesssim d(y,x)$. By this, we obtain
$$\delta^l +d(y,z_\alpha^l)\lesssim d(y,x)+\delta^l\lesssim d(y,x)$$
and hence
$$V(z_\alpha^l,y)\sim V(y,z_\alpha^l)\lesssim V(y,x)\sim V(x,y).$$
From this,  we deduce that,
for any fixed $\Gamma\in(0,\infty)$,
and for any $k\geq l$ and $x\in Q_\alpha^l$,
\begin{align*}
&|Q_k(f_{\alpha,l}^{(2)})(x)|\\
&\quad=\left|\int_{\cx\setminus B_\alpha^l}Q_k(x,y)[f(y)-m_{B_\alpha^l}(f)]\,d\mu(y)\right|\\
&\quad\lesssim \int_{\cx\setminus B_\alpha^l}\frac{1}{V_{\delta^k}(x)+V(x,y)}
\left[\frac{\delta^k}{\delta^k+d(x,y)}\right]^\Gamma|f(y)-m_{B_\alpha^l}(f)|\,d\mu(y)\\
&\quad\lesssim \sum_{j=0}^\infty\int_{3A_0^2C_0\delta^{l-j}
\leq d(z_\alpha^l,y)<3A_0^2C_0\delta^{l-j-1}}\frac{1}{V_{\delta^k}(x)+V(x,y)}
\left[\frac{\delta^k}{\delta^k+d(x,y)}\right]^\Gamma|f(y)-m_{B_\alpha^l}(f)|\,d\mu(y)\\
&\quad\lesssim \delta^{(k-l)\Gamma}\sum_{j=0}^\infty \delta^{j\Gamma}
\frac{1}{V(z_\alpha, 3A_0^2C_0\delta^{l-j-1})}
\int_{B(z_\alpha, 3A_0^2C_0\delta^{l-j-1})}|f(y)-m_{B_\alpha^l}(f)|\,d\mu(y)\\
&\quad\lesssim \delta^{(k-l)\Gamma} \sum_{j=0}^\infty \delta^{j\Gamma}
\left[\frac{1}{V(z_\alpha, 3A_0^2C_0\delta^{l-j-1})}
\int_{B(z_\alpha, 3A_0^2C_0\delta^{l-j-1})}
|f(y)-m_{B(z_\alpha, 3A_0^2C_0\delta^{l-j-1})}(f)|\,d\mu(y)\right.\\
&\quad\qquad \left.\vpz{\frac{1}{V(z_\alpha, 3A_0^2C_0\delta^{l-j-1})}}+
|m_{B(z_\alpha, 3A_0^2C_0\delta^{l-j-1})}(f)-m_{B_\alpha^l}(f)|\right].
\end{align*}
Notice that, for any $j\in\zz_+$,
\begin{align*}
&|m_{B(z_\alpha, 3A_0^2C_0\delta^{l-j-1})}(f)-m_{B_\alpha^l}(f)|\\
&\qquad\leq \frac{1}{\mu(B_\alpha^l)}\int_{B_\alpha^l}|f(y)
-m_{B(z_\alpha, 3A_0^2C_0\delta^{l-j-1})}(f)|\,d\mu(y)\\
&\qquad\leq \frac{\mu(B(z_\alpha, 3A_0^2C_0\delta^{l-j-1}))}{\mu(B_\alpha^l)}
\frac{1}{\mu(B(z_\alpha, 3A_0^2C_0\delta^{l-j-1}))}\\
&\quad\qquad\times\int_{B(z_\alpha, 3A_0^2C_0\delta^{l-j-1})}|f(y)
-m_{B(z_\alpha, 3A_0^2C_0\delta^{l-j-1})}(f)|\,d\mu(y)\\
&\qquad\lesssim \delta^{-j\omega}\|f\|_{\BMO(\cx)}.
\end{align*}
Using this and choosing $\Gamma \in(\omega,\infty)$,  we conclude that
$$|Q_k(f_{\alpha,l}^{(2)})(x)|\lesssim\delta^{(k-l)\Gamma}
\sum_{j=0}^\infty\left[\delta^{j\Gamma}+\delta^{j(\Gamma-\omega)}\right]
\|f\|_{\BMO(\cx)}\lesssim\delta^{(k-l)\Gamma}\|f\|_{\BMO(\cx)}.$$
Thus,
\begin{align}\label{f_a_2}
\left\{\frac{1}{\mu(Q_\alpha^l)}\int_{Q_\alpha^l}
\sum_{k=l}^\infty |Q_k(f_{\alpha,l}^{(2)})(x)|^2\,d\mu(x)\right\}^{1/2}
&\lesssim \|f\|_{\BMO(\cx)}
\left[\sum_{k=l}^\infty\delta^{2(k-l)\Gamma}\right]^{1/2}\\
&\sim \|f\|_{\BMO(\cx)}.\notag
\end{align}
From \eqref{f_a_3}, \eqref{f_a_1} and \eqref{f_a_2},
we deduce that, for any $l\in\zz$ and $\alpha\in\ca_l$,
\begin{align*}
\left\{\frac{1}{\mu(Q_\alpha^l)}\int_{Q_\alpha^l}
\sum_{k=l}^\infty |Q_k(f)(x)|^2\,d\mu(x)\right\}^{1/2}
&\lesssim \left\{\frac{1}{\mu(Q_\alpha^l)}\int_{Q_\alpha^l}
\sum_{k=l}^\infty |Q_k(f_{\alpha,l}^{(1)})(x)|^2\,d\mu(x)\right\}^{1/2}\\
&\qquad +\left\{\frac{1}{\mu(Q_\alpha^l)}\int_{Q_\alpha^l}
\sum_{k=l}^\infty |Q_k(f_{\alpha,l}^{(2)})(x)|^2\,d\mu(x)\right\}^{1/2}\\
&\lesssim \|f\|_{\BMO(\cx)},
\end{align*}
which implies  that
$$\|f\|_{\dot{F}^0_{\infty,2}(\cx)}\lesssim\|f\|_{\BMO(\cx)}$$
and hence $f\in\dot{F}^0_{\infty,2}(\cx)$.
This finishes the proof of (i).

Now we show (ii). Assume $f\in\dot{F}^0_{\infty,2}(\cx)$.
By Proposition \ref{12.2.7}, we have $f\in (\cggi)'$ with $\beta$ and $\gamma$ as in \eqref{12.1.2}.
Then, from \cite[Theorem 4.16]{hlyy}, it follows that there exists a sequence
$\{\overline{Q}_k\}_{k=-\infty}^\infty$
of bounded linear integral operators on $L^2(\mathcal{X})$,
whose kernels satisfy \eqref{4.23x},
\eqref{4.23a} and \eqref{4.23y} for the second  variable, such that,
for any $h\in \mathring{\cg}(\eta,\eta)$
$$h= \sum_{k=-\infty}^\infty Q_k\overline{Q}_kh \qquad \text{in}\quad \cggi.$$
For any $j\in\zz$ and $x\in\cx$, define
$$\dot{S}^j(f)(x):=\left\{\sum_{k=j}^\infty\int_{B(x,\delta^k)}
|Q_k(f)(y)|^2\frac{d\mu(y)}{V_{\delta^k}(x)}\right\}^{1/2}$$
and, for any $x\in\cx$, define
$$\dot{\overline{S}}(f)(x):=\left\{\sum_{k=-\infty}^\infty\int_{B(x,\delta^k)}
\left|\overline{Q}_k(f)(y)\right|^2\frac{d\mu(y)}{V_{\delta^k}(x)}\right\}^{1/2}.$$
Using an argument similar to that used in the proof of
\cite[Theorem 5.10]{hhllyy}, we conclude that, for any $h\in H^1(\cx)$,
\begin{equation}\label{ws_h}
\left\|\dot{\overline{S}}(h)\right\|_{L^1(\cx)}\lesssim\|h\|_{H^1(\cx)}.
\end{equation}
For any $x\in\cx$, define
\begin{equation}\label{stop_l}
l(x):=\inf\left\{j\in\zz:\ \dot{S}^j(f)(x)\leq A\dot{\cm}_2^0(f)(x)\right\},
\end{equation}
where $\dot{\cm}_2^0(f)$ is as in \eqref{m_q_s}
and $A$ is a positive constant which is determined later.

We claim that there exists a positive constant $C$ such that,
for any $y\in\cx$, $l\in\zz$ and $A$ large enough,
\begin{equation}\label{mu_mu}
\mu\left(\left\{z\in B(y,\delta^l):\ \dot{S}^l(f)(z) \leq A\dot{\cm}_2^0(f)(z)\right\}\right)\geq C\mu(B(y,\delta^l)).
\end{equation}
Indeed, let $B_0:=B(y,\delta^l)$. Then,
for any $x\in B_0$ and $z\in B(x,\delta^l)$,
$$d(y,z)\leq A_0[d(y,x)+d(x,z)]< 2A_0\delta^l,$$
which implies that
$$\bigcup_{x\in B_0}B(x,\delta^l)\subset B(y, 2A_0\delta^l).$$
Let
$$P:=\bigcup_{\{\alpha\in\ca_l:\ Q_\alpha^l\cap B(y,2A_0\delta^{l})\neq\emptyset\}}Q_\alpha^l.$$
For any $\alpha\in\ca_l$
such that $Q_\alpha^l\cap B(y,2A_0\delta^{l})\neq\emptyset$,
let $w\in Q_\alpha^l\cap B(y,2A_0\delta^{l})$. Then, for any $z\in Q_\alpha^l$,
$$d(y,z)\leq A_0[d(y,w)+d(w,z)] \leq 2A_0^2\delta^l
+A_0[d(w,z_\alpha^l)+d(z_\alpha^l)]\leq2A_0^2(1+C_0)\delta^l.$$
Thus, $P\subset B(y,2A_0^2(1+C_0)\delta^l)$. Moreover, since
$B_0=B(y,\delta^l)\subset B(y,2A_0\delta^l)\subset P$, it follows that
$\mu(P)\sim \mu(B_0)$.
From this and the Fubini theorem,  we deduce  that
there exist positive constants $C_1$ and $C_2$, independent of $x$,
$l$ and $f$, such that
\begin{align*}
\frac{1}{\mu(B_0)}\int_{B_0}\lf[\dot{S}^{l}(f)(x)\r]^2\,d\mu(x)
&=\frac{1}{\mu(B_0)}\sum_{k=l}^\infty\int_{B_0}
\int_{B(x,\delta^k)}|Q_k(f)(u)|^2\frac{d\mu(u)}{V_{\delta^k}(x)}\,d\mu(x)\\
&\leq C_1 \frac{1}{\mu(P)}\sum_{k=l}^\infty\int_{P}
\int_{B_0}|Q_k(f)(u)|^2\mathbf{1}_{B(z,\delta^k)}(x)\frac{d\mu(x)}{V_{\delta^k}(x)}\,d\mu(u)\\
&\leq C_2 \frac{1}{\mu(P)}\int_P \sum_{k=l}^\infty |Q_k(f)(u)|^2\,d\mu(u).
\end{align*}
For any $Q_\alpha^l\subset P$, we know that $Q_\alpha^l\cap B(y,2A_0\delta^l)\neq \emptyset$.
Choosing $\widetilde{w}\in Q_\alpha^l\cap B(y,2A_0\delta^l)$, then, for any $z\in B(y,\dz^l)$, we have
$$d(\widetilde{w},z)\leq A_0[d(\widetilde{w},y)+d(y,z)]
<(2A_0^2+A_0)\delta^l,$$
which implies $\widetilde{w}\in Q_\alpha^l\cap B(z,(2A_0^2+A_0)\delta^l)\neq\emptyset$.
From this, it follows that $Q_\alpha^l\subset E^l(z)$ and hence $P\subset E^l(z)$,
where $E^l(z)$ is as in \eqref{6.8x} with $x$ replaced by $z$.
Moreover, it is easy to see that
$$\mu(E^l(z))\sim \mu(B(z,\delta^l))\sim\mu(B_0).$$
By these, we conclude that there exists a positive constant $C_3$
such that
$$\frac{1}{\mu(B_0)}\int_{B_0}\lf[\dot{S}^{l}(f)(x)\r]^2\,d\mu(x)
\leq C_3\frac{1}{\mu(E^l(z))}\int_{E^l(z)}\sum_{k=l}^\infty |Q_k(f)(u)|^2\,d\mu(u)
\leq C_3\left[\dot{\cm}_2^0(f)(z)\right]^2.$$
From the arbitrariness of $z$, we further deduce that
$$\frac{1}{\mu(B_0)}\int_{B_0}\lf[\dot{S}^{l}(f)(x)\r]^2\,d\mu(x)\leq C_3
\left[\inf_{z\in B_0}\dot{\cm}_2^0(f)(z)\right]^2.$$
By this, we conclude that,
if $0<\inf_{x\in B_0 }\dot{\cm}_2^0(f)(x)<\infty$, then
\begin{align*}
\left[\inf_{x\in B_0 }\dot{\cm}_2^0(f)(x)\right]^2
\int_{\{z\in B_0:\ \dot{S}^l(f)(z) > A\dot{\cm}_2^0(f)(z)\}}\,d\mu(z)
&\leq  A^{-2}\int_{B_0}\lf[\dot{S}^{l}(f)(x)\r]^2\,d\mu(x)\\
&\leq A^{-2}C_3\mu(B_0)\left[\inf_{x\in B_0 }\dot{\cm}_2^0(f)(x)\right]^2.
\end{align*}
Thus, we have
$$\mu\left(\left\{z\in B_0:\ \dot{S}^l(f)(z) > A\dot{\cm}_2^0(f)(z)\right\}\right) \leq A^{-2}C_3\mu(B_0).$$
Choosing $A^2>C_3$,  we then obtain
$$\mu\left(\left\{z\in B_0:\ \dot{S}^l(f)(z) \leq A\dot{\cm}_2^0(f)(z)\right\}\right) \geq (1-A^{-2}C_3)\mu(B_0),$$
which proves \eqref{mu_mu}  in this case. Notice    that,
 if $\inf_{x\in B_0 }\dot{\cm}_2^0(f)(x)=0$
 or $\inf_{x\in B_0 }\dot{\cm}_2^0(f)(x)=\infty$, it is easy to see that
\eqref{mu_mu} also holds true, which completes the proof of \eqref{mu_mu}.

Let $h\in \mathring{\cg}(\eta,\eta)$ and $\|h\|_{H^1(\cx)} \leq 1$.
Then, from \cite[Theorem 5.10]{hhllyy}, the Fubini theorem, the H\"older inequality,
\eqref{mu_mu}, \eqref{ws_h} and Lemma \ref{e_mqs}, we deduce that
\begin{align*}
|\langle f, h\rangle|
&= \left|\left\langle f, \sum_{k=-\infty}^\infty Q_k\overline{Q}_kh\right\rangle\right|
=\left|\sum_{k=-\infty}^\infty\left\langle Q_k^\ast(f),
\overline{Q}_k(h) \right\rangle\right|\\
&\leq \sum_{k=-\infty}^\infty\int_{\cx} |Q_k^\ast(f)(y)
\overline{Q}_k(h)(y)|\,d\mu(y)\\
&\lesssim \int_{\cx}\left[\sum_{k=l(x)}^\infty \int_{B(x,\delta^k)}|Q_k^\ast(f)(y)
\overline{Q}_k(h)(y)|\frac{d\mu(y)}{V_{\delta^k}(x)}\right]\,d\mu(x)\\
&\lesssim \int_{\cx}\dot{S}^{l(x)}(f)(x)\dot{\overline{S}}(f)(x)\,d\mu(y)
\lesssim \int_{\cx}\dot{\cm}_2^0(f)(x)\dot{\overline{S}}(f)(x)\,d\mu(x)\\
&\lesssim \lf\|\dot{\cm}_2^0(f)\r\|_{L^\infty(\cx)}\lf\|\dot{\overline{S}}(h)\r\|_{L^1(\cx)}
\lesssim\|f\|_{\dot{F}_{\infty,2}^0(x)}\|h\|_{H^1(\cx)}.
\end{align*}
Combining this, Lemma \ref{12.15.2} and a density argument,
we then complete the proof of (ii) and hence  of  Theorem \ref{bmo_f}.
\end{proof}

Now we consider the relationship between homogeneous Besov and Triebel--Lizorkin spaces
and their inhomogeneous counterparts.

\begin{theorem}\label{ihlless}
Let $\beta,\ \gamma\in(0,\eta)$ with $\eta$ as in
Definition \ref{10.23.2}, and $s\in(0,\beta\wedge\gamma)$.
\begin{enumerate}
\item[{\rm(I)}] Let $p\in [1, \infty]$ and  $q\in(0,\infty]$.
\begin{enumerate}
\item[{\rm(I)$_{\rm 1}$}] For any $f\in\ihb$, there exists a $g\in L^p(\cx)\cap\hb$ such that
$$f=g\quad \text{in}\quad (\icgg)'$$
and
$$\|g\|_{\hb}+\|g\|_{L^p(\cx)}\leq C \|f\|_{\ihb},$$
where $C$ is a positive constant independent of $f$.
\item[{\rm(I)$_{\rm 2}$}] If $f\in \hb\cap L^p(\cx)$, then $f\in\ihb$ and
$$\|f\|_{\ihb}\leq C\lf[\|f\|_{\hb}+\|f\|_{L^p(\cx)}\r],$$
where $C$ is a positive constant independent of $f$.
\end{enumerate}
\end{enumerate}
\begin{enumerate}
\item[{\rm(II)}] Let $p\in [1, \infty)$ and  $q\in(p(s,\beta\wedge\gamma),\infty]$ with
$p(s,\beta\wedge\gamma)$ as in \eqref{pseta}.
\begin{enumerate}
\item[{\rm(II)$_{\rm 1}$}] For any $f\in\ihf$, there exists a $g\in L^p(\cx)\cap\hf$ such that
$$f=g\quad \text{in}\quad (\icgg)'$$
and
$$\|g\|_{\hf}+\|g\|_{L^p(\cx)}\leq C \|f\|_{\ihf},$$
where $C$ is a positive constant independent of $f$.
\item[{\rm(II)$_{\rm 2}$}] If $f\in \hf\cap L^p(\cx)$, then $f\in\ihf$ and
$$\|f\|_{\ihf}\leq C \lf[\|f\|_{\hf}+\|f\|_{L^p(\cx)}\r],$$
where $C$ is a positive constant independent of $f$.
\end{enumerate}
\end{enumerate}
\begin{enumerate}
\item[{\rm(III)}] Let $q\in[1,\infty]$.
\begin{enumerate}
\item[{\rm(III)$_{\rm 1}$}] For any $f\in\ihfi$, there exists a $g\in \hfi\cap L^\infty(\cx)$ such that
$$f=g\quad \text{in}\quad (\icgg)'$$
and
$$\|g\|_{\hfi}+\|g\|_{L^\infty(\cx)}\leq C \|f\|_{\ihfi},$$
where $C$ is a positive constant independent of $f$.
\item[{\rm(III)$_{\rm 2}$}] If $f\in \hfi\cap L^\infty(\cx)$, then $f\in\ihfi$ and
$$\|f\|_{\ihfi}\leq C\lf[\|f\|_{\hfi}+\|f\|_{L^\infty(\cx)}\r],$$
where $C$ is a positive constant independent of $f$.
\end{enumerate}
\end{enumerate}
\end{theorem}

\begin{proof}
We first prove (I)$_{\rm 1}$ when $p\in[1,\infty)$. Let $\{P_k\}_{k\in\zz}$ be a 1-exp-ATI.
For any $k\in\zz$, letting $Q_k:= P_k-P_{k-1}$, then, by Definitions \ref{1-exp-ati}
and \ref{1-exp-iati} below, we know that $\{Q_k\}_{k\in\zz}$ is an exp-ATI
and $\{P_0\}\cup\{Q_k\}_{k\in\nn}$ is an exp-IATI.
For any $f\in\ihb$, we have $f\in(\icgg)'$ with $s$, $\beta$
and $\gamma$ satisfying  \eqref{10.19.3}.
By Lemma \ref{ih_c_crf}, we know that there exist an $N\in\nn$
and a sequence $\{\widetilde{Q}_k\}_{k=0}^\infty$ of
bounded linear integral operators on $L^2(\mathcal{X})$ such that
\begin{equation}\label{fiexp}
f = \widetilde{Q}_0P_0f+\sum_{k=1}^\infty \widetilde{Q}_kQ_kf  \quad \text {in}\quad (\icgg)'.
\end{equation}
Moreover, for any $k\in\zz_+$, the kernel of $\widetilde{Q}_k$ satisfies
\eqref{4.23x}, \eqref{4.23y} and the following integral condition: for any $x\in\mathcal{X}$,
\begin{equation}\label{n}
\int_{\mathcal{X}}\widetilde{Q}_k(x,y)\,d\mu(y)=
\int_{\mathcal{X}}\widetilde{Q}_k(y,x)\,d\mu(y)=\begin{cases}
1 &\text{if } k \in \{0,\dots,N\},\\
0 &\text{if } k\in \{N+1,N+2,\ldots\}.
\end{cases}
\end{equation}
For any $n\in\nn$, let
\begin{equation}\label{fn}
f_n:=\widetilde{Q}_0P_0f+\sum_{k=1}^n \widetilde{Q}_kQ_kf.
\end{equation}
It is easy to see that, for  any $n\in\nn$, $f_n$ is a measurable function on $\cx$.
If $q\in(0,1)$, by \cite[Proposition 2.7(iii)]{hmy08} and \eqref{r},
we conclude that, for any $m$, $n\in\nn$ and $m>n$,
\begin{align}\label{fmfn1}
\|f_m-f_n\|_{L^p(\cx)}&=\left\|\sum_{k=n+1}^m \widetilde{Q}_kQ_kf\right\|_{L^p(\cx)}
\lesssim \sum_{k=n+1}^m \left\|Q_kf\right\|_{L^p(\cx)}\\
&\lesssim \delta^{ns}\left[\sum_{k=n+1}^m \delta^{-ksq}\left\|Q_kf\right\|_{L^p(\cx)}^q\right]^{1/q}.\notag
\end{align}
If $q\in[1,\infty]$, by the H\"older inequality, we also obtain,
for any $m$, $n\in\nn$ and $m>n$,
\begin{align}\label{fmfn2}
\|f_m-f_n\|_{L^p(\cx)}&=\left\|\sum_{k=n+1}^m \widetilde{Q}_kQ_kf\right\|_{L^p(\cx)}
\lesssim \sum_{k=n+1}^m \left\|Q_kf\right\|_{L^p(\cx)}\\
&\lesssim \sum_{k=n+1}^m \delta^{-ks/2}\left\|Q_kf\right\|_{L^p(\cx)}
\lesssim\left(\sum_{k=n+1}^m \delta^{ksq'/2}\right)^{1/q'}
\left[\sum_{k=n+1}^m \delta^{-ksq}\left\|Q_kf\right\|_{L^p(\cx)}^q\right]^{1/q}\notag\\
&\lesssim \delta^{ns/2}\left[\sum_{k=n+1}^m
\delta^{-ksq}\left\|Q_kf\right\|_{L^p(\cx)}^q\right]^{1/q}.\notag
\end{align}
From this and $f\in\ihb$, we deduce that $\|f_m-f_n\|_{L^p(\cx)}\to0$ as $m,\ n\to\infty$.
Similarly, we can also show that, for any $n\in\nn$,  $\|f_n\|_{L^p(\cx)}\lesssim \|f\|_{\ihb}$
and hence $\{f_n\}_{n\in\nn}$ is a Cauchy sequence in $L^p(\cx)$.
By the completeness of $L^p(\cx)$ with $p\in[1,\infty)$, let $g:=\lim_{n\to\infty}f_n$ in $L^p(\cx)$. Then, for any $\varphi\in\icgg$,
by the H\"older inequality, we have
\begin{align}\label{feg1}
0&\leq|\langle f-g,\varphi\rangle|\leq|\langle f-f_n,\varphi\rangle|
+|\langle f_n-g,\varphi\rangle|\\
&\leq
|\langle f-f_n,\varphi\rangle|+\|f_n-g\|_{L^p(\cx)}\|\varphi\|_{L^{p'}(\cx)}\to 0\notag
\end{align}
as $n\to\fz$, which implies $f=g$ in $(\icgg)'$,
where $1/p+1/p'=1$.
Moreover, using an argument similar to that used in the estimations of  \eqref{fmfn1}
and \eqref{fmfn2}, we find that, for any $p\in[1,\infty)$ and
$q\in(0,\infty]$,
\begin{align}\label{lplpb}
\|g\|_{L^p(\cx)}&=\lim_{n\to\infty}\|f_n\|_{L^p(\cx)}\\
&\lesssim \|P_0f\|_{L^p(\cx)}+\sum_{k=1}^N\|Q_kf\|_{L^p(\cx)}
+\lim_{n\to\infty}\left[\sum_{k=N+1}^n \delta^{-ksq}
\left\|Q_kf\right\|_{L^p(\cx)}^q\right]^{1/q}\notag\\
&\lesssim \|P_0f\|_{L^p(\cx)}+\sum_{k=1}^N\|Q_kf\|_{L^p(\cx)}+\|f\|_{\ihb}.\notag
\end{align}
By \eqref{10.15.5}, \eqref{10.15.6}, \eqref{qkpk}, \eqref{qkpkm}
and  an  argument similar to that used in the  estimation of \eqref{10.15.4}, we conclude that
\begin{align}\label{p0}
\|P_0f\|_{L^p(\cx)}&=\left\{\sum_{\alpha\in\ca_0}\sum_{m=1}^{N(0,\alpha)}
\int_{Q_\alpha^{0,m}}|P_0f(x)|^pd\mu(x)\right\}^{1/p}\\
&\leq \left\{\sum_{\alpha\in\ca_0}\sum_{m=1}^{N(0,\alpha)}
\mu(Q_\alpha^{0,m})\left[\sup_{z\in Q_\alpha^{0,m}}|P_0f(z)|\right]^p\right\}^{1/p}\notag\\
&\lesssim \left\{\sum_{\alpha \in \ca_0}\sum_{m=1}^{N(0,\alpha)}
\mu(Q_\alpha^{0,m})\left[m_{Q_\alpha^{0,m}}(|P_0f|)\r]^p\r.\notag\\
&\left.\qquad+\sum_{k=1}^{N}\sum_{\alpha \in \ca_k}\sum_{m=1}^{N(k,\alpha)}
\mu\left(\qa\right)\left[m_{\qa}(|Q_{k}(f)|)\r]^p\right\}^{1/p}\notag\\
&\qquad+\left\{\sum_{k=N+1}^\infty\delta^{-ksq}
\left[\sum_{\alpha \in \ca_{k}}\sum_{m=1}^{N(k,\alpha)}\mu\left(\qa\right)
\left\{\inf_{z\in\qa}|Q_{k}(f)(z)|\right\}^p\right]^{q/p}\right\}^{1/q}\notag\\
&\lesssim \|f\|_{\ihb}\notag
\end{align}
and, for any $k\in\{1,\dots,N\}$,
\begin{equation}\label{qkf}
\|Q_kf\|_{L^p(\cx)}\lesssim \|f\|_{\ihb}.
\end{equation}
Combining this with \eqref{lplpb} and \eqref{p0}, we know that
\begin{equation}\label{lpb}
\|g\|_{L^p(\cx)}\lesssim \|f\|_{\ihb}.
\end{equation}
This, combined with \cite[Proposition 2.7(iii)]{hmy08} and $f=g$ in $(\icgg)'$, implies that
\begin{align}\label{glessf}
\|g\|_{\hb}&=\left[\sum_{k=-\infty}^\infty\delta^{-ksq}\|Q_kg\|_{L^p(\cx)}^q\right]^{1/q}\\
&\lesssim \left[\sum_{k=-\infty}^N\delta^{-ksq}\|Q_kg\|_{L^p(\cx)}^q\right]^{1/q}
+\left[\sum_{k=N+1}^\infty\delta^{-ksq}\|Q_kg\|_{L^p(\cx)}^q\right]^{1/q}\notag\\
&\lesssim \|g\|_{L^p(\cx)}+ \left[\sum_{k=N+1}^\infty
\delta^{-ksq}\|Q_kf\|_{L^p(\cx)}^q\right]^{1/q}\notag
\lesssim \|f\|_{\ihb},
\end{align}
which completes the proof of (I)$_{\rm 1}$ when $p\in[1,\infty)$.

Next we show (I)$_{\rm 1}$ when $p=\infty$. By \eqref{fiexp}, \eqref{fn},
Lemma 2.5(ii) and the H\"older inequality when $q\in[1,\infty]$,
or \eqref{r} when $q\in(0,1)$,
we know that, for any $m$, $n\in\nn$ with $m>n$, and $x\in\cx$,
\begin{align*}
|f_m(x)-f_n(x)|&=\left|\sum_{k=n+1}^m \widetilde{Q}_kQ_kf\right|
= \left|\sum_{k=n+1}^m \int_{\cx}\widetilde{Q}_k(x,y)Q_kf(y)\,d\mu(y)\right|\\
&\leq\sum_{k=n+1}^m\|Q_kf\|_{L^\infty(\cx)}\int_{\cx}\left|\widetilde{Q}_k(x,y)\right|\,d\mu(y)\\
&\lesssim \delta^{ns}\left[\sum_{k=n+1}^m\delta^{-ksq}\|Q_kf\|_{L^\infty(\cx)}^q\right]^{1/q}
\lesssim \delta^{ns}\|f\|_{\ihbi},
\end{align*}
which implies that $\|f_m-f_n\|_{L^\infty(\cx)}\to0$ as $m,\ n\to \infty$.
Similarly, we can show that, for any $n\in\nn$,  $\|f_n\|_{L^p(\cx)}\lesssim \|f\|_{\ihbi}$ and
hence $\{f_n\}_{n\in\nn}$ is a Cauchy sequence in $L^\infty(\cx)$.
Let $g:=\lim_{n\to\infty} f_n$ in $L^\infty(\cx)$. Using an  argument similar to
that used in the estimations of \eqref{feg1},
\eqref{lpb} and \eqref{glessf}, we conclude that
$$f=g \quad \text{in} \quad (\icgg)'$$
and
$$\|g\|_{\hbi}+\|g\|_{L^\infty(\cx)}\lesssim \|f\|_{\ihbi};$$
we omit the details. This finishes the proof of (I)$_1$ when $p=\infty$ and hence of (I)$_1$.

Now we prove  (I)$_{\rm 2}$. Let $f\in \hb\cap L^p(\cx)$.
Notice that, since $p\in[1,\infty]$, from the  H\"older inequality and
\cite[Proposition 2.7(iii)]{hmy08}, it follows that, for $N\in\nn$ as in Lemma \ref{icrf},
\begin{align}\label{iformer}
&\left\{\sum_{\alpha \in \ca_0}\sum_{m=1}^{N(0,\alpha)}
\mu(Q_\alpha^{0,m})\left[m_{Q_\alpha^{0,m}}(|P_0f|)\r]^p+
\sum_{k=1}^{N}\sum_{\alpha \in \ca_k}\sum_{m=1}^{N(k,\alpha)}
\mu\left(\qa\right)\left[m_{\qa}(|Q_{k}f|)\r]^p\right\}^{1/p}\\
&\quad\leq\left\{\sum_{\alpha \in \ca_0}\sum_{m=1}^{N(0,\alpha)}
\mu(Q_\alpha^{0,m})m_{Q_\alpha^{0,m}}(|P_0f|^p)+
\sum_{k=1}^{N}\sum_{\alpha \in \ca_k}\sum_{m=1}^{N(k,\alpha)}
\mu\left(\qa\right)m_{\qa}(|Q_{k}f|^p)\right\}^{1/p}\notag\\
&\quad=\left[\|P_0f\|_{L^p(\cx)}^p +\sum_{k=1}^N\|Q_kf\|_{L^p(\cx)}^p\right]^{1/p}
\lesssim \|f\|_{L^p(\cx)}.\notag
\end{align}
By this, we find that
\begin{align*}
\|f\|_{\ihb}&=\left\{\sum_{\alpha \in \ca_0}\sum_{m=1}^{N(0,\alpha)}
\mu(Q_\alpha^{0,m})\left[m_{Q_\alpha^{0,m}}(|P_0f|)\r]^p+
\sum_{k=1}^{N}\sum_{\alpha \in \ca_k}\sum_{m=1}^{N(k,\alpha)}
\mu\left(\qa\right)\left[m_{\qa}(|Q_{k}f|)\r]^p\right\}^{1/p}\\
&\qquad +\left[\sum_{k=N+1}^\infty\delta^{-ksq}\|Q_kf\|_{L^p(\cx)}^q\right]^{1/q}\\
&\lesssim \|f\|_{L^p(\cx)}+\|f\|_{\hb},
\end{align*}
which shows that (I)$_{\rm 2}$ holds true.

Next we prove (II)$_{1}$. For any $f\in\ihf$,
we have $f\in(\icgg)'$ with $s$, $\beta$ and $\gamma$ satisfying \eqref{10.19.3}.
By \cite[Proposition 2.7(ii)]{hmy08} and the H\"older inequality when $q\in[1,\infty]$,
or \eqref{r} when $q\in(p(s, \beta\wedge\gamma),1)$, we find that,
for any $m,\ n\in\nn$ with $m>n$,
\begin{align*}
\|f_m-f_n\|_{L^p(\cx)}&=\sup_{h\in L^{p'}(\cx),\ \|h\|_{L^{p'}(\cx)}\leq 1}|\langle f_m-f_n, h\rangle|=\sup_{h\in L^{p'}(\cx),\ \|h\|_{L^{p'}(\cx)}\leq 1} \left|\left\langle
\sum_{k=n+1}^m\widetilde{Q}_kQ_kf, h\right\rangle\right|\\
&= \sup_{h\in L^{p'}(\cx),\ \|h\|_{L^{p'}(\cx)}\leq 1}
\left|\sum_{k=n+1}^m\langle Q_k(f), \widetilde{Q}_k^\ast(h)\rangle\right|\\
&\leq \sup_{h\in L^{p'}(\cx),\ \|h\|_{L^{p'}(\cx)}\leq 1} \int_{\cx}
\sum_{k=n+1}^m|Q_k(f)(x)|\left|\widetilde{Q}_k^\ast(h)(x)\right|\,d\mu(x)\\
&\leq \delta^{ns}\sup_{h\in L^{p'}(\cx),\ \|h\|_{L^{p'}(\cx)}\leq 1}
\left\|\left\{\sum_{k=n+1}^m\delta^{-ksq}|Q_k(f)|^q\right\}^{1/q}\right\|_{L^p(\cx)}
\|M(h)\|_{L^{p'}(\cx)}\notag\\
&\lesssim \delta^{ns}\left\|\left\{\sum_{k=n+1}^m\delta^{-ksq}|Q_k(f)|^q\right\}^{1/q}\right\|_{L^p(\cx)}
\lesssim \delta^{ns}\|f\|_{\ihf},
\end{align*}
which implies that $\|f_m-f_n\|_{L^p(\cx)}\to0$ as $m,\ n\to \infty$.
Similarly, we can also show that, for any $n\in\nn$,  $\|f_n\|_{L^p(\cx)}\lesssim \|f\|_{\hf}$ and
hence $\{f_n\}_{n\in\nn}$ is a Cauchy sequence in $L^p(\cx)$.
By the completeness of $L^p(\cx)$ with $p\in[1,\infty)$, let $g:=\lim_{n\to\infty} f_n$ in $L^p(\cx)$. Similarly to the estimations of \eqref{feg1},
\eqref{p0} and \eqref{qkf}, we conclude that
$$f=g \quad \text{in} \quad (\icgg)'$$
and
\begin{equation}\label{gfihf}
\|g\|_{L^p(\cx)}\lesssim \|f\|_{\ihf}.
\end{equation}
Moreover, we obtain
\begin{align}\label{gjf}
\|g\|_{\hf}&=\left\|\left\{\sum_{k=-\infty}^\infty\delta^{-ksq}|Q_k(g)|^q\right\}^{1/q}\right\|_{L^p(\cx)}\\
&\lesssim \left\|\left\{\sum_{k=-\infty}^N\delta^{-ksq}|Q_k(g)|^q\right\}^{1/q}\right\|_{L^p(\cx)}
+\left\|\left\{\sum_{k=N+1}^\infty\delta^{-ksq}|Q_k(f)|^q\right\}^{1/q}\right\|_{L^p(\cx)}\notag\\
&\lesssim \left\|\left\{\sum_{k=-\infty}^N\delta^{-ksq}|Q_k(g)|^q\right\}^{1/q}\right\|_{L^p(\cx)}
+\|f\|_{\ihf}\notag\\
&=: \mathrm{J}+\|f\|_{\ihf}.\notag
\end{align}
When $p/q\in(0,1]$, by \eqref{r} and \cite[Proposition 2.7(iii)]{hmy08}, we have
\begin{equation}\label{j1}
\mathrm{J}^p\lesssim \sum_{k=-\infty}^N\delta^{-ksp}\|Q_kg\|_{L^p(\cx)}^p
\lesssim\|g\|_{L^p(\cx)}^p\sum_{k=-\infty}^N\delta^{-ksp}\lesssim\|g\|_{L^p(\cx)}^p,
\end{equation}
while when $p/q\in(1,\infty)$, by the Minkowski inequality, \cite[Proposition 2.7(iii)]{hmy08}, we  also have
\begin{equation}\label{j2}
\mathrm{J}\lesssim \left[\sum_{k=-\infty}^N\delta^{-ksq}\|Q_kg\|_{L^p(\cx)}^q\right]^{1/q}
\lesssim \|g\|_{L^p(\cx)}\left[\sum_{k=-\infty}^N\delta^{-ksq}\right]^{1/q}\lesssim  \|g\|_{L^p(\cx)}.
\end{equation}
Combining this with \eqref{gfihf}, \eqref{gjf} and  \eqref{j1}, we conclude that
$$\|g\|_{L^p(\cx)}+\|g\|_{\hf}\lesssim \|f\|_{\ihf},$$
which completes the proof of (II)$_{\rm 1}$.

Now we show (II)$_{\rm 2}$.  Let $f\in \hf\cap L^p(\cx)$.
Using an argument similar to that used in the estimation of  \eqref{iformer}, we find that,
for $N\in\nn$ as in Lemma \ref{icrf},
\begin{align*}
\|f\|_{\ihf}&=\left\{\sum_{\alpha \in \ca_0}\sum_{m=1}^{N(0,\alpha)}
\mu\left(Q_\alpha^{0,m}\right)\left[m_{Q_\alpha^{0,m}}(|P_0f|)\r]^p+
\sum_{k=1}^{N}\sum_{\alpha \in \ca_k}\sum_{m=1}^{N(k,\alpha)}
\mu\left(\qa\right)\left[m_{\qa}(|Q_{k}f|)\r]^p\right\}^{1/p}\\
&\qquad +\left\|\left[\sum_{k=N+1}^\infty \delta^{-ksq}|Q_k(f)|^q\right]^{1/q}\right\|_{\lp}\\
&\lesssim \|f\|_{L^p(\cx)}+\|f\|_{\hf},
\end{align*}
which implies that (II)$_{\rm 2}$ holds true.

Next we show (III)$_{\rm 1}$.  For any $f\in\ihfi$,
we have $f\in(\icgg)'$ with $s,\ \beta$ and $\gamma$ satisfying \eqref{12.5.4}.
Let $\{f_n\}_{n=1}^\fz$ be as \eqref{fn}. By \eqref{fiexp},  we know that, for any $m,\ n\in\nn$
with $m>n$ and $x\in\cx$,
\begin{align}\label{fmfni1}
|f_m(x)-f_n(x)|&=\left|\sum_{k=n+1}^m \widetilde{Q}_kQ_kf\right|\\
&\leq\sum_{k=n+1}^m \sum_{\alpha\in\ca_k}\sum_{m=1}^{N(k,\alpha)}
\int_{\qa}\left|\widetilde{Q}_k(x,y)\right||Q_kf(y)|d\mu(y)\notag\\
&\lesssim\sum_{k=n+1}^m \sum_{\alpha\in\ca_k}\sum_{m=1}^{N(k,\alpha)}
\mu\left(\qa\right)\left|\widetilde{Q}_k(x,\ya)\right|\sup_{y\in\qa}|Q_kf(y)|,\notag
\end{align}
where we chose $\ya\in\qa$ such that $\sup_{y\in\qa}|\widetilde{Q}_k(x,y)|\leq 2|\widetilde{Q}_k(x,\ya)|$.
Notice that Lemma \ref{12.3.2} implies that, for any $k\in\{N+1,N+2,\ldots\}$
with $N\in\nn$ as in Lemma \ref{icrf},
$\alpha\in\ca_k$
and $m\in\{1,\dots,N(k,\alpha)\}$,
$$\left\{\delta^{-ksq}\left[\sup_{y\in\qa}|Q_kf(y)|\right]^q\right\}^{1/q}\lesssim\|f\|_{\ihfi}.$$
From this, the H\"older inequality, \eqref{4.23x} and Lemma \ref{9.14.1},
we deduce that, for any $x\in\cx$,
\begin{align}\label{fmfni2}
&|f_m(x)-f_n(x)|\\
&\quad\lesssim \sum_{k=n+1}^m\left\{\sum_{\alpha\in\ca_k}
\sum_{m=1}^{N(k,\alpha)}\delta^{ksq'}\mu\left(\qa\right)\left|\widetilde{Q}_k(x,\ya)\right|\right\}^{1/q'}\notag\\
&\qquad\quad\times\left\{\sum_{\alpha\in\ca_k}\sum_{m=1}^{N(k,\alpha)}\delta^{-ksq}
\mu\left(\qa\right)\left|\widetilde{Q}_k\left(x,\ya\right)\right|
\left[\sup_{y\in\qa}|Q_kf(y)|\right]^q\right\}^{1/q}\notag\\
&\quad\lesssim \|f\|_{\ihfi} \sum_{k=n+1}^m\delta^{ks}\sum_{\alpha\in\ca_k}
\sum_{m=1}^{N(k,\alpha)}\mu\left(\qa\right)
\left|\widetilde{Q}_k\left(x,\ya\right)\right|\notag\\
&\quad\lesssim \|f\|_{\ihfi} \sum_{k=n+1}^m\delta^{ks}\sum_{\alpha\in\ca_k}\sum_{m=1}^{N(k,\alpha)}\mu\left(\qa\right)
\frac{1}{V_{\delta^k}(x)+V(x,\ya)}\left[\frac{\delta^k}{\delta^k+d(x,\ya)}\right]^\gamma\notag\\
&\quad\lesssim \delta^{ns}\|f\|_{\ihfi},\notag
\end{align}
which implies that $\|f_m-f_n\|_{L^\infty(\cx)}\lesssim
\delta^{ns}\|f\|_{\ihfi}\to 0$ as  $m,\ n\to \infty$.
Similarly, we can also show that, for any $n\in\nn$,  $\|f_n\|_{L^\fz(\cx)}\lesssim \|f\|_{\ihfi}$ and
hence $\{f_n\}_{n\in\nn}$ is a Cauchy sequence in $L^\infty(\cx)$.
By the completeness of $L^\infty(\cx)$, there exists a $g\in L^\infty(\cx)$ such that
$g=\lim_{n\to\infty} f_n$ in $L^\infty(\cx)$.
Using an  argument similar to that used in the estimation of \eqref{feg1}, we conclude that
$$f=g \quad \text{in} \quad (\icgg)'.$$
Moreover, we have
\begin{align}\label{g123}
\|g\|_{L^\infty(\cx)}&=\lim_{n\to\infty}\|f_n\|_{L^\infty(\cx)}=\lim_{n\to\infty}
\left\|\widetilde{Q}_0P_0f+\sum_{k=1}^n \widetilde{Q}_kQ_kf\right\|_{L^\infty(\cx)}\\
&\leq \|\widetilde{Q}_0P_0f\|_{L^\infty(\cx)}+\sum_{k=1}^N\left\|\widetilde{Q}_kQ_kf\right\|_{L^\infty(\cx)}
+\lim_{n\to\infty}\left\|\sum_{k=N+1}^n \widetilde{Q}_kQ_kf\right\|_{L^\infty(\cx)}\notag\\
&=:\mathrm{G_1}+\mathrm{G_2}+\mathrm{G_3}.\notag
\end{align}
From an argument similar to that used in the estimations of \eqref{fmfni1} and \eqref{fmfni2},
we deduce that
\begin{equation}\label{g3}
\mathrm{G_3}\lesssim \|f\|_{\ihfi}.
\end{equation}
By \eqref{4.23x}, Lemma \ref{9.14.1} and  Definition \ref{ihfi}, we know that, for any $x\in\cx$,
\begin{align*}
\left|\widetilde{Q}_0P_0f(x)\right|&\leq \sum_{\alpha\in\ca_0}\sum_{m=1}^{N(k,\alpha)}
\int_{\qo}\left|\widetilde{Q}_0(x,y)\right||P_0f(y)|d\mu(y)\\
&\lesssim \sum_{\alpha\in\ca_0}\sum_{m=1}^{N(k,\alpha)}\mu\left(\qo\right)
\sup_{y\in Q^{0,m}_\alpha}\left|\widetilde{Q}_0(x,y)\right|m_{\qo}(|P_0f|)\\
&\lesssim\|f\|_{\ihfi} \sum_{\alpha\in\ca_0}\sum_{m=1}^{N(0,\alpha)}
\mu\left(\qo\right)
\frac{1}{V_1(x)+V(x,y_\alpha^{0,m})}\left[\frac{1}{1+d(x,y_\alpha^{0,m})}\right]^\gamma\\
&\lesssim \|f\|_{\ihfi},
\end{align*}
which implies that
\begin{equation}\label{g1}
\mathrm{G_1}\lesssim \|f\|_{\ihfi}.
\end{equation}
Using an argument similar to that used in the estimation of \eqref{g1}, we find that
\begin{equation}\label{g2}
\mathrm{G_2}\lesssim \|f\|_{\ihfi}.
\end{equation}
By \eqref{g123} through \eqref{g2}, we know that
\begin{equation}\label{gi}
\|g\|_{L^\infty(\cx)}\lesssim\|f\|_{\ihfi}.
\end{equation}
Now we estimate $\|g\|_{\hfi}$. Indeed, by $f=g$ in $(\icgg)'$, we have
\begin{align}\label{ghfi}
\|g\|_{\hfi}&=\sup_{l \in \zz} \sup_{\alpha\in\ca_l}\left[\frac{1}{\mu(Q_\alpha^l)}
\int_{Q_\alpha^l}\sum_{k=l}^\infty\delta^{-ksq}|Q_kg(x)|^q\,d\mu(x)\right]^{1/q}\\
&\lesssim \sup_{l \in \nn} \sup_{\alpha\in\ca_l}\left[\frac{1}{\mu(Q_\alpha^l)}
\int_{Q_\alpha^l}\sum_{k=l}^\infty\delta^{-ksq}|Q_kg(x)|^q\,d\mu(x)\right]^{1/q}\notag\\
&\qquad +\sup_{l \in \zz\setminus \nn} \sup_{\alpha\in\ca_l}\left[\frac{1}{\mu(Q_\alpha^l)}
\int_{Q_\alpha^l}\sum_{k=l}^\infty\delta^{-ksq}|Q_kg(x)|^q\,d\mu(x)\right]^{1/q}\notag\\
&\lesssim \|f\|_{\ihfi} +\sup_{l \in \zz\setminus \nn} \sup_{\alpha\in\ca_l}\left[\frac{1}{\mu(Q_\alpha^l)}
\int_{Q_\alpha^l}\sum_{k=l}^0\delta^{-ksq}|Q_kg(x)|^q\,d\mu(x)\right]^{1/q}\notag\\
&\qquad +\sup_{l \in \zz\setminus \nn} \sup_{\alpha\in\ca_l}\left[\frac{1}{\mu(Q_\alpha^l)}
\int_{Q_\alpha^l}\sum_{k=1}^\infty\delta^{-ksq}|Q_kf(x)|^q\,d\mu(x)\right]^{1/q}\notag\\
&=: \|f\|_{\ihfi}+\mathrm{G_4}+\mathrm{G_5}.\notag
\end{align}
From \eqref{4.23x} and Lemma \ref{6.15.1}(ii), we deduce that, for any $k\in\zz$ and $x\in\cx$,
$$|Q_kg(x)|\lesssim \|g\|_{L^\infty(\cx)},$$
which, together with \eqref{gi}, implies that
\begin{equation}\label{g4}
\mathrm{G_4}\lesssim\|g\|_{L^\infty(\cx)}\sup_{l \in \zz\setminus \nn}
\sup_{\alpha\in\ca_l}\left[\frac{1}{\mu(Q_\alpha^l)}
\int_{Q_\alpha^l}\sum_{k=l}^0\delta^{-ksq}\,d\mu(x)\right]^{1/q}\lesssim\|f\|_{\ihfi}.
\end{equation}
To estimate $\mathrm{G_5}$, for any $l \in \zz\setminus \nn$ and $\alpha\in\ca_l$, let
$$I_l^\alpha:=\{\beta\in\ca_1:\ Q_{\beta}^1\subset Q_\alpha^l\}.$$
By Theorem \ref{10.22.1}(ii), it is easy to see
$\bigcup_{\beta\in I_l^\alpha} Q_{\beta}^1 =Q_\alpha^l$
and
$\sum_{\beta\in I_l^\alpha}\mu(Q_{\beta}^1)=\mu(Q_\alpha^l)$.
From these facts, we deduce that
\begin{align}\label{g5}
\mathrm{G_5}&\sim \sup_{l \in \zz\setminus \nn} \sup_{\alpha\in\ca_l}\left[\frac{1}{\mu(Q_\alpha^l)}
\sum_{\beta\in I_l^\alpha}\mu(Q_{\beta}^1)\frac{1}{\mu(Q_{\beta}^1)}
\int_{Q_\beta^1}\sum_{k=1}^\infty\delta^{-ksq}|Q_kf(x)|^q\,d\mu(x)\right]^{1/q}\\
&\lesssim \|f\|_{\ihfi}\sup_{l \in \zz\setminus \nn} \sup_{\alpha\in\ca_l}\left[\frac{1}{\mu(Q_\alpha^l)}
\sum_{\beta\in I_l^\alpha}\mu(Q_{\beta}^1)\right]^{1/q}\lesssim  \|f\|_{\ihfi}.\notag
\end{align}
By \eqref{gi} through \eqref{g5}, we obtain
$$\|g\|_{\hfi}+\|g\|_{L^\infty(\cx)}\lesssim \|f\|_{\ihfi},$$
which completes the proof of (III)$_{\rm 1}$.

Finally we show (III)$_{\rm 2}$. Let $f\in \hfi\cap L^\infty(\cx)$.
Notice that, by \eqref{4.23x} and Lemma \ref{6.15.1}(ii), for any $\alpha\in\ca_0$
and $m\in\{1,\dots,N(0,\alpha)\}$,
$$m_{\qo}(|P_0f|)\leq\|P_0f\|_{L^\infty(\cx)}\lesssim \|f\|_{L^\infty(\cx)}$$
and, for any $k\in\{1,\dots,N\}$, $\alpha\in\ca_k$ and $m\in\{1,\dots,N(k,\alpha)\}$,
$$m_{\qa}(|Q_kf|)\leq\|Q_kf\|_{L^\infty(\cx)}\lesssim \|f\|_{L^\infty(\cx)}.$$
From this, it follows that, for $N\in\nn$ as in Lemma \ref{icrf},
\begin{align*}
\|f\|_{\ihfi}&=\max\left\{\sup_{\alpha\in\ca_0}\sup_{m\in\{1,\dots,N(0,\alpha)\}}m_{\qo}(|P_0f|),
\sup_{k\in\{1,\dots,N\}}\sup_{\alpha \in \ca_k}\sup_{m\in\{1,\dots,N(k,\alpha)\}}m_{\qa}(|Q_kf|),\r.\\
&\qquad\left.\sup_{l \in \nn,\ l>N} \sup_{\alpha\in\ca_l}
\left[\frac{1}{\mu(Q_\alpha^l)}\int_{Q_\alpha^l}\sum_{k=l}^\infty
\delta^{-ksq}|Q_k(f)(x)|^q\,d\mu(x)\right]^{1/q}\right\}\\
&\lesssim \|f\|_{L^\infty(\cx)}+\|f\|_{\hfi},
\end{align*}
which implies that (III)$_2$ holds true.
This finishes the proof of Theorem \ref{ihlless}.
\end{proof}

\subsection{Relationships among inhomogeneous Besov spaces, Triebel--Lizorkin spaces and other function spaces}

In this section, we concentrate on relationships among  inhomogeneous Besov spaces,
Triebel--Lizorkin spaces and other function spaces.
We do not need to assume that $\mu(\cx)=\infty$ in this
section.

\begin{theorem}\label{soq2}
Let $\beta,\ \gamma\in(0,\eta)$ with $\eta$ as in Definition \ref{10.23.2}, and $p\in(1,\infty)$.
Then $F_{p,2}^0(\cx)=L^p(\cx)$ in the following sense:
\begin{enumerate}
\item[{\rm(i)}] there exists a positive constant $C$ such that, for any $f\in L^p(\cx)$,
$$\|f\|_{F_{p,2}^0(\cx)}\leq C\|f\|_{L^p(\cx)}.$$
\item[{\rm(ii)}] for any $f\in F_{p,2}^0(\cx)$, there exists a $g\in L^p(\cx)$ such that
$$f=g\quad\text{in}\quad (\icgg)' \quad\text{and}\quad \|g\|_{L^p(\cx)}\leq C\|f\|_{F_{p,2}^0(\cx)},$$
where $C$ is a positive constant independent of $f$.
\end{enumerate}
\end{theorem}

\begin{proof}
We first prove (i). Assume $f\in L^p(\cx)$ with $p\in(1,\infty)$,
and that $\{Q_k\}_{k=0}^\fz$ is an exp-IATI.
Notice that, since $p\in(1,\infty)$, from the H\"older inequality
and \cite[Proposition 2.7(iii)]{hmy08},
it follows that, for $N\in\nn$ as in Lemma \ref{icrf},
\begin{align*}
&\sum_{k=0}^N\sum_{\alpha \in \ca_k}\sum_{m=1}^{N(k,\alpha)}\mu\lf(Q_\az^{k,m}\r)
\left[m_{Q_\az^{k,m}}\left(|Q_k(f)|\right)\r]^p\\
&\quad=\sum_{k=0}^N\sum_{\alpha \in \ca_k}\sum_{m=1}^{N(k,\alpha)}
\frac{1}{[\mu(Q_\az^{k,m})]^{p-1}}\left[\int_{Q_\az^{k,m}}|Q_k(f)(y)|\,d\mu(y)\right]^{p}\\
&\quad\leq \sum_{k=0}^N\sum_{\alpha \in \ca_k}\sum_{m=1}^{N(k,\alpha)}
\frac{1}{[\mu(Q_\az^{k,m})]^{p-1}}\lf[\mu\lf(Q_\az^{k,m}\r)\r]^{p/p'}\int_{Q_\az^{k,m}}|Q_kf(y)|^p\,d\mu(y)\\
&\quad=\sum_{k=0}^N\|Q_k(f)\|_{L^p(\cx)}^p\lesssim \|f\|_{L^p(\cx)}^p.
\end{align*}
By this and an argument similar to
that used in the proof of Theorem \ref{l_btl}(i), we find that, for $N\in\nn$ as in Lemma \ref{icrf},
\begin{equation*}
\|f\|_{F_{p,2}^0(\cx)}=\left\{\sum_{k=0}^\fz\sum_{\alpha \in \ca_0}\sum_{m=1}^{N(0,\alpha)}\mu\left(\qa\right)
\left[m_{\qa}\left(|Q_kf|\right)\r]^p\r\}^{1/p}+\left\|\left[\sum_{k=N+1}^\infty |Q_kf|^2\right]^{1/2}\right\|_{\lp}
\ls\|f\|_{L^p(\cx)},
\end{equation*}
which prove (i).

Now we show (ii). By \eqref{fiexp}, \eqref{fn} and an argument similar to that used in the estimation of \eqref{f_n_lp},
we conclude that $\{f_n\}_{n\in\nn}$ is a Cauchy sequence in $L^p(\cx)$.
By the completeness of $L^p(\cx)$ with $p\in(1,\infty)$, let
$$g:=\lim_{n\to\infty} f_n\quad\text{in}\quad L^p(\cx).$$
Using an argument similar to that used in the estimations of \eqref{feg1},
\eqref{p0} and \eqref{qkf}, we find that
$f=g$ in $(\icgg)'$
and
$\|g\|_{L^p(\cx)}\lesssim \|f\|_{F_{p,2}^0(\cx)}$. This finishes the proof of Theorem \ref{soq2}.
\end{proof}

\begin{remark}
By the Littlewood--Paley $g$-function characterization of $h^p(\cx)$
in \cite[Theorem 5.7]{hyy19}, we know that,
for any given $p\in (\omega/(\omega+\eta),1]$,
$F_{p,2}^0(\cx)=h^p(\cx)$ with equivalent quasi-norms.
Therefore, Theorem \ref{soq2} further complements this
via showing that this relation also holds true for any given $p\in(1,\infty)$.
\end{remark}

Now we establish the relationship among $C^s(\cx)$,
$B^s_{\infty,\infty}(\cx)$ and $F^s_{\infty,\infty}(\cx)$.
By Proposition \ref{proihfi}(iii),  $B^s_{\infty,\infty}(\cx)=F^s_{\infty,\infty}(\cx)$ with equivalent norms.
The following theorem states the relationship between $C^s(\cx)$ and
$B^s_{\infty,\infty}(\cx)$.

\begin{theorem}\label{ic_b}
Let $\beta,\ \gamma\in(0,\eta)$ with $\eta$ as in Definition
\ref{10.23.2}, and $s\in(0,\beta\wedge\gamma)$.
Then $C^s(\cx)=B^s_{\infty,\infty}(\cx)$ in the following sense:
\begin{enumerate}
\item[{\rm(i)}] there exists a positive constant $C$ such that, for any $f\in C^s(\cx)$,
$$\|f\|_{B^s_{\infty,\infty}(\cx)}\leq C\|f\|_{C^s(\cx)};$$
\item[{\rm(ii)}] for any $f\in B^s_{\infty,\infty}(\cx)$,
there exists a $g\in C^s(\cx)$ such that
$$f=g\quad\text{in}\quad (\icgg)' \quad\text{and}\quad
\|g\|_{C^s(\cx)}\leq C\|f\|_{B^s_{\infty,\infty}(\cx)},$$
where $C$ is a positive constant independent of $f$.
\end{enumerate}
\end{theorem}

\begin{proof}
We first show (i). Assume $f \in C^s(\cx)$ with $s\in(0,\beta\wedge\gamma)$,
where $\beta$ and $\gamma$ are as in this theorem.
Let $\{Q_k\}_{k\in\nn}$ be an exp-IATI.
By Definition \ref{ih}(i), we know that
\begin{align*}
\|f\|_{B^s_{\infty,\infty}(\cx)}&=
\max\left\{
\sup_{k\in\{0,\dots,N\}}\sup_{\alpha\in\ca_k}\sup_{m\in\{1,\ldots,N(k,\alpha)\}}m_{\qa}(|Q_kf|),
\sup_{k\in\{N+1,N+2,\ldots\}}\delta^{-ks}\|Q_k(f)\|_{L^\infty(\cx)}\right\},
\end{align*}
where $N\in\nn$ is as in Lemma \ref{icrf}.
By an argument similar to that used in the proof of Theorem \ref{c_b}(i),
we conclude that
\begin{equation}\label{licdot}
\sup_{k\in\{N+1,N+2,\dots\}}\delta^{-ks}\|Q_k(f)\|_{L^\infty(\cx)}\lesssim\|f\|_{\dot{C}^s(\cx)}.
\end{equation}
Moreover, it is easy to see that
$$\sup_{k\in\{0,\dots,N\}}\sup_{\alpha\in\ca_k}
\sup_{m\in\{k,\dots,N(k,\alpha)\}}m_{\qa}(|Q_kf|)
\leq \|f\|_{L^\infty(\cx)},$$
which, combined with \eqref{licdot}, implies that
$$\|f\|_{B^s_{\infty,\infty}(\cx)}\lesssim\|f\|_{\dot{C}^s(\cx)}+\|f\|_{L^\infty(\cx)}\sim\|f\|_{C^s(\cx)}$$
and hence completes  the proof of (i).

We next prove (ii). Assume $f\in B^s_{\infty,\infty}(\cx)$.
Let  $\{Q_k\}_{k\in\nn}$ be an exp-IATI
and $\{f_n\}_{n\in\nn}$ as in \eqref{fn}.
By the proof of Theorem \ref{ihlless}(I), we know that
$\{f_n\}_{n\in\nn}$ is a Cauchy sequence in $L^\infty(\cx)$.
By the completeness of $L^\infty(\cx)$, letting $g:=\lim_{n\to\infty} f_n$ in $L^\infty(\cx)$,
then, also by  the proof of Theorem \ref{ihlless}(I),
we find that
$f=g$ in  $(\icgg)'$
and
$\|g\|_{L^\infty(\cx)}\lesssim \|f\|_{B^s_{\infty,\infty}(\cx)}.$
To finish the proof of (ii), it suffices to show that
\begin{equation}\label{gihb}
\|g\|_{\dot{C}^s(\cx)}\lesssim \|f\|_{B^s_{\infty,\infty}(\cx)}.
\end{equation}
Indeed, for any $x,\ y\in\cx$ with $d(x,y)\geq (2A_0)^{-1}\delta^N$,
where $N\in\nn$ is as in Lemma \ref{icrf} [see also \eqref{n}],
\begin{equation}\label{gihb1}
|g(x)-g(y)|\leq\|g\|_{L^\infty(\cx)}\lesssim \|g\|_{L^\infty(\cx)} [d(x,y)]^s.
\end{equation}
For any $x,\ y\in\cx$ with  $0<d(x,y)<(2A_0)^{-1}\delta^N$,
there exists a unique $l_0\in\nn$ with $l_0\geq N$
such that $(2A_0)^{-1}\delta^{l_0+1}\leq d(x,y) <(2A_0)^{-1}\delta^{l_0}$.
Then, for any $x$, $y\in\cx$, we write
\begin{align}\label{gihb2}
|g(x)-g(y)|&=\left|\sum_{k=0}^\infty \widetilde{Q}_kQ_kf(x)-
\sum_{k=0}^\infty \widetilde{Q}_kQ_kf(y)\right|\\
&\leq\left|\sum_{k=0}^{l_0} \widetilde{Q}_kQ_kf(x)
-\sum_{k=0}^{l_0} \widetilde{Q}_kQ_kf(y)\right|
+\left|\sum_{k=l_0+1}^{\infty} \widetilde{Q}_kQ_kf(x)\right|
+\left|\sum_{k=l_0+1}^{\infty} \widetilde{Q}_kQ_kf(y)\right|\notag\\
&=:g_1(x,y)+g_2(x)+g_3(y).\notag
\end{align}
For the term $g_1(x,y)$, by \eqref{4.23y}, $s\in(0,\beta\wedge\gamma)$
with $\beta$ and $\gamma$ as in this theorem, and Lemma \ref{6.15.1}(ii),
we conclude that
\begin{align*}
g_1(x,y)&=\left|\sum_{k=0}^{l_0}\int_{\cx}\left[\widetilde{Q}_k(x,z)-\widetilde{Q}_k(y,z)\right]Q_kf(z)\,d\mu(z)\right|\\
&\lesssim \sum_{k=0}^{l_0}\int_{\cx} \left[\frac{d(x,y)}{\delta^k+d(x,z)}\right]^\beta
\frac{\delta^k}{V_{\delta^k}(x)+V(x,z)}\left[\frac{\delta^k}{\delta^k+d(x,z)}\right]^\gamma|Q_kf(z)|\,d\mu(z)\\
&\lesssim\|f\|_{B^s_{\infty,\infty}(\cx)}[d(x,y)]^s\sum_{k=0}^{l_0}\int_{\cx}
\left[\frac{d(x,y)}{\delta^k+d(x,z)}\right]^{\beta-s}
\frac{\delta^k}{V_{\delta^k}(x)+V(x,z)}\left[\frac{\delta^k}{\delta^k+d(x,z)}\right]^{\gamma+s}\,d\mu(z)\\
&\lesssim \|f\|_{B^s_{\infty,\infty}(\cx)}[d(x,y)]^s
\sum_{k=0}^{l_0}\delta^{l_0(\beta-s)}\delta^{-k(\beta-s)}
\lesssim\|f\|_{B^s_{\infty,\infty}(\cx)}[d(x,y)]^s.
\end{align*}
For the term $g_2(x)$, from \eqref{4.23x},
Lemma \ref{6.15.1}(ii) and $\delta^{l_0}\lesssim d(x,y)$,
we deduce that
\begin{align*}
g_2(x)&\lesssim \sum_{k=l_0+1}^\infty\int_{\cx}\lf|\widetilde{Q}_0(x,z)\r||Q_kf(z)|\,d\mu(z)\\
&\lesssim \|f\|_{B^s_{\infty,\infty}(\cx)}\sum_{k=l_0+1}^\infty\delta^{ks}
\int_{\cx}\frac{1}{V_{\delta^k}(x)+V(x,z)}\left[\frac{\delta^k}{\delta^k+d(x,z)}\right]^\gamma\,d\mu(z)\\
&\lesssim \|f\|_{B^s_{\infty,\infty}(\cx)}\delta^{l_0s}\lesssim \|f\|_{B^s_{\infty,\infty}(\cx)}[d(x,y)]^s.
\end{align*}
Similarly, for the term $g_3(y)$, we obtain
$$g_3(y)\lesssim \|f\|_{B^s_{\infty,\infty}(\cx)}[d(x,y)]^s.$$
Combining these estimates,
we conclude that, for any $x,\ y\in\cx$,
$$|g(x)-g(y)|\lesssim \|f\|_{B^s_{\infty,\infty}(\cx)}[d(x,y)]^s,$$
which implies that \eqref{gihb} holds true
and hence completes the proof of Theorem \ref{ic_b}.
\end{proof}

At the end of this section, we establish the relationship
between $\bmo(\cx)$ and $F^0_{\infty,2}(\cx)$.
Let us begin with the notions of 1-exp-IATIs (see, for instance, \cite[Definition 3.1]{hyy19}),
$\bmo(\cx)$ (see, for instance, \cite[Defiinition 2.1]{dy12})
and the local Hardy space $h^p(\cx)$ (see, for instance, \cite[Section 3]{hyy19}).

\begin{definition}\label{1-exp-iati}
Let $\eta\in(0,1)$ be as in Definition \ref{10.23.2}.
A sequence $\{P_k\}_{k=0}^\infty$ of bounded linear integral operators
on $L^2(\cx)$ is called  an \emph{inhomogeneous approximation of the identity with exponential
decay and integration 1} (for short, 1-exp-IATI) if $\{P_k\}_{k=0}^\infty$ has the following properties:
\begin{enumerate}
\item[{\rm(i)}] for any $k\in\nn$, $P_k$
satisfies (ii) and (iii) of Definition \ref{10.23.2} but without the term
$$\exp\left\{-\nu\left[\max\{d(x,\cy^k),d(y,\cy^k)\}\right]^a\right\};$$
\item[{\rm(ii)}] for any $k\in\nn$ and $x\in\cx$,
$$\int_\cx P_k(x,y)\,d\mu(y)=1=\int_\cx P_k(y,x)\,d\mu(y);$$
\item[{\rm(iii)}] letting $Q_0:=P_0$ and, for any $k\in\zz_+$,
$Q_k:=P_k-P_{k-1}$, then $\{Q_k\}_{k\in\nn}$ is an exp-IATI.
\end{enumerate}
\end{definition}

\begin{definition}\label{h1bmo}
Let $\cx$ be a space of homogeneous type.
\begin{enumerate}
\item[{\rm(i)}]  The \emph{space $\bmo(\cx)$}  is defined by setting
$$\bmo(\cx):=\left\{f\in L^1_{\loc}(\cx):\ \|f\|_{\bmo(\cx)}:=\sup_{\text{ball}\ B\subset\cx}
\frac{1}{\mu(B)}\int_B|f(x)-c_B|\,d\mu(x)<\infty\right\},$$
where, for any ball $B\subset\cx$,
$$c_B:=\begin{cases}
m_B(f) &\text{if } r_B\leq R,\\
0 &\text{if } r_B>R
\end{cases}$$
with some fixed $R \in(0,\infty)$.
\item[{\rm(ii)}] Let $\{P_k\}_{k\in\zz}$ be a 1-exp-IATI.
The \emph{local radial maximal function} $\cm^+_0(f)$ of $f$ is defined by setting, for any $x\in\cx$,
$$\cm^+_0(f)(x):=\max\left\{\max_{k\in\{0,\dots,N\}}\left\{\sum_{\alpha\in\ca_k}
\sum_{m=1}^{N(k,\alpha)}\sup_{z\in\qa}|P_kf(z)|\mathbf{1}_{\qa}(x)\right\},
\sup_{k\in\{N+1,N+2,\dots\}}|P_kf(x)|\right\},$$
where $N\in\nn$ is as in Lemma \ref{icrf}.
For any $p\in(0,\infty)$, the \emph{local Hardy space} $h^p(\cx)$ is defined by setting
$$h^p(\cx):=\left\{f\in(\icgg)':\ \|f\|_{h^p(\cx)}:=\|\cm^+_0(f)\|_{L^p(\cx)}\right\}.$$
\end{enumerate}
\end{definition}

In \cite[Lemma 6.1]{dy12}, Dafni and Yue proved  that the space
$\bmo(\cx)$ is independent of the choice of $R$.
The following theorem establishes the relationship
between $\bmo(\cx)$ and $F^0_{\infty,2}(\cx)$.

\begin{theorem}\label{locbmo_f}
Let $\eta$ be as in Definition \ref{10.23.2}.
Then $\bmo(\cx)=F^0_{\infty,2}(\cx)$ in the following sense:
\begin{enumerate}
\item[{\rm(i)}] there exists a positive constant $C$ such that, for any $f\in \bmo(\cx)$,
$$\|f\|_{F^0_{\infty,2}(\cx)}\leq C\|f\|_{\bmo(\cx)};$$
\item[{\rm(ii)}] for any given $f\in F^0_{\infty,2}(\cx)$,
the linear functional
$$L_f :\  g \mapsto L_f(g):= \langle f, g\rangle,$$
initially defined for any $g\in\cg(\eta,\eta)$, has a bounded linear extension on $h^1(\cx)$.
Moreover, there exists a positive constant $C$ such that, for any $f\in F^0_{\infty,2}(\cx)$,
$$\|L_f\|_{(h^1(\cx))^\ast}\leq C\|f\|_{F^0_{\infty,2}(\cx)}.$$
\end{enumerate}
\end{theorem}

To prove Theorem \ref{locbmo_f}, we need a technical lemma.
Let us begin with the following notion.
\begin{definition}
Let $\{Q_k\}_{k\in\zz_+}$ be an exp-IATI. For any $l\in\zz_+$ and $x\in\cx$, let
$E^l(x)$ be as in \eqref{6.8x}.
Let $s\in\rr$ and $q\in(0,\infty]$.
The \emph{inhomogeneous maximal function} $\mathcal M^s_q(f)$ for any $f\in (\icgg)'$ is defined by setting,
for any $x\in\cx$,
\begin{equation}\label{im_q_s}
\cm_q^s(f)(x)
:=\sup_{l\in\zz_+}\left[\frac{1}{\mu(E^l(x))}
\int_{E^l(x)}\sum_{k=l}^\infty\delta^{-ksq}|Q_k(f)(y)|^q\,d\mu(y)\right]^{1/q}
\end{equation}
with the usual modification made when $q=\infty$.
\end{definition}

The following lemma  was obtained in \cite[Proposition 6.27]{hmy08} on RD-spaces,
but the same proof still works for  spaces of homogeneous type;
we omit the details.

\begin{lemma}\label{e_imqs}
Let $\beta,\ \gamma\in(0,\eta)$ with $\eta$ as in Definition \ref{10.23.2}, and
$s\in(-(\beta\wedge\gamma),\beta\wedge\gamma)$.
Then $f\in F^s_{\infty,q}(\cx)$ if and only if $f\in(\icgg)'$
and $\cm_q^s(f)\in L^\infty(\cx)$.
Moreover, there exists a positive constant $C$  such that, for any $f\in\ihfi$,
$$C^{-1}\|f\|_{F^s_{\infty,q}(\cx)}
\leq \|\cm_q^s(f)\|_{L^\infty(\cx)}
\leq C\|f\|_{F^s_{\infty,q}(\cx)}.$$
\end{lemma}

Now we show Theorem \ref{locbmo_f}.

\begin{proof}[Proof of Theorem \ref{locbmo_f}]
We first show (i). Assume $f\in\bmo(\cx)$ and
$\beta,\ \gamma\in(0,\eta)$ with $\eta$ as in Definition \ref{10.23.2}.
Let $\{Q_k\}_{k=0}^\infty$ be an exp-IATI.
By $\cg_0^\eta(\beta,\gamma)\subset h^1(\cx)$
and $(h^1(\cx))^\ast=\bmo(\cx)$, we know that $f\in(\icgg)'$.
Notice that, by the definition of $\qa$, we have, for any $k\in\zz_+$,
$\alpha\in\ca_k$ and $m\in\{1,\dots, N(k,\alpha)\}$, $\mu(\qa)\sim\mu(Q_\alpha^k)$,
which, together with the H\"older inequality, implies that
\begin{align*}
m_{\qa}(|Q_kf|)&=\frac{1}{\mu(\qa)}\int_{\qa}|Q_kf(x)|\,d\mu(x)
\lesssim \frac{1}{\mu(Q_\alpha^k)}\int_{Q_\alpha^k}|Q_kf(x)|\,d\mu(x)\\
&\lesssim \left[\frac{1}{\mu(Q_\alpha^k)}\int_{Q_\alpha^k}|Q_kf(x)|^2\,d\mu(x)\right]^{1/2}\\
&\lesssim \left[\frac{1}{\mu(Q_\alpha^k)}\int_{Q_\alpha^k}
\sum_{l=k}^\infty|Q_lf(x)|^2\,d\mu(x)\right]^{1/2}
\end{align*}
and hence
\begin{equation}\label{equa_ifi}
\|f\|_{F^0_{\infty,2}(\cx)}\lesssim\sup_{l\in\zz_+}\sup_{\alpha\in\ca_l}
\left[\frac{1}{\mu(Q_\alpha^k)}\int_{Q_\alpha^k}\sum_{k=l}^\infty|Q_kf(x)|^2\,d\mu(x)\right]^{1/2}.
\end{equation}
Fix $l\in\zz_+$ and $\alpha\in \ca_l$. Define $B_\alpha^l:=B(z_\alpha^l, 3A_0^2C_0\delta^l)$ and write
$$f=\left[f-m_{B_\alpha^l}(f)\right]\mathbf{1}_{B_\alpha^l}
+\left[f-m_{B_\alpha^l}(f)\right]\mathbf{1}_{\cx\setminus B_\alpha^l}
+m_{B_\alpha^l}(f)
=:f_{\alpha,l}^{(1)}+f_{\alpha,l}^{(2)}+f_{\alpha,l}^{(3)}.$$
Choose $R:=6A_0^2C_0$
in  Definition \ref{h1bmo}(i). Then, using an  argument similar to that
used in the proof of Theorem \ref{bmo_f},
we conclude that, for any $l\in\zz_+$ and $\alpha\in\ca_l$,
\begin{equation}\label{fbmoi}
\left\{\frac{1}{\mu(Q_\alpha^l)}\int_{Q_\alpha^l}
\sum_{k=l}^\infty \left|Q_k(f_{\alpha,l}^{(1)})(x)\right|^2\,d\mu(x)\right\}^{1/2}
\lesssim \|f\|_{\bmo(\cx)}
\end{equation}
and
\begin{equation}\label{fbmoii}
\left\{\frac{1}{\mu(Q_\alpha^l)}\int_{Q_\alpha^l}
\sum_{k=l}^\infty \left|Q_k(f_{\alpha,l}^{(2)})(x)\right|^2\,d\mu(x)\right\}^{1/2}
\lesssim \|f\|_{\bmo(\cx)}.
\end{equation}
Choosing $R\in(3A_0^2C_0\delta, 3A_0^2C_0)$, we then find  that,  for $l=0$,
\begin{equation}\label{fbmoiii}
\left\{\frac{1}{\mu(Q_\alpha^0)}\int_{Q_\alpha^0}
\sum_{k=0}^\infty \left|Q_k(f_{\alpha,l}^{(3)})(x)\right|^2\,d\mu(x)\right\}^{1/2}
\lesssim \left|f_{\alpha,l}^{(3)}\right|\lesssim \|f\|_{\bmo(\cx)}
\end{equation}
and, for any $l\in\nn$,
\begin{equation}\label{fbmoiv}
\left\{\frac{1}{\mu(Q_\alpha^l)}\int_{Q_\alpha^l}
\sum_{k=l}^\infty \left|Q_k(f_{\alpha,l}^{(3)})(x)\right|^2\,d\mu(x)\right\}^{1/2}
=0\leq\|f\|_{\bmo(\cx)}.
\end{equation}
By \eqref{equa_ifi} through \eqref{fbmoiv},
we conclude that $f\in F^0_{\infty,2}(\cx)$
and $\|f\|_{F^0_{\infty,2}(\cx)}\lesssim\|f\|_{\bmo(\cx)}$.
Thus, we complete the proof of (i).

Using an argument similar to that used in the proof of Theorem \ref{bmo_f}(ii),
we then obtain (ii) and we omit the details here,
which completes the proof of Theorem \ref{locbmo_f}.
\end{proof}

\section{Boundedness of Calder\'on--Zygmund operators}\label{s5}

In this section, we concentrate on the boundedness of
Calder\'on--Zygmund operators on Besov and Triebel--Lizorkin spaces.
Let us recall the notions of $C^s_b(\cx)$ and
$\mathring{C}^s_b(\cx)$.
For any $s\in (0,1]$, let $C^s(\cx)$  be as in Definition \ref{c_s} and define
$$C^s_b(\cx):=\left\{f\in C^s(\cx):\ f\ \text{has bounded support}\right\}$$
and
$$\mathring{C}^s_b(\cx):=\left\{f\in C^s_b(\cx):\ \int_{\cx}f(x)\,d\mu(x)=0\right\},$$
both equipped with the strict inductive limit topology induced by $\|\cdot\|_{C^s(\cx)}$ (see, for instance, \cite[p.\,273]{ms79b}).
Let $(C^s_b(\cx))'$ [resp., $(\mathring{C}^s_b(\cx))'$]
be the set of all continuous linear functionals on $C^s_b(\cx)$ [resp., $\mathring{C}^s_b(\cx)$],
equipped with the weak-$\ast$ topology.

Now we recall the notion of normalized bump functions on $\cx$ (see, for instance,
\cite[p.\,546]{ns} and \cite[p.\,126]{hmy08}).

\begin{definition}
Let $\epsilon\in(0,1]$. A function $\varphi$ on $\cx$ is called an \emph{$\epsilon$-bump function associated
to a ball $B(x_0, r)$}, for some $x_0\in\cx$ and $r\in(0,\infty)$, if $\varphi$ is supported in $B(x_0, r)$ and there
exists a positive constant $C$, independent of $\varphi$, such that, for any $s\in(0,\epsilon]$,
\begin{equation}\label{7.0x}
\|\varphi\|_{\dot{C}^s(\cx)}\leq Cr^{-s}.
\end{equation}
The minimal positive constant $C$ in \eqref{7.0x} is called
the \emph{$\epsilon$-bump constant} of $\varphi$ and, if $C=1$,
then $\varphi$ is called a  \emph{normalized $\epsilon$-bump function}.
\end{definition}

We now introduce the notion of Calder\'on--Zygmund operators of order $\epsilon$ on a space
$\cx$ of homogeneous type for any given $\ez\in(0,1]$.

\begin{definition}\label{4.11.2}
Let $\epsilon\in(0,1]$. A linear operator $T$, which is continuous from $C_b^{\epsilon'}(\cx)$ to
$(C_b^{\epsilon'}(\cx))'$ for any $\epsilon'\in(0,\epsilon)$, is called a \emph{Calder\'on--Zygmund operators of order
$\epsilon$} if $T$ has a distributional kernel $K$,  which is locally integrable away
from the diagonal of $\cx\times\cx$, and satisfies the following conditions:
\begin{enumerate}
\item[{\rm(i)}] if $\phi_1,\ \phi_2\in C_b^{\epsilon'}(\cx)$ for some  $\epsilon'\in(0,\epsilon)$
have disjoint supports, then
$$\langle T\phi_1,\phi_2\rangle=\int_{\cx}\int_{\cx}K(x,y)\phi_1(y)\phi_2(x)\,d\mu(x)\,d\mu(y);$$
\item[{\rm(ii)}] if $\phi$ is a normalized $\epsilon$-bump function associated to a ball with radius $r$,
then there exists a positive constant $C$, independent of $\phi$, such that, for any
$\epsilon'\in(0,\epsilon)$, $\|T\phi\|_{\dot{C}^{\epsilon'}(\cx)}\leq Cr^{-\epsilon'}$.
\item[{\rm(iii)}] there exists a positive constant $C$ such that, for any $x,\ y\in\cx$ with $x\neq y$,
$$|K(x,y)|\leq C\frac{1}{V(x,y)}$$
and, for any $x,\ x',\ y\in\cx$ with $d(x,x')\leq d(x,y)/(2A_0)$ and $x\neq y$,
$$|K(x,y)-K(x',y)|\leq C\frac{[d(x,x')]^\epsilon}{V(x,y)[d(x,y)]^\epsilon};$$
\item[{\rm(iv)}] conditions (i), (ii) and (iii) also hold true with $x$ and $y$ interchanged, that is, these
conditions also hold true for the adjoint operator $T^\ast$ defined by setting, for any
$\ez'\in(0,\ez)$ and $\phi,\ \psi\in C^{\ez'}_b(\cx)$,
$\langle T^\ast \phi,\psi\rangle :=\langle T\psi,\phi\rangle$.
\end{enumerate}
\end{definition}

We have the following conclusion of Calder\'on--Zygmund operators of order $\eta$.

\begin{theorem}\label{12.15.1}
Let $\beta,\ \gamma\in(0,\eta)$ with $\eta$ as in Definition \ref{10.23.2},
$s\in(-(\beta\wedge\gamma), \beta\wedge\gamma)$,
$p(s,\beta\wedge\gamma)$ be as in \eqref{pseta}
and $T$ a Calder\'on--Zygmund
operator of order $\eta$.
\begin{enumerate}
\item[{\rm(i)}] If $p\in(p(s,\bz\wedge\gz),\fz]$, $\bz$ and $\gz$ satisfy \eqref{6.14.1} and
\eqref{6.14.1x}, and $q\in(0,\fz]$, then $T$ is bounded on $\dot B^s_{p,q}(\cx)$ when
$\max\{p,q\}<\fz$, and from $\hb\cap\mathring{\cg}(\eta,\eta)$ to $\hb$ when $\max\{p,q\}=\infty$, where
$\hb$ is viewed as a space of $(\cggi)'$.

\item[{\rm(ii)}] If $p\in(p(s,\bz\wedge\gz),\fz)$, $\bz$ and  $\gz$ satisfy \eqref{6.14.1} and \eqref{6.14.1x}, and
$q\in(p(s,\bz\wedge\gz),\fz)]$, then $T$ is bounded on $\dot F^s_{p,q}(\cx)$ when
$q\neq\fz$, and from $\dot F^s_{p,\fz}(\cx)\cap\mathring{\cg}(\eta,\eta)$ to $\dot F^s_{p,\fz}(\cx)$,
where $\hf$ is viewed as a space of $(\cggi)'$.
\end{enumerate}
\end{theorem}

To prove Theorem \ref{12.15.1}, we need the following several technical lemmas.

\begin{lemma}\label{12.15.2}
Let $\beta,\ \gamma\in(0,\eta)$ with $\eta$ as in Definition \ref{10.23.2},
and $s\in(-(\beta\wedge\gamma), \beta\wedge\gamma)$.
Let $\hb$ and $\hf$, as subspaces of $(\mathring{\mathcal{G}}^\eta_0(\beta,\gamma))'$
with $\beta$ and $\gamma$ satisfying \eqref{6.14.1}
via replacing $\widetilde{\beta}$ and $\widetilde{\gamma}$
respectively by $\beta$ and $\gamma$, be as in Definition \ref{h}. Let
$p(s,\beta\wedge\gamma)$ be as in \eqref{pseta}.
Then $\mathring{\cg}(\eta,\eta)$ is dense
in $\hb$ when $p\in(p(s,\beta\wedge\gamma),\infty)$ and $q\in(0,\infty)$,
and in $\hf$ when $p,\ q\in(p(s,\beta\wedge\gamma),\infty)$,
where $p(s,\beta\wedge\gamma)$ is as in \eqref{pseta}.
\end{lemma}

\begin{proof}
Let $\{Q_k\}_{k\in\zz}$ be an exp-AT{\rm I} and $f\in\hb$ with $s,\ p$ and $q$ as in this lemma.
Then, for any $k\in\zz$, $\alpha\in\ca_k$ and $m\in\{1,\dots,N(k,\alpha)\}$,
suppose that $\ya$ is an arbitrary point in $\qa$. Then, by Lemma \ref{crf},
we know that there exists a sequence $\{\widetilde{Q}_k\}_{k=-\infty}^\infty$
of bounded linear integral operators on $L^2(\mathcal{X})$ such that
$$f(\cdot) = \sum_{k=-\infty}^\infty\sum_{\alpha \in \ca_k}
\sum_{m=1}^{N(k,\alpha)}\mu\left(\qa\right)\widetilde{Q}_k(\cdot,\ya)Q_k(f)\left(\ya\right) \qquad\text{in}\quad (\cggi)'.$$
For any $k\in\zz$, since $\ca_k$ is at
most countable, it follows that there exists a sequence of finite sets, $\{\ca_k^n\}_{n\in\nn}$,
satisfying that $\ca_k^n\subset\ca_k^{n+1}$ for any $n\in\nn$,
and $\bigcup_{n\in\nn}\ca_k^n=\ca_k$. For any $n\in\nn$ and $x\in\cx$, define
$$f_n(x) := \sum_{|k|<n}\sum_{\alpha \in \ca_k^n}\sum_{m=1}^{N(k,\alpha)}
\mu\left(\qa\right)\widetilde{Q}_k(x,\ya)Q_k(f)\left(\ya\right).$$
By Lemma \ref{crf}, it is easy to show $f_n\in\mathring{\cg}(\eta,\eta)$.
Moreover, we have
\begin{align*}
f(\cdot)-f_n(\cdot)&=\left[\sum_{|k|\geq n}\sum_{\alpha \in \ca_k}
\sum_{m=1}^{N(k,\alpha)}+\sum_{|k|<n}\sum_{\alpha \notin \ca_k^n}
\sum_{m=1}^{N(k,\alpha)}\right]\mu\left(\qa\right)\widetilde{Q}_k(\cdot,\ya)Q_k(f)\left(\ya\right).
\end{align*}
From this and an argument similar to that used in
the estimations of \eqref{es-ex}, \eqref{6.2.1} and \eqref{6.2.2}, we deduce that
\begin{align*}
\|f-f_n\|_{\hb}&\lesssim \left\{\sum_{|k|\geq n}\delta^{-ksq}\left[
\sum_{\alpha \in \ca_k}\sum_{m=1}^{N(k,\alpha)}\mu\left(\qa\right)
\left|Q_k(f)\left(\ya\right)\right|^p\right]^{q/p}\right.\\
&\qquad+\left.\sum_{|k|<n}\delta^{-ksq}\left[\sum_{\alpha \notin \ca_k^n}
\sum_{m=1}^{N(k,\alpha)}\mu\left(\qa\right)
\left|Q_k(f)\left(\ya\right)\right|^p\right]^{q/p}\right\}^{1/q}\\
&\lesssim \|f\|_{\hb}.
\end{align*}
By this, $p,\ q\in(0,\fz)$ and the dominated convergence theorem, we conclude that
$\|f-f_n\|_{\hb}\to 0$ as $n\to \infty$. Thus, $\mathring{\cg}(\eta,\eta)$ is dense in $\hb$.
The proof that
$\mathring{\cg}(\eta,\eta)$ is dense in $\hf$ is similar and we omit the details.
This finishes the proof of Lemma \ref{12.15.2}.
\end{proof}

\begin{lemma}\label{4.8.1}
For any $k\in \zz$, there exists a non-negative measurable function
$\Phi_k:\ \cx\times\cx\to[0,\infty)$ such that
\begin{enumerate}
\item[{\rm(i)}] for any $x\in\cx$,
$\supp(\Phi_k(x,\cdot)) \subset B(x,C_1\delta^k)$;
\item[{\rm(ii)}] for any $x,\ y_1,\ y_2\in\cx$,
$$|\Phi_k(x,y_1)-\Phi_k(x,y_2)|\leq C_2\left[\frac{d(y_1,y_2)}{\delta^k}\right]^\eta;$$
\item[{\rm(iii)}] for any $x\in\cx$ and $y\in B(x,\delta^k)$,
$\Phi_k(x,y)=1$,
\end{enumerate}
where $C_1$ and  $C_2$ are positive constants independent of $k$, $x$, $y_1$ and $y_2$.
Moreover, for any $k\in\zz$,
\begin{equation}\label{4.8.9}
\|\Phi_k\|_{L^\infty(\cx\times\cx)} = 1.
\end{equation}
\end{lemma}

\begin{proof}
For any $k\in\zz$ and $\alpha\in\ca_k$, by  \cite[Corollary 4.2]{ah13}, we know that
there exists a function $\varphi^k_\alpha:\ \cx\to \rr$ satisfying
$\mathbf{1}_{B(z_\alpha^k, 4A_0^2C_0\delta^k)}\leq \varphi^k_\alpha
\leq\mathbf{1}_{B(z_\alpha^k, 4A_0^2C_0\delta^{k-1})}$ and,
for any $x,\ y\in \cx$,
\begin{equation}\label{4.8.2}
|\varphi^k_\alpha(x)-\varphi^k_\alpha(y)|\lesssim \left[\frac{d(x,y)}{\delta^k}\right]^\eta,
\end{equation}
where the implicit positive constant is independent of $k$ and $\alpha$.
For any $k\in\zz$ and $x,\ y\in \cx$, define
$$\widetilde{\Phi}_k(x,y):=\sum_{\alpha\in\ca_k}\varphi^k_\alpha(x)\varphi^k_\alpha(y).$$
We claim that there exist positive constants $C_1$ and $C_2$ such that,
for any $k\in\zz$ and $x\in\cx$,
\begin{equation}\label{4.8.3}
\supp\lf(\widetilde{\Phi}_k(x,\cdot)\r)\subset B(x,C_1\delta^k)
\end{equation}
and, for any $x,\ y_1,\ y_2\in\cx$,
\begin{equation}\label{4.8.4}
\left|\widetilde{\Phi}_k(x,y_1)-\widetilde{\Phi}_k(x,y_2)\right|\leq C_2
\left[\frac{d(y_1,y_2)}{\delta^k}\right]^\eta.
\end{equation}
To this end, fix $k\in\zz$ and $x\in\cx$. By (i) and (iii) of
Theorem \ref{10.22.1}, we know that there exists a unique
$\alpha_0\in\ca_k$ such that $x\in Q_{\alpha_0}^k
\subset B(z_{\alpha_0}^k, 2A_0C_0\delta^k)\subset B(z_{\alpha_0}^k, 4A_0^2C_0\delta^{k-1})$.
Let
$$\ca_k(x):=\{\alpha\in\ca_k:\ x\in B(z_\alpha^k,4A_0^2C_0\delta^{k-1})
\cap B(z_{\alpha_0}^k, 4A_0^2C_0\delta^{k-1})\}.$$
Notice that, if $\alpha\in\ca_k\setminus\ca_k(x)$, then $\varphi_\alpha^k(x)=0$. Therefore,
\begin{equation}\label{4.8.7}
\widetilde{\Phi}_k(x,y)=\sum_{\alpha\in\ca_k(x)}\varphi^k_\alpha(x)\varphi^k_\alpha(y).
\end{equation}
For any $\alpha\in\ca_k(x)$, we have $x\in B(z_\alpha^k,4A_0^2C_0\delta^{k-1})
\cap B(z_{\alpha_0}^k, 4A_0^2C_0\delta^{k-1})$.
By this, we conclude that, for any
$z \in B(z_\alpha^k,4A_0^2C_0\delta^{k-1})$,
\begin{align}\label{4.8.5}
d(z,z_{\alpha_0}^k)&\leq A_0[d(z,z_\alpha^k)+d(z_\alpha^k,z_{\alpha_0}^k)]
\leq4A_0^3C_0\delta^{k-1}+A_0^2[d(z_\alpha^k,x)+d(x,z_{\alpha_0}^k)]\\
&\leq(4A_0^3C_0+8A_0^4C_0)\delta^{k-1},\notag
\end{align}
which implies that
$$\bigcup_{\alpha\in\ca_k(x)}B(z_\alpha^k,4A_0^2C_0\delta^{k-1})
\subset B(z_{\alpha_0}^k, (4A_0^3C_0+8A_0^4C_0)\delta^{k-1}).$$
From this, Lemma \ref{11.14.2} and the fact that $(4A_0^3C_0+8A_0^4C_0)/
(4A_0^2C_0)=A_0+2A_0^2>1$, we deduce
that there exists a positive constant $M_0$, independent of $x$ and $k$, such that $\#\ca_k(x)\leq M_0$.
By this, \eqref{4.8.7} and \eqref{4.8.2}, we obtain, for any $y_1,\ y_2\in\cx$,
\begin{align}\label{4.8.6}
\left|\widetilde{\Phi}_k(x,y_1)-\widetilde{\Phi}_k(x,y_2)\right|
&=\left|\sum_{\alpha\in\ca_k(x)}\varphi^k_\alpha(x)\varphi^k_\alpha(y_1)
-\sum_{\alpha\in\ca_k(x)}\varphi^k_\alpha(x)\varphi^k_\alpha(y_2)\right|\\
&\leq \sum_{\alpha\in\ca_k(x)}|\varphi^k_\alpha(y_1)-\varphi^k_\alpha(y_2)|
\lesssim \left[\frac{d(y_1,y_2)}{\delta^k}\right]^\eta,\notag
\end{align}
where the implicit positive constant is
independent of $k$, $x$, $y_1$ and $y_2$.

Moreover, if $\widetilde{\Phi}_k(x,y)\neq 0$, then, by \eqref{4.8.7}, we know that there exists an
$\alpha\in\ca_k(x)$ such that $\varphi_\alpha^k(y)\neq 0$, which further implies that
$y\in B(z_\alpha^k,4A_0C_0\delta^k)$. From this, together with \eqref{4.8.5},
it follows that
$$d(y,x)\leq A_0[d(y,z_{\alpha_0}^k)+d(z_{\alpha_0}^k,x)]
\leq (4A_0^4C_0+8A_0^5C_0+4A_0^2C_0)\delta^{k-1},
$$
which implies that
\begin{equation}\label{4.8.8}
\supp\lf(\widetilde{\Phi}_k(x,\cdot)\r)\subset B(x,C_1\delta^k),
\end{equation}
where $C_1:=(4A_0^3C_0+8A_0^4C_0+4A_0C_0)\delta^{-1}$.

For any $k\in\zz$ and $x,\ y\in\cx$, define
$$\Phi_k(x,y):=\min\left\{\widetilde{\Phi}_k(x,y),1\right\}.$$
By \eqref{4.8.6} and \eqref{4.8.8}, we know that $\Phi_k$ satisfies (i), (ii) and \eqref{4.8.9}.

Now we show that $\Phi_k$ satisfies (iii). Indeed, for any
fixed $k\in\zz$ and $x\in\cx$,
let $\alpha_0\in\ca_k$ be such that
$x\in Q_{\alpha_0}^k$ as in the above. Then,
for any $y\in B(x,\delta^k)$, we have
$$d(y,z_{\alpha_0}^k)\leq A_0[d(y,x)+d(x,z_{\alpha_0}^k)]
\leq A_0\delta^k+2A_0^2C_0\delta^k<4A_0^2C_0\delta^k,$$
which implies that, for any $y \in B(x,\delta^k)$, $\varphi_{\alpha_0}^k(y)=1$. Notice that
$\varphi_{\alpha_0}^k(x)=1$ and hence
$\widetilde{\Phi}_k(x,y)\geq\varphi_{\alpha_0}^k(x)\varphi_{\alpha_0}^k(y)=1$,
which further implies that $\Phi_k(x,y)=1$.
This finishes the proof of Lemma \ref{4.8.1}.
\end{proof}

Using Lemma \ref{4.8.1} and some ideas similar to those used in the proof of \cite[Lemma 3.5.1]{ns},
we can prove the following lemma.

\begin{lemma}\label{4.9.1}
Let $\theta\in(0,\infty)$, $\eta\in(0,1)$ be as in Definition \ref{10.23.2}, and
$\{Q_k\}_{k\in\zz}$ an {\rm $\exp$-ATI}. Then, for any $k\in\zz$, there exists a sequence
$\{Q_{k,j}\}_{j=0}^\infty$ of functions on $\cx\times\cx$ such that
\begin{enumerate}
\item[{\rm(i)}]  for any $x\in\cx$,
\begin{equation}\label{5.8x}
Q_k(x,\cdot)=\sum_{j=0}^\infty \delta^{j\theta}Q_{k,j}(x,\cdot)
\end{equation}
both pointwisely and in $ C^{\eta'}(\cx)$ with any given $\eta'\in(0,\eta]$;
\item[{\rm(ii)}] for any $k\in\zz$, $j\in\zz_+$ and $\eta'$ as in (i), $Q_{k,j}(x,\cdot)$
is an  $\eta'$-bump function associated to the  ball $B(x,C_1\delta^{k-j})$,
with its $\eta'$-bump constant independent of $k$ and $j$,
where $C_1$ is as in Lemma \ref{4.8.1}(i). Moreover, for any $k\in\zz$, $j\in\zz_+$ and $x\in\cx$,
\begin{equation}\label{van-qkj}
\int_\cx Q_{k,j}(x,y)\,d\mu(y)=0.
\end{equation}
\end{enumerate}
\end{lemma}

\begin{proof}
Let $\{\Phi_k\}_{k\in\zz}$ be as in Lemma \ref{4.8.1}. For any $k\in\zz$ and $x,\ y\in\cx$, define
$$A_{k,0}(x,y):=\Phi_k(x,y)Q_k(x,y)$$
and, for any $k\in\zz$, $j\in\nn$ and $x,\ y\in\cx$, define $$A_{k,j}(x,y):=[\Phi_{k-j}(x,y)-\Phi_{k-j+1}(x,y)]Q_k(x,y).$$
Notice that, for any given $x$, $y\in\cx$, there exists a $j_0\in\zz_+$ such that
$d(x,y)<\delta^{k-j_0}$. Then, by Lemma \ref{4.8.1}(iii), we know that, for any $j\in
\{j_0,j_0+1,\dots\}$, $\Phi_{k-j}(x,y)=1$,
which implies that, for any $j\in\{ j_0+1,
j_0+2,\dots\}$,
$$A_{k,j}(x,y)=0.$$
By this, we conclude that, for any $x,\ y\in\cx$,
\begin{equation}\label{4.9.2}
\sum_{j=0}^\infty A_{k,j}(x,y)=\sum_{j=0}^{j_0} A_{k,j}(x,y)=\Phi_{k-j_0}(x,y)Q_k(x,y)=Q_k(x,y).
\end{equation}

We claim that, for any fixed $\theta\in(0,\infty)$,  and any $k\in\zz$,
$j\in\zz_+$ and $x,\ y\in\cx$,
\begin{equation}\label{4.9.4}
|A_{k,j}(x,y)|\lesssim \delta^{j\theta}\frac{1}{V_{\delta^{k-j}}(x)}.
\end{equation}
Indeed, if $j=0$, \eqref{4.9.4} is obviously true by
\eqref{4.8.9} and the size condition of
$Q_k$.
Now we assume $j\in\nn$. From the size condition of $Q_k$, we deduce that, for any
fixed $\Gamma \in (0,\infty)$, and any $k\in\zz$ and $x,\ y\in\cx$,
$$|Q_k(x,y)|\lesssim\frac{1}{V_{\delta^k}(x)}\left[\frac{\delta^k}{\delta^k+d(x,y)}\right]^\Gamma.$$
By Lemma \ref{4.8.1}(iii), we find that, for any $j\in\zz_+$ and $y\in B(x, \delta^{k-j})$,
$\Phi_{k-j}(x,y)=1$, which implies that, for any $j\in\nn$
and $y\in B(x, \delta^{k-j+1})$, $A_{k,j}(x,y)=0$.
Moreover, for any $y\in\cx$ satisfying $d(x,y) \geq \delta^{k-j+1}$, we have
\begin{equation}\label{4.9.3}
\frac{\delta^k}{\delta^k+d(x,y)}\lesssim \delta^{j}.
\end{equation}
Let $\Gamma :=\omega+\theta$. Then,  by \eqref{4.8.9}, \eqref{4.9.3}
and  $V_{\delta^{k-j}}(x)\lesssim\delta^{-j\omega}V_{\delta^k}(x)$,  we obtain
\begin{equation}\label{4.10.7}
|A_{k,j}(x,y)|\lesssim |Q_k(x,y)|\lesssim\delta^{j\theta}\frac{1}{V_{\delta^{k-j}}(x)},
\end{equation}
which shows \eqref{4.9.4}.

For any $k\in\zz$, $j\in\zz_+$ and $x\in\cx$, define
$$a_{k,j}(x):=\int_{\cx}A_{k,j}(x,y)d\mu(y).$$
Notice that, by Lemma \ref{4.8.1}(i), for any $k\in\zz$, $j\in\zz_+$ and $x\in\cx$, we have
\begin{equation}\label{4.9.5}
\supp(A_{k,j}(x,\cdot))\subset B(x,C_1\delta^{k-j}),
\end{equation}
where $C_1$ is as in Lemma \ref{4.8.1}(i).
From this and \eqref{4.9.4}, we deduce that, for any $k\in\zz$ and $x\in\cx$,
\begin{equation}\label{4.10.2}
\int_{B(x,C_1\delta^{k-j})}|A_{k,j}(x,y)|\,d\mu(y)\lesssim \delta^{j\theta}\frac{1}{V_{\delta^{k-j}}(x)}\mu(B(x,C_1\delta^{k-j}))\lesssim \delta^{j\theta}
\end{equation}
and hence
\begin{equation*}
\sum_{j=0}^\infty\int_{B(x,C_1\delta^{k-j})}
|A_{k,j}(x,y)|\,d\mu(y)\lesssim \sum_{j=0}^\infty\delta^{j\theta}
\frac{1}{V_{\delta^{k-j}}(x)}\mu(B(x,C_1\delta^{k-j}))\lesssim 1.
\end{equation*}
Thus, by this,  the Fubini theorem, \eqref{4.9.2} and the cancellation of $Q_k$,
we know that, for any $k\in\zz$ and $x\in\cx$,
\begin{equation}\label{4.9.6}
\sum_{j=0}^\infty a_{k,j}(x)
=\int_{\cx}\sum_{j=0}^\infty A_{k,j}(x,y)\,d\mu(y)
=\int_{\cx}Q_k(x,y)d\mu(y)=0.
\end{equation}
For any $k\in\zz$, $j\in\zz_+$ and $x\in\cx$, define
$$s_{k,j}(x):=\sum_{l=0}^j a_{k,l}(x).$$
From \eqref{4.9.6}, it follows that
$s_{k,j}=-\sum_{l=j+1}^\infty a_{k,l}$.
Then, for any $k\in\zz$, $j\in\zz_+$ and $x,\ y\in\cx$, define
$$\phi_{k,j}(x,y):=\Phi_{k-j}(x,y)\left[\int_{\cx}\Phi_{k-j}(x,z)d\mu(z)\right]^{-1},$$
$$\widetilde{A}_{k,j}(x,y):=A_{k,j}(x,y)-a_{k,j}(x)
\phi_{k,j}(x,y)+s_{k,j}(x)[\phi_{k,j}(x,y)-\phi_{k,j+1}(x,y)]$$
and
$$Q_{k,j}(x,y):=\delta^{-j\theta}\widetilde{A}_{k,j}(x,y).$$
By the definition of $\phi_{k,j}$, it is easy to see
that $\int_\cx\phi_{k,j}(x,y)\,d\mu(y)=1$. From this, we deduce that \eqref{van-qkj} holds true.
Moreover, observe that, by \eqref{4.9.6}, for any $k\in\zz$ and $x,\ y\in\cx$,
\begin{align*}
&\sum_{j=0}^\infty s_{k,j}(x)[\phi_{k,j}(x,y)-\phi_{k,j+1}(x,y)]\\
&\quad=\sum_{j=0}^\infty\sum_{l\leq j} a_{k,l}(x)\phi_{k,j}(x,y)
+\sum_{j=0}^\infty\sum_{l> j} a_{k,l}(x)\phi_{k,j+1}(x,y)\\
&\quad=\sum_{l=0}^\fz\sum_{j\geq l} a_{k,l}(x)\phi_{k,j}(x,y)+\sum_{l=1}^\infty
\sum_{j<l} a_{k,l}(x)\phi_{k,j+1}(x,y)\\
&\quad=\sum_{l=1}^\infty a_{k,l}(x)\left[\sum_{j\geq l}
\phi_{k,j}(x,y)+\sum_{j< l}\phi_{k,j+1}(x,y)\right]+a_{k,0}(x)\sum_{j=0}^\infty \phi_{k,j}(x,y)\\
&\quad=\sum_{l=1}^\infty a_{k,l}(x)\left[\sum_{j=1}^\infty
\phi_{k,j}(x,y)+\phi_{k,l}(x,y)\right]+a_{k,0}(x)\sum_{j=1}^\infty
\phi_{k,j}(x,y)+a_{k,0}(x)\phi_{k,0}(x,y)\\
&\quad=\sum_{l=0}^\infty a_{k,l}(x)\phi_{k,l}(x,y)
+\left[\sum_{l=0}^\infty a_{k,l}(x)\right]\left[\sum_{j=1}^\infty \phi_{k,j}(x,y)\right]
=\sum_{l=0}^\infty a_{k,l}(x)\phi_{k,l}(x,y),
\end{align*}
which, combined with \eqref{4.9.2}, further implies that
\begin{align*}
\sum_{j=0}^\infty\delta^{j\tz}Q_{k,j}(x,y)&=\sum_{j=0}^\infty \widetilde{A}_{k,j}(x,y)\\
&=\sum_{j=0}^\infty A_{k,j}(x,y)-\sum_{j=0}^\infty a_{k,j}(x)\phi_{k,j}(x,y)
+\sum_{j=0}^\infty s_{k,j}(x)[\phi_{k,j}(x,y)-\phi_{k,j+1}(x,y)]\\
&=Q_k(x,y).
\end{align*}
This shows that \eqref{5.8x} holds true pointwisely.

Now we prove that $\{Q_{k,j}\}_{j=0}^\infty$ satisfies (ii).
Indeed, by Lemma \ref{4.8.1}(i) and \eqref{4.10.2},  we easily know that,
for any $k\in\zz$, $j\in\zz_+$ and $x\in\cx$,
$$\supp (Q_{k,j}(x,\cdot)) \subset B(x, C_1\delta^{k-j}).$$
We claim that, for any
$\eta'\in(0,\eta]$, $k\in\zz$, $j\in\zz_+$ and $x\in\cx$,
\begin{equation}\label{4.10.1}
\|Q_{k,j}(x,\cdot)\|_{\dot{C}^{\eta'}(\cx)}\lesssim \delta^{-\eta'(k-j)}\frac{1}{V_{\delta^{k-j}}(x)}.
\end{equation}
Indeed, from the definition of $s_{k,j}$ and \eqref{4.10.2}, we deduce that, for any $k\in\zz$, $j\in\zz_+$ and
$x\in\cx$,
\begin{equation}\label{4.10.4}
|s_{k,j}(x)|\lesssim \delta^{j\theta}.
\end{equation}
By the definition of $\phi_{k,j}$ and Lemma \ref{4.8.1}(iii), we know that,
for any $k\in\zz$, $j\in\zz_+$ and $x,\ y\in\cx$,
$$|\phi_{k,j}(x,y)|\leq \left[\int_{B(x,\delta^{k-j})}1\,d\mu(y)\right]^{-1}= \frac{1}{V_{\delta^{k-j}}(x)}.$$
Thus, we conclude that, for any $k\in\zz$, $j\in\zz_+$ and $x,\ y\in\cx$,
\begin{equation}\label{sizeqkj}
|Q_{k,j}(x,y)|\lesssim \frac{1}{V_{\delta^{k-j}}(x)}.
\end{equation}
From this, we deduce that, for any $x,\ y_1,\ y_2\in\cx$ with $d(y_1,y_2)\geq\delta^{k-j}$,
\begin{equation}\label{4.10.3}
\frac{|Q_{k,j}(x,y_1)-Q_{k,j}(x,y_2)|}{[d(y_1,y_2)]^{\eta'}}\lesssim \delta^{-\eta'(k-j)}\frac{1}{V_{\delta^{k-j}}(x)}.
\end{equation}
Moreover, by (ii) and (iii) of Lemma \ref{4.8.1}, we know that,
for any $k\in\zz$, $j\in\zz_+$ and  $x,\ y_1,\ y_2\in\cx$ with $d(y_1,y_2)<\delta^{k-j}$,
\begin{align}\label{4.10.5}
|\phi_{k,j}(x,y_1)-\phi_{k,j}(x,y_2)|&=|\Phi_{k-j}(x,y_1)-\Phi_{k-j}(x,y_2)|\left[\int_{\cx}\Phi_{k-j}(x,z)\,
d\mu(z)\right]^{-1}\\
&\lesssim \left[\frac{d(y_1,y_2)}{\delta^{k-j}}\right]^\eta \frac{1}{V_{\delta^{k-j}}(x)}.\notag
\end{align}
Now, we estimate the regularity of $\{A_{k,j}\}_{j\in\zz_+}$.
When $j=0$, from \eqref{10.23.3}, \eqref{10.23.4}, \eqref{4.8.9} and Lemma \ref{4.8.1}(ii),
we deduce  that, for any $k\in\zz$ and $x,\ y_1,\ y_2\in\cx$,
\begin{align}\label{4.10.9}
|A_{k,0}(x,y_1)-A_{k,0}(x,y_2)|&=|\Phi_k(x,y_1)Q_k(x,y_1)-\Phi_k(x,y_1)Q_k(x,y_1)|\\
&=|\Phi_k(x,y_1)-\Phi_k(x,y_2)||Q_k(x,y_1)|\notag\\
&\qquad+|\Phi_k(x,y_2)||Q_k(x,y_1)-Q_k(x,y_2)|\notag\\
&\lesssim  \left[\frac{d(y_1,y_2)}{\delta^k}\right]^\eta \frac{1}{V_{\delta^{k}}(x)}.\notag
\end{align}

When $j\in\nn$, we consider four cases.

{\it Case 1) $d(x,y_1) < (2A_0)^{-1}\delta^{k-j+1}$ and $d(x,y_2) < \delta^{k-j+1}$.}
In this case, by Lemma \ref{4.8.1}(iii), we conclude that
\begin{equation}\label{4.10.11}
|A_{k,j}(x,y_1)-A_{k,j}(x,y_2)|=0.
\end{equation}

{\it Case 2) $d(x,y_1) < (2A_0)^{-1}\delta^{k-j+1}$ and $d(x,y_2) \geq \delta^{k-j+1}$.}
In this case, we have $A_{k,j}(x,y_1)=0$ and
$$d(y_1,y_2)\geq A_0^{-1}[d(x,y_2)-A_0d(x,y_1)]\gtrsim \delta^{k-j}.$$
From this, Lemma \ref{4.8.1}(iii) and \eqref{4.10.7}, we deduce that
\begin{align}\label{4.10.6}
|A_{k,j}(x,y_1)-A_{k,j}(x,y_2)|&=|A_{k,j}(x,y_2)|
=|[\Phi_{k-j}(x,y_2)-\Phi_{k-j+1}(x,y_2)]Q_k(x,y_2)|\\
&\lesssim \delta^{j\theta}\frac{1}{V_{\delta^{k-j}}(x)}
\lesssim\left[\frac{d(y_1,y_2)}{\delta^{k-j}}\right]^\eta\delta^{j\theta}\frac{1}{V_{\delta^{k-j}}(x)}\notag.
\end{align}

{\it Case 3) $d(x,y_1) \geq (2A_0)^{-1}\delta^{k-j+1}$ and $d(x,y_2) < (4A_0^2)^{-1}\delta^{k-j+1}$.}
In this case, we conclude that $A_{k,j}(x,y_2)=0$.
Using an argument similar to that used in the estimation of \eqref{4.10.6},
we find that
\begin{align}\label{4.10.8}
|A_{k,j}(x,y_1)-A_{k,j}(x,y_2)|&=|A_{k,j}(x,y_1)|
=|[\Phi_{k-j}(x,y_1)-\Phi_{k-j+1}(x,y_1)]Q_k(x,y_1)|\\
&\lesssim \delta^{j\theta}\frac{1}{V_{\delta^{k-j}}(x)}
\lesssim\left[\frac{d(y_1,y_2)}{\delta^{k-j}}\right]^\eta\delta^{j\theta}
\frac{1}{V_{\delta^{k-j}}(x)}\notag.
\end{align}

{\it Case 4) $d(x,y_1) \geq (2A_0)^{-1}\delta^{k-j+1}$ and $d(x,y_2) \geq (4A_0^2)^{-1}\delta^{k-j+1}$.}
In this case, we have
\begin{align*}
&|A_{k,j}(x,y_1)-A_{k,j}(x,y_2)|\\
&\qquad=|[\Phi_{k-j}(x,y_1)-\Phi_{k-j+1}(x,y_1)]Q_k(x,y_1)\\
&\qquad\qquad-[\Phi_{k-j}(x,y_2)-\Phi_{k-j+1}(x,y_2)]Q_k(x,y_2)|\\
&\qquad\leq |\Phi_{k-j}(x,y_1)Q_k(x,y_1)-\Phi_{k-j}(x,y_2)Q_k(x,y_2)|\\
&\qquad\qquad+|\Phi_{k-j+1}(x,y_1)Q_k(x,y_1)-\Phi_{k-j+1}(x,y_2)Q_k(x,y_2)|,
\end{align*}
which, together  with \eqref{4.10.7},
(ii) and (iii) of Lemma \ref{4.8.1},
and an argument similar to that used in the estimation of  \eqref{4.10.9},
further implies that
\begin{equation}\label{4.10.10}
|A_{k,j}(x,y_1)-A_{k,j}(x,y_2)|\lesssim\left[\frac{d(y_1,y_2)}
{\delta^{k-j}}\right]^\eta\delta^{j\theta}\frac{1}{V_{\delta^{k-j}}(x)}.
\end{equation}

From \eqref{4.10.9} through \eqref{4.10.10}, we deduce that,
for any $k\in\zz$, $j\in\zz_+$
and $x,\ y_1,\ y_2\in\cx$ with $d(y_1,y_2)<\delta^{k-j}$,
\begin{equation}\label{4.25.4}
|A_{k,j}(x,y_1)-A_{k,j}(x,y_2)|\lesssim\left[\frac{d(y_1,y_2)}{\delta^{k-j}}\right]^\eta
\delta^{j\theta}\frac{1}{V_{\delta^{k-j}}(x)},
\end{equation}
which, combined with \eqref{4.10.2}, \eqref{4.10.4} and \eqref{4.10.5}, implies that,
for any $k\in\zz$, $j\in\zz_+$,  $x,\ y_1,\ y_2\in\cx$ with $0<d(y_1,y_2)<\delta^{k-j}$,
\begin{align}\label{4.11.1}
\frac{|Q_{k,j}(x,y_1)-Q_{k,j}(x,y_2)|}{[d(y_1,y_2)]^{\eta'}}
&\leq \delta^{-j\theta}[d(y_1,y_2)]^{\eta-\eta'}\frac{1}{[d(y_1,y_2)]^{\eta}}[|A_{k,j}(x,y_1)-A_{k,j}(x,y_2)|\\
&\qquad+|a_{k,j}(x)||\phi_{k,j}(x,y_1)-\phi_{k,j}(x,y_2)|\notag\\
&\qquad+|s_{k,j}(x)||\phi_{k,j}(x,y_1)-\phi_{k,j}(x,y_2)|\notag\\
&\qquad+|s_{k,j}(x)||\phi_{k,j+1}(x,y_1)-\phi_{k,j+1}(x,y_2)|]\notag\\
&\lesssim \delta^{-\eta'(k-j)}\frac{1}{V_{\delta^{k-j}}(x)}.\notag
\end{align}
Combining this with \eqref{sizeqkj}, we know that \eqref{4.10.1} holds true.

Finally, by \eqref{4.10.1}, we find that, for any fixed $\eta'\in(0,\eta]$ and any $N\in\nn$,
\begin{align}\label{ceta}
\left\|Q_k(x,\cdot)-\sum_{j=0}^N \delta^{j\theta}Q_{k,j}(x,\cdot)\right\|_{\dot{C}^{\eta'}(\cx)}
&=\left\|\sum_{j=N+1}^\infty \delta^{j\theta}Q_{k,j}(x,\cdot)\right\|_{\dot{C}^{\eta'}(\cx)}\\
&\lesssim\sum_{j=N+1}^\infty\delta^{j\theta} \delta^{-\eta'(k-j)}\frac{1}{V_{\delta^{k-j}}(x)}\notag\\
&\lesssim \delta^{-k\eta'}\frac{1}{V_{\delta^{k}}(x)}\dz^{(\eta'+\theta)N}\to 0\notag
\end{align}
as $N\to\infty$ and, by \eqref{sizeqkj},
\begin{align}\label{linf}
\left\|Q_k(x,\cdot)-\sum_{j=0}^N \delta^{j\theta}Q_{k,j}(x,\cdot)\right\|_{L^\infty(\cx)}
&=\left\|\sum_{j=N+1}^\infty \delta^{j\theta}Q_{k,j}(x,\cdot)\right\|_{L^\infty(\cx)}\\
&\lesssim\sum_{j=N+1}^\infty\delta^{j\theta} \frac{1}{V_{\delta^{k-j}}(x)}
\lesssim \delta^{j\theta N}\frac{1}{V_{\delta^{k}}(x)}\to 0\notag
\end{align}
as $N\to\fz$, which implies that \eqref{5.8x} holds true in $ C^{\eta'}(\cx)$. This finishes the proof
of Lemma \ref{4.9.1}.
\end{proof}

\begin{remark}\label{exoft}
Let $T$ be a Calder\'on--Zygmund  operator of order $\eta$ with $\eta$
as in Definition \ref{10.23.2}. By Lemma \ref{4.9.1} and Definition \ref{4.11.2}(ii),
we find that, for any $\eta'\in(0,\eta)$, $k\in\zz$,
$j\in\zz_+$ and $x\in\cx$, $T(Q_{k,j}(x,\cdot))\in \dot{C}^{\eta'}(\cx)$.
Then, from Lemma \ref{4.9.1} and Definition \ref{4.11.2}(ii), it follows that,
for any $k\in\zz$, $j\in\zz_+$, $\varphi\in \mathring{\cg}(\beta,\gamma)$
with $\beta,\ \gamma\in(0,\eta)$, and $x\in\cx$,
\begin{align}\label{tqp}
|\langle T(Q_{k,j}(x,\cdot)),\varphi\rangle|
&=\left|\int_{\cx}T(Q_{k,j}(x,\cdot))(y)\varphi(y)\,d\mu(y)\right|\\
&=\left|\int_{\cx}[T(Q_{k,j}(x,\cdot))(y)-T(Q_{k,j}(x,\cdot))(x_0)]\varphi(y)\,d\mu(y)\right|\notag\\
&\lesssim \|T(Q_{k,j}(x,\cdot))\|_{\dot{C}^{\eta'}(\cx)}
\|\varphi\|_{\mathring{\cg}(\beta,\gamma)}\notag\\
&\qquad\times\int_\cx [d(x_0,y)]^{\eta'}\frac{1}{V_1(x_0)+V(x_0,y)}
\left[\frac{1}{1+d(x_0,y)}\right]^\gamma\,d\mu(y)\notag\\
&\lesssim \|T(Q_{k,j}(x,\cdot))\|_{\dot{C}^{\eta'}(\cx)}
\|\varphi\|_{\mathring{\cg}(\beta,\gamma)}
\lesssim \delta^{-\eta'(k-j)}\frac{1}{V_{\delta^k}(x)}\|\varphi\|_{\mathring{\cg}(\beta,\gamma)},\notag
\end{align}
where $x_0$ is the same as in $\mathring{\cg}(\beta,\gamma)$ and
we chose $\eta'\in(0,\gamma)$.
For any $N\in\nn$, let
$$\widetilde{Q}_{k,N}:=\sum_{j=0}^N\delta^{j\theta}Q_{k,j}.$$
Then, for any $M,\  N\in\nn$ with $M < N$, we have
\begin{align*}
\left|\left\langle T\left(\widetilde{Q}_{k,M}(x,\cdot)\right)
-T\left(\widetilde{Q}_{k,N}(x,\cdot)\right),\varphi\right\rangle\right|
&\leq\sum_{j=M+1}^N|\langle T(Q_{k,j}(x,\cdot)),\varphi\rangle|\\
&\lesssim \sum_{j=M+1}^N\|T(Q_{k,j}(x,\cdot))\|_{\dot{C}^{\eta'}(\cx)}
\|\varphi\|_{\mathring{\cg}(\beta,\gamma)}\\
&\lesssim \delta^{-\eta'k}\frac{1}{V_{\delta^k}(x)}\|\varphi\|_{\mathring{\cg}(\beta,\gamma)}
\sum_{j=M+1}^\infty\delta^{\eta'j}\\
&\lesssim \delta^{-\eta'k}\delta^{\eta'M}\frac{1}{V_{\delta^k}(x)}\|\varphi\|_{\mathring{\cg}(\beta,\gamma)},
\end{align*}
which implies that $\{\langle T(\widetilde{Q}_{k,N}(x,\cdot)),\varphi\rangle\}_{N=0}^\infty$ is a Cauchy
sequence of $\cc$. Thus,
$$
\lim_{N\to\infty}\lf\langle T\lf(\widetilde{Q}_{k,N}(x,\cdot)\r),\varphi\r\rangle
$$
exists. Define
\begin{equation}\label{exp-t}
\langle T(Q_k(x,\cdot)),\varphi\rangle :=\lim_{N\to\infty}\left\langle
T\left(\widetilde{Q}_{k,N}(x,\cdot)\right),\varphi\right\rangle.\end{equation}
It is easy to see that $\langle T(Q_k(x,\cdot)),\varphi\rangle$ is independent of the choice of
$\{Q_{k,j}\}_{j=0}^\infty$. By \eqref{tqp}, we conclude that
$T(Q_k(x,\cdot))\in (\mathring{\cg}(\beta,\gamma))'$ with $\beta,\ \gamma\in(0,\eta)$.
\end{remark}

In \cite[Proposition 2.12]{hmy08}, Han et al.\ introduced an expansion $\widetilde{T}$ of $T$
from $C^\eta(\cx)$ to $(\mathring{C}_b^\eta(\cx))'$.
Namely, for any $\varphi\in\mathring{C}_b^\eta(\cx)$,
suppose $\supp \varphi \subset B(x_0,r)$  for some $x_0\in\cx$ and $r\in(0,\infty)$.
Choose $\psi\in C_b^\eta(\cx)$ such that $\mathbf{1}_{B(x_0,2A_0r)}
\leq \psi\leq\mathbf{1}_{B(x_0,4A_0^2r)}$.
Define
\begin{align}\label{exp-t2}
\left\langle \widetilde{T}(Q_k(x,\cdot)),\varphi\right\rangle &:=\langle T(\psi Q_k(x,\cdot)),\varphi\rangle \\
&\qquad + \int_\cx\int_\cx[K(w,y)-K(x_0,y)][1-\psi(y)]Q_k(x,y)\varphi(w)\,d\mu(w)\,d\mu(y).\notag
\end{align}
In \cite[Proposition 2.12]{hmy08}, it was also proved that the expansion \eqref{exp-t2} is independent of the choice of
$\psi$. We point out that, for any  $\varphi\in\mathring{C}_b^\eta(\cx)$,
\eqref{exp-t} coincides with \eqref{exp-t2}, which is proved  in the following proposition.

\begin{proposition}\label{tttp}
Let $\eta$ be as in Definition \ref{10.23.2},
$\{Q_k\}_{k\in\zz}$ an {\rm exp-ATI} and $\varphi\in\mathring{C}_b^\eta(\cx)$ with
$\supp\vz\subset B(x_0,r)$ for some $x_0\in\cx$ and $r\in(0,\fz)$. For any $x\in\cx$, suppose that
$\langle T(Q_k(x,\cdot)),\varphi\rangle$
and $\langle \widetilde{T}(Q_k(x,\cdot)),\varphi\rangle$ are,
respectively, as in \eqref{exp-t} and \eqref{exp-t2}. Then
\begin{equation}\label{ttt}
\left\langle \widetilde{T}(Q_k(x,\cdot)),\varphi\right\rangle=\langle T(Q_k(x,\cdot)),\varphi\rangle.
\end{equation}
\end{proposition}

\begin{proof}
Let $\eta$ be as in Definition \ref{10.23.2} and
$\varphi\in\mathring{C}_b^\eta(\cx)$ with
$\supp\vz\subset B(x_0,r)$ for some $x_0\in\cx$ and $r\in(0,\fz)$.
By \cite[Corollary 4.2]{ah13}, we know that, for any $N\in\nn$,
there exists a $\psi_N\in C_b^\eta(\cx)$ such that
$$\mathbf{1}_{B(x_0,C_1[\delta^{k-N}+d(x_0,x)])}
\leq \psi_N\leq\mathbf{1}_{B(x_0,2A_0C_1[\delta^{k-N}+d(x_0,x)])},$$
where $C_1$ is as in Lemma \ref{4.8.1}(i) and, for any $y_1,\ y_2\in\cx$,
$$|\psi_N(y_1)-\psi_N(y_2)|\lesssim
\left[\frac{d(y_1,y_2)}{\delta^{k-N}+d(x_0,x)}\right]^\eta,$$
which implies that $\sup_{N\in\zz_+}\|\psi_N\|_{C^\eta(\cx)}\lesssim[d(x_0,x)]^{-\eta}$.
Choose $N_0\in\nn$ sufficiently large such that $C_1[\delta^{k-N_0}+d(x_0,x)]>2A_0r$. Then, by
\cite[Proposition 2.12]{hmy08},
we conclude that, for any $N\in\{N_0,N_0+1,\ldots\}$ and $x\in\cx$,
\begin{align*}
\left\langle \widetilde{T}(Q_k(x,\cdot)),\varphi\right\rangle &=\langle T(\psi_N Q_k(x,\cdot)),\varphi\rangle \\
&\qquad + \int_\cx\int_\cx[K(w,y)-K(x_0,y)][1-\psi_N(y)]Q_k(x,y)\varphi(w)\,d\mu(w)d\mu(y).
\end{align*}
For any $k\in\zz$ and $N\in\nn$, let $\widetilde{Q}_{k,N}$ be as in Remark \ref{exoft}.
From this, we deduce that, for any $k\in\zz$ and $x\in\cx$,
\begin{align*}
\left\langle \widetilde{T}(Q_k(x,\cdot)),\varphi\right\rangle &=\lim_{N\to\infty}
\langle T(\psi_N Q_k(x,\cdot)),\varphi\rangle \\
&\qquad +\lim_{N\to\infty} \int_\cx\int_\cx[K(w,y)-K(x_0,y)][1-\psi_N(y)]Q_k(x,y)
\varphi(w)\,d\mu(w)d\mu(y).\\
&= \lim_{N\to\infty}\langle T(\psi_N[Q_k(x,\cdot)-\widetilde{Q}_{k,N}(x,\cdot)]),
\varphi\rangle+\lim_{N\to\infty}\langle T(\psi_N \widetilde{Q}_{k,N}(x,\cdot)),\varphi\rangle\\
&\qquad + \lim_{N\to\infty}\int_\cx\int_\cx[K(w,y)-K(x_0,y)]
[1-\psi_N(y)]Q_k(x,y)\varphi(w)\,d\mu(w)d\mu(y)\\
&=: \mathrm{I}_1+\mathrm{I}_2+\mathrm{I}_3.
\end{align*}

We first estimate $\mathrm{I}_2$. By Lemma \ref{4.9.1}(ii), we find that, for any $N\in\zz_+$,
$\supp \widetilde{Q}_{k,N}(x,\cdot)\subset B(x,C_1\delta^{k-N})$. For any $y\in B(x,C_1\delta^{k-N})$, we have
$$d(y,x_0)\leq A_0[d(y,x)+d(x,x_0)]\leq A_0C_1[\delta^{k-N}+d(x_0,x)],$$
which implies that $\psi_N \widetilde{Q}_{k,N}(x,\cdot)=\widetilde{Q}_{k,N}(x,\cdot)$
and hence, by \eqref{exp-t}, we further obtain
$$
\mathrm{I}_2=\lim_{N\to\infty}\lf\langle T\lf(\widetilde{Q}_{k,N}(x,\cdot)\r),\varphi\r\rangle
=\langle T(Q_k(x,\cdot),\varphi)\rangle.
$$

Now we estimate $\mathrm{I}_1$. For any $N\in\zz_+$ and $x\in\cx$,
\begin{align*}
&\|\psi_N[Q_k(x,\cdot)-\widetilde{Q}_{k,N}(x,\cdot)]\|_{C^\eta(\cx)}\\
&\quad=\sup_{y_1,\ y_2\in\cx,\ y_1\neq y_2}
\frac{|\psi_N(y_1)[Q_k(x,y_1)-\widetilde{Q}_{k,N}(x,y_1)]-\psi_N(y_2)[Q_k(x,y_2)-\widetilde{Q}_{k,N}(x,y_2)]|}
{[d(y_1,y_2)]^\eta}\\
&\quad\leq \sup_{y_1,\ y_2\in\cx,\ y_1\neq y_2}
\frac{|\psi_N(y_1)[Q_k(x,y_1)-\widetilde{Q}_{k,N}(x,y_1)]-\psi_N(y_1)[Q_k(x,y_2)-\widetilde{Q}_{k,N}(x,y_2)]|}
{[d(y_1,y_2)]^\eta}\\
&\quad\qquad +\sup_{y_1,\ y_2\in\cx,\ y_1\neq y_2}
\frac{|\psi_N(y_1)[Q_k(x,y_2)-\widetilde{Q}_{k,N}(x,y_2)]-\psi_N(y_2)[Q_k(x,y_2)-\widetilde{Q}_{k,N}(x,y_2)]|}
{[d(y_1,y_2)]^\eta}\\
&\quad\leq\|\psi_N\|_{L^\infty(\cx)}
\left\|Q_k(x,\cdot)-\widetilde{Q}_{k,N}(x,\cdot)\right\|_{C^\eta(\cx)}
+\|\psi_N\|_{C^\eta(\cx)}
\left\|Q_k(x,\cdot)-\widetilde{Q}_{k,N}(x,\cdot)\right\|_{L^\infty(\cx)},
\end{align*}
which, together with Lemma \ref{4.9.1}(i), implies $\lim_{N\to\infty}\|\psi_N[Q_k(x,\cdot)
-\widetilde{Q}_{k,N}(x,\cdot)]\|_{C^\eta(\cx)}=0$.
From this and the assumption that $T$ is continuous on $C_b^\eta(\cx)$,
we deduce that $\mathrm{I}_1=0$.

Finally we estimate $\mathrm{I}_3$.
Observe that there exists an $N_1\in\zz_+$ such that $r\leq2^{-1}C_1[\delta^{k-N_1}+d(x_0,x)]$.
Then, for any $w\in B(x_0,r)$ and $y\notin B(x_0,A_0C_1[\delta^{k-N_1}+d(x_0,x)])$,
we have
$$d(w,x_0)<r\leq2^{-1}C_1[\delta^{k-N_1}+d(x_0,x)]\leq (2A_0)^{-1}d(x_0,y).$$
From this and Definition \ref{4.11.2}(iii), we deduce that,
for any $N\in\{N_1,N_1+1,\dots\}$ and $x\in\cx$,
\begin{align*}
&\int_\cx\int_\cx|K(w,y)-K(x_0,y)||1-\psi_{N}(y)||Q_k(x,y)||\varphi(w)|\,d\mu(w)\,d\mu(y)\\
&\quad\lesssim\int_{\cx\setminus B(x_0,A_0C_1[\delta^{k-N_1}+d(x_0,x)])}\int_{B(x_0,r)}
\frac{[d(x_0,w)]^\eta}{V(x_0,y)[d(x_0,y)]^\eta}|Q_k(x,y)||\varphi(w)|\,d\mu(w)d\mu(y)\\
&\quad \lesssim \frac{1}{\mu(B(x_0,A_0C_1[\delta^{k-N_1}+d(x_0,x)]))}
\mu(B(x_0,r))\|Q_k(x,\cdot)\|_{L^1(\cx)}\|\varphi\|_{L^\infty(\cx)}.
\end{align*}
By this and the Lebesgue dominated convergence theorem, we conclude that
$\mathrm{I}_3=0$.
From the estimates of $\mathrm{I}_1$, $\mathrm{I}_2$ and $\mathrm{I}_3$,
we deduce that \eqref{ttt} holds true, which  completes the proof of Proposition \ref{tttp}.
\end{proof}

Now we show Theorem \ref{12.15.1}.

\begin{proof}[Proof of Theorem \ref{12.15.1}]
Let $f\in \mathring{\cg}(\eta,\eta)$ and
$\{Q_k\}_{k=-\infty}^\infty$ be an exp-AT{\rm I}. For any $k\in\zz$,
$\alpha\in\ca_k$ and $m\in\{1,\dots,N(k,\alpha)\}$, suppose that
$\ya$ is an arbitrary point in $\qa$. Then, for
any fixed $\beta,\ \gamma\in(0,\eta)$,
by Lemma \ref{crf}, we know that there exists a sequence
$\{\widetilde{Q}_k\}_{k=-\infty}^\infty$ of bounded linear operators on $L^2(\mathcal{X})$ such that
$$f(\cdot) = \sum_{k'=-\infty}^\infty\sum_{\alpha' \in \ca_{k'}}
\sum_{m'=1}^{N(k',\alpha')}\mu\left(\qap\right)Q_{k'}\left(\cdot,\yap\right)\widetilde{Q}_{k'}(f)
\left(\yap\right) \qquad\text{in}\quad \cggi.$$
Let $\{P_k\}_{k=-\infty}^\infty$ be  another exp-AT{\rm I}.
From Remark \ref{exoft}, we deduce that, for any $k\in\zz$ and $x\in\cx$,
$$P_kT(x,\cdot)=[T^\ast(P_k(x,\cdot))]^\ast\in(\cggi)'.$$
Thus, for any $k\in\zz$ and $x\in\cx$,
$$P_kT(f)(x) =\sum_{k'=-\infty}^\infty\sum_{\alpha' \in \ca_{k'}}
\sum_{m'=1}^{N(k',\alpha')}\mu\left(\qap\right)\left(P_kTQ_{k'}\left(\cdot,\yap\right)\right)(x)
\widetilde{Q}_{k'}(f)\left(\yap\right).$$
Now, for any $k',\ k\in\zz$, $\alpha'\in\ca_{k'}$,
$m'\in\{1,\dots,N(k',\alpha')\}$, $x\in\cx$ and $\yap \in \qap$, let
\begin{equation}
I_{k,k'}(x,\yap):=\left(P_kTQ_{k'}\left(\cdot,\yap\right)\right)(x).
\end{equation}
We claim that, for any fixed $\eta'\in(0,\eta)$,
and any $k'\in\zz$, $\alpha'\in\ca_{k'}$,
$m'\in\{1,\dots,N(k',\alpha')\}$ and $x\in\cx$,
\begin{equation}\label{12.16.1}
\left|I_{k,k'}\left(x,\yap\right)\right|\lesssim\delta^{|k-k'|\eta'}\frac{1}{V_{\delta^{k\wedge k'}}(x)+V(x,\yap)}
\left[\frac{\delta^{k\wedge k'}}{\delta^{k\wedge k'}+d(x,\yap)}\right]^{\eta'}.
\end{equation}

To show \eqref{12.16.1}, by symmetry, we only need to consider the case $k\leq k'$.
To this end, by Lemma \ref{4.9.1},
Remark \ref{exoft} and the Lebesgue dominated convergence theorem,
we have
\begin{equation}\label{ptq}
\left(P_kTQ_{k'}\left(\cdot,\yap\right)\right)(x)=\sum_{j=0}^\infty\sum_{j'=0}^\infty
\delta^{j\theta}\delta^{j'\theta}\left(P_{k,j}TQ_{k',j'}\left(\cdot,\yap\right)\right)(x).
\end{equation}
We claim that, for any $k,\ k'\in\zz$, $j,\ j'\in\zz_+$, $\alpha'\in\ca_{k'}$,
$m'\in\{1,\dots,N(k',\alpha')\}$ and $x\in\cx$,
\begin{equation}\label{12.16.2}
\left|\left(P_{k,j}TQ_{k',j'}\left(\cdot,\yap\right)\right)(x)\right|\lesssim\delta^{[(k'-j')-(k-j)]\eta'}
\frac{\delta^{-j'\omega}}{V_{\delta^{k}}(x)+V(x,\yap)}\left[
\frac{\delta^{k-j-j'}}{\delta^k+d(x,\yap)}\right]^{\eta'}.
\end{equation}

Now we show \eqref{12.16.2}. If $d(x,\yap)\leq3A_0^2C_1\delta^{k-j-j'}$
with $C_1$ as in  Lemma \ref{4.8.1}(i),
we then have
$$\delta^k+d(x,\yap)\lesssim \delta^{k-j-j'}+d(x,\yap)\lesssim\delta^{k-j-j'}$$
and
$$\frac{1}{V_{\delta^{k-j}}(x)}\lesssim \delta^{-j'\omega}\frac{1}{V_{\delta^{k-j-j'}}(x)}
\lesssim \delta^{-j'\omega}\frac{1}{V_{\delta^{k}}(x)+V(x,\yap)}.$$
From these, \eqref{sizeqkj}, Definition
\ref{4.11.2}(ii) and Lemma \ref{4.9.1}(ii), we deduce that, for any $\alpha'\in\ca_{k'}$,
$m'\in\{1,\dots,N(k',\alpha')\}$ and $x\in\cx$,
\begin{align*}
\left|\left(P_{k,j}TQ_{k',j'}\left(\cdot,\yap\right)\right)(x)\right|
&=\left|\int_\cx(T^\ast (P_{k,j}(x,\cdot)))^\ast(y)Q_{k',j'}\left(y,\yap\right)\,d\mu(y)\right|\\
&\leq\int_\cx\left|(T^\ast (P_{k,j}(x,\cdot)))^\ast(y)-(T^\ast (P_{k,j}(x,\cdot)))^\ast\left(\yap\right)\right|\\
&\qquad\times\left|Q_{k',j'}\left(y,\yap\right)\right|\,d\mu(y)\notag\\
&\lesssim \int_\cx\|(T^\ast (P_{k,j}(x,\cdot)))^\ast\|_{\dot{C}^{\eta'}(\cx)}\left[d\left(y,\yap\right)\right]^{\eta'}\notag\\
&\qquad\times\frac{1}{V_{\delta^{k'-j'}}(y)+V(y,\yap)}
\left[\frac{\delta^{k'-j'}}{\delta^{k'-j'}+d(y,\yap)}\right]^\eta\,d\mu(y)\notag\\
&\lesssim \delta^{[(k'-j')-(k-j)]\eta'}\frac{1}{V_{\delta^{k-j}}(x)}\notag\\
&\lesssim\delta^{[(k'-j')-(k-j)]\eta'}
\frac{\delta^{-j'\omega}}{V_{\delta^{k}}(x)+V(x,\yap)}\left[\frac{\delta^{k-j-j'}}
{\delta^k+d(x,\yap)}\right]^{\eta'}.
\end{align*}
If $d(x,\yap)>3A_0^2C_1\delta^{k-j-j'}$,
it is easy to see that
$$B(x,C_1\delta^{k-j})\cap B(\yap,C_1\delta^{k'-j'})=\emptyset.$$
Moreover, for any $z\in B(x,C_1\delta^{k-j})$ and $y\in B(\yap,C_1\delta^{k'-j'})$, we have
\begin{align*}
d(z,\yap)&\geq A_0^{-1}[d(x,\yap)-A_0d(x,z)]
\geq A_0^{-1}(3A_0^2C_1\delta^{k-j-j'}-A_0C_1\delta^{k-j})\\
&\geq 2A_0C_1\delta^{k-j-j'}\geq 2A_0d(y,\yap)
\end{align*}
and
$$d(x,\yap)\sim d(x,\yap)-A_0d(x,z)\lesssim d(z,\yap).$$
From these and Definition \ref{4.11.2}(iii), we deduce that
$$\left|K(z,y)-K\left(z,\yap\right)\right|\lesssim \frac{[d(y,\yap)]^\eta}{V(z,\yap)[d(z,\yap)]^\eta }
\lesssim  \frac{\delta^{(k'-j')\eta'}}{V(x,\yap)[d(x,\yap)]^{\eta'}},$$
which, combined  with
$$d(x,\yap)\gtrsim \delta^{k-j-j'}+d(x,\yap)\gtrsim \delta^{k}+d(x,\yap),$$
implies that
\begin{align*}
\left|\left(P_{k,j}TQ_{k',j'}\left(\cdot,\yap\right)\right)(x)\right|&=\left|\int_{\cx}
\int_{\cx}\left[K(z,y)-K\left(z,\yap\right)\right]
P_{k,j}(z,x)Q_{k',j'}\left(y,\yap\right)\,d\mu(y)d\mu(z)\right|\\
&\lesssim \left[\frac{\delta^{k'-j'}}{d(x,\yap)}\right]^{\eta'}\frac{1}{V(x,\yap)}\\
&\lesssim \delta^{[(k'-j')-(k-j)]\eta'}
\frac{\delta^{-j'\omega}}{V_{\delta^{k}}(x)+V(x,\yap)}
\left[\frac{\delta^{k-j-j'}}{\delta^k+d(x,\yap)}\right]^{\eta'}.
\end{align*}
This shows \eqref{12.16.2}. From \eqref{12.16.2} and  \eqref{ptq} with
$\theta:=2(\omega+\eta)$,  we deduce that
\eqref{12.16.1} holds true.

Notice that \eqref{12.16.1} is analogous to \eqref{pp} in the proof of
Lemma \ref{5.12.1}. Then we can use an  argument similar to
that used in the proof of Lemma \ref{5.12.1} to obtain that $\|Tf\|_{\hb}\lesssim \|f\|_{\hb}$ and
$\|Tf\|_{\hf}\lesssim \|f\|_{\hf}$; we omit the details here.
Finally, by Lemma \ref{12.15.2} and a density argument, we complete the
proof of Theorem \ref{12.15.1}.
\end{proof}

\begin{remark}
Let $\cx$ be a regular space.
Suppose that $T$ is a Calder\'on--Zygmund operator
whose distribution kernel only satisfies (i), (ii) and (iii) of Definition \ref{4.11.2}. On one hand,
by \cite[Theorem 5.2]{hs}, we know that the conclusion of Theorem \ref{12.15.1} in this case
holds true for any $s\in(0,\beta\wedge\gamma)$.
On the other hand, if $T$ is a Calder\'on--Zygmund operator and
the distribution kernel of $T^\ast$ only satisfies (i), (ii) and (iii) of Definition \ref{4.11.2},
by \cite[Theorem 5.2]{hs}, we also know that the conclusion of
Theorem \ref{12.15.1} in this case holds true for
any $s\in(-(\beta\wedge\gamma),0)$.
These facts indicate that our assumptions on $T$  in Theorem \ref{12.15.1}
is stronger when $s\in(-(\beta\wedge\gamma),\beta\wedge\gamma)$.
However, our underlying spaces are more general than regular spaces considered  in \cite{hs}.
Moreover, our result also covers
the corresponding one when $\cx$ is an RD-space (see \cite[Theorem 5.23]{hmy06}).
\end{remark}

We now turn to the boundedness of Calder\'on--Zygmund operators on $\ihb$ and $\ihf$.
To this end, let us first introduce
the inhomogeneous Calder\'on--Zygmund operators.

\begin{definition}\label{4.12.1}
Let $\epsilon\in(0,1]$ and $\sigma \in (0,\infty)$. A linear operator $T$,
which is continuous from $C_b^{\epsilon'}(\cx)$ to $(C_b^{\epsilon'}(\cx))'$
for any $\epsilon'\in(0,\epsilon)$, is called an \emph{inhomogeneous Calder\'on--Zygmund operator of order
$(\epsilon,\sigma$)} if $T$ has a distributional kernel $K$,
which is locally integrable away from the diagonal of $\cx\times\cx$,
and satisfies the conditions (i) through (iv) of Definition \ref{4.11.2}
and the following additional conditions:
\begin{enumerate}
\item[{\rm(ii)$'$}] if $\phi$ is a normalized $\epsilon$-bump function associated to a ball with radius $r$,
then there exists a positive constant $C$, independent of $\phi$, such that $\|T\phi\|_{L^\infty(\cx)}\leq C$.
\item[{\rm(iii)$'$}] there exists a positive constant $C$ such that, for any $x,\ y\in\cx$ with $d(x,y)>1$,
$$|K(x,y)|\leq C\frac{1}{V(x,y)}\frac{1}{[d(x,y)]^\sigma}.$$
\end{enumerate}
\end{definition}

We have the following conclusion on the boundedness of inhomogeneous Calder\'on--Zygmund
operators of order $(\eta,\sigma)$.

\begin{theorem}\label{4.12.2}
Let $\beta,\ \gamma\in(0,\eta)$ with $\eta$ as in Definition \ref{10.23.2} and
$s\in(-(\beta\wedge\gamma),\beta\wedge\gamma)$.
\begin{enumerate}
\item[{\rm(i)}] Suppose $p\in(p(s,\bz\wedge\gz),\fz]$,
$q\in(0,\fz]$, $\bz$ and  $\gz$ satisfy \eqref{10.19.3}
and \eqref{4.23.x}, and $T$  is  an inhomogeneous Calder\'on--Zygmund operators of order $(\eta,\sigma)$
with $\sigma\in(\omega(1/p-1)_+,\fz)$, where $p(s,\bz\wedge\gz)$ and $\omega$ are, respectively, as in
\eqref{eq-doub} and \eqref{pseta}. Then $T$ is bounded on $\ihb$ when $\max\{p,q\}<\fz$,
and from $\ihb\cap\cg(\eta,\eta)$ to $\ihb$ when $\max\{p,q\}=\infty$,
where $\ihb$ is viewed as a subspace of $(\icgg)'$;
\item[{\rm(ii)}] Suppose $p\in(p(s,\bz\wedge\gz),\fz)$, $q\in(p(s,\bz\wedge\gz),\fz]$, $\bz$ and $\gz$
satisfy \eqref{10.19.3} and \eqref{4.23.x}, and $T$ is  an inhomogeneous Calder\'on--Zygmund operators
of order $(\eta,\sigma)$ with $\sigma\in(\omega(1/p-1)_+,\fz)$. Then $T$ is bounded on $\ihf$ when $q\neq\fz$,
and from $F^s_{p,\fz}(\cx)\cap\cg(\eta,\eta)$ to $F^s_{p,\fz}(\cx)$,
where $\ihb$ is viewed as a subspace of $(\icgg)'$.
\end{enumerate}
\end{theorem}

To prove Theorem \ref{4.12.2}, we also need the following several technical lemmas.
The following lemma is similar to  Lemma \ref{12.15.2},
and the details of its proof are omitted here.

\begin{lemma}\label{4.12.3}
Let $p,\ q\in(0,\infty]$, $\beta,\ \gamma\in(0,\eta)$ with $\eta$ as in Definition \ref{10.23.2}, and
$s\in(-(\beta\wedge\gamma),\beta\wedge\gamma)$ satisfy \eqref{10.19.3} and \eqref{4.23.x}. Let $\ihb$ and $\ihf$ be as in Definition \ref{ih}. Then $\cg(\eta,\eta)$ is dense in $\ihb$ when $p\in(p(s,\beta\wedge\gamma),\infty)$,
and $q\in(0,\infty)$, and dense in $\ihf$ when $p,\ q\in(p(s,\beta\wedge\gamma),\infty)$,
where $p(s,\beta\wedge\gamma)$ is as in \eqref{pseta}.
\end{lemma}

\begin{lemma}\label{4.25.1}
Let $\theta\in(0,\infty)$, $\eta\in(0,1)$ be as in Definition \ref{10.23.2},
and $\{Q_k\}_{k=0}^\infty$ an {\rm exp-IATI}.
Then, for any $k\in\zz_+$, there exists a sequence
$\{Q_{k,j}\}_{j=0}^\infty$ of functions on $\cx\times\cx$ such that
\begin{enumerate}
\item[{\rm(i)}] for any $x\in\cx$
$$Q_k(x,\cdot)=\sum_{j=0}^\infty \delta^{j\theta}Q_{k,j}(x,\cdot)$$
both pointwisely and in $C^{\eta'}(\cx)$ with any given $\eta'\in(0,\eta]$;
\item[{\rm(ii)}] for any $k,\ j\in\zz_+$ and $\eta'$ as in (i), $Q_{k,j}(x,\cdot)$
is an $\eta'$-bump function associated to the ball $B(x,C_1\delta^{k-j})$,
with its $\eta'$-bump constant independent of $k$ and $j$,
where $C_1$ is as in Lemma \ref{4.8.1}(i). Moreover, for any $k\in\nn$, $j\in\zz_+$ and $x\in\cx$,
\begin{equation}\label{ivan-qkj}
\int_\cx Q_{k,j}(x,y)\,d\mu(y)=0.
\end{equation}
\end{enumerate}
\end{lemma}

\begin{proof}
When $k\in\nn$, we use an argument similar to
that used in the proof of Lemma \ref{4.9.1} and
obtain the desired conclusion; we omit the details.
It suffices to show this lemma when $k=0$.
Let $\Phi_0$ be as in Lemma \ref{4.8.1}.
For any $x,\ y\in\cx$, let
$$A_{0,0}(x,y):=\Phi_0(x,y)Q_0(x,y)$$
and, for any $j\in\nn$,
$$A_{0,j}(x,y):=[\Phi_{-j}(x,y)-\Phi_{-j+1}(x,y)]Q_0(x,y).$$
Moreover, for any $j\in\zz_+$ and
$x,\ y\in\cx$, let
$$Q_{0,j}(x,y):=\delta^{-j\theta}A_{0,j}(x,y).$$
By Lemma \ref{4.8.1}(i),  it is easy to see that, for any $j\in\zz_+$ and $x\in\cx$,
$$\supp (Q_{0,j}(x,\cdot)) \subset B(x, C_1\delta^{-j}).$$
Then, using an argument similar to
that used in the estimations of \eqref{4.9.2} and \eqref{4.9.4},
we conclude that, for any $x,\ y\in\cx$,
\begin{equation}\label{4.25.2}
\sum_{j=0}^\infty \delta^{j\theta}Q_{0,j}(x,y)=\sum_{j=0}^{\infty} A_{0,j}(x,y)=Q_0(x,y)
\end{equation}
and, for any $j\in\zz_+$ and $x,\ y\in\cx$,
\begin{equation*}
A_{0,j}(x,y)\lesssim \delta^{j\theta}\frac{1}{V_{\delta^{-j}}(x)}.
\end{equation*}
From this and  an  argument similar to that
used in the estimations of \eqref{4.10.3}
and \eqref{4.25.4}, we deduce that, for
any $\eta'\in(0,\eta]$, $j\in\zz_+$ and $x,\ y_1,\ y_2\in\cx$,
\begin{equation*}
|A_{0,j}(x,y_1)-A_{0,j}(x,y_2)|\lesssim\left[\frac{d(y_1,y_2)}{\delta^{-j}}\right]^{\eta'}
\delta^{j\theta}\frac{1}{V_{\delta^{-j}}(x)},
\end{equation*}
which implies that, for
any $\eta'\in(0,\eta]$,  $j\in\zz_+$ and $x\in\cx$,
\begin{equation}\label{qoceta}
\|Q_{0,j}(x,\cdot)\|_{\dot{C}^{\eta'}(\cx)}=\delta^{-j\theta}\|A_{0,j}(x,\cdot)\|_{\dot{C}^{\eta'}(\cx)}
\lesssim \delta^{\eta'j}\frac{1}{V_{\delta^{-j}}(x)}.
\end{equation}
Finally, using an argument similar to that used in the estimations of \eqref{ceta} and \eqref{linf},
we obtain
$$
\lim_{N\to \infty} \left\|Q_0(x,\cdot)-\sum_{j=0}^N
\delta^{j\theta}Q_{0,j}(x,\cdot)\right\|_{C^{\eta'}(\cx)}=0,
$$
which completes the proof of Lemma
\ref{4.25.1}.
\end{proof}

\begin{remark}\label{expofit}
Let $\theta,\ \sigma\in(0,\infty)$ and $\beta,\ \gamma\in(0,\eta)$ with $\eta$  as in Definition \ref{10.23.2}.
Let $T$ be an inhomogeneous Calder\'on--Zygmund operator of order $(\eta,\sigma)$.
Using an argument similar to that used in Remark \ref{exoft},
we conclude that, for any exp-IATI $\{Q_k\}_{k=0}^\infty$ and $k\in\zz_+$,
$Q_kT:=\lim_{N\to\infty}(T^\ast(\sum_{j=0}^N\delta^{j\theta}Q_{k,j}(x,\cdot)))^\ast\in(\icgg)'$.
Moreover, by the Lebesgue dominated convergence theorem,
we have, for any $k,\ k'\in\zz_+$,
$$Q_kTQ_{k'}=\sum_{j=0}^\infty\sum_{j'=0}^\infty \delta^{j\theta}\delta^{j'\theta}Q_{k,j}TQ_{k',j'}.$$

We also point out that, in \cite[Proposition 2.25]{hmy08}, for any inhomogeneous Calder\'on--Zygmund
operator $T$, an expansion $\widetilde{T}$ from $C^\eta(\cx)$ to $(C_b^\eta(\cx))'$ with $\cx$ being an
RD-space was introduced. Using an argument similar to that used in the proof of Proposition \ref{tttp},
we deduce that, for any $k\in\zz_+$, $\varphi\in C_b^\eta(\cx)$ and $x\in\cx$,
$$\left\langle \widetilde{T}(Q_k(x,\cdot)),\varphi\right\rangle
=\left\langle \lim_{N\to\infty}T\left(\sum_{j=0}^N\delta^{j\theta}Q_{k,j}(x,\cdot)\right),\varphi\right\rangle;$$
we omit the details here.
\end{remark}

Now we prove  Theorem \ref{4.12.2}.

\begin{proof}[Proof of Theorem \ref{4.12.2}]
Let $\eta$ be as in Definition \ref{10.23.2},
$f\in\cg(\eta,\eta)$ and  $\{Q_k\}_{k=0}^\infty$ be an
exp-{\rm I}AT{\rm I}. Using Remark \ref{expofit}, Lemma \ref{icrf}
and an argument similar to that used in the proof of Theorem
\ref{12.15.1}, we know that, for any $k\in\zz_+$ and $x\in\cx$,
\begin{align*}
Q_k(Tf)(x)&=\sum_{\alpha' \in \ca_0}\sum_{m'=1}^{N(0,\alpha')}
\int_{\qop}(Q_kTQ_0(\cdot,y))(x)\,d\mu(y)\widetilde{Q}^{0,m'}_{\alpha',1}(f)\\
&\qquad+\sum_{k'=1}^{N'}\sum_{\alpha' \in \ca_{k'}}\sum_{m'=1}^{N(k',\alpha')}
\mu\left(\qap\right)\left(Q_kTQ_{k'}\left(\cdot,\yap\right)\right)(x)\widetilde{Q}^{k',m'}_{\alpha',1}(f)\\
& \qquad+\sum_{k'=N'+1}^\infty\sum_{\alpha' \in \ca_{k'}}
\sum_{m'=1}^{N(k',\alpha')}\mu\left(\qap\right)\left(Q_kTQ_{k'}\left(\cdot,\yap\right)\right)(x)
\widetilde{Q}_{k'}f\left(\yap\right),
\end{align*}
where all the notation is as in Lemma \ref{icrf}.

For any $k',\ k\in\zz_+$, $\alpha'\in\ca_{k'}$,  $m'\in\{1,\dots,N(k',\alpha')\}$, $x\in\cx$ and $y \in \qap$, let
$$J_{k,k'}(x,y):=(Q_kTQ_{k'}(\cdot,y))(x).$$
We now claim that, for any $k,\ k'\in\{0,\dots, N\}$ and $x,\ y\in\cx$,
\begin{equation}\label{4.12.4}
|J_{k,k'}(x,y)|\lesssim\left[\frac{1}{1+d(x,y)}\right]^\sigma\frac{1}{V_1(x)+V(x,y)};
\end{equation}
for any fixed $\eta'\in(0,\eta)$, and any $k\in \{N+1,N+2,\dots\}$ and
$k'\in\{0,\dots,N\}$, $k\in\{0,\dots, N\}$ and $k'\in \{N+1,N+2,\dots\}$,
or $k,\ k'\in \{N+1,N+2,\dots\}$, and any $x,\ y\in\cx$,
\begin{equation}\label{4.12.6}
|J_{k,k'}(x,y)|\lesssim\delta^{|k-k'|\eta'}\frac{1}{V_{\delta^{k\wedge k'}}(x)
+V(x,y)}\left[\frac{\delta^{k\wedge k'}}{\delta^{k\wedge k'}+d(x,y)}\right]^{\eta'}.
\end{equation}
Obviously, using an argument similar
to that used in the estimation of  \eqref{12.16.1},
we obtain \eqref{4.12.6} when $k,\ k'\in \{N+1,N+2,\dots\}$.
To show \eqref{4.12.4} and \eqref{4.12.6} when $k\in \{N+1,N+2,\dots\}$ and
$k'\in\{0,\dots,N\}$, or $k\in\{0,\dots, N\}$ and $k'\in \{N+1,N+2,\dots\}$,
by symmetry, we may assume $k\leq N$ and $k\leq k'$.
From Remark \ref{expofit}, we deduce that, for any $x,\ y\in\cx$,
\begin{equation}\label{qitq}
J_{k,k'}(x,y)=\sum_{j=0}^\infty\sum_{j'=0}^\infty\delta^{j\theta}\delta^{j'\theta}(Q_{k,j}TQ_{k',j'}(\cdot,y))(x).
\end{equation}
We claim that, for any $k,\ k'\in\{0,\dots, N\}$, $j,\ j'\in\zz_+$ and $x,\ y\in\cx$,
\begin{equation}\label{lessn}
|(Q_{k,j}TQ_{k',j'}(\cdot,y))(x)|\lesssim\delta^{-j(\omega+\sigma)}
\delta^{-j'(\omega+\sigma)}\left[\frac{1}{1+d(x,y)}\right]^\sigma\frac{1}{V_1(x)+V(x,y)}
\end{equation}
and, for any $k\in\{0,\dots, N\}$, $k'\in \{N+1,N+2,\dots\}$, $j,\ j'\in\zz_+$,
$\alpha'\in\ca_{k'}$, $m'\in\{1,N(k',\alpha')\}$ and $x\in\cx$,
\begin{equation}\label{kn}
\left|\left(Q_{k,j}TQ_{k',j'}\left(\cdot,\yap\right)\right)(x)\right|\lesssim\delta^{[(k'-j')-(k-j)]\eta'}
\frac{\delta^{-j'\omega}}{V_{\delta^{k}}(x)+V(x,\yap)}\left[\frac{\delta^{k-j-j'}}
{\delta^k+d(x,\yap)}\right]^{\eta'}.
\end{equation}
Here and thereafter, for any $k\in\zz_+$, $\{Q_{k,j}\}_{j=0}^\fz$ is as in Lemma \ref{4.25.1} with
$\theta\in(0,\fz)$ determined later.
Observe that the estimation of \eqref{12.16.2} only needs the cancellation of $Q_{k',j'}$.
Then, using an argument similar to that used in
the estimation of \eqref{12.16.2}, we conclude that \eqref{kn} holds true.

To prove \eqref{lessn}, we consider two cases.

{\it Case 1) $d(x,y)<3A_0^2C_1\delta^{-N}\delta^{k-j'-j}$ with $C_1$ as in Lemma \ref{4.8.1}(i)}.
By \eqref{4.10.1}, \eqref{qoceta} and  Definition \ref{4.12.1}(ii)$'$,
we know that, for any $k\in\{0,\dots,N\}$ and $j\in\zz_+$,
$$\|(T^\ast (Q_{k,j}(x,\cdot)))^\ast\|_{L^\infty(\cx)}
\lesssim \frac{1}{V_{\delta^{k-j}}(x)}\lesssim \frac{1}{V_1(x)},$$
which implies that
$$|(Q_{k,j}TQ_{k',j'}(\cdot,y))(x)|=\left|\int_\cx(T^\ast (Q_{k,j}(x,\cdot)))^\ast
Q_{k',j'}(z,y)\,d\mu(z)\right|\lesssim \frac{1}{V_1(x)}.$$
Notice that, when $d(x,y)<3A_0^2\delta^{-N}\delta^{k-j'-j}$,
then $1+d(x,y)\lesssim \delta^{-j'-j}$. From this, we deduce that
$$\frac{1}{V_1(x)}\lesssim\delta^{-j\omega}\delta^{-j'\omega}
\frac{1}{V_{\delta^{-j'-j}}(x)}\lesssim\delta^{-j\omega}\delta^{-j'\omega}\frac{1}{V_1(x)+V(x,y)}.$$
By this, we further conclude that, for any $k,\ k'\in\{0,\dots, N\}$ and $j,\ j'\in\zz_+$,
\begin{equation}\label{xyclose}
|(Q_{k,j}TQ_{k',j'}(\cdot,y))(x)|\lesssim\delta^{-j\omega}\delta^{-j'\omega}\frac{1}{V_1(x)+V(x,y)}
\left[\frac{\delta^{-j'-j}}{1+d(x,y)}\right]^\sigma.
\end{equation}

{\it Case 2)} $d(x,y)\geq3A_0^2C_1\delta^{-N}\delta^{k-j'-j}$. In this case, it is easy to see that
$$B(x,C_1\delta^{k-j})\cap B(y,C_1\delta^{k'-j'})=\emptyset.$$
Moreover, for any $z\in B(y,C_1\delta^{k'-j'})$ and $w\in B(x,C_1\delta^{k-j})$,
\begin{align}\label{4.13.3}
d(w,y)&\geq A_0^{-1}[d(x,y)-A_0d(w,x)]\geq(2A_0)^{-1}d(x,y)
\end{align}
and
\begin{equation}\label{4.13.4}
d(w,z)\geq A_0[d(w,y)-A_0d(y,z)]\geq(2A_0)^{-1}d(w,y).
\end{equation}
From \eqref{4.13.3} and \eqref{4.13.4}, we deduce that
$$V(w,z)\gtrsim V(w,y)\sim V(y,w)\gtrsim V(y,x)\sim V(x,y).$$
Besides, we also find that
\begin{align*}
d(w,z)&\geq A_0^{-1}[d(x,z)-A_0d(x,w)]\\
&\geq A_0^{-1}\{A_0^{-1}[d(x,y)-A_0d(y,z)]-A_0d(x,w)\}\\
&>A_0^{-1}[3A_0C_1\delta^{-N}\delta^{k-j'-j}-\delta^{k'-j'}-A_0\delta^{k-j}]
\geq 1.
\end{align*}
Thus, by Definition \ref{4.12.1}(iii)$'$, we have
\begin{align*}
|(Q_{k,j}TQ_{k',j'}(\cdot,y))(x)|&\lesssim \int_{\cx}\int_{\cx}|Q_{k,j}(w,x)|
\frac{1}{V(w,z)}\frac{1}{[d(w,z)]^\sigma}|Q_{k',j'}(z,y)|\,d\mu(w)d\mu(z)\\
&\lesssim \frac{1}{V(x,y)}\frac{1}{[d(x,y)]^\sigma}\\
&\lesssim\delta^{-j\omega}\delta^{-j'\omega}\frac{1}{V_1(x)+V(x,y)}
\left[\frac{\delta^{-j'-j}}{1+d(x,y)}\right]^\sigma,
\end{align*}
which, together with \eqref{xyclose}, implies
that \eqref{lessn} holds true. From \eqref{qitq} and  \eqref{lessn}
via choosing $\theta > \omega+\sigma$, we deduce that \eqref{4.12.4} holds true.
Thus, \eqref{4.12.4} and \eqref{4.12.6} hold true.

Using these estimates and an argument similar to that used in the
estimation of  \eqref{6.9.1}, we obtain
$\|Tf\|_{\ihb}\lesssim\|f\|_{\ihb}$ and $\|Tf\|_{\ihf}\lesssim\|f\|_{\ihf}$.
Finally, by a density argument, we complete  the proof of Theorem \ref{4.12.2}.
\end{proof}

At the end of this section, we consider the boundedness
of Calder\'on--Zygmund operators on $\hfi$ and $\ihfi$.

\begin{theorem}\label{cz}
Let $\beta,\ \gamma\in(0,\eta)$ with $\eta$ as in Definition \ref{10.23.2},
$s\in (-(\beta\wedge\gamma),\beta\wedge\gamma)$
and $q\in(p(s,\beta\wedge\gamma),\infty]$ with $p(s,\beta\wedge\gamma)$
as in \eqref{pseta}. Let $\hfi$ be as in Definition \ref{hfi}
and $T$  a Calder\'on--Zygmund operator of  order $\eta$
as in Definition \ref{4.11.2}. Then $T$ is bounded
from $\hfi\cap\mathring{\cg}(\eta,\eta)$ to $\hfi$.
\end{theorem}

\begin{proof}
By \eqref{12.16.1}, which is analogous to \eqref{pp} in the proof of
Lemma \ref{5.12.1}, we can use an argument similar to that used in the
proof of Lemma \ref{11.14.2} to obtain
$\|Tf\|_{\hfi}\lesssim \|f\|_{\hfi}$; we
omit the details. This finishes the proof of Theorem \ref{cz}.
\end{proof}

\begin{theorem}\label{czi}
Let $\beta,\ \gamma\in(0,\eta)$ with $\eta$ as in Definition \ref{10.23.2},
$s\in (-(\beta\wedge\gamma),\beta\wedge\gamma)$
and $q\in(p(s,\beta\wedge\gamma),\infty]$ with $p(s,\beta\wedge\gamma)$
as in \eqref{pseta}. Let $\ihfi$ be as in Definition \ref{ihfi}
and $T$  a Calder\'on--Zygmund operator of order $(\eta,\sigma)$
with $\sigma\in(0,\infty)$  as
in Definition \ref{4.12.1}. Then $T$ is bounded  from $\ihfi\cap\mathring{\cg}(\eta,\eta)$ to $\ihfi$.
\end{theorem}

\begin{proof}
Using \eqref{4.12.4}, \eqref{4.12.6} and an  argument similar
to that used in the proof of Lemma \ref{12.3.2}, we obtain
$\|Tf\|_{\ihfi}\lesssim \|f\|_{\ihfi}$; we omit the details.
This finishes the proof of Theorem \ref{czi}.
\end{proof}

\bigskip

\noindent Fan Wang, Ziyi He and Dachun Yang (Corresponding author)

\medskip

\noindent Laboratory of Mathematics and Complex Systems (Ministry of Education of China),
School of Mathematical Sciences, Beijing Normal University, Beijing 100875, People's Republic of China

\smallskip

\noindent{\it E-mails:} \texttt{fanwang@mail.bnu.edu.cn} (F. Wang)

\noindent\phantom{{\it E-mails:} }\texttt{ziyihe@mail.bnu.edu.cn} (Z. He)

\noindent\phantom{{\it E-mails:} }\texttt{dcyang@bnu.edu.cn} (D. Yang)

\bigskip

\noindent Yongsheng Han

\medskip

\noindent Department of Mathematics, Auburn University, Auburn, AL 36849-5310, USA

\smallskip

\noindent{\it E-mail:} \texttt{hanyong@auburn.edu}


\begin{thebibliography}{99}

\bibitem{am15} R. Alvarado and M. Mitrea, Hardy Spaces on Ahlfors-Regular Quasi Metric Spaces. A Sharp Theory,
Lecture Notes in Mathematics 2142, Springer, Cham, 2015.

\vspace{-0.3cm}

\bibitem{ah13} P. Auscher and T. Hyt\"onen,
Orthonormal bases of regular wavelets in spaces of homogeneous type,
Appl. Comput. Harmon. Anal. 34 (2013), 266--296.

\vspace{-0.3cm}

\bibitem{ah13-2} P. Auscher and T. Hyt\"onen, Addendum to orthonormal bases of regular wavelets
in spaces of homogeneous type [Appl. Comput. Harmon. Anal. 34(2) (2013) 266--296],
Appl. Comput. Harmon. Anal. 39 (2015), 568--569.

\vspace{-0.3cm}

\bibitem{b47} S. N. Bern\v{s}te\v{\i}n, On properties of homogeneous functional classes, Dokl. Acad. Nauk SSSR
(N. S.) 57 (1947), 111--114.

\vspace{-0.3cm}

\bibitem{b59} O. V. Besov, On some families of functional spaces.
Imbedding and extension theorems, Dokl.
Acad. Nauk SSSR 126 (1959), 1163--1165.

\vspace{-0.3cm}

\bibitem{b61} O. V. Besov, Investigation of a class of function spaces in connection with imbedding and
extension theorems, Trudy Mat. Inst. Steklov. 60 (1961), 42--81.

\vspace{-0.3cm}

\bibitem{bin} O. V. Besov, V. P. Il'in and S. M. Nikol'ski\u{\i},
Integralnye Predstavleniya Funktsii i Teoremy Vlozheniya, Second edition, Fizmatlit ``Nauka'', Moscow, 1996.

\vspace{-0.3cm}

\bibitem{b11} H. Brezis, Functional Analysis, Sobolev Spaces and Partial Differential Equations, Universitext, Springer, New York, 2011.

\vspace{-0.3cm}

\bibitem{bbd18} H.-Q. Bui, T. A. Bui and X. T. Duong, Weighted Besov and Triebel--Lizorkin spaces associated
to operators and applications, Forum Math. Sigma 8 (2020), e11, 95 pp.

\vspace{-0.3cm}

\bibitem{bd20} T. A. Bui and X. T. Duong, Sharp weighted estimates for square functions associated
to operators on spaces of homogeneous type, J. Geom. Anal. 30 (2020), 874--900.

\vspace{-0.3cm}

\bibitem{bdk} T. A. Bui, X. T. Duong and L. D. Ky, Hardy spaces associated to critical functions and
applications to $T1$ theorems, J. Fourier Anal. Appl. 26 (2020), Article number 27, 67 pp.

\vspace{-0.3cm}

\bibitem{bdl} T. A. Bui, X. T. Duong and F. K. Ly, Maximal function characterizations for new local Hardy type
spaces on spaces of homogeneous type, Trans. Amer. Math. Soc. 370 (2018), 7229--7292.

\vspace{-0.3cm}

\bibitem{bdl20} T. A. Bui, X. T. Duong and F. K. Ly,
Maximal function characterizations for Hardy spaces on spaces of homogeneous type
with finite measure and applications, J. Funct. Anal. 278 (2020), 108423. 55 pp.

\vspace{-0.3cm}

\bibitem{cw71} R. R. Coifman and G. Weiss, Analyse Harmonique
Non-Commutative sur Certains Espaces Homog\`enes,
(French) \'Etude de Certaines Int\'egrales Singuli\`eres, Lecture Notes in Mathematics 242, Berlin--New York,
Springer-Verlag, 1971.

\vspace{-0.3cm}

\bibitem{cw77} R. R. Coifman and G. Weiss, Extensions of Hardy spaces and their use in analysis, Bull. Amer.
Math. Soc. 83 (1977), 569--645.

\vspace{-0.3cm}

\bibitem{dy12} G. Dafni and H. Yue, Some characterizations of
local $\bmo$ and $h^1$ on metric measure spaces, Anal. Math. Phys. 2 (2012), 285--318.

\vspace{-0.3cm}

\bibitem{dh09} D. Deng and Y. Han, Harmonic Analysis on Spaces
of Homogeneous Type. With a Preface by Yves
Meyer, Lecture Notes in Mathematics 1966, Berlin, Springer-Verlag, 2009.

\vspace{-0.3cm}

\bibitem{dhy04} D. Deng, Y. Han and D. Yang, Inhomogeneous Plancherel--P\^olya inequalities on spaces of homogeneous
type and their applications, Commun. Contemp. Math. 6 (2004), 221--243.

\vspace{-0.3cm}

\bibitem{dy03} X. Y. Duong and L. Yan, Hardy spaces of spaces
of homogeneous type, Proc. Amer. Math. Soc. 131
(2003), 3181--3189.

\vspace{-0.3cm}

\bibitem{fj90} M. Frazier and B. Jawerth, A discrete transform and decompositions of distribution space,
J. Funct. Anal. 93 (1990), 34--170.

\vspace{-0.3cm}

\bibitem{fmt19} X. Fu, T. Ma and D. Yang, Real-variable characterizations of Musielak--Orlicz Hardy spaces on
spaces of homogeneous type, Ann. Acad. Sci. Fenn. Math. 45 (2020), 343--410.

\vspace{-0.3cm}

\bibitem{fy18} X. Fu and D. Yang, Wavelet characterizations of the atomic Hardy space
$H^1$ on spaces of homogeneous type, Appl. Comput. Harmon. Anal. 44 (2018), 1--37.

\vspace{-0.3cm}

\bibitem{fyl17} X. Fu, D. Yang and Y. Liang, Products of functions in $\BMO(\cx)$ and $H^1_\at(\cx)$ via
wavelets over spaces of homogeneous type, J. Fourier Anal. Appl. 23 (2017), 919--990.

\vspace{-0.3cm}

\bibitem{glmy14} L. Grafakos, L. Liu, D. Maldonado and D. Yang, Multilinear analysis on metric spaces,
Dissertationes Math. 497 (2014), 1--121.

\vspace{-0.3cm}

\bibitem{gly08} L. Grafakos, L. Liu and D. Yang, Maximal function characterizations of Hardy spaces on
RD-spaces and their applications, Sci. China Ser. A 51 (2008), 2253--2284.

\vspace{-0.3cm}

\bibitem{gly09} L. Grafakos, L. Liu and D. Yang, Vector-valued singular integrals and maximal functions on
spaces of homogeneous type, Math. Scand. 104 (2009), 296--310.

\vspace{-0.3cm}

\bibitem{gly091} L. Grafakos, L. Liu and D. Yang, Radial maximal function characterizations for Hardy
spaces on RD-spaces, Bull. Soc. Math. France 137 (2009), 225--251.


\vspace{-0.3cm}

\bibitem{hhhlp20} Ya. Han, Yo. Han, Z. He, J. Li and C. Pereyra, Geometric characteriztions of embedding theorems --- for
Sobolev, Besov, and Triebel--Lizorkin spaces on spaces of homogeneous type --- via orthonormal wavelets,
J. Geom. Anal. (to appear).

\vspace{-0.3cm}

\bibitem{hhl16} Ya. Han, Yo. Han and J. Li, Criterion of the boundedness of singular integrals on spaces of
homogeneous type, J. Funct. Anal. 271 (2016), 3423--3464.

\vspace{-0.3cm}

\bibitem{hhl18} Ya. Han, Yo. Han and J. Li, Geometry and Hardy spaces on spaces of homogeneous type in the
sense of Coifman and Weiss, Sci. China Math. 60 (2017), 2199--2218.

\vspace{-0.3cm}

\bibitem{h97} Y. Han, Inhomogeneous Calder\'on reproducing formula on spaces of homogeneous type,
J. Geom. Anal. 7 (1997), 259--284.

\vspace{-0.3cm}

\bibitem{hlw}  Y. Han, J. Li and L. A. Ward, Hardy space theory on
spaces of homogeneous type via orthonormal wavelet bases, Appl. Comput.
Harmon. Anal. 45 (2018), 120--169.

\vspace{-0.3cm}

\bibitem{hly99} Y. Han, S. Lu and D. Yang, Inhomogeneous Besov and Triebel--Lizorkin spaces on spaces of homogeneous
type, Approx. Theory Appl. (N.S.) 15 (3) (1999), 37--65.

\vspace{-0.3cm}

\bibitem{hly99b} Y. Han, S. Lu and D. Yang, Inhomogeneous Triebel--Lizorkin spaces on spaces of homogeneous type,
Math. Sci. Res. Hot-Line 3 (9) (1999), 1--29.

\vspace{-0.3cm}

\bibitem{hly01} Y. Han, S. Lu and D. Yang, Inhomogeneous discrete Calder\'on reproducing formulas for spaces of
homogeneous type, J. Fourier Anal. Appl. 7 (2001), 571--600.

\vspace{-0.3cm}

\bibitem{hmy06} Y. Han, D. M\"uller and D. Yang, Littlewood--Paley characterizations for Hardy spaces on
spaces of homogeneous type, Math. Nachr. 279 (2006), 1505--1537.

\vspace{-0.3cm}

\bibitem{hmy08} Y. Han, D. M\"uller and D. Yang, A theory of Besov and Triebel--Lizorkin spaces on metric
measure spaces modeled on Carnot--Carath\'eodory spaces, Abstr.
Appl. Anal. 2008, Art. ID 893409, 1--250.

\vspace{-0.3cm}

\bibitem{hs} Y. S. Han and E. T. Sawyer, Littlewood--Paley theory on spaces of homogeneous type and the
classical function spaces, Mem. Amer. Math. Soc. 110 (1994), no. 530, 1--126.

\vspace{-0.3cm}

\bibitem{hy02} Y. Han and D. Yang, New characterizations and applications of inhomogeneous Besov and
Triebel--Lizorkin spaces on homogeneous type spaces and fractals, Dissertationes Math. (Rozprawy Mat.) 403
(2002), 1--102.

\vspace{-0.3cm}

\bibitem{hy03} Y. Han and D. Yang,
Some new spaces of Besov and Triebel--Lizorkin type on homogeneous spaces,
Studia Math. 156 (2003), 67--97.

\vspace{-0.3cm}

\bibitem{hhllyy} Z. He, Y. Han, J. Li,
L. Liu, D. Yang and W. Yuan, A complete real-variable theory of Hardy spaces on spaces of homogeneous
type, J. Fourier Anal. Appl. 25 (2019), 2197--2267.

\vspace{-0.3cm}

\bibitem{hlyy} Z. He, L. Liu, D. Yang and W. Yuan, New Calder\'on reproducing formulae with exponential decay
on spaces of homogeneous type, Sci. China Math. 62 (2019), 283--350.

\vspace{-0.3cm}

\bibitem{hwy20} Z. He, F. Wang and D. Yang, Wavelet characterizations of Besov and Triebel--Lizorkin spaces on spaces
of homogeneous type and their applications, Submitted.

\vspace{-0.3cm}

\bibitem{hyy19} Z. He, D. Yang and W. Yuan, Real-variable characterizations of local Hardy spaces on spaces
of homogeneous type, Math. Nachr. (to appear) or arXiv:1908.01911.

\vspace{-0.3cm}

\bibitem{hyz09} G. Hu, D. Yang and Y. Zhou, Boundedness of singular integrals in Hardy spaces on spaces of
homogeneous type, Taiwanese J. Math. 13 (2009), 91--135.

\vspace{-0.3cm}

\bibitem{hk} T. Hyt\"onen and A. Kairema, Systems of dyadic cubes in a doubling metric space, Colloq. Math.
126 (2012), 1--33.

\vspace{-0.3cm}

\bibitem{ht14} T. Hyt\"onen and O. Tapiola,
Almost Lipschitz-continuous wavelets in
metric spaces via a new randomization of dyadic cubes,
J. Approx. Theory 185 (2014), 12--30.

\vspace{-0.3cm}

\bibitem{jn} P. Jaming and F. Negreira,  A Plancherel--Polya inequality in Besov spaces on spaces of
homogeneous type, J. Geom. Anal. 29 (2019), 1571--1582.

\vspace{-0.3cm}

\bibitem{j94} A. Jonsson, Besov spaces on closed subsets of $\rn$, Trans. Amer. Math. Soc. 341 (1994),
355--370.


\vspace{-0.3cm}

\bibitem{kyz10} P. Koskela, D. Yang and Y. Zhou, A characterization of
Haj\l asz--Sobolev and Triebel--Lizorkin spaces via grand Littlewood--Paley functions,
J. Funct. Anal. 258 (2010), 2637--2661.

\vspace{-0.3cm}

\bibitem{kyz11} P. Koskela, D. Yang and Y. Zhou, Pointwise characterizations of Besov
and Triebel--Lizorkin spaces and quasiconformal mappings, Adv. Math. 226 (2011), 3579--3621.


\vspace{-0.3cm}

\bibitem{ky15} L. D. Ky, On the product of functions in $BMO$ and $H^1$ over spaces of homogeneous type, J.
Math. Anal. Appl. 425 (2015), 807--817.

\vspace{-0.3cm}

\bibitem{lcfy17} L. Liu, D.-C. Chang, X. Fu and D. Yang, Endpoint boundedness of commutators on spaces of
homogeneous type, Appl. Anal. 96 (2017), 2408--2433.

\vspace{-0.3cm}

\bibitem{lcfy18} L. Liu, D.-C. Chang, X. Fu and D. Yang, Endpoint estimates of linear commutators on
Hardy spaces over spaces of homogeneous type, Math. Methods Appl. Sci. 41 (2018), 5951--5984.

\vspace{-0.3cm}

\bibitem{lyy18} L. Liu, D. Yang and W. Yuan, Bilinear decompositions for products of Hardy and Lipschitz
spaces on spaces of homogeneous type, Dissertationes Math. 533 (2018), 1--93.

\vspace{-0.3cm}

\bibitem{l72} P. I. Lizorkin, Operators connected with fractional differentiation, and classes of
differentiable functions, (Russian) Studies in the Theory of Differentiable Functions of Several Variables
and Its Applications {\rm I}V, Trudy Mat. Inst. Steklov 117 (1972), 212--243.

\vspace{-0.3cm}

\bibitem{l74} P. I. Lizorkin, Properties of functions of the spaces $\Lambda_{p,\theta}^r$, Trudy Mat. Inst.
Steklov. 131 (1974), 158--181.

\vspace{-0.3cm}

\bibitem{ms79} R. A. Mac\'{i}as and C. Segovia, Lipschitz functions on spaces of homogeneous type,
Adv. Math. 33 (1979), 257--270.

\vspace{-0.3cm}

\bibitem{ms79b} R. A. Mac\'{i}as and C. Segovia, A decomposition into atoms of distributions on spaces of
homogeneous type, Adv. Math. 33 (1979), 271--309.

\vspace{-0.3cm}

\bibitem{m92} Y. Meyer, Wavelets and Operators, Cambridge Studies
in Advanced Mathematics 37, Cambridge University Press, Cambridge, 1992.

\vspace{-0.3cm}

\bibitem{my97} Y. Meyer and R. R. Coifman, Wavelets, Calder\'{o}n--Zygmund and Multilinear Operators,
Cambridge Studies in Advanced Mathematics 48, Cambridge University Press, Cambridge, 1997.

\vspace{-0.3cm}

\bibitem{my09} D. M\"uller and D. Yang, A difference characterization of Besov and
Triebel--Lizorkin spaces on RD-spaces, Forum Math. 21 (2009), 259--298.

\vspace{-0.3cm}

\bibitem{ns} A. Nagel and E. M. Stein, On the product theory of singular integrals, Rev. Mat. Iberoamericana
20 (2004), 531--561.

\vspace{-0.3cm}

\bibitem{n08} E. Nakai, A generalization of Hardy spaces $H^p$ by using atoms,
Acta Math. Sin. (Engl. Ser.) 24 (2008), 1243--1268.


\vspace{-0.3cm}

\bibitem{n17} E. Nakai, Singular and fractional integral operators on preduals of Campanato
spaces with variable growth condition, Sci. China Math. 60 (2017), 2219--2240.


\vspace{-0.3cm}

\bibitem{ny97} E. Nakai and K. Yabuta, Pointwise multipliers for functions of weighted bounded mean
oscillation on spaces of homogeneous type, Math. Japon. 46 (1997), 15--28.

\vspace{-0.3cm}

\bibitem{n51} S. M. Nikol'ski\u{\i}, Inequalities for entire analytic functions of finite order and their
application to the theory of differentiable functions of several variables, Trudy Mat. Inst. Steklov 38
(1951), 244--278.

\vspace{-0.3cm}

\bibitem{p73} J. Peetre, Remarques sur les espaces de Besov. Le cas $0 < p < 1$, (French) C. R. Acad. Sci.
Paris S\'er. A-B 277 (1973), 947--950.

\vspace{-0.3cm}

\bibitem{p75} J. Peetre, On spaces of Triebel--Lizorkin type, Ark. Mat. 13 (1975), 123--130.

\vspace{-0.3cm}

\bibitem{p76} J. Peetre, New Thoughts on Besov Spaces, Duke University Mathematics Series, Duke University
Press, Durham, 1976.

\vspace{-0.3cm}

\bibitem{Sa18} Y. Sawano, Theory of Besov Spaces, Developments in Mathematics 56, Springer, Singapore, 2018.

\vspace{-0.3cm}

\bibitem{sy18} L. Song and L. Yan, Maximal function characterizations for Hardy spaces associated with
nonnegative self-adjoint operators on spaces of homogeneous type, J. Evol. Equ. 18 (2018), 221--243.

\vspace{-0.3cm}

\bibitem{s93} E. M. Stein, Harmonic Analysis: Real-Variable Methods, Orthogonality, and Oscillatory Integrals,
Princeton Mathematical Series 43, Monographs in Harmonic Analysis III, Princeton University Press, Princeton, NJ, 1993.

\vspace{-0.3cm}

\bibitem{t64} M. H. Taibleson, On the theory of Lipschitz spaces of distributions on Euclidean $n$-space. I.
Principal properties, J. Math. Mech. 13 (1964), 407--479.

\vspace{-0.3cm}

\bibitem{t65} M. H. Taibleson, On the theory of Lipschitz spaces of distributions on Euclidean $n$-space. II.
Translation invariant operators, duality, and interpolation, {\rm J}. Math. Mech. 14 (1965), 821--839.

\vspace{-0.3cm}

\bibitem{t66} M. H. Taibleson, On the theory of Lipschitz spaces of distributions on Euclidean $n$-space. III.
Smoothness and integrability of Fourier transforms, smoothness of convolution kernels,
{\rm J}. Math. Mech. 15 (1966), 973--981.

\vspace{-0.3cm}

\bibitem{t73}H. Triebel, Spaces of distributions of Besov type on Euclidean $n$-space. Duality, interpolation,
Ark. Mat. 11 (1973), 13--64.

\vspace{-0.3cm}

\bibitem{t83} H. Triebel, Theory of Function Spaces. I, Monographs in Mathematics 78, Birkh\"auser Verlag,
Basel, 1983.

\vspace{-0.3cm}

\bibitem{t92} H. Triebel, Theory of Function Spaces. II, Monographs in Mathematics 84, Birkh\"auser Verlag,
Basel, 1992.

\vspace{-0.3cm}

\bibitem{t06} H. Triebel, Theory of Function Spaces. III, Monographs in Mathematics 100, Birkh\"auser
Verlag, Basel, 2006.

\vspace{-0.3cm}

\bibitem{w88} H. Wallin, New and old function spaces, in: Function Spaces and Applications (Lund, 1986), 99--114,
Lecture Notes in Mathematics 1302, Springer, Berlin, 1988.

\vspace{-0.3cm}

\bibitem{y031} D. Yang, Frame characterizations of Besov and Triebel--Lizorkin spaces on spaces of homogeneous type
and their applications, Georgian Math. J. 9 (2002), 567--590.

\vspace{-0.3cm}

\bibitem{y032} D. Yang, Besov spaces on spaces of homogeneous type and fractals, Studia Math. 156 (2003),
15--30.

\vspace{-0.3cm}

\bibitem{y033} D. Yang, $T1$ theorems on Besov and Triebel--Lizorkin spaces on spaces of homogeneous type and their
applications, Z. Anal. Anwendungen 22 (2003), 53--72.

\vspace{-0.3cm}

\bibitem{y034} D. Yang, Riesz potentials in Besov and Triebel--Lizorkin spaces over spaces of homogeneous type, Potential
Anal. 19 (2003), 193--210.

\vspace{-0.3cm}

\bibitem{y035} D. Yang, Embedding theorems of Besov and Triebel--Lizorkin spaces on spaces of homogeneous type, Sci. China
Ser. A 46 (2003), 187--199.

\vspace{-0.3cm}

\bibitem{y041} D. Yang, Localization principle of Triebel--Lizorkin spaces on spaces of homogeneous type, Rev. Mat. Complut.
17 (2004), 229--249.

\vspace{-0.3cm}

\bibitem{y042} D. Yang, Real interpolations for Besov and Triebel--Lizorkin spaces on spaces of homogeneous type,
Math. Nachr. 273 (2004), 96--113.

\vspace{-0.3cm}

\bibitem{y051}D. Yang, Some new inhomogeneous Triebel--Lizorkin spaces on metric measure spaces and their various
characterizations, Studia Math. 167 (2005), 63--98.

\vspace{-0.3cm}

\bibitem{y052} D. Yang, Some new Triebel--Lizorkin spaces on spaces of homogeneous type and their frame characterizations,
Sci. China Ser. A 48 (2005), 12--39.

\vspace{-0.3cm}

\bibitem{yy19}S. Yang and D. Yang, Atomic and maximal function characterizations of Musielak--Orlicz--Hardy spaces
associated to non-negative self-adjoint operators on spaces of homogeneous type, Collect. Math. 70 (2019), 197--246.

\vspace{-0.3cm}

\bibitem{yz08} D. Yang and Y.  Zhou, Boundedness of sublinear operators in Hardy spaces on RD-spaces via atoms,
J. Math. Anal. Appl. 339 (2008), 622--635.

\vspace{-0.3cm}

\bibitem{yz10} D. Yang and Y. Zhou, Radial maximal function characterizations of Hardy spaces on RD-spaces and
their applications, Math. Ann. 346 (2010), 307--333.

\vspace{-0.3cm}

\bibitem{yz11} D. Yang and Y. Zhou, New properties of Besov and Triebel--Lizorkin spaces on RD-spaces,
Manuscripta Math 134 (2011), 59--90.

\vspace{-0.3cm}

\bibitem{ysy} W. Yuan, W. Sickel and D. Yang, Morrey and Campanato Meet Besov, Lizorkin and Triebel,
Lecture Notes in Mathematics 2005, Springer-Verlag, Berlin, 2010.

\vspace{-0.3cm}

\bibitem{zyh20} X. Zhou, D. Yang and Z. He,
A real-variable theory of Hardy--Lorentz spaces on spaces of homogeneous type, Submitted.

\vspace{-0.3cm}

\bibitem{zsy16} C. Zhuo, Y. Sawano and D. Yang, Hardy spaces with variable exponents
on RD-spaces and applications, Dissertationes Math. 520 (2016), 1--74.

\vspace{-0.3cm}

\bibitem{z45} A. Zygmund, Smooth functions, Duke Math. J. 12 (1945), 47--76.

\end{thebibliography}
\end{document}